\documentclass[12pt]{amsart}

\usepackage[text={400pt,660pt},centering]{geometry}

\usepackage{color}
\usepackage{esint,amssymb}
\usepackage{graphicx}
\usepackage{MnSymbol}
\usepackage{mathtools}
\usepackage[colorlinks=true, pdfstartview=FitV, linkcolor=blue, citecolor=blue, urlcolor=blue,pagebackref=false]{hyperref}
\usepackage{microtype}

\usepackage{bm}
\usepackage{scalerel} 
\usepackage{dsfont}

\definecolor{darkgreen}{rgb}{0,0.5,0}
\definecolor{darkblue}{rgb}{0,0,0.7}
\definecolor{darkred}{rgb}{0.9,0.1,0.1}

\makeatletter

\def\chaptermark#1{}

\def\chapter{%
  \if@openright\cleardoublepage\else\clearpage\fi
  \thispagestyle{plain}\global\@topnum\z@
  \@afterindenttrue \secdef\@chapter\@schapter}

\def\@chapter[#1]#2{\refstepcounter{chapter}%
  \ifnum\c@secnumdepth<\z@ \let\@secnumber\@empty
  \else \let\@secnumber\thechapter \fi
  \typeout{\chaptername\space\@secnumber}%
  \def\@toclevel{0}%
  \ifx\chaptername\appendixname \@tocwriteb\tocappendix{chapter}{#2}%
  \else \@tocwriteb\tocchapter{chapter}{#2}\fi
  \chaptermark{#1}%
  \addtocontents{lof}{\protect\addvspace{10\p@}}%
  \addtocontents{lot}{\protect\addvspace{10\p@}}%
  \@makechapterhead{#2}\@afterheading}
\def\@schapter#1{\typeout{#1}%
  \let\@secnumber\@empty
  \def\@toclevel{0}%
  \ifx\chaptername\appendixname \@tocwriteb\tocappendix{chapter}{#1}%
  \else \@tocwriteb\tocchapter{chapter}{#1}\fi
  \chaptermark{#1}%
  \addtocontents{lof}{\protect\addvspace{10\p@}}%
  \addtocontents{lot}{\protect\addvspace{10\p@}}%
  \@makeschapterhead{#1}\@afterheading}
\newcommand\chaptername{Part}

\def\@makechapterhead#1{\global\topskip 7.5pc\relax
  \begingroup
  \fontsize{\@xivpt}{18}\bfseries\centering
    \ifnum\c@secnumdepth>\m@ne
      \leavevmode \hskip-\leftskip
      \rlap{\vbox to\z@{\vss
          \centerline{\normalsize\mdseries
              \uppercase\@xp{
              }
              }
          \vskip 3pc}}\hskip\leftskip\fi
     \fontsize{\@xivpt}{14}\chaptername \enspace \thechapter. \enspace #1\par \endgroup
  \skip@34\p@ \advance\skip@-\normalbaselineskip
  \vskip\skip@ }
\def\@makeschapterhead#1{\global\topskip 7.5pc\relax
  \begingroup
  \fontsize{\@xivpt}{18}\bfseries\centering
  #1\par \endgroup
  \skip@34\p@ \advance\skip@-\normalbaselineskip
  \vskip\skip@ }
\def\appendix{\par
  \c@chapter\z@ \c@section\z@
  \let\chaptername\appendixname
  \def\thechapter{\@Alph\c@chapter}}

\newcounter{chapter}

\newif\if@openright

\makeatother
\renewcommand{\thechapter}{\Roman{chapter}}

\newtheorem{theorem}{Theorem}
\newtheorem{proposition}{Proposition}
\newtheorem{lemma}[proposition]{Lemma}
\newtheorem{corollary}[proposition]{Corollary}

\theoremstyle{definition}
\newtheorem{remark}[proposition]{Remark}

\newtheorem{definition}[proposition]{Definition}

\newcommand{\cref}[1]{Corollary~\ref{c.#1}}

\numberwithin{equation}{section}
\numberwithin{proposition}{section}
\numberwithin{figure}{section}

\newcommand{\Z}{\mathbb{Z}}
\newcommand{\N}{\mathbb{N}}
\newcommand{\R}{\mathbb{R}}

\newcommand{\A}{\mathcal{A}}
\newcommand{\Ahom}{\overline{\A}}

\newcommand{\E}{\mathbb{E}}
\renewcommand{\P}{\mathbb{P}}
\newcommand{\F}{\mathcal{F}}
\newcommand{\Zd}{\mathbb{Z}^d}
\newcommand{\Rd}{{\mathbb{R}^d}}

\newcommand{\ep}{\varepsilon}
\newcommand{\eps}{\varepsilon}

\renewcommand{\a}{\mathbf{a}}
\renewcommand{\b}{\mathbf{b}}
\newcommand{\ahom}{{\overbracket[1pt][-1pt]{\a}}}  

\renewcommand{\subset}{\subseteq}

\newcommand{\cu}{{\scaleobj{1.2}{\square}}}
\newcommand{\cut}{{\scaleobj{1.2}{\boxbox}}}
\newcommand{\cc}{\mathbf{c}}
\newcommand{\ga}{\gamma}

\renewcommand{\fint}{\strokedint}

\newcommand{\Ll}{\left}
\newcommand{\Rr}{\right}

\DeclareMathOperator{\dist}{dist}

\DeclareMathOperator*{\osc}{osc}

\DeclareMathOperator{\supp}{supp}
\DeclareMathOperator{\spn}{span}

\newcommand{\X}{\mathcal{X}}  
\newcommand{\Y}{\mathcal{Y}} 
\newcommand{\W}{\mathbf{W}}
\newcommand{\V}{\mathbf{V}}

\renewcommand{\bar}{\overline}
\renewcommand{\tilde}{\widetilde}

\newcommand{\indc}{\mathds{1}}

\newcommand{\mcl}{\mathcal}
\newcommand{\msf}{\mathsf}
\newcommand{\al}{\alpha}
\newcommand{\be}{\beta}
\newcommand{\CC}{\mathbf{C}}

\renewcommand{\O}{\mathcal{O}}

\newcommand{\de}{\delta}
\newcommand{\td}{\widetilde}
\newcommand{\1}{\mathds{1}}
\newcommand{\Phid}{\Phi^{(\delta)}}
\newcommand{\Jd}{{I^{(\delta)}}}
\newcommand{\tJd}{{\tilde{I}^{(\delta)}}}

\newcommand{\Add}{\msf{Add}}
\newcommand{\Fluc}{\msf{Fluc}}
\newcommand{\Dual}{\msf{Dual}}
\newcommand{\Loc}{\msf{Loc}}

\newcommand{\Phidt}{\Phi^{(\tilde \delta)}}

\newcommand{\J}{I}

\makeatletter
\def\@tocline#1#2#3#4#5#6#7{\relax
  \ifnum #1>\c@tocdepth 
  \else
    \par \addpenalty\@secpenalty\addvspace{#2}%
    \begingroup \hyphenpenalty\@M
    \@ifempty{#4}{%
      \@tempdima\csname r@tocindent\number#1\endcsname\relax
    }{%
      \@tempdima#4\relax
    }%
    \parindent\z@ \leftskip#3\relax \advance\leftskip\@tempdima\relax
    \rightskip\@pnumwidth plus4em \parfillskip-\@pnumwidth
    #5\leavevmode\hskip-\@tempdima
      \ifcase #1
       \or\or \hskip 1em \or \hskip 2em \else \hskip 3em \fi%
      #6\nobreak\relax
    \dotfill\hbox to\@pnumwidth{\@tocpagenum{#7}}\par
    \nobreak
    \endgroup
  \fi}
\makeatother

\begin{document}

\title{The additive structure of elliptic homogenization}

\vspace{-1cm}

\begin{abstract}
One of the principal difficulties in stochastic homogenization is transferring  quantitative ergodic information from the coefficients to the solutions, since the latter are nonlocal functions of the former. In this paper, we address this problem in a new way, in the context of linear elliptic equations in divergence form, by showing that certain quantities associated to the energy density of solutions are essentially additive. As a result, we are able to prove quantitative estimates on the weak convergence of the gradients, fluxes and energy densities of the first-order correctors (under blow-down) which are optimal in both scaling and stochastic integrability. The proof of the additivity is a bootstrap argument, completing the program initiated in~\cite{AKM}: using the regularity theory recently developed for stochastic homogenization, we reduce the error in additivity as we pass to larger and larger length scales. In the second part of the paper, we use the additivity to derive central limit theorems for these quantities by a reduction to sums of independent random variables. In particular, we prove that the first-order correctors converge, in the large-scale limit, to a variant of the Gaussian free field. 
\end{abstract}

\author[S. Armstrong]{Scott Armstrong}
\address[S. Armstrong]{Universit\'e Paris-Dauphine, PSL Research University, CNRS, UMR [7534], CEREMADE, Paris, France}
\curraddr{Courant Institute of Mathematical Sciences, New York University, 251 Mercer St., New York 10012}
\email{scotta@cims.nyu.edu}

\author[T. Kuusi]{Tuomo Kuusi}
\address[T. Kuusi]{Department of Mathematics and Systems Analysis, Aalto University, Finland}
 \email{tuomo.kuusi@aalto.fi}

\author[J.-C. Mourrat]{Jean-Christophe Mourrat}
\address[J.-C. Mourrat]{Ecole normale sup\'erieure de Lyon, CNRS, Lyon, France}
\email{jean-christophe.mourrat@ens-lyon.fr}

\date{submitted February 1, 2016; revised October 6, 2016}

\maketitle

\vspace{-0.847cm}

\tableofcontents

\chapter{Additive structure}
\label{part.one}

\section{Introduction}

\subsection{Summary of main results}
We prove optimal quantitative estimates for the first-order correctors in stochastic homogenization for linear, uniformly elliptic equations of the form
\begin{equation} \label{e.pde}
-\nabla \cdot \left( \a(x) \nabla u \right) = 0 \quad \mbox{in} \ U \subseteq\Rd.
\end{equation}
The first-order correctors are the unique (up to an additive constant) solutions $\phi_e \in H^1_{\mathrm{loc}} (\Rd)$, for each  vector $e\in\Rd$, of the equation
\begin{equation}
-\nabla \cdot\left( \a(x) \left( e + \nabla \phi_e \right) \right) = 0 \quad \mbox{in} \ \Rd
\end{equation}
such that $\nabla\phi_e$ is a $\Zd$-stationary, mean-zero gradient field. Here the coefficients $\a(x)$ are assumed to be a $\Zd$-stationary random field, valued in the set of real, symmetric $d$-by-$d$ matrices with eigenvalues belonging to the interval $[1,\Lambda]$, and are sampled by a probability measure $\P$ which satisfies a unit range of dependence (see below in Section~\ref{ss.assumptions} for the precise assumptions). 

\smallskip

Obtaining estimates on the  first-order correctors is a fundamental objective in the quantitative theory of elliptic homogenization, for reasons including the following: (i) they represent the first term in the two-scale asymptotic expansion for a solution in the length scale~$\ep$, (ii) the homogenized coefficients $\ahom$ can be defined in terms of the expectation of their flux,
\begin{equation*} \label{}
\ahom e := \E \left[ \int_{[0,1]^d} \a(x)\left(e+\nabla \phi_e(x)\right)\,dx \right],
\end{equation*}
(iii) estimates on the correctors imply estimates on the error in homogenization (e.g., for the Dirichlet problem), and
(iv) in view of the regularity theory~\cite{AS,GNO3}, we know that arbitrary solutions of~\eqref{e.pde} can be approximated by functions of the form $x\mapsto e\cdot x + \phi_e(x)$ in the same way that harmonic functions can be approximated by affine functions (see Proposition~\ref{p.regularity}, below). 

\smallskip

The following theorem is a summary of our main results concerning estimates of the first-order correctors. Before stating it, we fix some notation. We denote by $\ahom$ the homogenized coefficients. The heat kernel for~$\ahom$ is 
\begin{equation}
\label{e.heatkernel}
\Phi(x,t):=(4\pi t)^{-\frac d2} (\det \ahom)^{\,-\frac12} \exp\left( - \,\frac{x\cdot \ahom^{\,-1} x}{4t}\right).
\end{equation}
We use the heat kernel for~$\ahom$ throughout the paper as a convenient density against which to measure spatial averages of certain random fields. Note that, for~$z \in \Rd$ and~$r > 0$,~$\Phi_{z,r} := \Phi(\, \cdot\, - z,r^2)$ has unit mass and a length scale of~$r$. Throughout the paper, we use the notation
$$
\int_{\Phi_{z,r}} f \, := \int_{\Rd} f(x) \Phi(x-z,r^2) \, dx,
$$
and replace $\Phi_{0,r}$ by $\Phi_r$ to lighten the notation.
If $X$ is a random variable and $s,\theta\in(0,\infty)$, then we define the statement 
\begin{equation*}
X \leq \O_s(\theta)
\end{equation*}
to mean that 
\begin{equation}
\label{e.def.Os}
\E \Ll[ \exp\left(\left(\theta^{-1} X_+\right)^s\right) \Rr] \le 2,
\end{equation}
where $X_+ := X \vee 0 = \max\{ X,0\}$.

\begin{theorem}[Optimal estimates for first-order correctors]
\label{t.firstcorrectors}
For each $s<2$, there exists $C(s,d,\Lambda)<\infty$ such that, for every $r \geq 1$ and $e\in\partial B_1$,
\begin{equation}
\label{e.gradient}
\left| \int_{\Phi_r} \nabla \phi_e \right| \leq \O_s\left( Cr^{-\frac d2} \right),
\end{equation}
\begin{equation}
\label{e.flux}
\left| \int_{\Phi_r} \a (e +  \nabla \phi_e) - \ahom e \right| \leq \O_s\left( Cr^{-\frac d2} \right),
\end{equation}
and
\begin{equation}
\label{e.energydensity}
\left| \int_{\Phi_r} \frac12 \left( e+ \nabla \phi_e\right) \cdot \a\left( e+ \nabla \phi_e\right)  - \frac12 e\cdot \ahom e \right| \\
 \leq \O_s\left( Cr^{-\frac d2} \right).
\end{equation}
Moreover, there exist $\ep(d,\Lambda)>0$ and, for every $s<2$, a constant $C(s,d,\Lambda)<\infty$ such that, for each $r\geq2$ and $e\in\partial B_1$,
\begin{equation}
\label{e.phiosc}
\left( \fint_{B_r} \left| \phi_e - \left( \phi_e \right)_{B_r} \right|^2 \right)^{\frac12} \leq 
\left\{ 
\begin{aligned}
& \O_s\left(C\log^{\frac12} r\right) & \mbox{if} & \ d=2,\\
& \O_{2+\ep}(C) & \mbox{if} & \ d>2.
\end{aligned}
\right. 
\end{equation}
In particular, if $d>2$, then $\phi_e$ exists as a $\Zd$-stationary random field. 
\end{theorem}

Theorem~\ref{t.firstcorrectors} gives a CLT scaling of spatial averages of the gradient, flux, and energy density of the first-order correctors. These estimates are optimal in the scaling ($r^{-\frac d2}$) and ``almost" optimal in terms of stochastic integrability (the condition $s<2$ cannot be improved past $s=2$). The final estimate on the sublinear growth (and boundedness, in $d>2$) of the correctors is also optimal, in every dimension, both in terms of the scaling and stochastic integrability. These are the first estimates for these quantities at the critical scale under the finite range of dependence assumption (regardless of stochastic integrability) as well as the first estimates to be (``almost") optimal in terms of stochastic integrability at the critical scale (under any assumption). As will become clear from the proof, the stochastic integrability of each of the estimates~\eqref{e.gradient},~\eqref{e.flux} and~\eqref{e.energydensity} may be improved to beyond Gaussian-type bounds ($s>2$) if we replace the critical scaling $r^{-\frac d2}$ with $r^{-\alpha}$ for any subcritical $\alpha< \frac d2$. Note that obtaining estimates at the critical scaling is necessary if there is any hope to characterize the fluctuations (e.g., by proving central limit theorems for these quantities). 

\smallskip

We mention that, as we were writing this paper, we became aware of a new and very interesting work of Gloria and Otto~\cite{GO5} which contains some similar, but weaker, results compared to Theorem~\ref{t.firstcorrectors}. Under the same assumption of finite range of dependence, they obtain sub-optimal versions of~\eqref{e.gradient},~\eqref{e.flux} and~\eqref{e.phiosc}. Namely, they prove~\eqref{e.gradient} and~\eqref{e.flux} with $\O_2(r^{-\alpha})$ replacing the right side of these estimates for sub-critical exponents $\alpha < \frac d2$, and they obtain~\eqref{e.phiosc} in $d=2$ with a scaling of $r^\ep$ for arbitrary $\ep>0$, and in dimensions $d>2$ for the integrability exponent $s=2$. While the technical details of their arguments are different from ours, the two approaches share, on a high level, a philosophy first outlined in~\cite{AKM} of using the regularity theory introduced in~\cite{AS} to accelerate the convergence of spatial averages of gradients, fluxes and energy densities of solutions.\footnote{Several months after this paper was submitted and posted to arXiv and before it was accepted, Gloria and Otto completed a substantial revision~\cite{GO6} of~\cite{GO5} in which they prove Theorem~\ref{t.firstcorrectors} as well as Theorem~\ref{t.main1}. Their analysis is based on a quantity they call the ``homogenization commutator'' which is closely related to the quantity $J$ considered here.}

\smallskip

The proof we give of Theorem~\ref{t.firstcorrectors} originates in the ideas from our previous papers~\cite{AS,AM,AKM} and completes the program initiated there. One of the main difficulties in understanding the statistical behavior of solutions to equations with random coefficients is to overcome the fact that solutions are nonlocal, nonlinear functions of the coefficients. One of the main themes of~\cite{AS,AM,AKM} is that energy-type quantities are much better behaved. Our point of view is that one should study these quantities first, and then derive properties of solutions as consequences. 

\smallskip

The energy quantity we focus on is denoted by $J$ and is defined below in~\eqref{e.Jintro}. We accelerate the rate of convergence of $J$ to its homogenized limit by a bootstrap (i.e., renormalization) argument, which relies on the higher regularity theory developed in~\cite{AS,GNO3}. Without employing abstract concentration inequalities, the bootstrap actually shows that these energy quantities are \emph{additive} between scales and that the energy densities of their minimizers are \emph{local}, up to small errors. The finite range of dependence condition then enters in the simplest possible way, telling us that the energy quantities are essentially sums of i.i.d.\ random variables, and thus concentration of measure becomes easy (and optimal). As we iterate the bootstrap argument (or equivalently, as we pass to larger and larger length scales), the additivity of the energy density  improves until it finally achieves the optimal scaling after a finite number of iterations. 

\smallskip

The energy quantity central to our study is defined, in the simplest case, for each $z\in\Rd$, $r\geq 1$ and $p,q\in\Rd$, by
\begin{equation}
\label{e.Jintro}
J(z,r,p,q) 
:= \sup_{u \in \A_1} \int_{\Phi_{z,r}}  \left( -\frac 12\nabla u \cdot \a\nabla u -p \cdot \a \nabla u + \nabla u \cdot \ahom q \right).
\end{equation}
Here $\A_1$ denotes the set of solutions of the equation which have sub-quadratic growth at infinity,
\begin{equation*} \label{}
\A_1:= \Big\{ u\in H^1_{\mathrm{loc}}(\Rd)\,:\, -\nabla \cdot\left( \a\nabla u \right) = 0 \ \mbox{in} \ \Rd,  \ \ \limsup_{r\to \infty} \frac{1}{r^{4}} \fint_{B_r} \left| u(x) \right|^2\,dx =0 \Big\}.
\end{equation*}
In fact, $\A_1$ is a $(d+1)$-dimensional (random) vector space spanned by the constant functions and those of the form $x\mapsto x\cdot e + \phi_e(x)$ (see Proposition~\ref{p.regularity}). The quantity~$J$ is a variant of the subadditive and superadditive quantities~$\mu$ and~$\nu$ which lie at the heart of the analysis of~\cite{AS}, and is identical to the quantity considered in~\cite{AKM} if the heat kernel is replaced by the characteristic function of a cube and $\A_1$ is replaced by the set of all solutions (cf.~\cite[Lemma 3.1]{AKM}). 

\smallskip

Theorem~\ref{t.firstcorrectors} is a simple consequence of the following result concerning the additivity of $J$. 

\smallskip

\begin{theorem}[Additive structure of $J$]
\label{t.main1}
For every $s < 1$, there exists a constant $C(s,d,\Lambda) < \infty$ such that the following statements hold.
\begin{enumerate}

\item[(i)] \emph{Additivity}. For every $R > r \ge 1$ and $p, q \in B_1$,
\begin{equation*}
\Ll| J(0,R,p,q) 
- \int_{\Phi_{\sqrt{R^2 - r^2}}} J(\, \cdot \,,r,p,q) \Rr| \le \O_s\Ll(C  r^{-d}\Rr).
\end{equation*}

\smallskip

\item[(ii)] \emph{Control of the expectation}. For every $r \ge 1$ and $p,q \in B_1$,
\begin{equation*}
\Ll| \E\left[J(0,r,p,q)\right] - \frac12 (q-p)\cdot \ahom  (q-p) \Rr|
\leq C r^{-d}.
\end{equation*}

\smallskip

\item[(iii)] \emph{CLT scaling of the fluctuations}. For every $r \ge 1$ and $p,q \in B_1$,
\begin{equation*}
\big|J (0,r,p,q)  - \E\left[ J(0,r,p,q)\right]\big| \le \O_{2s}\left(C r^{-\frac d 2}\right).
\end{equation*}

\smallskip

\item[(iv)] \emph{Localization}. For every $\delta,\eps > 0$, there exist $C(\delta,\eps,s,d,\Lambda)<\infty$ and, for every $r\geq 1$ and $p,q\in B_1$, an $\F(B_{r^{1+\delta}})$-measurable random variable $J^{(\delta)}(0,r,p,q)$ such that, for every $\gamma \in \left(0,\frac d{2s}\wedge \left( \frac d2(1+\delta)+\delta\right) -\ep\right]$,
\begin{equation*} 
\left| J(0,r,p,q) - J^{(\delta)}(0,r,p,q)\right| \leq \O_{2s}\left( Cr^{-\gamma}\right)
\end{equation*}
and, for every $\gamma \in \left(0, 2-\ep\right]$,
\begin{equation*} 
\left| J(0,r,p,q) - J^{(\delta)}(0,r,p,q)\right| \leq \O_1\left( Cr^{-\gamma}\right).
\end{equation*}
\end{enumerate}
\end{theorem}

Each of the estimates of Theorem~\ref{t.main1} is optimal in both scaling and stochastic integrability, with the exception of the localization statement in~(iv). While any exponent larger than $\frac d2$ is satisfactory in the sense that it shows the localization error is of strictly lower order compared to the CLT scaling of the quantities in (i) and (iii), we expect the second estimate to hold for the exponent $\gamma = d$. Note that the second localization estimate is almost optimal in $d=2$, but is of no use if $d\ge 4$ since the first localization estimate becomes stronger. 

\smallskip

Theorem~\ref{t.main1} contains a great deal more information than Theorem~\ref{t.firstcorrectors}. For instance, the additivity of~$J$ allows us to prove a central limit theorem for~$J$ itself (and, as a result, each of the quantities in the first three estimates of Theorem~\ref{t.firstcorrectors}) in a relatively simple way, by mimicking the usual CLT argument for sums of i.i.d.\ random variables (rather than resorting to the more complicated and more commonly used machinery developed for nonlinear functions of i.i.d.\ random variables). Since we have an identity linking the gradient of $J$ to the spatial averages of the gradient and flux its maximizers (see~\eqref{e.gradient-J}, below), one may expect that a CLT for $J$ would imply a CLT for the spatially averaged gradient of its maximizers. This is explained in complete detail in the second part of this paper, where we derive the scaling limit of the first-order correctors by demonstrating their convergence to a variant of the Gaussian free field.

\smallskip

Finally, we comment on our hypotheses. Since the methods in this paper are of more interest than the results we can state for any particular model, we have simplified the assumptions for clarity and readability. For instance, symmetry of the coefficients can be removed in a straightforward way (e.g., by using the techniques of~\cite{AM}), and the arguments can be adapted to the case of uniformly elliptic systems, as opposed to scalar equations. The assumption that $\a(x)$ satisfy a finite range of dependence can be replaced by other mixing conditions and our arguments will yield quantitative bounds appropriate to the particular mixing assumption. Indeed, one of the advantages of our methods compared to previous ones is that we can treat essentially any mixing condition and produce optimal quantitative estimates. 

\subsection{Background and comparison to previous works}
The \emph{qualitative} theory of stochastic homogenization for linear elliptic equations in divergence form was completed in the early 1980s by Papanicolaou and Varadhan~\cite{PV1}, Kozlov~\cite{K1}, Yurinskii~\cite{Y1} and, later, using variational methods, by Dal Maso and Modica~\cite{DM1,DM2}. Each of these results state roughly that, $\P$-almost surely, solutions of
\begin{equation*}
-\nabla \cdot \left( \a(x) \nabla u \right) = 0 \quad \mbox{in} \ U 
\end{equation*}
converge in $L^2(U)$ (with $L^2$ suitably normalized relatively to the size of $U$) as the domain $U$ becomes large to those of a deterministic, constant-coefficient equation
\begin{equation*} \label{}
-\nabla \cdot \left( \ahom \nabla u \right) = 0. 
\end{equation*}
The proof of convergence in each of these works is based on the application of the ergodic theorem, and therefore applies under the sole assumption that~$\P$ is ergodic with respect to translations of the coefficients -- a much more general assumption than the finite range of dependence we assume here. For the same reason, these arguments do not give quantitative information concerning the speed of homogenization. 

\smallskip

Developing a \emph{quantitative} theory of stochastic homogenization poses a greater challenge compared to the qualitative theory, due to the difficulty mentioned above of understanding the dependence of the solutions on the coefficients. It is only very recently that satisfactory progress has been made in overcoming this basic obstacle, and by now there are essentially two alternative programs. The first has its origins in an unpublished paper of Naddaf and Spencer~\cite{NS}, and is based on probabilistic machinery more commonly used in statistical physics~\cite{NS2}, namely concentration inequalities, such as spectral gap or logarithmic Sobolev inequalities, which provide a way to quantitatively measure the dependence of the solutions on the coefficients. This approach has been developed extensively by Gloria, Otto and their collaborators~\cite{GO1,GO2,GO3,GNO,GNO2,MO}, who proved optimal quantitative bounds on the scaling of the first-order correctors (including their sublinear growth and spatial averages of their energy density). In particular, they were the first to obtain estimates for the correctors at the critical scalings, albeit with suboptimal stochastic integrability (typically finite moment bounds) and with somewhat restrictive ergodic assumptions. 
Later, central limit theorems for the spatial averages of the gradients and the energy densities of the correctors were obtained using these techniques~\cite{N,MoO,MN,GN}. 

\smallskip

These important and influential results were the first to give a complete quantitative picture of the behavior of the first-order correctors on \emph{any} stochastic model, and have inspired a huge amount of subsequent research. However, there are two downsides to an approach based on concentration inequalities. The first is that the theory does not apply to general coefficient fields, even under the  strongest and most natural ergodic assumptions (such as finite range of dependence). This is because concentration inequalities are only available for probability measures having a special structure, such as that of an underlying product space (like the random checkerboard and Poisson point cloud models). The second is that a reliance on concentration inequalities makes it more difficult to obtain estimates which are optimal in \emph{stochastic integrability}. Indeed, as of this writing, the best estimates obtained by this method are exponential in stochastic integrability, but not Gaussian (see~\cite{GNO3}). Before the regularity theory was introduced in~\cite{AS}, the best estimates available (see \cite{GNO,MO}) were finite moment bounds. 

\smallskip

The second program is the one which began in the works~\cite{AS,AM,AKM} and is completed here. Compared to the work of Gloria, Otto and their collaborators (and with the exception of the very recent work~\cite{GO5} mentioned above), the fundamental problem of how one transfers quantitative ergodic information from the coefficients to the solutions is handled by a different mechanism, as described in the previous subsection: the use of concentration inequalities is replaced by a bootstrap argument, driven by the regularity theory, revealing an additive structure of the energy densities.

\subsection{Outline of the paper}
The paper is split into two parts. The first one focuses on the proofs of Theorems~\ref{t.firstcorrectors} and~\ref{t.main1}, which give the size of the fluctuations of $J$ and the first-order correctors. The second part ``goes to next-order'' by characterizing their scaling limits. 

\smallskip

In the next section, we state the precise assumptions and fix some notation used throughout the paper. Section~\ref{s.regularity} contains a summary of the regularity theory, stated in complete generality. In Section~\ref{s.additivestructure}, we introduce a higher-order version of the energy quantity $J$ and give its basic properties. Section~\ref{s.bootstrap} contains an outline of the bootstrap argument and thus a roadmap of the rest of Part~\ref{part.one}. The main ingredients in the bootstrap argument are proved in Sections~\ref{s.basecase} (the base case of the induction),~\ref{s.fluc-subopt} (improvement of fluctuations),~\ref{s.additivity} (improvement of additivity) and~\ref{s.localization} (improvement of localization). The proofs of Theorems~\ref{t.firstcorrectors} and~\ref{t.main1} are completed in Section~\ref{s.finalcorrectors}. The outline of Part~\ref{part.two} is given at the end of Section~\ref{s.heuristics}.

\section{Assumptions and notation}
\label{s.notation}

\subsection{Assumptions}
\label{ss.assumptions}
We work in the Euclidean space in a fixed dimension $d \geq 2$ and consider, for a fixed $\Lambda \geq 1$, the space of coefficient fields valued in the real symmetric $d$-by-$d$ matrices~$\mathbb{S}^d$ which satisfy
\begin{equation}
\label{e.ue}
\left| \xi\right|^2 \leq \xi \cdot \a(x) \xi \leq \Lambda \left| \xi\right|^2, \quad \forall  \xi\in\Rd.
\end{equation}
We denote by $\Omega$ the collection of all such coefficient fields:
\begin{equation*}
\Omega:= \left\{ \a(\cdot)\,:\, \a:\Rd \to \mathbb{S}^{d} \ \mbox{is Lebesgue measurable and satisfies~\eqref{e.ue}} \right\}.
\end{equation*}
We endow $\Omega$ with the translation group $\{\tau_y\}_{y\in\Rd}$, which acts on $\Omega$ via
\begin{equation*}
(\tau_y\a)(x) := \a(x+y),
\end{equation*}
and with the family $\{ \F(U)\}$ of $\sigma$-algebras on $\Omega$, with $\F(U)$ defined for each Borel subset $U\subseteq \Rd$ by
\begin{multline*} \label{}
\F(U) := \mbox{$\sigma$-algebra on $\Omega$ generated by the family of maps} \\
\a \mapsto \int_{U} p\cdot \a(x)q \, \varphi(x) \,dx, \quad p,q \in\Rd, \ \varphi\in C^\infty_c(\Rd).
\end{multline*}
We think of~$\F(U)$ as encoding the information about the behavior of the coefficients in~$U$. The largest of these~$\sigma$-algebras is~$\F:=\F({\Rd})$. The translation group may be naturally extended to~$\F$ itself by defining
\begin{equation*}
\tau_yA:= \left\{ \tau_y\a\,:\, \a\in A \right\}, \quad A\in\F,
\end{equation*}
and to any random element $X$ by setting $(\tau_zX)(\a):= X(\tau_z\a)$. 

\smallskip

Throughout the paper, we consider a probability measure~$\P$ on $(\Omega,\F)$ which is assumed to satisfy the following two conditions:
\begin{enumerate}

\item[(P1)] $\P$ is invariant under $\Zd$-translations: for every $z\in \Zd$ and $A\in \F$,
\begin{equation*} 
\P \left[ A \right] = \P \left[ \tau_z A \right].
\end{equation*}

\smallskip

\item[(P2)] $\P$ has a unit range of dependence: for every pair of Borel subsets $U, V\subseteq \Rd$ with $\dist(U,V) \geq 1$, 
\begin{equation*}
\mbox{$\F(U)$ \  and \ $\F(V)$ \ are \ $\P$-independent.}
\end{equation*}
\end{enumerate}
The expectation of an $\F$-measurable random variable $X$ with respect to $\P$ is denoted by~$\E\left[X \right]$.

\subsection{General notation}
For $z \in \Rd$ and $r > 0$, we write $B_r(z)$ (or simply $B_r$ if $z = 0$) for the open Euclidean ball of center $z$ and radius $r$. We denote cubes of side length $r > 0$ by
\begin{equation*}
\cu_r:= \left( -\tfrac 12r,\tfrac 12r\right)^d, \qquad \cu_r(z):= z + \cu_r.
\end{equation*}
We recall that we write $\Phi_{z,r}:= \Phi(\cdot - z,r^2)$ where $\Phi$ is defined in \eqref{e.heatkernel},
and use the notation
\begin{equation}
\label{e.not.int}
\int_{\Phi_{z,r}} f = \int_{\Phi_{z,r}} f(x) \, dx  := \int_{\R^d} f(x) \Phi_{z,r}(x) \, dx,
\end{equation}
$$
\|f\|_{L^2(\Phi_{z,r})} := \Ll( \int_{\Phi_{z,r}}\left| f\right|^2 \Rr) ^{\frac12}.
$$
When $z = 0$, we simply write $\Phi_r$ instead of $\Phi_{0,r}$ in these expressions. We will sometimes also use a truncation of the mask $\Phi_{z,r}$ defined by
\begin{equation}
\label{e.Phideltadef}
\Phi^{(\delta)}_{z,r}:= \left\{ \begin{aligned}
& \Phi_{z,r} & \mbox{in} & \ B_{r^{1+\delta}}(z), \\
& 0 & \mbox{in} & \ \Rd\setminus \overline B_{r^{1+\delta}}(z),
\end{aligned} \right.
\end{equation}
where $\delta > 0$. We also use the notation in \eqref{e.not.int} with $\Phi_{z,r}$ replaced by $\Phid_{z,r}$, and may write $\Phid_r$ in place of $\Phid_{0,r}$ for brevity.

For a measurable set $E\subseteq \Rd$, we denote the Lebesgue measure of $E$ by $|E|$ unless $E$ is a finite set, in which case $|E|$ denotes the cardinality of $E$. For a bounded Lipschitz domain $U \subset \Rd$ with $|U| < \infty$, we write $\fint_U := |U|^{-1} \int_U$, and for $p\in [1,\infty)$, we
denote the normalized $L^p(U)$ norm of a function $f\in L^p(U)$ by 
\begin{equation} 
\label{e.def.nL}
\| f \|_{\underline{L}^p(U)} := \left( \fint_U \left| f(x )\right|^p\,dx\right)^{\frac1p}.
\end{equation}
It is also convenient to denote $\| f \|_{\underline{L}^\infty(U)} := \| f \|_{L^\infty(U)}$. For a vector-valued $F\in L^p(U;\Rd)$, we write $\| F \|_{\underline{L}^p(U)}:= \| | F|  \|_{\underline{L}^p(U)}$. For $f \in L^1(U)$, we write $(f)_U := \fint_U f$. 

\smallskip

We denote the set of solutions of our equation in a domain $U\subseteq\Rd$ by
\begin{equation*}
\A(U) := \left\{ u \in H^1_{\mathrm{loc}}(U) \,:\, \forall v\in H^1_0(U), \ \ \int_U \nabla v(x)\cdot \a(x) \nabla u(x)\,dx = 0 \right\}. 
\end{equation*}
For every $k \in \N$, we set
\begin{equation}
\label{e.def.Ak}
\A_k := \Ll\{u \in \A(\R^d) \, : \, \lim_{r \to \infty} r^{-(k+1)} \|u\|_{\underline L^2(B_r)} = 0 \Rr\}.
\end{equation}
We  denote by $\Ahom(U)$ the set of solutions of the homogenized equation and let
\begin{equation}
\label{e.def.Ahomk}
\Ahom_k := \mbox{the set of $\ahom$-harmonic polynomials of degree at most $k$.}
\end{equation}
It is convenient to use the notation 
\begin{equation}
\label{e.def.Ak-Phi}
\A_k(\Phi_{z,r}) := \left\{ p \in \A_k : \left\| \nabla p \right\|_{L^2(\Phi_{z,r})} \leq 1 \right\},
\end{equation}
with $\Ahom_k(\Phi_{z,r})$ defined in an analogous way. 
For $m \le k$, we denote by $\pi_{z,m}$ the projection of $\Ahom_k$ onto $\Ahom_m$ such that
\begin{equation}
\label{e.def.pim}
\pi_{z,m} (p) = \sum_{n = 0}^m \frac 1 {n!} \nabla^n p(z)(\cdot - z)^{\otimes n}\,.
\end{equation}
Above the interpretation for the tensor product is
\begin{equation*} 
\frac 1{n!} \nabla^n p(x) z^{\otimes n} = \sum_{\stackrel{j_1,\ldots,j_d \in \N_0}{j_1+\cdots j_d = n}} \frac1{j_1! \cdots j_d!} \partial_{x_1}^{j_1} \cdots  \partial_{x_d}^{j_d} p(x) z_1^{j_1} \cdots z_d^{j_d}\,.
\end{equation*}
Note that $\phi_{z,m}$ is just the $m$th order Taylor approximation of $p$ centered at $z$. 
We also define 
\begin{equation}
\label{e.defPk}
\mathcal{P}_k:= \mbox{the set of polynomials of degree at most $k$.}
\end{equation}

\smallskip

We define $\mathbb{L}^2_{\mathrm{pot}}$ to be the set of $\Zd$-stationary random fields which, for each realization of the coefficients, are the gradient of a function in $H^1_{\mathrm{loc}}(\Rd)$. That is, $\mathbb{L}^2_{\mathrm{pot}}$ is the set of functions $\Omega \to L^2_{\mathrm{loc}}$ of the form
\begin{equation} 
\label{e.def.L2pot}
\a \mapsto \mathbf{f}(\cdot,\a) 
\end{equation}
such that, for each $z\in\Zd$, the random fields $\mathbf{f}(\cdot,\a)$ and $\mathbf{f}(\cdot,\tau_z\a)$ have the same law (with respect to $\P$) and, for $\P$-almost every $\a\in\Omega$, there exists $u\in H^1_{\mathrm{loc}}(\Rd)$ such that $\mathbf{f}(\cdot,\a) = \nabla u$. (Note that $u$ itself may not be unique and in particular is not required to be stationary.)

\subsection{Notation for random variables}
We next discuss some notation used throughout for measuring the size and stochastic integrability of random variables. For every exponent $s \in (0,\infty)$, $\theta > 0$ and $\F$-measurable random variable $X$, we recall that we write
$$
X \le \O_s(\theta) 
$$
to mean that \eqref{e.def.Os} holds. We likewise write 
\begin{equation*}
X\leq Y +  \O_s(\theta)  \quad \iff \quad X-Y \leq \O_s(\theta)
\end{equation*}
and
\begin{equation*}
X = Y + \O_s(\theta) \quad\iff \quad X-Y \leq \O_s(\theta) \ \ \mbox{and} \ \ Y-X \leq \O_s(\theta). 
\end{equation*}
The reader may have noticed that the $\O_s(\theta)$ notation is just a different way of writing bounds for random variables with respect to certain \emph{Orlicz} norms. The usage of $\O_s(\theta)$ lightens the notation, making many computations much more readable, like the ``big-$O$" notation it evokes. 

\smallskip

We record a few elementary properties of $\O_s$-bounded random variables.
\begin{remark}
\label{r.change-s}
For every $s < s' \in (0,\infty)$ 
and every random variable $X$,
$$
\Ll\{
\begin{array}{l}
X \mbox{ takes values in }  [0,1] \\
X \le \O_s(\theta)
\end{array}
\Rr.
\quad  \implies \quad X \le \O_{s'}(\theta^{\frac s {s'}}).
$$
Indeed, under the above assumptions, we have
$$
\E\Ll[\exp( (\theta^{\frac s {s'}} X)^{s'}) \Rr] \le \E\Ll[\exp( (\theta X)^{s})\Rr] \le 2.
$$
As an example, since by~\eqref{e.bounded-J}, $J$ is bounded, for each $s<1$, we may replace $\O_s\Ll(C  r^{-d}\Rr)$ by $\O_1\Ll(C  r^{-sd}\Rr)$ in part~(i) of Theorem~\ref{t.main1}, and $\O_{2s} ( C r^{-\frac d 2} )$ by $\O_2 ( C r^{-\frac {sd} 2} )$ in part (iii) of this theorem.
\end{remark} 
\begin{remark}
\label{r.multiply}
If $X_i \le \O_{s_i}(\theta_i)$ for $i \in\{1,2\}$, then
\begin{equation*} 
X_1 X_2 \le \O_{\frac{s_1 s_2}{s_1 + s_2}} \left( \theta_1\theta_2 \right)\,.
\end{equation*}
Indeed, we may assume $X_i \ge 0$, and then observe that by Young's and H\"older's inequalities,
\begin{align*}
\lefteqn{
\E\Ll[\exp\Ll(\Ll[(\theta_1 \theta_2)^{-1} X_1 X_2\Rr]^\frac{s_1 s_2}{s_1 + s_2}\Rr) \Rr] 
} \qquad & \\
& \le \E\Ll[\exp\Ll(\frac{s_1}{s_1 + s_2}(\theta_1^{-1} X_1)^{s_1} +  \frac{s_2}{s_1 + s_2} (\theta_2^{-1} X_2)^{s_2} \Rr) \Rr] \\
& \le \E\Ll[\exp\Ll(\theta_1^{-1} X_1)^{s_1}\Rr)\Rr]^{\frac{s_1}{s_1 + s_2}}\, \E\Ll[\exp\Ll(\theta_2^{-1} X_2)^{s_2}\Rr)\Rr]^{\frac{s_2}{s_1 + s_2}} \\
& \le 2.
\end{align*}
\end{remark}

In the following lemma, we check that an average of random variables bounded by $\O_s(\theta)$ is bounded by $\O_s(C\theta)$ for a constant~$C$ depending only on~$s$. This is used throughout the paper without further mention. 

\begin{lemma}
\label{l.sum-O}
\emph{(i)} Let $s \ge 1$, $\mu$ be a measure over an arbitrary measurable space $E$, let $\theta : E \to \R_+$ be a measurable function and $(X(x))_{x \in E}$ be a jointly measurable family of nonnegative random variables such that for every $x \in E$, $X(x) = \O_s(\theta(x))$. We have
$$
\int X \, d\mu = \O_s \Ll( \int \theta \, d \mu \Rr) .
$$
\emph{(ii)} For every $s \in (0,1)$, there exists $C(s) < 1$ such that the following holds. Let $\mu$ be a probability measure over an arbitrary measurable space $E$ and $(X(x))_{x \in E}$ be a jointly measurable family of nonnegative random variables such that for every $x \in E$, $X(x) = \O_s(1)$. We have
$$
\int X \, d \mu = \O_s(C).
$$
\end{lemma}
\begin{proof}
We start with the proof of (i). Without loss of generality, we can assume $\int \theta \, d \mu < \infty$, and by homogeneity, we can further assume that $\int \theta \, d \mu = 1$. 
The function $x \mapsto \exp(x^s)$ is convex on $\R_+$. By Jensen's inequality,
\begin{equation*} 
\E\left[\exp\left(  \left(\int X \, d\mu \right)^{s} \right) \right]  \leq \E\left[\int \exp \left( (\theta^{-1} X)^s  \right) \, \theta \, d \mu \right] 
 \leq 2.
\end{equation*}
We now turn to the second statement. Let $s \in (0,1)$ and $t_s := \Ll( \frac{1-s} s \Rr)^{\frac 1 s}$. 
The function $x \mapsto \exp((x + t_s)^s)$ is convex on $\R_+$. By Jensen's inequality,
\begin{align*} 
\E\left[\exp\left(  \left(\int X \, d\mu \right)^{s} \right) \right]  & \leq \E\left[\exp \left( \left( \int X  \, d\mu  + t_s \right)^{s}  \right) \right]  
\\ & \leq \E\left[\int \exp \left( X^s  + t_s^s \right) \, d \mu \right] 
\\ & \leq 2 \exp\left( t_s^s\right)\,.
\end{align*}
For $\sigma \in (0,1)$ sufficiently small in terms of $s$, we thus have
\begin{equation*}
\E\left[\exp\left( \sigma \left(\int X \, d\mu \right)^{s}  \right) \right] \leq \E\left[\exp\left(\left(\int X \, d\mu \right)^{s} \right) \right]^\sigma
\leq  \Ll[2 \exp \Ll( t_s^s\Rr)\Rr]^\sigma \le 2\,. \qedhere
\end{equation*}
\end{proof}

\section{Higher-order regularity theory and Liouville theorems}
\label{s.regularity}

\subsection{Regularity theory}

One of the main tools in this paper is the higher regularity theory summarized in the following proposition. Recall that $\A_k$ and $\Ahom_k$ are defined in~\eqref{e.def.Ak} and~\eqref{e.def.Ahomk}, respectively.

\begin{proposition}[Instrinsic $C^{k,1}$ regularity]
\label{p.regularity}
Fix $s \in (0,d)$. There exist an exponent $\delta(s,d,\Lambda)\in \left( 0,1 \right)$ and a random variable $\X_s$ satisfying the estimate
\begin{equation}
\label{e.X}
\X_s \leq \O_s\left(C(s,d,\Lambda)\right)
\end{equation}
such that the following statements hold:
\begin{enumerate}
\item[(i)] For every $k\in\N$, there exists $C(k,d,\Lambda)<\infty$ such that, for every $u \in \A_k$, there exists $p\in \overline{\A}_k$ such that, for every $R\geq \X_s$,
\begin{equation} \label{e.liouvillec}
\left\| u - p \right\|_{\underline{L}^2(B_R)} \leq C R^{-\delta} \left\| p \right\|_{\underline{L}^2(B_R)}.
\end{equation}

\item[(ii)] For every $p\in \overline{\A}_k$, there exists $u\in \A_k$ satisfying~\eqref{e.liouvillec} for every $R\geq \X_s$. 

\item[(iii)] There exists $C(k,d,\Lambda)<\infty$ such that, for every $R\geq 2\X_s$ and $u\in \A(B_R)$, there exists $\phi \in \A_k(\Rd)$ such that, for every $r \in \left[ \X_s, \frac12 R \right]$, we have the estimate
\begin{equation}
\label{e.intrinsicreg}
\left\| u - \phi \right\|_{\underline{L}^2(B_r)} \leq C \left( \frac r R \right)^{k+1} 
\left\| u \right\|_{\underline{L}^2(B_R)}.
\end{equation}
\end{enumerate}
\end{proposition}

Proposition~\ref{p.regularity} is a deterministic consequence of the regularity theory for stochastic homogenization introduced in~\cite{AS}. Indeed, it is a classical fact that interior regularity estimates are linked to Liouville-type theorems, as one can recover the latter from the former by a simple iteration procedure. Thus, and as we have pointed out previously in~\cite[Remark 2.4]{AKM} and clarify here, the mesoscopic regularity estimates proved in~\cite{AS,AM,AKM} imply the Liouville-type result given in Proposition~\ref{p.regularity}(i) and~(ii) by a straightforward (and deterministic) analysis argument. Writing the regularity estimates in terms of ``intrinsic polynomials" (denoted here by $\A_k(\Rd)$) goes back to the original formulation of Avellaneda and Lin~\cite{AL1,AL4} for equations with periodic coefficients. The papers~\cite{GNO3,FO1} were the first to write regularity estimates in the form of~\eqref{e.intrinsicreg} in the stochastic setting and to give a complete proof of the Liouville results. 

\smallskip

We continue by recalling two previous results proved in~\cite{AS,AM,AKM} before showing how to derive Proposition~\ref{p.regularity} from them. 

\begin{proposition}[{\cite{AS,AM}}]
\label{p.errorestimate}
Fix a Lipschitz domain $U \subseteq B_1$, $s\in (0,d)$ and $\ep>0$. There exists an exponent $\delta(s,\ep,d,\Lambda)\in (0,1)$ and a random variable $\mathcal{R}_s$ satisfying the estimate
\begin{equation}
\label{e.minimalradiusEE}
\mathcal{R}_s \leq \O_s(C(U,s,\ep,d,\Lambda))
\end{equation}
such that, for every $r\geq \mathcal{R}_s$, $f\in W^{1,2+\ep}(rU)$ and solutions $u,\overline{u} \in f+H^1_0(rU)$ of 
\begin{equation*}
-\nabla \cdot \left( \a\nabla u \right) = 0 \quad \mbox{and} \quad -\nabla \cdot \left( \ahom \nabla \overline{u} \right) = 0 \quad \mbox{in} \ rU,
\end{equation*}
we have the estimate
\begin{equation}
\label{e.quenchedEE}
\frac1{r}  \left\| u - \overline{u} \right\|_{\underline{L}^2(rU)} \leq r^{-\delta} \left\| \nabla f \right\|_{\underline{L}^{2+\ep}(rU)}. 
\end{equation}
\end{proposition}

We next state the mesoscopic $C^{k,1}$ regularity estimate, which is a deterministic consequence of Proposition~\ref{p.errorestimate} (see~\cite[Lemma 2.5]{AKM}). Here the result is in a slightly more general form than what appears in~\cite{AKM}, although the statement here is what~\cite[Lemma 2.5]{AKM} actually gives.

\begin{proposition}[{\cite[Theorem 2.1]{AKM}}]
Fix $s \in (0,d)$. There exist an exponent $\delta(s,d,\Lambda)\in (0,1)$, a random variable $\X_s$ satisfying the estimate
\begin{equation}
\label{e.minimalradius}
\X_s \leq \O_s(C'(s,d,\Lambda)),
\end{equation}
and, for each $k \in \N$, a constant $C(k,d,\Lambda) < \infty$ 
such that, for every $R\geq 2\X_s$, $v \in \A(B_R)$ and $r \in \left[ \X_s, \frac12R\right]$,
\begin{equation}
\label{e.mesoregularity} 
\inf_{w\in \Ahom_k}\left\| v-w \right\|_{\underline{L}^2(B_r)} \leq C \left( \frac rR \right)^{k+1} \inf_{w\in \Ahom_k}\left\| v-w \right\|_{\underline{L}^2(B_R)} + C r^{-\delta} \left( \frac rR \right) \left\| v \right\|_{\underline{L}^2(B_R)}.
\end{equation}
\label{p.mesoregularity}
\end{proposition}

We now present the proof of Proposition~\ref{p.regularity}. 

\begin{proof}[{Proof of Proposition~\ref{p.regularity}}]

\emph{Step 1.} 
We first remark that if $p$ is an $\ahom$-harmonic polynomial, then $\tilde p$, defined as $\tilde p(y) = p(\ahom^{\,-\frac12} y)$, is harmonic. It is easy to see the orthogonality of two harmonic monomials of different homogeneity, since in the polar coordinates their angular parts are eigenfunctions for the Laplace-Beltrami operator of the sphere with different eigenvalues. Therefore, in the original coordinates, two $\ahom$-harmonic monomials of different degree are orthogonal with respect to the inner product of $L^2\left( \{x \in \R^d \, : \, x \cdot \ahom^{-1} x < r^2 \}\right)$ for every $r>0$.  


\smallskip

Let $\X_s$ denote the random variable appearing in Proposition~\ref{p.mesoregularity}. Without loss of generality, we will assume that $\X_s$ also bounds the random variable $\mathcal{R}_s$ given by Proposition~\ref{p.errorestimate} with the choice of domain $U = B_1$ and $\eps = 1$. 
We choose $\theta \in \left(0,\frac12\right]$ small, and for every $j \in \N$, set $r_j := \theta^{-j} r_0$ with $ r_0 \in \left[ \X_s , \frac12R\right]$. In order to use the orthogonality of the polynomials, we use the balls $B^j := \{x \in \R^d \, :\, x \cdot \ahom^{-1} x < r_j^2 \}$. For $u \in \A(B^{j+1})$, 
using the ellipticity of $\ahom$, an application of Proposition~\ref{p.mesoregularity} and the triangle inequality yields 
the existence of $p_{j}, p_{j+1} \in \Ahom_{k+1}$, which are the best $\Ahom_{k+1}$-polynomial approximations of $u$ in $B^j$ and $B^{j+1}$, respectively, such that 
\begin{align*} \label{}
\left\| u -  p_{j} \right\|_{\underline{L}^2(B^{j})} 
& \leq C \theta^{k+2} \left\| u - p_{j+1} \right\|_{\underline{L}^2(B^{j+1})}  + C \theta r_j^{-\delta} \left\| u\right\|_{\underline{L}^2(B^{j+1})} 
 \\  & \le  C \left( \theta^{k+2} +   \theta r_j^{-\delta} \right)  \left\| u - p_{j+1} \right\|_{\underline{L}^2(B^{j+1})}  + C\theta r_j^{-\delta}  \left\| p_{j+1}\right\|_{\underline{L}^2(B^{j+1})} 
\end{align*}
with $C = C(k,d,\Lambda)$. We choose $\theta \leq (2C)^{-1/\delta}$ and $r_0 := \max\{\X_s, C^{1/\delta} \theta^{-(k+2)/\delta} \}$, so that the previous inequality becomes
\begin{equation*} 
\left\| u -  p_{j} \right\|_{\underline{L}^2(B^{j})} 
 \leq  \theta^{k+2-\delta} \left\| u - p_{j+1} \right\|_{\underline{L}^2(B^{j+1})}  + r_j^{-\delta} \left\| p_{j+1}\right\|_{\underline{L}^2(B^{j+1})} .
 \end{equation*}
Iterating then gives, for $n>0$ and $u \in \A(B^{j+n})$, that 
\begin{multline}  \label{e.iterbasic}  
\left\| u - p_{j} \right\|_{\underline{L}^2(B^{j})}
\\  \leq \left( \theta^{k+2-\delta}\right)^{n}  \left\| u -p_{j+n} \right\|_{\underline{L}^2(B^{j+n})} 
 + r_{j}^{-\delta} \sum_{h=j+1}^{j+n} \left( \theta^{k+2} \right)^{h-(j+1)}  \left\| p_h  \right\|_{\underline{L}^2(B^{h})} .
\end{multline}
This serves as the basic inequality for our argument. 

\smallskip

\emph{Step 2.} The proof proceeds by induction. We prove (i), (ii) and (iii') together in an induction on~$k$, where (iii') is the following weaker version of (iii):
\begin{enumerate}
\item[(iii')] For each $\alpha\in (0,1)$ and $k\in\N$, there exists $C(k,\alpha,d,\Lambda)<\infty$ such that, for every $R\geq 2\X_s$ and $u\in \A(B_R)$, there exists $\phi \in \A_k(\Rd)$ such that, for every $r \in \left[ \X_s, \frac12 R \right]$,
\begin{equation*}
\left\| u - \phi \right\|_{\underline{L}^2(B_r)} \leq C \left( \frac r R \right)^{k+\alpha} 
\left\| u \right\|_{\underline{L}^2(B_R)}.
\end{equation*}
\end{enumerate}

By the notation (i)$_{k}$, (ii)$_{k}$ and (iii')$_k$ we mean that the corresponding statements hold up to degree $k$. Our induction assumption states that  (i)$_{k-1}$, (ii)$_{k-1}$ and (iii')$_{k-1}$ are all valid. These are trivial for $k=1$ since the constants are both $\ahom$- and $\a$-harmonic, and this is our initial step for the induction. We will first show that (ii)$_{k}$  holds, then (i)$_{k}$, and finally (iii')$_k$. In the last step, we will demonstrate that~(i)$_{k+1}$, (ii)$_{k+1}$ and~(iii')$_{k+1}$ actually imply~(iii)$_k$. 

\smallskip

\emph{Step 3.} In this step, we show that (ii)$_k$ is valid using (i)$_{k-1}$, (ii)$_{k-1}$ and (iii')$_{k-1}$. Due to (ii)$_{k-1}$ and orthogonality properties of $\ahom$-harmonic polynomials, we may assume that  $q \in \Ahom_k$ is a homogeneous polynomial of degree $k$. 
First, by Proposition~\ref{p.errorestimate}, we find for each $m \in \N$ a solution $u_m \in \A(B^m)$ such that 
\begin{equation} \label{e.u_m local}
\left\| u_m - q \right\|_{\underline{L}^2(B^{m})} \leq C r_m^{-\delta} \left\| q \right\|_{\underline{L}^2(B^{m})} \,.
\end{equation}
Let $w_m := u_{m+1} - u_{m}$,  and let $\phi_m \in \A_{k-1}(\R^d)$ be given by (iii')$_{k-1}$ for~$w_m$. Then~(iii')$_{k-1}$ together with the triangle inequality and the previous display imply 
\begin{align*} 
\left\| w_m -\phi_m  \right\|_{\underline{L}^2(B^{j})} & \leq C \left(\theta^{k-\delta/2}\right)^{j-m} \left\| w_m  \right\|_{\underline{L}^2(B^{m})} 
\\ & \leq C r_{m}^{-\delta} \left(\theta^{k-\delta/2}\right)^{j-m} \left\| q  \right\|_{\underline{L}^2(B^{m})} 
= C r_j^{-\delta} \left(\theta^{\delta/2} \right)^{m-j}  \left\| q  \right\|_{\underline{L}^2(B^{j})}  
\end{align*}
for all $j \leq m$.
By (i)$_{k-1}$, there exists $p_{\phi_m} \in \Ahom_{k-1}$ such that $\left\| p_{\phi_m} - \phi_m \right\|_{\underline{L}^2(B^{n})} \leq c r_n^{-\delta}  \left\| p_{\phi_m}  \right\|_{\underline{L}^2(B^{n})}$ for every $n \in \N$. By the triangle inequality and the two previous displays, 
\begin{multline*}
\left\| p_{\phi_m}   \right\|_{\underline{L}^2(B^{m})} \le  2 \left\| {\phi_m}   \right\|_{\underline{L}^2(B^{m})} \\
\le  2 \left\| w_m  \right\|_{\underline{L}^2(B^{m})} + 2 \left\| w_m -\phi_m  \right\|_{\underline{L}^2(B^{m})} 
\leq  C r_{m}^{-\delta} \left\| q  \right\|_{\underline{L}^2(B^{m})} \,. 
\end{multline*}
Therefore we have, for any $n>m$,
\begin{align*} 
\left\| \phi_m  \right\|_{\underline{L}^2(B^{n})} & \leq C  \left(\frac{r_n}{r_m} \right)^{k-1} \left\| p_{\phi_m}  \right\|_{\underline{L}^2(B^{m})} 
\\ & \leq C \theta^{(m-n)(k-1)} r_{m}^{-\delta} \left\| q  \right\|_{\underline{L}^2(B^{m})}  = C \theta^{(n-m)(1-\delta)}  r_n^{-\delta} \left\| q  \right\|_{\underline{L}^2(B^{n})}\,.
\end{align*}
Set $v_n = u_n - \sum_{m=1}^{n-1} \phi_m$, so that $v_n - u_j = \sum_{m=j}^{n-1} (w_m-\phi_m) - \sum_{m=1}^{j-1} \phi_m$. We get
\begin{multline*} 
\left\| v_n - q  \right\|_{\underline{L}^2(B^{j})}  \leq \left\| u_j - q  \right\|_{\underline{L}^2(B^{j})} + \sum_{m=j}^{n-1} \left\| w_m -\phi_m \right\|_{\underline{L}^2(B^{j})} + \sum_{m=1}^{j-1} \left\| \phi_m \right\|_{\underline{L}^2(B^{j})}
\\  \leq C r_j^{-\delta} \left\| q  \right\|_{\underline{L}^2(B^{j})} \left(1 + \sum_{m=0}^{n-1-j} \left(\theta^{\delta/2} \right)^{m}   +  \sum_{m=1}^{j-1} \left( \theta^{1-\delta}\right)^{m} \right) \leq C r_j^{-\delta} \left\| q  \right\|_{\underline{L}^2(B^{j})}
\end{multline*}
whenever $j<n$. After letting $n \to \infty$, we find $u \in \A_k(\R^d)$ satisfying (ii)$_k$, appealing to weak convergence in $H_{\rm loc}^1$ and a diagonal argument.  

\smallskip

\emph{Step 4.} We now prove (i)$_k$ using (i)$_{k-1}$ and (ii)$_k$. 
Let us denote by $P_m$ the projection of polynomials to $m^{th}$-degree homogenous polynomials, i.e. $\nabla^j P_m p(0) = 0$ whenever $j \neq m$, and let $Q_m =\sum_{h=0}^m P_h$.  
For $u \in \A_k$, since
\begin{equation} \label{e.Liouvgrowth} 
\lim_{r \to \infty} r^{-(k+1)} \|u\|_{\underline L^2(B_r)} = 0,
\end{equation}
letting $n \to \infty$ in~\eqref{e.iterbasic} yields that the first term on the right in~\eqref{e.iterbasic} tends to zero and that the sum is convergent.  We obtain
\begin{equation}  \label{e.iterbasic1}  
\left\| u - p_{j} \right\|_{\underline{L}^2(B^{j})}
 \leq r_j^{-\delta} \omega_{j,k+1,1} \,, \qquad \omega_{j,m,\sigma} :=  \sum_{h=j}^{\infty} \left( \theta^{k+1+\sigma} \right)^{h-j}  \left\| Q_{m} p_h  \right\|_{\underline{L}^2(B^{h})} \,,
\end{equation}
with $\omega_{j,k+1,\sigma} < \infty$, for all $j \in \N$ and $\sigma > 0$, by~\eqref{e.Liouvgrowth}.
We get consequently by the triangle inequality that 
\begin{equation} \label{e.iter004}  
\left\| p_{j} - p_{j+1} \right\|_{\underline{L}^2(B^{j})} \leq C r_{j}^{-\delta} \omega_{j,k+1,1}  \,.
\end{equation}
Since $p$ is in $\Ahom_{k+1}$, $P_m p \in  \Ahom_{m}$ for all $m \in \N$, and hence the orthogonality of the $\ahom$-harmonic monomials implies that 
\begin{align} \label{e.iter005}
 \left\| P_{k+1} (p_{j} - p_{j+1}) \right\|_{\underline{L}^2(B_1)} 
& \leq  C  r_{j}^{-(k+1+\delta)} \omega_{j,k,1}
+ C r_j^{-\delta} \sum_{h=j}^{\infty} \theta^{h-j}  \left\| P_{k+1} p_h  \right\|_{\underline{L}^2(B_1)} \,.
\end{align}
Summation then yields 
\begin{align*} 
\sum_{m=j}^\infty \left\| P_{k+1} (p_{m} - p_{m+1}) \right\|_{\underline{L}^2(B_1)}  & \leq 
C  r_{j}^{-(k+1+\delta)} \sum_{m=j }^\infty \left( \theta^{k+1+\delta} \right)^{m-j} \omega_{m,k,1} 
\\ & \qquad + C r_j^{-\delta} \sum_{m=j}^\infty \left( \theta^{\delta} \right)^{m-j} 
\sum_{h=m}^{\infty} \theta^{h-m}  \left\| P_{k+1} p_h  \right\|_{\underline{L}^2(B_1)} \,.
\end{align*}
We can rearrange the sums on the right as
\begin{align} \label{e.tpiab00000} 
\sum_{m=j }^\infty \left( \theta^{k+1+\delta} \right)^{m-j} \omega_{m,k,1} & = 
\sum_{m=j}^{\infty}    \left( \theta^{k+1+\delta} \right)^{m-j}  \sum_{h=m}^\infty \left( \theta^{k+2} \right)^{h-m}  \left\| Q_{k} p_h  \right\|_{\underline{L}^2(B^{h})} 
\\ \notag & \leq C \sum_{m=j}^{\infty}    \left( \theta^{k+1+\delta} \right)^{m-j}   \left\| Q_{k} p_m  \right\|_{\underline{L}^2(B^{m})} 
\\ \notag & = C \omega_{j,k,\delta},
\end{align}
and, using also the triangle inequality, 
\begin{align*} 
\lefteqn{\sum_{m=j}^\infty \left( \theta^{\delta} \right)^{m-j} 
\sum_{h=m}^{\infty} \theta^{h-m}  \left\| P_{k+1} p_h  \right\|_{\underline{L}^2(B_1)}   \leq C 
\sum_{m=j}^\infty \left( \theta^{\delta} \right)^{m-j} \left\| P_{k+1} p_m  \right\|_{\underline{L}^2(B_1)} } \qquad &
\\ & \leq C \left\| P_{k+1} p_j  \right\|_{\underline{L}^2(B_1)}  + \sum_{m=j}^\infty \left( \theta^{\delta} \right)^{m-j} \sum_{h=m}^\infty \left\| P_{k+1} (p_{h} - p_{h+1} ) \right\|_{\underline{L}^2(B_1)}  
\\ & \leq C \left\| P_{k+1} p_j  \right\|_{\underline{L}^2(B_1)}   + C \sum_{h=j}^\infty \left\| P_{k+1} (p_{h} - p_{h+1} ) \right\|_{\underline{L}^2(B_1)} \,.
\end{align*}
Combining the last three displays, using $Cr_0^{-\delta}\leq \frac12$ and reabsorption yields
\begin{equation} \label{e.iter0005}
\sum_{m=j}^\infty \left\| P_{k+1} (p_{m} - p_{m+1}) \right\|_{\underline{L}^2(B_1)}  \leq C r_{j}^{-(k+1+\delta)}  \omega_{j,k,\delta} + C r_j^{-\delta} \left\| P_{k+1} p_j  \right\|_{\underline{L}^2(B_1)} \,.
\end{equation}
We then apply the hypothesis~\eqref{e.Liouvgrowth} in the form
$$\left\| P_{k+1}  p_{h}  \right\|_{\underline{L}^2(B_1)} = r_h^{-(k+1)} \left\| P_{k+1}  p_{h}  \right\|_{\underline{L}^2(B^h)} \leq 2 
r_h^{-(k+1)} \left\| u  \right\|_{\underline{L}^2(B^h)}     \to 0
$$
as $h \to \infty$. Hence the triangle inequality and reabsorption give
\begin{equation} \label{e.iter006}
\left\| P_{k+1} p_{j}  \right\|_{\underline{L}^2(B^j)} =  r_{j}^{k+1}  \left\| P_{k+1} p_{j}  \right\|_{\underline{L}^2(B_1)} 
\leq  C  r_{j}^{-\delta}  \omega_{j,k,\delta}\,.
\end{equation}
Inserting this back to~\eqref{e.iter004} proves 
\begin{equation} \label{e.iter007}  
 \omega_{j,k+1,1} \leq C \omega_{j,k,\delta}  \quad \mbox{and} \quad \left\| Q_k( p_{j+1} - p_{j} ) \right\|_{\underline{L}^2(B^{j})} \leq C r_{j}^{-\delta} \omega_{j,k,\delta}  \,.
\end{equation}
We have hence reduced the degree of the approximative polynomials. 

\smallskip

We next estimate the growth of $\omega_{j,k,\sigma}$ and $\omega_{j,k-1,\sigma}$. By~\eqref{e.iter007}, the triangle inequality and polynomial growth, for $\sigma \in (-1, \delta]$,  we have
\begin{align} \label{e.omega iter k}
\omega_{j+1,k,\sigma} & = \sum_{h=j+1}^\infty (\theta^{k+1+\sigma})^{h-(j+1)} \left\| Q_{k} p_h  \right\|_{\underline{L}^2(B^{h})} 
 \\ \notag & \leq \sum_{h=j+1}^\infty (\theta^{k+1+\sigma})^{h-(j+1)} \left( \left\| Q_{k} p_{h-1}\right\|_{\underline{L}^2(B^{h})}  + \left\| Q_{k} (p_h-p_{h-1})  \right\|_{\underline{L}^2(B^{h})}   \right)
\\ \notag & \leq \theta^{-k}  \sum_{h=j}^\infty (\theta^{k+1+\sigma})^{h-j } \left(\left\| Q_{k} p_{h}\right\|_{\underline{L}^2(B^{h})}  +   \left\| Q_{k} (p_{h+1}-p_{h})  \right\|_{\underline{L}^2(B^{h})}  \right)
\\ \notag & \leq   \theta^{-k} \omega_{j,k,\sigma} + C r_{j}^{-\delta} \sum_{h=j}^\infty (\theta^{k+1+\sigma})^{h-j } \omega_{h,k,\delta}   
\\ \notag & \leq   \left( \theta^{-k} + \theta^{\delta j} \right) \omega_{j,k,\sigma}\,,
\end{align}
where the sum on the second last line was estimated as in~\eqref{e.tpiab00000} using $C r_{0}^{-\delta} \leq 1$. Iteration then yields, for $n>j>0$, 
\begin{equation} \label{e.omega k res}
\omega_{n,k,\sigma} \leq C \left( \theta^{-k} \right)^{n-j} \omega_{j,k,\sigma} \leq C \left( \theta^{-k} \right)^{n} \omega_{0,k,\sigma} \,.
\end{equation}
Furthermore, we obtain by the same computation as in~\eqref{e.omega iter k}, appealing also to~\eqref{e.omega k res} and $C r_{0}^{-\delta} \leq 1$,
\begin{equation*} \label{e.omega iter k-1}
\omega_{j+1,k-1,\sigma} \leq   \theta^{1-k} \omega_{j,k-1,\sigma}  + C r_{j}^{-\delta} \omega_{j,k,\sigma} \leq \theta^{1-k} \omega_{j,k-1,\sigma} + \left( \theta^{\delta-k} \right)^{j}
\omega_{0,k,\sigma}  \,.
\end{equation*}
After an iteration we conclude that, for $\sigma \in (-1, \delta]$,
\begin{equation} \label{e.omega k-1 res}
 \left\| Q_{k-1} p_j \right\|_{\underline{L}^2(B^j)} \leq  \omega_{j,k-1,\sigma} \leq C \left( \theta^{\delta-k} \right)^{j} \omega_{0,k,\sigma}\,.
\end{equation}
Connecting~\eqref{e.iterbasic1},~\eqref{e.iter006} and~\eqref{e.iter007} with~\eqref{e.omega k res} and~\eqref{e.omega k-1 res} therefore gives
\begin{equation} \label{e:(ii)_k is almost over}
\sup_{j \in \N} r_j^{\delta-k}  \left\| u - P_k p_j  \right\|_{\underline{L}^2(B^j)} < + \infty\,.
\end{equation}

\smallskip

To continue, observe that the second estimate in~\eqref{e.iter007} is symmetric to~\eqref{e.iter004}. Thus, completely analogously to how~\eqref{e.iter0005} was obtained, we deduce that
\begin{equation} \label{e.iter0007}
\sum_{m=j}^\infty \left\| P_{k} (p_{m} - p_{m+1}) \right\|_{\underline{L}^2(B_1)}  \leq C r_{j}^{-(k+\delta)}  \omega_{j,k-1,\delta-1} + C r_j^{-\delta} \left\| P_{k} p_j  \right\|_{\underline{L}^2(B_1)} \,.
\end{equation}
Since the term on the right is bounded by~\eqref{e.omega k-1 res}, $\{ P_k p_j \}_j$ is a Cauchy sequence in $L^2(B_1)$, with limit $\tilde p = P_k \tilde p$. By the triangle inequality,
\begin{align}  \label{e.iter008}
 \left\| P_{k} (p_{j} - \tilde p) \right\|_{\underline{L}^2(B^j)} 
& \leq  C  r_{j}^{-\delta} \left( \omega_{j,k-1,\delta-1} +   \left\| P_{k}  \tilde p \right\|_{\underline{L}^2(B^j)}  \right)\,.
\end{align}
Putting together~\eqref{e.omega k-1 res},~\eqref{e:(ii)_k is almost over} and~\eqref{e.iter008} yields 
\begin{equation} \label{e:(ii)_k is OVER}
\sup_{j \in \N} r_j^{\delta-k}  \left\| u - \tilde p  \right\|_{\underline{L}^2(B^j)} < + \infty\,.
\end{equation}
To conclude the step, by (ii)$_k$ we finally find $\tilde u \in \A_k$ such that $\left\| \tilde u - \tilde p  \right\|_{\underline{L}^2(B^j)} \leq C r_j^{-\delta} \left\| \tilde p  \right\|_{\underline{L}^2(B^j)}$ for all $j \in \N$. This together with~\eqref{e:(ii)_k is OVER} gives
\begin{equation} \label{e:(ii)_k is TOTALLY OVER}
\sup_{j \in \N} r_j^{\delta-k}  \left\| u - \tilde u  \right\|_{\underline{L}^2(B^j)} < + \infty\,,
\end{equation}
which means that $u - \tilde u \in \A_{k-1}$. Thus, by (i)$_{k-1}$ there is $q \in \Ahom_{k-1}$ such that  
$\left\| u- \tilde u - q  \right\|_{\underline{L}^2(B^j)} \leq C r_j^{-\delta} \left\| q  \right\|_{\underline{L}^2(B^j)}$ for all $j \in \N$. Now (i)$_k$ follows with $p = q + \tilde p \in \Ahom_k$ simply by the triangle inequality and the orthogonality of $\tilde p$ and $q$:
\begin{multline*} 
\left\| u- p  \right\|_{\underline{L}^2(B^j)}  \leq \left\| u - \tilde u - q  \right\|_{\underline{L}^2(B^j)} + \left\| \tilde u - \tilde p  \right\|_{\underline{L}^2(B^j)} 
 \\  \leq  C r_j^{-\delta} \left(  \left\| q  \right\|_{\underline{L}^2(B^j)}^2 +  \left\| \tilde p  \right\|_{\underline{L}^2(B^j)}^2\right)^{\frac12} =  C r_j^{-\delta} \left\|  p  \right\|_{\underline{L}^2(B^j)}\,.
\end{multline*}

\smallskip

\emph{Step 5.} Assuming that (i)$_k$ and (ii)$_k$  hold, we will show that also (iii')$_{k}$ is true. Let $n \in \N$ be so large that $B^n \subset B_R$ and $B^{n+1} \cap B_R^c \neq \emptyset$. 
Then the ellipticity of~$\ahom$ yields $|B^{n}| \geq c |B_R|$. 
Observe that for any $u \in \A(B^n)$ we have by~\eqref{e.iterbasic} that there exists $p \in \Ahom_{k}$ such that 
$$
\left\| u - p  \right\|_{\underline{L}^2(B^{j-1})} \leq C\theta^{k+1}  \left\| u  \right\|_{\underline{L}^2(B^{j})}\,.
$$ 
We now choose $\theta$ possibly smaller so that $2 C \theta^{1-\alpha} \le 1$, where $\alpha$ is as in (iii')$_{k}$.  Define sequences $\{ u_j\}$, $\{ p_j\}$ and $\{ \phi_j\}$ recursively by setting $u_n :=u$ and, for every $j \in \{0,\ldots,n-1\}$, selecting $p_{j}\in \Ahom_k$ by way of the previous display to satisfy
\begin{equation*} \label{e.pickpolypj}
 \left\| u_j - p_{j-1} \right\|_{\underline{L}^2(B^{j-1})} 
\leq \frac12\theta^{k+\alpha} \left\| u_j  \right\|_{\underline{L}^2(B^j) }  \quad \mbox{and} \quad \left\| p_{j-1} \right\|_{L^\infty(B^j)} \leq C\left\| u_j  \right\|_{\underline{L}^2(B^j) }\,.
\end{equation*}
Then pick $\phi_{j-1} \in \A_k(\Rd)$ using the assumption (ii)$_{k}$ satisfying
\begin{equation*} 
\left\| p_{j-1} - \phi_{j-1} \right\|_{\underline{L}^2(B^{j-1})} \leq C r_j^{-\delta} \left\| p_{j-1} \right\|_{\underline{L}^2(B^{j-1})}\,,
\end{equation*}
and set $u_{j-1}:= u_j - \phi_{j-1}$. The triangle inequality and the above estimates imply that 
\begin{equation*}
\left\| u_{j-1}  \right\|_{\underline{L}^2(B^{j-1})} \leq \left( \frac{1}{2} \theta^{k+\alpha} + C r_j^{-\delta} \right) \left\| u_j \right\|_{\underline{L}^2(B^j)}.
\end{equation*}
Demanding $4 C r_0^{-\delta} \leq \theta^{k+\alpha}$, the previous inequality gives after iteration that, for all $j\in \{ 0,\ldots,n\}$,
\begin{equation*}
\left\| u_{j}  \right\|_{\underline{L}^2(B^{j})} \leq  \left( \frac{3}{4}  \theta^{k+\alpha} \right)^j \left\| u \right\|_{\underline{L}^2(B^n)}.
\end{equation*}
We have hence shown that, for every $r\in\left[ C\vee \X_s,\frac12 R\right]$, 
\begin{equation*}
\inf_{\phi \in \A_k(\Rd) } \left\| u - \phi \right\|_{\underline{L}^2(B_r)} \leq C\left( \frac rR \right)^{k+\alpha} \left\| u \right\|_{\underline{L}^2(B_{R})}.
\end{equation*}
The restriction $r\geq C\vee \X_s$ can be relaxed to $r\geq \X_s$ at the expense of increasing the constant prefactor. We thus obtain, for every $r\in \left[ \X_s,\frac12R\right]$, the estimate
\begin{equation}
\label{e.iiiww}
\inf_{\phi \in \A_k(\Rd) } \left\| u - \phi \right\|_{\underline{L}^2(B_r)} \leq C\left( \frac rR \right)^{k+\alpha} \left\| u \right\|_{\underline{L}^2(B_{R})}.
\end{equation}
To complete the proof of (iii')$_{k}$, we need to check that we can select $\phi\in \A_k(\Rd)$ independent of the radius $r$. Let $\phi_r\in \A_k(\Rd)$ achieve the infimum on the left side of~\eqref{e.iiiww}. Then by the triangle inequality,
\begin{equation*}
\left\| \phi_r - \phi_{2r} \right\|_{\underline{L}^2(B_r)} \leq C\left( \frac rR \right)^{k+\alpha} \left\| u \right\|_{\underline{L}^2(B_{R})}.
\end{equation*}
Using statements (i)$_k$ and (ii)$_k$ implying that any $\phi \in \A_k(\R^d)$ satisfies $\left\| \phi \right\|_{\underline{L}^2(B_s)} \leq C \left( \frac sr \right)^k \left\| \phi \right\|_{\underline{L}^2(B_s)}$ for $s \geq r$, we obtain that, for every $s\geq r\geq \X_s$,
\begin{equation*}
\left\| \phi_r - \phi_{2r} \right\|_{\underline{L}^2(B_s)} \leq C\left( \frac sr \right)^k \left( \frac rR \right)^{k+\alpha} \left\| u \right\|_{\underline{L}^2(B_{R})} \leq C\left( \frac sr \right)^{-\alpha} \left( \frac sR \right)^{k+\alpha} \left\| u \right\|_{\underline{L}^2(B_{R})}. 
\end{equation*}
Summing the previous inequality over dyadic radii yields, for every $s\geq r\geq \X_s$,
\begin{equation*}
\left\| \phi_r - \phi_{s} \right\|_{\underline{L}^2(B_s)}  \leq C \left( \frac sR \right)^{k+\alpha} \left\| u \right\|_{\underline{L}^2(B_{R})}. 
\end{equation*}
In particular, if we take $\phi:= \phi_{\X_s}$ then we obtain, for every $s\geq \X_s$,
\begin{equation*}
\left\| u - \phi \right\|_{\underline{L}^2(B_s)} \leq \left\| u - \phi_s \right\|_{\underline{L}^2(B_s)} + \left\| \phi - \phi_s \right\|_{\underline{L}^2(B_s)}\leq C\left( \frac sR \right)^{k+\alpha} \left\| u \right\|_{\underline{L}^2(B_{R})}.
\end{equation*}
This completes the proof of (iii')$_{k}$.

\smallskip

\emph{Step 6.} We finish the proof of the proposition by showing that (i)$_{k+1}$, (ii)$_{k+1}$ and (iii')$_{k+1}$ imply (iii)$_k$. Fix $R\geq 2\X_s$ and $u\in \A(B_R)$. Select first 
$\psi \in \A_{k+1}(\Rd)$ such that, for every $r \in \left[ \X_s, \frac12 R \right]$, 
\begin{equation}
\label{e.kk11}
\left\| u - \psi \right\|_{\underline{L}^2(B_r)} \leq C \left( \frac r R \right)^{k+1+\alpha} 
\left\| u \right\|_{\underline{L}^2(B_R)}\,.
\end{equation}
Let $p_\psi \in \Ahom_{k+1}$ be the approximating polynomial of $\psi$, that is $\left\|\psi - p_\psi \right\|_{\underline{L}^2(B_r)} \leq Cr^{-\delta} \left\| p_\psi \right\|_{\underline{L}^2(B_r)} $, provided by (i)$_{k+1}$, and take $\tilde \psi$ be the corrector in $\A_{k+1}$ corresponding the monomial $P_{k+1} p_\psi$, i.e.  $\left\|\tilde \psi - P_{k+1} p_\psi \right\|_{\underline{L}^2(B_r)} \leq Cr^{-\delta} \left\| P_{k+1} p_\psi \right\|_{\underline{L}^2(B_r)} $, given by (ii)$_{k+1}$. Clearly $\tilde \psi $ has growth of degree $k+1$:
\begin{multline} \label{e.kk1111}
\left\| \tilde \psi \right\|_{\underline{L}^2(B_r)} \leq 2 \left\| P_{k+1} p_\psi \right\|_{\underline{L}^2(B_r)}  \leq 2 \left(\frac rR \right)^{k+1} \left\| P_{k+1} p_\psi \right\|_{\underline{L}^2(B_R)} 
\\  \leq 2 \left(\frac rR \right)^{k+1} \left\| p_\psi \right\|_{\underline{L}^2(B_R)}  \leq C \left(\frac rR \right)^{k+1} \left\| u \right\|_{\underline{L}^2(B_R)}\,.
\end{multline}
Moreover, it is easy to see that $\phi := \psi - \tilde \psi$ belongs to $\A_{k}$ and this is our candidate for the corrector in the statement (iii)$_k$. Indeed, we have by~\eqref{e.kk11},~\eqref{e.kk1111} and the triangle inequality  that 
\begin{align*} 
\left\| u - \phi \right\|_{\underline{L}^2(B_r)} &  \leq C\left\| \tilde \psi \right\|_{\underline{L}^2(B_r)} +  C\left\| u - \psi \right\|_{\underline{L}^2(B_r)}  
 \\ &  \leq C \left(\left(\frac rs \right)^{k+1} +   \left( \frac r R \right)^{k+1+\alpha} \right) \left\| u \right\|_{\underline{L}^2(B_R)} \leq C \left(\frac rs \right)^{k+1} \left\| u \right\|_{\underline{L}^2(B_R)}  \,.
\end{align*}
The proof is complete. 
\end{proof}

\subsection{Multiscale Poincar\'e inequality and consequences}

The next lemma is a ``multiscale Poincar\'e" inequality, which is a variant of~\cite[Proposition 5.1]{AKM} and~\cite[Lemma 3.3]{AGK}. This inequality gives a convenient way to transfer bounds on the spatial averages of the gradient of a function to the oscillation of the function itself. 

\begin{lemma}
\label{l.mspoincare}
Let $w$ be a solution of the parabolic equation
\begin{equation*}
\partial_t w - \nabla \cdot \left( \ahom \nabla w\right) = 0 \quad \mbox{in} \ \Rd \times (0,\infty)
\end{equation*}
satisfying $w(\cdot,0) \in L^2(\Psi_R)$ and 
\begin{equation*}
\int_{0}^{R^2} \int_{\Psi_R} \left| w(y,t) \right|^2\,dy\,dt < \infty,
\end{equation*}
where $\Psi_R$ is the function
\begin{equation}
\label{e.PsiR}
\Psi_R(x):= R^{-d} \exp\left( - \frac{|x|}{R} \right).
\end{equation}
Then there exists $C(d,\Lambda)< \infty$ such that, for every $\sigma \in (0,1]$,
\begin{multline*}
 \int_{\Psi_R}  \left| w(y,0) \right|^2\,dy \\
 \leq C\int_{\Psi_R} \left| w(y,(\sigma R)^2) \right|^2\,dy 
 +C \int_{0}^{(\sigma R)^2} \int_{\Psi_R} \left| \nabla w(y,t) \right|^2\,dy\,dt.
 \end{multline*}
\end{lemma}
\begin{proof}
We compute, for any $\ep>0$
\begin{align*}
& \left| \partial_t \int_{\Psi_R} \frac12 \left| w(y,t) \right|^2\,dy \right| 
 = \left| \int_{\Rd} \nabla \left( \Psi_Rw(\cdot,t)\right)(y) \cdot \ahom \nabla w(y,t)\,dy \right| \\
& \qquad  \qquad \leq C \int_{\Rd} \Psi_R(y) \left( \frac1\ep\left| \nabla w(y,t) \right|^2 + \ep \left(\frac{\left| \nabla \Psi_R(y) \right| }{\Psi_R(y)}\right)^2 \left|w(y,t) \right|^2\right)\,dy\\
& \qquad \qquad \leq \frac{C}{\ep} \int_{\Psi_R} \left| \nabla w(y,t) \right|^2\,dy + \frac{C\ep}{R^{2}} \int_{\Psi_R}  \left| w(y,t) \right|^2\,dy.
\end{align*}
Integrating with respect to $t$ yields
\begin{align*}
\lefteqn{
\sup_{t\in [0,(\sigma R)^2]}\int_{\Psi_R}  \left| w(y,\cdot) \right|^2\,dy 
} \qquad & \\
& \leq \int_{\Psi_R} \left| w(y,(\sigma R)^2) \right|^2\,dy + \int_{0}^{(\sigma R)^2} \left| \partial_t \int_{\Psi_R} \left| w(y,t) \right|^2\,dy \right| \,dt \\
 & \leq \int_{\Psi_R}\left| w(y,(\sigma R)^2) \right|^2\,dy +  \frac{C}{\ep}\int_{0}^{(\sigma R)^2} \int_{\Psi_R} \left| \nabla w(y,t) \right|^2\,dy\,dt
  \\ & \qquad 
+ \int_0^{(\sigma R)^2} \frac{C\ep}{R^{2}} \int_{\Psi_R}  \left| w(y,t) \right|^2\,dy \,dt.
\end{align*}
Now taking $\ep =c$ sufficiently small, we can absorb the last term on the right side to obtain
\begin{align*}
 \int_{\Psi_R}  \left| w(y,0) \right|^2\,dy 
 & \leq \sup_{t\in [0,(\sigma R)^2] } \int_{\Psi_R}  \left| w(y,t) \right|^2\,dy \\
 & \leq C \int_{\Psi_R} \left| w(y,(\sigma R)^2) \right|^2\,dy +  C \int_{0}^{(\sigma R)^2} \int_{\Psi_R} \left| \nabla w(y,t) \right|^2\,dy\,dt. \\
\end{align*}
This completes the proof of the lemma. 
\end{proof}

We next specialize the previous lemma to elements of~$\A_m(\Rd)$. We obtain that elements of $\A_m(\Rd)$ behave like normal polynomials in the sense that their spatial averages bound their oscillation.

\begin{lemma} 
\label{l.mspoincare2}
Let $s\in (0,d)$, let $\X_s$ be the random variable in Proposition~\ref{p.regularity}, and let $m\in\N$. Then there exist a constant $\sigma_0(m,d,\Lambda) \in \left(0,\frac12 \right]$, such that, for every $r \geq \X_s$, $v\in \A_m(\Rd)$, and $\sigma \in (0,\sigma_0]$ 
there are constants $C(\sigma,m,d,\Lambda)<\infty$ and  $\theta(\sigma,m,d,\Lambda) \in \left(0,\frac12 \right]$ such that
\begin{equation}  \label{e.mspoincare2}
\left\| v \right\|_{L^2(\Psi_r)}^2  \leq C \fint_{B_{r/\theta} } \left| \int_{\Phi_{y,\sigma r}} v(z) \, dz  \right|^2 \, dy \,.
\end{equation}
\end{lemma}

\begin{proof}
Define
\begin{equation*}
w (y,t):= \int_{\Phi_{y,\sqrt{t}}} v(z)\,dz  \, ,
\end{equation*}
which is the solution of the parabolic equation
\begin{equation*}
\left\{
\begin{aligned}
& \partial_t w   - \nabla \cdot \left( \ahom \nabla w \right) = 0 & \mbox{in} & \ \Rd \times (0,\infty), \\
& w  = v & \mbox{on} & \ \Rd \times \{0 \}.
\end{aligned}
\right.
\end{equation*}
By Proposition~\ref{p.regularity}, for every  $r \geq\X_s$, there exists a unique polynomial $q \in \Ahom_m$ such that, for every $S \geq R \geq r$, we have 
\begin{equation} \label{e.upolygrowth1}
\left\| v \right\|_{\underline{L}^2(B_S)} \leq C \left\| q \right\|_{\underline{L}^2(B_S)} \leq C\left(\frac{S}{R}\right)^m \left\| q \right\|_{\underline{L}^2(B_R)} 
\leq C\left(\frac{S}{R}\right)^m \left\| v \right\|_{\underline{L}^2(B_R)} \,.
\end{equation}
In particular, the measure $v^2 \,dx$ is a doubling measure and $v$ has polynomial growth. 

\smallskip

Our starting point is that  Lemma~\ref{l.mspoincare} gives, for all $\sigma \in (0,1]$,  
\begin{multline}
\label{e.intintime00}
 \int_{\Psi_{r}}  \left| v (y) \right|^2\,dy  
 \\ \leq C \int_{\Psi_{r}} \left| w(y,(\sigma r)^2) \right|^2 \,dy 
 +C \int_{0}^{(\sigma r)^2} \int_{\Psi_{r}} \left| \nabla w(y,t) \right|^2\,dy\,dt.
\end{multline}

\smallskip

\emph{Step 1.} We first claim that there is a small $\sigma_0 = \sigma_0(m,d,\Lambda) \in (0,1]$ such that
\begin{equation} \label{e.additivityb100}
C \int_{0}^{(\sigma_0 r)^2} \int_{\Psi_{r}} \left| \nabla w(y,t) \right|^2\,dy\,dt \leq \frac14  \int_{\Psi_r}  \left| v (y) \right|^2\,dy\,.
\end{equation}
so that by absorbing it back onto the left side,~\eqref{e.intintime00} can be improved to 
\begin{equation} \label{e.intintime10}
\int_{\Psi_{r}}  \left| v (y) \right|^2\,dy  \\
 \leq C \int_{\Psi_r } \left| w(y,(\sigma r)^2) \right|^2\,dy 
\end{equation}
for all $\sigma \in (0,\sigma_0]$. 
To prove~\eqref{e.additivityb100}, we first get by H\"older's inequality that 
\begin{equation*} 
\int_{\Psi_{r}} \left| \nabla w(y,t) \right|^2\,dy
  \leq \int_{\Psi_{r}} \int_{\Phi_{y,\sqrt{t}}} \left|\nabla v(z) \right|^2 \, dz \, dy\,.
\end{equation*}
We then notice that the right-hand side can be rewritten with the aid of Fubini's theorem, taking into account the definitions of $\Psi_r$ and $\Phi_{\sqrt{t}}$,  for all $\ep>0$ as
\begin{multline*} 
 \int_{0}^{\ep r^2} \int_{\Psi_{r}} \int_{\Phi_{y,\sqrt{t}}} \left|\nabla v(z) \right|^2  \, dz \, dy \,dt
\\  = \int_{\Psi_{r}} \left|\nabla v (z) \right|^2 \int_{0}^{\ep r^2} \int_{\Phi_{z,\sqrt{t}} } \exp\left(\frac{|\ahom^{\,-\frac12}z|}{r} - \frac{|\ahom^{\,-\frac12}y|}{r}\right)   \, dy \, dt \, dz\,.
\end{multline*}
We analyze the integral in the middle. By the triangle inequality and the ellipticity of $\ahom$ we obtain
\begin{align*} 
&  \int_{0}^{\ep r^2} \int_{\Phi_{z,\sqrt{t}} } \exp\left(\frac{|\ahom^{\,-\frac12}z|}{r} - \frac{|\ahom^{\,-\frac12}y|}{r}\right)   \, dy \, dt 
\\ & \qquad  \leq \int_{0}^{\ep r^2} \int_{ \R^d } \exp\left(\frac{|\ahom^{\,-\frac12}(z-y)|}{r}\right) \Phi_{\sqrt{t}}(z-y)  \, dy \, dt 
\\ & \qquad  \leq C \int_{0}^{\ep r^2} t^{-\frac d2} \int_{ \R^d }  \exp\left(\frac{|y|}{r} - \frac{|y|^2}{t} \right) \, dy \, dt 
\\ & \qquad  \leq C \int_{0}^{\ep r^2} t^{-\frac d2} \left| B_{\frac{2t}{r}}\right| \, dt +   C \int_{0}^{\ep r^2}  \int_{ \Phi_{0,\sqrt{2t}}}\, dy \, dt 
\\ & \qquad  \leq C \ep r^2 \,.
\end{align*}
Therefore, combining the above three displays, we arrive at
\begin{equation} \label{e.mspoincare201}
\int_{0}^{\ep r^2} \int_{\Psi_{r}} \left| \nabla w(y,t) \right|^2\,dy \, dt \leq C\ep r^2 \int_{\Psi_{r}} \left|\nabla v(z) \right|^2 \, dz\,.
\end{equation}
Furthermore, we have the layer-cake formula, for any $g \in L^1(\Psi_r)$ and $\Omega\subset \R^d$, 
\begin{equation} \label{e.psilayer}
\int_{\Omega} \Psi_{r}(z) g(z) \, dz = \frac1{r^{d+1}} \int_{0}^\infty \exp\left(-\frac{\lambda}{r} \right) \int_{\Omega \cap \left(\ahom^{\frac12} B_\lambda\right)} g(z) \, dz \, d\lambda\,.
\end{equation}
Using the Caccioppoli estimate, the doubling property~\eqref{e.upolygrowth1} and the ellipticity of $\ahom$ we deduce that
\begin{equation*} 
 \int_{\ahom^{\frac12} B_{R}} \left|\nabla v(y) \right|^2 \, dy  
 \leq  \frac{C}{R^2} \int_{\ahom^{\frac12}  B_{2R}} \left| v(y) \right|^2 \, dy  
  \leq   \frac{C}{r^2}  \int_{\ahom^{\frac12} B_{R}} \left| v(y) \right|^2 \, dy 
\end{equation*}
for any $R \geq r$. Thus the layer-cake formula~\eqref{e.psilayer} yields
\begin{equation*} 
r^2 \int_{\Psi_{r}} \left|\nabla v(z) \right|^2 \, dz  \leq C \int_{\Psi_{r}} \left| v(z) \right|^2 \, dz  \,.
\end{equation*}
Now our claim~\eqref{e.additivityb100} follows from~\eqref{e.intintime00},~\eqref{e.mspoincare201} and the above display provided we take $\ep =\sigma_0^2$ sufficiently small. 

\smallskip

\emph{Step 2.}
We next show that for any $\ep \in (0,1)$ and $\sigma \in (0,\sigma_0]$ there are constants $C(\ep,\sigma,m,d,\Lambda)< \infty$ and $\theta(\ep,\sigma,m,d,\Lambda) \in (0,1)$ such that
\begin{equation} \label{e.psitail00}
 \int_{\Psi_{r}} \left| w(y,(\sigma r)^2) \right|^2\,dy 
 \leq  C \fint_{B_{r/\theta}} \left| w(y,(\sigma r)^2) \right|^2\,dy +  \ep \int_{\Psi_{r}} \left| v(y) \right|^2 \, dy\,.
\end{equation}
This together with~\eqref{e.intintime10} proves our claim by taking small enough $\ep$, which then also fixes the parameter $\theta$. 
We first decompose the integral on the left as  
\begin{multline}  \label{e.psitail10a} 
 \int_{\Psi_{r}} \left| w(y,(\sigma r)^2) \right|^2\,dy  
 \\  \leq C_\theta  \fint_{B_{r/\theta}} \left| w(y,(\sigma r)^2) \right|^2\,dy + \int_{\R^d \setminus B_{r/\theta}}  \Psi_{r}(y) \left| \int_{\Phi_{y,\sigma r}} v(z) \, dz \right|^2 \, dy \,.
\end{multline}
As in Step 1, with the aid of the triangle and H\"older's inequalities we obtain
\begin{align}  \label{e.psitail10} 
& \int_{\R^d \setminus B_{r/\theta}}   \Psi_{r}(y) \left| \int_{\Phi_{y,c r}} v(z) \, dz \right|^2 \, dy 
\\  \nonumber & \qquad  
\leq \int_{\R^d \setminus B_{r/\theta}} \Psi_{r}(z) \left| v(z) \right|^2  \int_{\Phi_{z,\sigma r}} \exp\left(\frac{|\ahom^{\,-\frac12}(y-z)|}{r}\right)\,dy \, dz 
\\  \nonumber & \qquad  
\leq  C \int_{\R^d \setminus B_{r/\theta}} \Psi_{r}(z) \left| v(z) \right|^2  \, dz \,.
\end{align}
Now the layer-cake formula~\eqref{e.psilayer} and the polynomial growth in~\eqref{e.upolygrowth1} imply
\begin{align} 
\label{e.psitail20}   
\int_{\R^d \setminus B_{r/\theta}} \Psi_{r}(z) \left| v(z) \right|^2  \, dz  & \leq \frac1{r^{d+1}} \int_{r/\theta}^\infty \exp\left(-\frac{\lambda}{r} \right) \int_{\ahom^{\frac12}  B_\lambda} |v(z)|^2 \, dz \, d\lambda 
\\\nonumber & \leq \frac{C}{r} \fint_{B_r} |v(z)|^2 \, dz  \int_{r/\theta}^\infty \exp\left(-\frac{\lambda}{r} \right)  \left(\frac{\lambda}{r}\right)^m \, d\lambda 
\,.
\end{align}
For any given $\tilde \ep \in (0,1)$, we may choose $\theta(\tilde \ep,m,d,\Lambda) > 0$ so small that
$$
\frac1r \int_{r/\theta}^\infty \exp\left(-\frac{\lambda}{r} \right)  \left(\frac{\lambda}{r}\right)^m \, d\lambda = \int_{1/\theta}^\infty \exp\left(-\lambda \right)  \lambda^m \, d\lambda  = \tilde \ep \,.
$$
Therefore, combining~\eqref{e.psitail10}   and~\eqref{e.psitail20} yields
$$
\int_{\R^d \setminus B_{r/\theta}}   \Psi_{r}(y) \left| \int_{\Phi_{y,\delta r}} v(z) \, dz \right|^2 \, dy  \leq C \tilde \ep \fint_{B_{r} } |v(z)|^2 \, dz
$$
with $C$ independent of $\tilde \ep$. Inserting this into~\eqref{e.psitail10a}, we deduce that~\eqref{e.psitail00} holds. This finishes the proof.
\end{proof}

\section{Higher-order energy quantities}
\label{s.additivestructure}

In this section, we introduce the main object of study in the paper and record some of its basic properties. It is a higher-order version of the quantity $J$ defined in the introduction, in which we replace the space $\A_1$ by $\A_k$ and allow $p,q$ to take values in $\Ahom_k$ rather than $\Rd$. 

\begin{definition}
\label{d.Jk}
For every~$k\in\N$, $z\in\Rd$, $r\geq 1$ and $p,q\in \overline \A_k$, we define
\begin{equation}
\label{e.Jk}
J_k(z,r,p,q) := \max_{u \in \A_k}  \int_{\Phi_{z,r}}\left(-\frac12 \nabla u \cdot\a  \nabla u - \nabla p\cdot \a\nabla u + \nabla u\cdot \ahom \nabla q \right),
\end{equation}
The function achieving the maximum in the definition of $J_k$ is denoted by
\begin{equation}
\label{e.def.v}
v_k(\cdot,z,r,p,q) := \mbox{element of $\A_k$ achieving the maximum for $J_k(z,r,p,q)$.}
\end{equation}
Note that existence of the maximizer~$v_k(\cdot,z,r,p,q)$ is ensured by the fact that~$\A_k$ is a finite-dimensional subspace of $L^2(\Rd,\Phi_{z,r}\,dx)$, by Proposition~\ref{p.regularity}. It is specified uniquely up to an additive constant by the uniform concavity of the maximization problem in the definition of~$J_k$. 
\end{definition}

\smallskip

In most of the paper, $k\in\N$ is a fixed integer. We remark  that Theorems~\ref{t.firstcorrectors} and~\ref{t.main1}, except for the second localization statement, can be proved if one chooses $k=1$ throughout. To keep the notation simple, we often suppress dependence on $k$, writing for instance $J(z,r,p,q)$ and $v(\cdot,z,r,p,q)$ instead of $J_k(z,r,p,q)$ and $v_k(\cdot,z,r,p,q)$, respectively. 

\smallskip

It is immediate that the map $(p,q) \mapsto J(z,r,p,q)$ is a quadratic form, and that the map $(p,q) \mapsto v(\cdot,z,r,p,q)$ is a linear map between the vector spaces $\Ahom_k\times \Ahom_k$ and $\A_k$. Notice that by~\eqref{e.ue} and Young's inequality, $J(z,r,p,q)$ is uniformly bounded: for every $z \in \Rd$, $r \ge 1$ and $p,q \in \Ahom_k$,
\begin{equation}
\label{e.bounded-J}
0 \le J(z,r,p,q) \le \Lambda \int_{\Phi_{z,r}} \Ll( |\nabla p|^2 + |\nabla q|^2 \Rr) .
\end{equation}

\smallskip

We next record the first and second variations of the optimization problem in the definition of $J$. For every $z\in\Rd$, $r\geq 1$, $p,q\in\Ahom_k$ and $u \in \A_k$, we write
\begin{equation*}
\mcl J(u,z,r,p,q) := \int_{\Phi_{z,r}}\left(-\frac12 \nabla u \cdot\a  \nabla u - \nabla p\cdot \a\nabla u + \nabla u\cdot \ahom \nabla q \right).
\end{equation*}

\begin{lemma}
\label{l.variations}
For every $z\in\Rd$, $r\geq 1$, $p,q\in\Ahom_k$ and $w \in \A_k$,
\begin{equation}
\label{e.first-var}
\int_{\Phi_{z,r}} \nabla w \cdot \a\nabla v(\cdot,z,r,p,q) = \int_{\Phi_{z,r}} \left( -\nabla p \cdot \a \nabla w + \nabla w \cdot \ahom \nabla q \right),
\end{equation}
and
\begin{equation}
\label{e.second-var}
J(z,r,p,q) - \mcl J(v(\cdot,z,r,p,q) + w, z,r,p,q) = \int_{\Phi_{z,r}} \frac 1 2 \nabla w \cdot \a \nabla w.
\end{equation}
\end{lemma}
\begin{proof}
For every $t \in \R$, let $v_t(\cdot) := v(\cdot,z,r,p,q) + t \, w(\cdot)$. We have
\begin{align*}
0 
&\leq \mcl J(v_0,z,r,p,q) - \mcl J(v_t,z,r,p,q) \\
& = \int_{\Phi_{z,r}} \Ll(\frac{t^2}2 \nabla w\cdot\a\nabla w + t  \left(\nabla w \cdot \a \nabla v_0 + \nabla p\cdot \a \nabla w -  \nabla w \cdot \ahom \nabla q\right)\Rr).
\end{align*}
Sending $t\to 0$ yields
\begin{equation*}
  \fint_{U} \left(\nabla w \cdot \a \nabla v_0 + \nabla p\cdot \a \nabla w -  \nabla w \cdot \ahom \nabla q\right) = 0,
\end{equation*}
which is~\eqref{e.first-var}.
The previous identity with $t=1$ then gives~\eqref{e.second-var}. 
\end{proof}

Identity~\eqref{e.first-var} implies that for $p,p',q,q' \in \Ahom_k$,
\begin{align*}
& \int_{\Phi_{z,r}} \left( -\nabla p \cdot \a \nabla v(\cdot,z,r,p',q') + \nabla v(\cdot,z,r,p',q') \cdot \ahom \nabla q \right) \\
& \qquad = \int_{\Phi_{z,r}}  \nabla v(\cdot,z,r,p,q) \cdot \a \nabla v(\cdot,z,r,p',q')  \\
& \qquad  = \int_{\Phi_{z,r}} \left( -\nabla p' \cdot \a \nabla v(\cdot,z,r,p,q) + \nabla v(\cdot,z,r,p,q) \cdot \ahom \nabla q' \right),
\end{align*}
and, in particular,
\begin{align}
\label{e.J-energy}
J(z,r,p,q) & = \frac12 \int_{\Phi_{z,r}} \nabla v(\cdot,z,r,p,q) \cdot\a  \nabla v(\cdot,z,r,p,q) \\
\notag
& = \frac12 \int_{\Phi_{z,r}} \left( -\nabla p \cdot \a \nabla v(\cdot,z,r,p,q) + \nabla v(\cdot,z,r,p,q) \cdot \ahom \nabla q \right).
\end{align}
We deduce from the first line above and \eqref{e.bounded-J} that
\begin{equation} 
\label{e.bounded-u}
\left\| \nabla v(\cdot,z,r,p,q) \right\|_{L^2(\Phi_{z,r})} \leq C \left( \left\| \nabla p \right\|_{L^2(\Phi_{z,r})} + \left\| \nabla q \right\|_{L^2(\Phi_{z,r})}  \right). 
\end{equation}
Moreover,
\begin{multline}
\label{e.polarization}
J(z,r,p+p',q+q') - J(z,r,p,q) - J(z,r,p',q') \\
= \int_{\Phi_{z,r}}  \left( -\nabla p' \cdot \a \nabla v(\cdot,z,r,p,q) + \nabla v(\cdot,z,r,p,q) \cdot \ahom \nabla q' \right).
\end{multline}
Since $(p,q) \mapsto J(z,r,p,q)$ is a quadratic form, its gradient $\nabla J(z,r,p,q)$ is a linear form on $\Ahom_k \times \Ahom_k$, and 
\begin{multline}
\nabla J(z,r,p,q)(p',q') \\
 = \int_{\Phi_{z,r}} \left( -\nabla p' \cdot \a \nabla v(\cdot,z,r,p,q) + \nabla v(\cdot,z,r,p,q) \cdot \ahom \nabla q' \right) .
 \label{e.gradient-J}
\end{multline}

The previous identity identifies the spatial averages of the gradient and the flux of $v(\cdot,z,r,p,q)$ with $\nabla J(z,r,p,q)$. In particular, it tells that the fluctuations of $J(z,r,p,q)$ are of the same order as the fluctuations of the spatial averages of the gradient and of the flux of its maximizers. This observation is the basis of the usefulness of $J$ itself and lies at the foundation of the arguments in Section~\ref{s.additivity}. 

\begin{remark}
\label{r.gradient}
The identity 
\begin{equation}
\label{e.polar.grad}
\nabla J(z,r,p,q)(p',q') = J(z,r,p+p',q+q') - J(z,r,p,q) - J(z,r,p',q')
\end{equation}
following from \eqref{e.polarization} and \eqref{e.gradient-J} holds with $J$ replaced by any quadratic form on $\Ahom_k \times \Ahom_k$. In particular, if $(p,q) \mapsto \td J(p,q)$ is any quadratic form on $\Ahom_k \times \Ahom_k$, then
$$
\sup_{p,q,p',q' \in \Ahom_k(\Phi_r)} \Ll|\nabla \td J(p,q)(p',q')\Rr| \le 6 \sup_{p,q \in \Ahom_k(\Phi_r)} |\td J(p,q)|.
$$
\end{remark}

\smallskip

We next observe that $J$ responds quadratically to perturbations near its maximum. This is obviously equivalent to~\eqref{e.second-var}, but we give it a separate lemma anyway for readability since it is used in this form many times in the paper. 

\begin{lemma}
\label{l.quadratic-response}
For every $z \in \Rd$, $r \ge 1$ $p,q\in\Ahom_k$ and $w_1,w_2 \in \A_k$,
\begin{equation}
\label{e.lowerUC}
 \frac14 \left\| \nabla w_1 - \nabla w_2 \right\|_{{L}^2(\Phi_{z,r})}^2
\leq  2J(z,r,p,q) - \mcl J( w_1,z,r,p,q) - \mcl J(w_2,z,r,p,q),
\end{equation}
\begin{equation}
\label{e.upperUC}
2\mcl J(w_1,z,r,p,q) - \mcl J(w_2,z,r,p,q) - J(z,r,p,q)
 \leq \frac{\Lambda}{4}  \left\| \nabla w_1 - \nabla w_2 \right\|_{{L}^2(\Phi_{z,r})}^2.
\end{equation}
\end{lemma}
\begin{proof}
For every $v_1,v_2 \in \A_k$,
\begin{multline*}
2\mcl J\left(\frac{v_1+v_2}2,z,r,p,q\right) - \mcl J\left(v_1,z,r,p,q\right) - \mcl J\left(v_2,z,r,p,q\right) 
\\ = \frac14 \int_{\Phi_{z,r}}  (\nabla v_1 - \nabla v_2) \cdot \a (\nabla v_1 - \nabla v_2).
\end{multline*}
We get~\eqref{e.lowerUC} from this by choosing $v_1 = w_1$, $v_2 = w_2$ and using the maximality of $J(z,r,p,q)$, while~\eqref{e.upperUC} follows similarly by choosing $v_1 = w_2$ and $v_2 = 2w_1 -w_2$. 
\end{proof}
\begin{lemma}
\label{l.C11}
For every $z \in \Rd$, $r\ge 1$ and $p_1, p_2, q_1, q_2 \in \Ahom_k$,
\begin{multline*}
0 \le J(z,r,p_1,q_1) + J(z,r,p_2,q_2) - 2 J \Ll( z,r,\frac{p_1 + p_2} 2, \frac{q_1 + q_2} 2 \Rr)\\
 \le \Lambda \int_{\Phi_{z,r}} \Ll( |\nabla (p_1 - p_2)|^2 + |\nabla (q_1 - q_2)|^2 \Rr) .
\end{multline*}
\end{lemma}
\begin{proof}
Since $(p,q) \mapsto J(z,r,p,q)$ is a quadratic form,
\begin{multline}
\label{e.parallelogram}
J(z,r,p_1,q_1) + J(z,r,p_2,q_2) - 2 J \Ll( z,r,\frac{p_1 + p_2} 2, \frac{q_1 + q_2} 2 \Rr) \\
= 2 J \Ll( z,r,\frac{p_1 - p_2} 2, \frac{q_1 - q_2} 2 \Rr) .
\end{multline}
The result then follows from \eqref{e.bounded-J}.
\end{proof}

As soon as $r$ is larger than a fixed random scale, the function $(p,q) \mapsto J(z,r,p,q)$ is uniformly convex in each variable separately. We record a much stronger form of this statement, and postpone its proof to Section~\ref{s.basecase}.

\begin{proposition}
\label{p.basecase}
There exist $\ep_0(d,\Lambda)\in\left(0,\frac12\right]$ and, for every $s \in (0,d)$, a random variable $\mathcal{Y}_s$ and a constant $C(s,k,d,\Lambda) < \infty$ satisfying 
\begin{equation}
\label{e.Ys}
\Y_s = \O_s\left( C \right)
\end{equation}
and such that, for every $r\geq \mathcal{Y}_s$ and $p,q\in \Ahom_k(\Phi_r)$,
\begin{equation}
\label{e.Jminimalrad}
\left| J(0,r,p,q) - \int_{\Phi_r} \frac12\left( \nabla p -\nabla q\right)\cdot  \ahom  \left( \nabla p -\nabla q\right)\right| 
\leq C r^{-\ep_0 (d-s)}.
\end{equation}
\end{proposition}

Proposition~\ref{p.basecase} has the following consequence for the uniform convexity of $J$.

\begin{corollary}
\label{c.convexity}
For every $s \in (0,d)$, there exist a constant $C(s,k,d,\Lambda) < \infty$ and a random variable $\mcl Y_s$ satisfying \eqref{e.Ys} and such that, for every $r \ge \mcl Y_s$ and $p$, $p_1$, $p_2$, $q$, $q_1$, $q_2 \in \Ahom_k$,
\begin{equation*}
J(0,r,p_1,q) + J(0,r,p_2,q) - 2 J \Ll( 0,r,\frac{p_1 + p_2} 2, q \Rr) \\
\ge  \frac 1 2 \int_{\Phi_{r}}  |\nabla (p_1 - p_2)|^2 ,
\end{equation*}
\begin{equation*}
J(0,r,p,q_1) + J(0,r,p,q_2) - 2 J \Ll( 0,r,p, \frac{q_1 + q_2} 2 \Rr) \\
\ge \frac 1 2 \int_{\Phi_{r}}  |\nabla (q_1 - q_2)|^2.
\end{equation*}
Moreover, there exists $r_0(k,d,\Lambda) < \infty$ such that for every $r \ge r_0$ and $p$, $p_1$, $p_2$, $q$, $q_1$, $q_2 \in \Ahom_k$,
\begin{equation*}
\E\Ll[J(0,r,p_1,q)\Rr] + \E\Ll[J(0,r,p_2,q)\Rr] - 2 \E\Ll[J \Ll( 0,r,\frac{p_1 + p_2} 2, q \Rr)\Rr] \\
\ge  \frac 1 4 \int_{\Phi_{r}}  |\nabla (p_1 - p_2)|^2 ,
\end{equation*}
\begin{equation*}
\E\Ll[J(0,r,p,q_1)\Rr] + \E\Ll[J(0,r,p,q_2)\Rr] - 2 \E\Ll[J \Ll( 0,r,p, \frac{q_1 + q_2} 2 \Rr)\Rr] \\
\ge \frac 1 4 \int_{\Phi_{r}}  |\nabla (q_1 - q_2)|^2.
\end{equation*}

\end{corollary}
\begin{proof}
By Proposition~\ref{p.basecase} and homogeneity, there exists $C(s,k,d,\Lambda) < \infty$ and $\mcl Y_s$ satisfying \eqref{e.Ys} such that for every $r \ge \mcl Y_s + C$ and $p \in \Ahom_k$,
$$
J(0,r,p,0) \ge \frac 1 4 \int_{\Phi_r} \nabla p \cdot \ahom \nabla p.
$$
The first inequality then follows by \eqref{e.parallelogram}. The second inequality is obtained in the same way. Since $J\ge 0$, we also have
$$
\E[J(0,r,p,0)] \ge \frac 1 4 \P[r \ge \mcl Y_s + C] \int_{\Phi_r} \nabla p \cdot \ahom \nabla p,
$$
so that for $r$ sufficiently large,
$$
\E[J(0,r,p,0)] \ge \frac 1 8  \int_{\Phi_r} \nabla p \cdot \ahom \nabla p,
$$
and we obtain the last two inequalities as before.
\end{proof}

The core of the analysis in Section~\ref{s.additivity} involves comparing the maximizers of $J$ on different scales. When performing this comparison, it is appropriate to select values of $p$ and $q$ (depending on the scale) so that the gradients and fluxes of the maximizers, e.g.~$v(\cdot,0,r,p_r,q_r)$ and $v(\cdot,0,R,p_R,q_R)$, have the same expected spatial averages. It is natural, therefore, to define linear maps $L_{z,r},L^*_{z,r} :\Ahom_k \to \Ahom_k$ which have the property that the expectation of the spatial average of~$\nabla v(\cdot,z,r,L_{z,r}^*p,L_{z,r}q)$ is $\nabla q-\nabla p$. We cannot do this exactly, but the following definition is motivated by Lemma~\ref{e.Lproperties} below. 

\begin{definition} \label{d.Ls}
For each $z \in \Rd$, $r \ge r_0(k,d,\Lambda)$ and $q \in \Ahom_k$, we let $L_{z,r}(q)\in \Ahom_k$ denote the unique (up to an additive constant) minimum of the (deterministic) function
\begin{equation}
\label{e.function-def-L}
q'  \longmapsto  \E[J(z,r,0,q')] - \int_{\Phi_{z,r}} \nabla q \cdot \ahom \nabla q'.
\end{equation}
Likewise, for each $p  \in \Ahom_k$, we let $L_{z,r}^*(p)$ denote the unique (up to an additive constant) minimum of the (deterministic) function
\begin{equation}
\label{e.function-def-L-star}
p'  \longmapsto  \E[J(z,r,p',0)] - \int_{\Phi_{z,r}} \nabla p \cdot \ahom \nabla p'.
\end{equation}
We fix the additive constants by requiring that $p(z) = L_{z,r}^*(p)(z)$ and $q(z) = L_{z,r}(q)(z)$. Note that the mappings $p \mapsto L_{z,r}^*(p)$ and $q\mapsto L_{z,r}(q)$ are linear. We define
\begin{equation}
\label{e.def.u}
u(\cdot,z,r,p,q) := v \big(\cdot,z,r,L^*_{z,r}p, L_{z,r}q \big),
\end{equation}
and
\begin{equation}
\label{e.def.tilde-J}
\J(z,r,p,q):= \mcl J\big(u(\cdot,z,r,p,q),z,r,p,q \big).
\end{equation}
Where the context requires that we make the dependence on $k$ explicit, we write $L_{z,r,k}$, $I_k$, $u_k$, and so forth.
\end{definition}
 
\begin{lemma}
\label{e.Lproperties}
For every $z \in \Rd$, $r \ge r_0(k,d,\Lambda)$ and $p,q,p',q' \in \Ahom_k$,
\begin{equation}
\label{e.I-spat-av}
\E \Ll[ \int_{\Phi_{z,r}} \nabla u(\cdot,z,r,0,q) \cdot \ahom \nabla q' \Rr]  = \int_{\Phi_{z,r}} \nabla q \cdot \ahom \nabla q'
\end{equation}
and
\begin{equation}
\label{e.I-spat-flux}
\E \Ll[ \int_{\Phi_{z,r}} \nabla p' \cdot \a \nabla u(\cdot,z,r,p,0)   \Rr]  = - \int_{\Phi_{z,r}} \nabla p \cdot \ahom \nabla p'.
\end{equation}
Moreover,
\begin{multline}
\label{e.EJL}
\E \left[ J\left(z,r,L_{z,r}^*p,L_{z,r}q \right)  \right] \\ 
=  \frac 12 \int_{\Phi_{z,r}} \nabla p \cdot \ahom \nabla L_{z,r}^*p 
+ \frac12 \int_{\Phi_{z,r}} \nabla q \cdot \ahom \nabla L_{z,r}q \\
+ \E \left[ \int_{\Phi_{z,r}} \nabla v(\cdot,0,r,L_{z,r}^*p,0) \cdot \a  \nabla v(\cdot,0,r,0,L_{z,r}q) \right].
\end{multline}
\end{lemma}
\begin{proof}
For $p,q,p',q' \in \Ahom_k$, the Euler-Lagrange equations for the minimization problems in Definition~\ref{d.Ls} read as
\begin{equation}
\label{e.obvious-EL}
\E[\nabla_q J(z,r,0,L_{z,r} q)(q')] = \int_{\Phi_{z,r}} \nabla q \cdot \ahom \nabla q'
\end{equation}
and
\begin{equation}
\label{e.obvious-EL.star}
\E[\nabla_p J(z,r,L^*_{z,r} p,0)(p')] = - \int_{\Phi_{z,r}} \nabla p \cdot \ahom \nabla p',
\end{equation}
respectively. The identities~\eqref{e.I-spat-av} and~\eqref{e.I-spat-flux} follow by these and~\eqref{e.gradient-J}. The final identity~\eqref{e.EJL} is obtained from~\eqref{e.J-energy},~\eqref{e.I-spat-av} and~\eqref{e.I-spat-flux}.
\end{proof}

We next show that both $L_{z,r}$ and $L^*_{z,r}$ are non-negative symmetric operators with respect to the natural scalar product.
\begin{lemma} 
\label{l.Lsymm}
For every $r \geq r_0(k,d,\Lambda)$ and $p,p' \in \Ahom_k$, we have that $L_{z,r}$ and $L^*_{z,r}$ are non-negative and satisfy
\begin{equation}
 \label{e.Lsymm1}
\int_{\Phi_{z,r}} \nabla L_{z,r} p \cdot \ahom \nabla p'
 = \int_{\Phi_{z,r}}  \nabla p \cdot \ahom \nabla L_{z,r}p'
\end{equation}
and
\begin{equation} 
\label{e.Lsymm2}
\int_{\Phi_{z,r}} \nabla L_{z,r}^* p \cdot \ahom \nabla p' 
= \int_{\Phi_{z,r}}  \nabla p \cdot \ahom \nabla L_{z,r}^* p'\,.
\end{equation}
\end{lemma}
\begin{proof}
Since 
\begin{equation*} 
\E\left[J(z,r,0,L_{z,r} p)\right] = \frac12 \int_{\Phi_{z,r}} \nabla p \cdot \ahom \nabla L_{z,r} p\,,
\end{equation*}
the non-negativity of $L_{z,r}$ follows from the fact that the  left side is non-negative. 
To obtain the symmetry of $L_{z,r}$, we use Lemma~\ref{e.Lproperties} and the first variation~\eqref{e.first-var}, which give
\begin{align*} 
\int_{\Phi_{z,r}}  \nabla p \cdot \ahom \nabla L_{z,r} p'   & = \E \left[  \int_{\Phi_{z,r}}  \nabla v(\cdot,z,r,0,L_{z,r}p) \cdot \ahom\nabla L_{z,r} p' \right] 
\\ & = \E \left[  \int_{\Phi_{z,r}}  \nabla v(\cdot,z,r,0,L_{z,r}p) \cdot \a\nabla v(\cdot,z,r,0,L_{z,r} p') \right]
\\ & = \E \left[  \int_{\Phi_{z,r}}  \nabla L_{z,r}p \cdot \ahom\nabla v(\cdot,z,r,0,L_{z,r} p')   \right] 
\\ & = \int_{\Phi_{z,r}}  \nabla L_{z,r} p \cdot \ahom \nabla p'    \,.
\end{align*}
This is~\eqref{e.Lsymm1}, and a similar computation gives~\eqref{e.Lsymm2}.
\end{proof}

\begin{lemma}
\label{l.snappingtheLs}
Suppose that $K\geq 1$ and $\theta >0$ are such that, for every $z\in\Rd$, $r\geq 1$ and $p,q \in \Ahom_k(\Phi_{z,r})$, 
\begin{equation}
\label{e.quad-close-hom}
\left| \E \left[ J(z,r,p,q) \right] - \int_{\Phi_{z,r}} \frac 1 2 \nabla (q-p) \cdot \ahom \nabla (q-p) \right| \leq Kr^{-\theta}. 
\end{equation}
Then for every $z \in \Rd$, $r \ge r_0(k,d,\Lambda)$ and $p,q\in\Ahom_k(\Phi_{z,r})$,
\begin{equation}
\label{e.snapping}
\left\|\nabla \Ll(q - L_{z,r} q\Rr)\right\|_{L^2(\Phi_{z,r})} 
+ \left\|\nabla \Ll(p - L_{z,r}^* p\Rr)\right\|_{L^2(\Phi_{z,r})} 
\leq 24 K r^{-\theta}.
\end{equation}
\end{lemma}
\begin{proof}
We only prove inequality \eqref{e.snapping} for $L_{z,r}$; the argument for $L^*_{z,r}$ is identical. 
By \eqref{e.quad-close-hom} and Remark~\ref{r.gradient}, we have, for every $p,q,p',q' \in \Ahom_k(\Phi_{z,r})$,
$$
\Ll| \E[\nabla J(z,r,p,q)(p',q')] - \int_{\Phi_{z,r}} \nabla(q'-p') \cdot \ahom \nabla (q-p) \Rr| \le 6Kr^{-\theta}. 
$$
In particular, by homogeneity, we have for every $q' \in \Ahom_k$ that 
$$
\Ll|\E[\nabla_q J(z,r,0,q)(q') ]- \int_{\Phi_{z,r}} \nabla q \cdot \ahom \nabla q'\Rr| \le 6 K r^{-\theta}  \, \|\nabla q'\|_{L^2(\Phi_{z,r})}. 
$$
By definition of $L_{z,r} q$, see \eqref{e.obvious-EL}, and linearity of $q \mapsto \nabla_q J(z,r,0,q)$, we deduce that for every $q' \in \Ahom_k$,
$$
\Ll|\E[\nabla_q J(z,r,0,q - L_{z,r} q )(q')]\Rr| \le 6 K r^{-\theta} \, \|\nabla q'\|_{L^2(\Phi_{z,r})}. 
$$
Since $J$ is a quadratic form, we have $\nabla_q J(z,r,0,q)(q) = 2 J(z,r,0,q)$, see \eqref{e.polar.grad}, and hence
$$
\E[J(z,r,0,q - L_{z,r} q )] \le 3 K r^{-\theta} \, \|\nabla(q - L_{z,r} q)\|_{L^2(\Phi_{z,r})}. 
$$
By the fourth inequality in Corollary~\ref{c.convexity} with $p = 0$ and $q_1 = q-L_{z,r} q = -q_2$, we get that for $r$ sufficiently large,
$$
\frac 1 4 \|\nabla(q - L_{z,r} q)\|_{L^2(\Phi_{z,r})}^2 \le 3 K r^{-\theta} \|\nabla(q - L_{z,r} q)\|_{L^2(\Phi_{z,r})},
$$
which is the announced result.
\end{proof}

\begin{corollary}
\label{c.L-identity}
There exists $\eps_1(d,\Lambda) > 0$ and $C(k,d,\Lambda) < \infty$ such that for every $z \in \Rd$, $r \ge r_0(k,d,\Lambda)$ 
and $p \in \Ahom_k(\Phi_{z,r})$,
\begin{equation*} \label{}
\left\|\nabla \Ll(q - L_{z,r} q\Rr)\right\|_{L^2(\Phi_{z,r})} 
+ \left\|\nabla \Ll(p - L_{z,r}^* p\Rr)\right\|_{L^2(\Phi_{z,r})}  \le C r^{-\eps_1}.
\end{equation*}
\end{corollary}
\begin{proof}
This is immediate from Proposition~\ref{p.basecase} and Lemma~\ref{l.snappingtheLs}.
\end{proof}

\begin{remark}
\label{r.L-bijection}
We may think of the function in \eqref{e.function-def-L} as defined on the quotient space $\Ahom_k / \Ahom_0$, and of $L_{z,r}$ as a mapping from $\Ahom_k / \Ahom_0$ to itself. Since $p \mapsto \|\nabla p\|_{L^2(\Phi_{z,r})}$ is a norm on this space, Corollary~\ref{c.L-identity} shows that $L_{z,r}$ : $\Ahom_k / \Ahom_0 \to \Ahom_k / \Ahom_0$ is a small perturbation of the identity as $r$ tends to infinity, and thus, in particular, that this mapping is bijective with bounded inverse for $r \ge r_0(k,d,\Lambda)$.
\end{remark}

We can easily express $I$ in terms of $J$ or vice versa. Indeed, by \eqref{e.second-var}, the definition of $u(\cdot,z,r,p,q)$ in \eqref{e.def.u} and \eqref{e.J-energy}, we have, for every $z \in \Rd$, $r\ge 1$ and $p,q \in \Ahom_k$,
\begin{align}
\label{e.I.as.J}
\lefteqn{
J(z,r,p,q) - I(z,r,p,q) 
} \qquad  &  \\
& = \frac 1 2 \int_{\Phi_{z,r}} \nabla((v-u)(\cdot,z,r,p,q)) \cdot \a \nabla((v-u)(\cdot,z,r,p,q)) \notag \\
& = J(z,r,p-L_{z,r}^*p,q-L_{z,r}q) .\notag
\end{align}
As a consequence, by \eqref{e.bounded-J} and Corollary~\ref{c.L-identity}, for $r \ge r_0(k,d,\Lambda)$,
\begin{equation}
\label{e.bounded-I}
|I(z,r,p,q)| \le  \Lambda \int_{\Phi_{z,r}} \Ll( |\nabla p|^2 + |\nabla q|^2 \Rr).
\end{equation}
Conversely, by iteration of the identity above, 
\begin{equation}
\label{e.J.as.I}
J(z,r,p,q) = \sum_{l = 0}^{\infty} I(z,r,(L^*_{z,r} - \mathrm{Id})^l p, (L_{z,r} - \mathrm{Id})^l q),
\end{equation}
with, for every $p,q \in \Ahom_k(\Phi_{z,r})$ and $r \ge r_0(k,d,\Lambda)$,
\begin{equation}
\label{e.tail.I.series}
\Ll| \sum_{l = m}^\infty I(z,r,(L^*_{z,r} - \mathrm{Id})^l p, (L_{z,r} - \mathrm{Id})^l q) \Rr| \le C r^{-\ep m},
\end{equation}
by Corollary~\ref{c.L-identity} and \eqref{e.bounded-I}.

\begin{definition}
\label{def.r0}
From now on, we fix the constant $r_0(k,d,\Lambda) < \infty$ to be the maximum of all the $r_0$'s given in the statements above, so that all of them hold simultaneously for $r \ge r_0$.
\end{definition}

We finally record a technical lemma that is used several times in Sections~\ref{s.additivity} and~\ref{s.localization}, which is essentially a version of the quadratic response (Lemma~\ref{l.quadratic-response}) for the quantity $I(z,r,p,q)$. 

\begin{lemma}
\label{l.Iquadresponse}
Suppose that $\theta\in (0,\infty)$ is such that, for every $z\in\Rd$, $r\geq r_0$ and $p,q\in \Ahom_k(\Phi_{z,r})$,
\begin{equation}
\label{e.assLtheta}
\left\|  \nabla L_{z,r}q - \nabla q \right\|_{L^2(\Phi_{z,r})}   + \left\|  \nabla L_{z,r}^*p - \nabla p \right\|_{L^2(\Phi_{z,r})} 
\leq Cr^{-\theta}. 
\end{equation}
Then, for every $z\in\Rd$, $r\geq r_0$, $p,q\in \Ahom_k(\Phi_{z,r})$ and $w\in \A_k$,
\begin{multline}
\label{e.Iquadresponse}
\left| I(z,r,p,q) - \mathcal{J}(\nabla w,z,r,p,q) \right| \\
\leq C\left\| \nabla u(\cdot,z,r,p,q) - \nabla w \right\|_{L^2(\Phi_{z,r})}^2 + Cr^{-\theta} \left\| \nabla u(\cdot,z,r,p,q) - \nabla w \right\|_{L^2(\Phi_{z,r})}.
\end{multline}
\end{lemma}
\begin{proof}
Fix $p,q\in \Ahom_m(\Phi_{z,r})$ and denote $u:= u(\cdot,z,r,p,q)$. Observe that the assumption~\eqref{e.assLtheta} and quadratic response (Lemma~\ref{l.quadratic-response}) yield
\begin{align*}
\lefteqn{
I(z,r,p,q)  = \mathcal{J} ( \nabla u,z,r,p,q)
} \qquad & \\ 
& = \mathcal{J} \left( \nabla u, z, r, L_{z,r}^*p, L_{z,r}q \right)
+ \int_{\Phi_{z,r}} \left( \a \nabla (L_{z,r}^*p - p) - \ahom \nabla(L_{z,r}q-q) \right) \cdot \nabla u \\
& \leq  \mathcal{J} \left( \nabla w, z, r, L_{z,r}^*p, L_{z,r}q \right) + C \int_{\Phi_{z,r}} \left| \nabla u -\nabla w \right|^2 \\
& \qquad + \int_{\Phi_{z,r}} \left( \a \nabla (L_{z,r}^*p - p) - \ahom \nabla(L_{z,r}q-q) \right) \cdot \nabla u \\
& = \mathcal{J}(\nabla w, z,r,p,q)
 + C \int_{\Phi_{z,r}} \left| \nabla u - \nabla w \right|^2 \\
& \qquad + \int_{\Phi_{z,r}} \left( \a \nabla (L_{z,r}^*p - p) - \ahom \nabla(L_{z,r}q-q) \right) \cdot \left( \nabla u -  \nabla w \right) \\
& \leq \mathcal{J}(\nabla w, z,r,p,q) + C \left\| \nabla u - \nabla w \right\|_{L^2(\Phi_{z,r})}^2 + C r^{-\theta}  \left\| \nabla u - \nabla w \right\|_{L^2(\Phi_{z,r})}.
\end{align*}
The reverse inequality also holds by almost the same computation and so we obtain the statement of the lemma.
\end{proof}

\section{The bootstrap outline}
\label{s.bootstrap}

In this section, we give the bootstrap argument at the core of the proof of Theorem~\ref{t.main1}. We begin with some definitions that are used throughout the rest of the paper. 

\smallskip

\begin{definition}[$\Add_k(s,\alpha)$]

For each $k\in\N$, $s,\alpha\in (0,\infty)$, we let $\Add_k(s,\alpha)$ denote the statement that there exists a constant $C(k,s,\alpha,d,\Lambda)<\infty$ such that, for every $z \in \Rd$, $r_0 \leq r < R$ and $p,q\in \Ahom_k(\Phi_R)$,
\begin{equation}
\label{e.def.add}
\J_k(z,R,p,q) = \int_{\Phi_{z,\sqrt{R^2 - r^2}}} \J_k(\cdot,r,p,q) + \O_s \left(C  r^{-\alpha} \right).
\end{equation}
\end{definition}

\begin{definition}[$\Fluc_k(s,\alpha)$]
For each $k\in\N$ and $s,\alpha\in (0,\infty)$,
we let $\Fluc_k(s,\alpha)$ denote the statement that there exists a constant $C(k,s,\alpha,d,\Lambda)<\infty$ such that, for every $z \in \Rd$, $r\geq1$ and $p,q\in \Ahom_k(\Phi_r)$,
\begin{equation*}
J_k(z,r,p,q) = \E \left[J_k(z,r,p,q)\right] + \O_s\left(Cr^{-\alpha}\right). 
\end{equation*}
\end{definition}

\begin{definition}[$\Dual_k(\alpha)$]
For each $k\in\N$ and $\alpha\in (0,\infty)$,
we let $\Dual_k(\alpha)$ denote the statement that there exists $C(k,\alpha,d,\Lambda)<\infty$ such that, for every $z \in \Rd$ and $r\geq 1$,
\begin{equation*} \label{}
\E \left[ \sup_{w\in\A_k(\Phi_{z,r})} \left|  \int_{\Phi_{z,r}}  \left( \a(x) - \ahom \right) \nabla w(x) \,dx \right| \right] \leq Cr^{-\alpha} . 
\end{equation*}
\end{definition}

\begin{definition}[$\Loc_k(s,\delta,\alpha)$]
For each $k\in\N$ and $s,\delta,\alpha\in (0,\infty)$, 
we denote by $\Loc_k(s,\delta,\alpha)$ the statement that there exists $C(k,s,\delta,\alpha,d,\Lambda) < \infty$ and, for every $z \in \Rd$, $r \geq 1$ and $p, q \in \Ahom_k(\Phi_{z,r})$, an $\mcl F\left(B_{r^{1+\de}}(z)\right)$-measurable random variable $J^{(\delta)}_k(z,r,p,q)$ such that
\begin{equation*}
J_k(z,r,p,q) = J^{(\delta)}_k(z,r,p,q) + \O_s(C r^{-\alpha}).
\end{equation*}
\end{definition}


We also let $\Add_k(s,\alpha-)$ denote the statement that $\Add_k(s,\beta)$ holds for every $\beta< \alpha$. We define $\Fluc_k(s,\alpha-)$, $\Dual_k(\al-)$ and $\Loc_k(s,\delta,\alpha-)$ similarly. 

\smallskip

Most of the effort in the paper (and the entirety of this and the four following sections) is focused on the proof of the following theorem, which is close to the statements of the main results. The proofs of the latter are finally completed in Section~\ref{s.finalcorrectors}.

\begin{theorem}
\label{t.additivitybelowd}
For every $s<1$ and $\delta >0$, we have that the following hold:
\begin{equation} 
\label{e.k=1}
\left\{ 
\begin{aligned}
& \Add_1\left(s,d\right), \quad 
\Fluc_1\left(2s,\tfrac d2 \right), \\
& \Dual_1\left(\tfrac d2\right), \quad
\Loc_1\left(2s,\delta, \left( \tfrac d2(1+\delta) + \delta \right)\wedge \tfrac {d} {2s}\,-\right) \, .
\end{aligned} 
\right.
\end{equation} 
Moreover, for every $k \in \N$, $s < 1$ and $\delta > 0$, the following hold:
\begin{equation}
\label{e.k.gen}
\Add_k(1,2-), \quad \Dual_k(1-), \quad \Loc_k(1,\de,2-),
\end{equation}
and
\begin{equation}
\label{e.fluctkopt}
\left\{
\begin{aligned}
& \Fluc_k(2s,1) \ & \mbox{if} \ d=2, \\
& \Fluc_k\left( \tfrac{4s}{3},\tfrac 32 \right) \   & \mbox{if} \ d=3, \\
& \Fluc_k(1,2-) \ & \mbox{if} \ d\ge 4.
\end{aligned}
\right.
\end{equation}
\end{theorem}

The proof of Theorem~\ref{t.additivitybelowd} is an immediate consequence of induction and the following six implications, which are stated here and proved later in the paper.

\smallskip

The first establishes the base case of the bootstrap argument. 

\begin{proposition}
\label{p.thebasecase}
There exists $\alpha_0(d,\Lambda)\in \left(0,\frac12\right]$ such that, for every $k\in\N$ and $t\geq 1$, the following hold:
\begin{equation*}
\Fluc_k\left(t,\frac{\alpha_0}t\right), \quad
\Add_k\left(t,\frac{\alpha_0}t\right)
\quad \mbox{and} \quad 
\Dual_k\left(\frac{\alpha_0}t\right).
\end{equation*}
\end{proposition}

Proposition~\ref{p.thebasecase} is proved in Section~\ref{s.basecase} and is based on the results from~\cite{AS,AM}. 

\smallskip

The next proposition asserts roughly that localization and additivity give us control of the fluctuations up to (almost) the same exponent or $\frac d2$, whichever is smaller. It is convenient to divide the statement into two parts, the first for suboptimal scales and the second at the optimal scale. The proof of this proposition is the only place in the paper in which we exploit stochastic cancellations inherited from the finite range dependence of the coefficient field. 

\begin{proposition}
\label{p.control-fluct}

Fix $k\in\N$, $s\in (1,2]$ and $\alpha\in (0,\infty)$. Then the following hold:

\begin{enumerate}

\item[(i)] For every $\beta \in \left(0,\alpha \wedge \frac d 2\right)$ and $\delta \in \left(0, \frac {(\alpha - \beta)(d-2\beta)}{\beta d}\right)$,
\begin{equation*} \label{}
\Loc_k(s,\delta, \alpha)
 \ \ \mbox{and} \ \ 
 \Add_k(s,\alpha)  
  \implies   
\Fluc_k(s,\beta).
\end{equation*}

\item[(ii)] For every $\delta > 0$ and $\alpha > \frac d 2(1+\delta)$, 
\begin{equation*}
\Loc_k\Ll(s,\de,\al\Rr) \ \  \mbox{and} \ \ \Add_k\left(s, \al \right) 
 \implies 
\Fluc_k\left(s,\tfrac d 2 \right) \,.
\end{equation*}
\end{enumerate}
\end{proposition}

The proof of Proposition~\ref{p.control-fluct} is given in Section~\ref{s.fluc-subopt}. 

\smallskip

The argument for the next proposition lies at the heart of the paper. It states that control of both the fluctuations $I_k$ as well as the correspondence between gradients and fluxes of elements of $\A_k$ implies the additivity of $I_k$ \emph{with an improved exponent.}

\begin{proposition}
\label{p.additivity}
For every $s\in (0,2]$ and $\alpha \in \left(0, \frac ds \right)\cap \left(0,\frac d2 \right]$,
\begin{equation*}
\Fluc_1(s,\alpha)  \ \  \mbox{and} \ \ \Dual_1(\alpha) 
  \implies  
\Add_1\left( \tfrac{s}2, 2\alpha \right)
\end{equation*}
and, for general $k\in\N$, 
\begin{equation*} \label{}
\Fluc_k(s,\alpha)  \ \  \mbox{and}  \ \  \Dual_k(\alpha) 
  \implies  
\Add_k\left( \tfrac{s\alpha}\beta, \beta \right)
\quad \mbox{for} \ \beta:= 2\alpha \wedge (\alpha+1).
\end{equation*}
Moreover, for each $\ep>0$, there exists an exponent $\eta(\ep,s,d,\Lambda)>0$ such that, if $\alpha \in \left(0, \left( \frac ds - \ep\right) \wedge \frac d2 \right]$, then
\begin{equation*} \label{}
\Fluc_k(s,\alpha) 
\  \mbox{and} \ 
\Dual_k(\alpha) 
  \implies  
\Add_k\left( s,\alpha+\eta \right).
\end{equation*}
\end{proposition}

The proof of Proposition~\ref{p.additivity} is given in Section~\ref{s.additivity}. 

\smallskip

The next proposition concerns the improvement of localization and is based on the regularity theory (Proposition~\ref{p.regularity}), the arguments developed in Section~\ref{s.additivity} for proving the Proposition~\ref{p.additivity}, and the identification of maximizers of $J_k$ with two-scale expansions in terms of the correctors.

\begin{proposition}
\label{p.localization}
For every $k\in\N$, $s\in (0,\infty)$, $\alpha \in \left(0,\frac d s \right)$ and $\delta > 0$,
\begin{equation*}
\Fluc_k(s,\alpha) 
\ \  \mbox{and} \ \
\Dual_k(\alpha)
  \implies  
\Loc_k\left(s,\delta, \left( \alpha(1+\delta)+\delta\right)\wedge \tfrac d s \,-\right).
\end{equation*}
Moreover, 
for every $s\in (0,2]$, $\alpha \in \left(0,  \frac ds \wedge 1\right)$ and $\delta>0$,
\begin{equation*} \label{}
\forall k\in\N, \ \Fluc_k(s,\alpha) 
\ \  \mbox{and} \ \ 
\Dual_k(\alpha)  
  \implies  
\forall k\in\N, \ \Loc_k\left(\tfrac s 2, \delta, 2\alpha- \right).
\end{equation*}
\end{proposition}

The proof of Proposition~\ref{p.localization} is given in Section~\ref{s.localization}. 

\smallskip

The next proposition concerns the improvement of the statement $\Dual_k(\alpha)$, which controls the correspondence between spatial averages of gradients and fluxes for elements of $\A_k$. It is obtained by comparing elements of $\A_k$ with two-scale expansions in terms of correctors.

\begin{proposition}
\label{p.min}
For every $s\in (0,2]$, $\alpha \in \left(0,\frac ds \right)\cap \left( 0,\frac d2\right]$ and $\beta\in (0,\alpha]$, 
\begin{equation*}
\Fluc_1(s,\alpha) 
\ \mbox{and} \ 
\Dual_1(\beta)
  \implies   
\Dual_1\left( 2\beta \wedge \alpha\right).
\end{equation*}
Moreover, for every $k\in\N$, $s\in (0,\infty)$, $\alpha \in \left(0,\frac ds \right)$ and $\beta\in (0,\alpha]$
\begin{equation*}
\Fluc_k(s,\alpha) 
\ \mbox{and} \ 
\Dual_k(\beta)
  \implies   
\Dual_k\left( \tfrac{2\beta}{\beta+1}\wedge \alpha\wedge1-\right).
\end{equation*}
\end{proposition}

The proof of Proposition~\ref{p.min} is also given in Section~\ref{s.localization}. The fact that the second statement does not allow $\Dual_k(\alpha)$ to be improved past $\alpha=1$ is the reason that the exponent in the bootstrap argument saturates, for $k>1$, at exponent $\alpha=1$.

\smallskip

We conclude this section by giving the bootstrap argument, demonstrating that the previous five propositions imply Theorem~\ref{t.additivitybelowd}. 

\begin{proof}[{Proof of Theorem~\ref{t.additivitybelowd}}]
For each $\alpha>0$, we let $\mathcal{S}_k(\alpha)$ denote the statement that 
\begin{equation*}
\Add_k(2,\alpha), \ \Fluc_k(2,\alpha), \ \mbox{and} \ \Dual_k(\alpha) \ \mbox{hold.}
\end{equation*}
According to Proposition~\ref{p.thebasecase}, there exists $\ep_0(d,\Lambda)>0$ such that, for every $k\in\N$, 
\begin{equation}
\label{e.Skbase}
\mathcal{S}_k(\ep_0) \ \mbox{holds.}
\end{equation}
In the first two steps, we prove the first assertion of the theorem, which refers to the case $k=1$. 

\smallskip

\emph{Step 1.} We show that, for every $\ep>0$, there exists $\gamma(\ep,d,\Lambda)>0$ such that, for every $\alpha\in \left( \ep,\frac d2-\ep\right]$, 
\begin{equation}
\label{e.inductionstepk1}
\mathcal{S}_1(\alpha) \implies \mathcal{S}_1(\alpha+\gamma).
\end{equation}
Applying Propositions~\ref{p.additivity} and \ref{p.localization}, we find that there exists $\eta(\ep,d,\Lambda)>0$ such that, for every $\delta >0$,
\begin{equation*}
\mathcal{S}_1(\alpha) 
\implies 
\Add_1(2,\alpha+\eta), \  \mbox{and} \ 
\Loc_1\left(2,\delta,\left(\alpha(1+\delta)+\delta \right)\wedge \tfrac d 2 -\right).
\end{equation*}
Using this, we then apply Proposition~\ref{p.control-fluct} with the parameters
\begin{equation*}
\delta:= \frac{\ep \eta}{2d^2} \quad \mbox{and} \quad \beta:= \left( \alpha +\frac12 \delta \right) \wedge \left( \alpha+ \frac12 \eta \right) \wedge \left( \frac d2 - \frac12 \ep \right)
\end{equation*}
to obtain that 
\begin{equation*}
\mathcal{S}_1(\alpha)  \implies \Fluc_1(2,\beta). 
\end{equation*}
Applying Proposition~\ref{p.min} and the previous display, we get that 
\begin{equation*}
\mathcal{S}_1(\alpha)  \implies\Dual_1(2\alpha\wedge \beta).
\end{equation*}
Since $2\alpha\geq \alpha+\ep$ and 
\begin{equation*}
\beta -\alpha \geq \frac12 \left( \delta\wedge\eta\wedge \ep \right) = \frac12  \left( \frac{\ep \eta}{2d^2}\wedge\eta\wedge \ep \right),
\end{equation*}
we have shown~\eqref{e.inductionstepk1} for $\gamma:=  \frac12  \left( \frac{\ep \eta}{2d^2}\wedge\eta\wedge \ep \right)$.

\smallskip

\emph{Step 2.} We complete the proof of the first statement of the theorem. By~\eqref{e.Skbase},~\eqref{e.inductionstepk1} and induction, we deduce that 
\begin{equation*}
\mathcal{S}_1\left( \tfrac d2\,- \right)  \ \mbox{holds.}
\end{equation*}
Applying Proposition~\ref{p.additivity} once more, we obtain that $\Add_1(1,d-)$ holds, and by Remark~\ref{r.change-s} and \eqref{e.bounded-I}, that for every $s \ge \frac 1 2$, 
\begin{equation*}
\Add_1\left(2s,\tfrac{d}{2s}-\right)
 \ \mbox{holds}.
\end{equation*}
By Proposition~\ref{p.localization}, we also get that for every $\delta > 0$ and $s < 1$,
\begin{equation*}
\Loc_1\left(2s,\delta,\left( \tfrac d2(1+\delta)+\delta \right)\wedge \tfrac{d}{2s} - \right) 
 \ \mbox{holds.}
\end{equation*}
For each $s \in \Ll[\frac 1 2, 1\Rr)$, we choose $\delta > 0$ sufficiently small that
$\frac{d}{2s} > \frac d 2 (1+\delta)$, 
and then apply the second statement of Proposition~\ref{p.control-fluct} to obtain that
\begin{equation*}
\Fluc_1\left(2s,\tfrac d2 \right)
 \ \mbox{holds.}
\end{equation*}
Now applying Proposition~\ref{p.min} and then the first statement of Proposition~\ref{p.additivity}, we obtain that, for every $s<1$,
\begin{equation*}
\Dual_1\left(\tfrac d2\right)
\  \mbox{and} \ 
\Add_1(s,d) \ \mbox{hold.}
\end{equation*}
 This completes the proof of~\eqref{e.k=1}.
 
 \smallskip
 
\emph{Step 3.} We argue that, for every $\ep>0$, there exists $\gamma(\ep,d,\Lambda)>0$ such that, for every $k\in\N$ and $\alpha \in \left[ \ep, 1-\ep \right]$,
\begin{equation}
\label{e.inductionstepkgen}
\mathcal{S}_k(\alpha) \implies \mathcal{S}_k(\alpha+\gamma). 
\end{equation}
The argument is almost the same as in Step~1, the only difference being that we use the second statement of Proposition~\ref{p.min} instead of the first. The details are therefore omitted.

\smallskip

\emph{Step 4.} By~\eqref{e.Skbase},~\eqref{e.inductionstepkgen} and induction, we obtain that 
\begin{equation*}
\forall k\in\N, \ \mathcal{S}_k(1-) \ \mbox{holds.}
\end{equation*}
By the second statements of Propositions~\ref{p.additivity} and \ref{p.localization}, for every $\de > 0$, 
\begin{equation}
\label{e.add.k}
\Add_k\left(1,2-\right) \ \mbox{and} \ \Loc_k(1,\de,2-)
 \ \mbox{hold},
\end{equation}
and therefore \eqref{e.k.gen} is proved. By Remark~\ref{r.change-s}, we also infer that for every $\de > 0$ and $s \ge 1$,
\begin{equation}
\label{e.add.k.d3}
\Add_k\left(s, \tfrac 2s-\right) \ \mbox{and} \ \Loc_k\left(s, \de, \tfrac 2s-\right)
 \ \mbox{hold}.
\end{equation}
Choosing
\begin{equation*}
\left\{
\begin{aligned}
& s < 2 & \mbox{if} \ d=2, \\
& s < \tfrac 4 3 & \mbox{if} \ d=3, \\
& s = 1 & \mbox{if} \ d\ge 4,
\end{aligned}
\right.
\end{equation*}
and then $\de > 0$ sufficiently small in terms of $s$, we obtain \eqref{e.fluctkopt} by an application of Proposition~\ref{p.control-fluct} (the second statement if $d \in \{2,3\}$, the first statement if $d \ge 4$). 
\end{proof}

\section{The base case}
\label{s.basecase}

In this section, we prove Proposition~\ref{p.thebasecase}, which establishes the base case of the induction argument explained in the previous section, as well as Proposition~\ref{p.basecase}. They are essentially a rephrasing of the results of~\cite{AS,AM}, which give an algebraic rate of convergence for certain subadditive energy quantities that are close to~$J$. We just need to reformulate these results in terms of integration against the heat kernel rather than with respect to Lebesgue measure in bounded domains. This is a somewhat routine exercise which resembles the arguments in the proof of~\cite[Proposition 4.1]{AS}, albeit in a simpler context. 

\smallskip

As will be shown, Proposition~\ref{p.thebasecase} is a simple consequence of Proposition~\ref{p.basecase}. The extra information provided by Proposition~\ref{p.basecase} will prove to be useful in Sections~\ref{s.additivity} and~\ref{s.localization}. 

\smallskip

In order to connect to the results of~\cite{AS}, we recall the subadditive quantities introduced there, which are defined, for each bounded domain $U\subseteq \Rd$ and $p,q\in \Rd$, by 
\begin{equation*} \label{}
\mu(U,q):= \inf_{u\in H^1(U)} \fint_U \left( \frac12 \nabla u(x) \cdot \a(x) \nabla u(x) - q\cdot \nabla u(x) \right)\,dx
\end{equation*}
and
\begin{equation*} \label{}
\nu(U,p):= \inf_{v\in H^1_0(U)} \fint_U \frac 12 \left(p+\nabla v(x) \right)\cdot \a(x) \left( p+\nabla v(x) \right)\,dx.
\end{equation*}
Note that by an easy integration by parts (cf.~\cite[Lemma 3.1]{AKM}), the latter can be written in the form
\begin{equation*} \label{}
\nu(U,p) = \sup_{v \in\A(U)} \fint_U \left( -\frac12 \nabla v(x)\cdot \a(x)\nabla v(x) - p\cdot \a(x)\nabla v(x) \right)\,dx.
\end{equation*}
Likewise, the set over which the infimum is taken in the definition of $\mu$ may be replaced by $\A(U)$. 
From these formulas, we see that $J$ is actually a variation of a combination of $\mu$ and $\nu$ in which the domain $U$ has been ``smoothed out" by replacing it with the heat kernel for~$\ahom$ and the admissible set~$\A(U)$ is replaced by $\A_k(U)$ for some $k\in\N$. Indeed, if we (abuse the notation and) define, for each $p,q\in\Rd$,
\begin{equation*}
J(U,p,q) := \sup_{w \in \A(U)} \fint_U \left( -\frac12 \nabla w(x)\cdot \a(x)\nabla w(x) - \left(  \a(x)p -\ahom q\right)\cdot\nabla w (x)\right) \, dx,
\end{equation*}
then it is easy to check (see~\cite[Lemma 3.1]{AKM}) that 
\begin{equation*}
J(U,p,q) = \nu(U,p) - \mu(U,\ahom q) - p\cdot \ahom q. 
\end{equation*}
Therefore,~\cite[Theorem 3.1]{AS} gives the existence of an exponent $\delta(d,\Lambda)\in \left(0,\frac12\right]$ and a constant $C(d,\Lambda)< \infty$ such that, for every $s\in (0,d)$, $p,q\in B_1$ and $n\in\N$, 
\begin{equation}
\label{e.basecasemunu}
\P \left[ \left| J(\cu_{3^n},p,q) - \frac12 (p-q) \cdot \ahom (p-q) \right| \geq C 3^{-n\delta(d-s)}t \right]  \leq C\exp\left( -3^{ns}t/C\right). 
\end{equation}
We will use~\eqref{e.basecasemunu} to obtain a similar estimate on the quantity $J(0,R,p,q)$ studied in this paper. 

\smallskip

We now define the random variable $\mathcal{Y}_s$ appearing in the statement of Proposition~\ref{p.basecase}. It is also used throughout Section~\ref{s.additivity}. 

\begin{definition}[{The random variable $\mathcal{Y}_s$}]
Fix $s\in (0,d)$. To gain some room, we take $s_1: =\frac12(s+d)$ and $s_2:=\frac12(s_1+d)$ so that $s<s_1<s_2<d$. We also fix the mesoscale exponent $\gamma := s_1/s_2 \in\left(\frac12,1\right)$ and define
\begin{multline*}
\tilde{\mathcal{Y}}_s
:= 
\sup\Bigg\{ 3^{n/\gamma} \, :\, n\in \N, \\
\sup_{z\in\Zd\cap \cu_{3^{2n}}} 
\sup_{p,q\in B_1}
\left| J\left(\cu_{3^n}(z), p,q\right) - \frac12 (p-q) \cdot \ahom (p-q) \right| \geq 3^{-n\delta(d-s_2)} \Bigg\}.
\end{multline*}
Finally, with $\X_s$ the random variable in the statement of Proposition~\ref{p.regularity}, we set
\begin{equation*} \label{}
\mathcal{Y}_s := \tilde{\mathcal{Y}}_s \vee \X_s.
\end{equation*}
\end{definition}

\smallskip

We first check that $\mathcal{Y}_s$ satisfies the estimate~\eqref{e.Ys}. 

\begin{lemma}
\label{l.Ys}
There exists $C(s,d,\Lambda)<\infty$ such that 
\begin{equation*} \label{}
\mathcal{Y}_s = \O_s(C). 
\end{equation*}
\end{lemma}
\begin{proof}
By a union bound and~\eqref{e.basecasemunu}, we estimate
\begin{align*}
\lefteqn{
\P \left[ \tilde{\mathcal{Y}}_s \geq R \right] 
} \quad & \\
& \leq \sum_{n\in\N,\, 3^n \geq R^\gamma} \sum_{z\in \Zd \cap \cu_{3^{2n}}} \, \sum_{p,q\in \{ e_1,\ldots,e_d\}} \\
& \qquad \P \left[ \left| J\left(\cu_{3^n}(z), p,q\right) - \frac12 (p-q) \cdot \ahom (p-q) \right| \geq c3^{-n\delta(d-s_2)}\right] \\
& \leq \sum_{n\in\N,\, 3^n \geq R^\gamma} C3^{2dn} \sup_{p,q\in B_1} \P \left[ \left| J\left(\cu_{3^n}, p,q\right) - \frac12 (p-q) \cdot \ahom (p-q) \right| \geq c3^{-n\delta(d-s_2)}\right] \\
& \leq C\sum_{n\in\N,\, 3^n \geq R^\gamma} 3^{2dn} \exp\left( -c3^{ns_2} \right)  \\
& \leq CR^{4d} \exp\left( -c R^{s_1} \right).
\end{align*}
Integrating this with respect to $R$, using $s_1>s$, yields 
\begin{equation*}
\E \left[ \exp\left( \tilde{\mathcal{Y}}_s^s \right) \right] \leq C(s,d,\Lambda), 
\end{equation*}
which implies that
\begin{equation}
\label{e.tYs}
\tilde{\mathcal{Y}}_s = \O_s(C). 
\end{equation}
We now obtain~\eqref{e.Ys} from~\eqref{e.X} and~\eqref{e.tYs}.
\end{proof}

\smallskip

The proof of~\eqref{e.Jminimalrad} is accomplished by approximating the integrals in the definition of $J(0,R,p,q)$ by a Riemann sum, and approximating each term using the convergence of $J$ for cubes guaranteed in the definition of $\mathcal{Y}_s$. We break the proof into the next two lemmas, one each for the upper and lower bounds.

\begin{lemma} 
\label{l.JbreakupUB}
There exist $\delta(d,\Lambda)\in \left(0,\frac12\right]$ and $C(s,k,d,\Lambda)<\infty$ such that,  for every $R\geq \mathcal{Y}_s$ and $p,q\in \Ahom_k(\Phi_R)$,
\begin{equation} \label{e.JbreakupUB}
J(0,R,p,q) 
\leq \int_{\Phi_R}\frac12 \left( \nabla p -\nabla q \right)  \cdot \ahom\left( \nabla p -\nabla q \right) +  CR^{-\delta(d-s)}. 
\end{equation}
\end{lemma} 
\begin{proof}
We need to fix some parameters. We take $\gamma(s,d)\in \left(\frac12,1\right)$ as in the definition of~$\tilde{\mathcal{Y}}_s$, above. We let $n$ be the integer satisfying $3^{n-1} \leq R^\gamma < 3^n$ and set $r:=3^n$. The hypothesis $R \geq\mathcal{Y}_s$ implies that, for every $m\geq n$, 
\begin{multline} \label{e.ngoodscale}
\sup_{z\in \Zd\cap \cu_{3^{2n}}}  \,
\sup_{p,q\in B_1}
\left| J\left(\cu_{3^m}(z), p,q\right) - \frac12 (p-q) \cdot \ahom (p-q) \right|  \\
\leq 3^{-m\delta(d-s_2)} \leq CR^{-\delta(d-s_2)/2}.
\end{multline}
Fix $\sigma \in \left(0,\frac12(1-\gamma)\right]$ to be selected below and set 
\begin{equation*} \label{}
S:= \left\lceil 3^{n(1+\sigma-\gamma)/\gamma} \right\rceil 3^n ,
\end{equation*}
so that $S\in r\Z$ and $S \simeq R^{1+\sigma}$. Note that
\begin{equation*} \label{}
(1+\sigma-\gamma)/\gamma \leq \frac12(1-\gamma)/\gamma \leq \frac12 < 1,
\end{equation*}
so that, with an eye toward~\eqref{e.ngoodscale}, we have $S\leq 3^{3n/2}< 3^{2n}$ and thus $\cu_S \subseteq \cu_{3^{2n}}$.

\smallskip

Observe that
\begin{multline*}
J(0,R,p,q) 
 \leq  \sup_{v\in \A_k} \int_{\cu_{S}} \Phi_R \left( -\frac12 \nabla v\cdot \a\nabla v - \left(\a \nabla p- \ahom \nabla q\right)\cdot \nabla v  \right) \\
+  \sup_{v\in \A_k} \int_{\Rd\setminus\cu_{S}} \Phi_R \left( -\frac12 \nabla v\cdot \a\nabla v - \left(\a \nabla p- \ahom \nabla q\right)\cdot \nabla v  \right).
\end{multline*}
We may brutally estimate the second term on the right, using the decay of $\Phi_R$ and $S/R\geq cR^\sigma$, by
\begin{multline}
\label{e.tailchoptcha}
\sup_{v\in \A_k} \int_{\Rd\setminus\cu_{S}} \Phi_R \left( -\frac12 \nabla v\cdot \a\nabla v - \left(\a \nabla p- \ahom \nabla q\right)\cdot \nabla v  \right)  \\
\leq C \int_{\Rd \setminus \cu_{S}}  \Phi_R \left( \left| \nabla p\right|^2+  \left| \nabla q\right|^2\right)
\leq  C\int_{\Rd \setminus \cu_{S}} \Phi_R(x) \left( \frac{|x|}{R} \right)^{2(k-1)}\, dx
 \leq CR^{-100}.
\end{multline}

We turn to the estimate of the first term. Fix $v \in \A_k$. For convenience, assume $k\geq 2$. For each $z\in r\Zd$, denote $\left( \Phi_R \right)_z:= \fint_{z+\cu_r} \Phi_R(x)\,dx$ and observe that, for constant $C(k,d,\Lambda)<\infty$ and every $z\in r\Zd \cap \cu_{S}$, we have
\begin{align*} \label{}
\osc_{z+\cu_r} \Phi_R
 \leq Cr \left\| \nabla  \Phi_R \right\|_{L^\infty(z+\cu_r)} 
 \leq \frac{Cr(|z|+Cr)}{R^2} \sup_{z+\cu_r} \Phi_R 
 &  \leq C\frac{rS}{R^2} \sup_{z+\cu_r} \Phi_R  \\
 &  \leq CR^{\sigma+\gamma-1}\sup_{z+\cu_r} \Phi_R.
\end{align*}
As $\sigma \leq \frac12(1-\gamma)$, after adding a large constant $C(s,k,d,\Lambda)$ to $\mathcal{Y}_s$ so that $R\geq \mathcal{Y}_s \geq C$, we have
\begin{equation*} \label{}
\osc_{z+\cu_r} \Phi_R \leq \frac12 \sup_{z+\cu_r} \Phi_R\,,
\end{equation*}
and then returning to the previous estimate, we find that, for every $z\in r\Zd \cap \cu_{S}$,
\begin{equation}
\label{e.PhiRoscgamma}
\osc_{z+\cu_r} \Phi_R \leq  CR^{-\frac12 (1-\gamma)} \inf_{z+\cu_r} \Phi_R.
\end{equation}
We also have, by the normalization of $p$ and $q$, for every $z\in r\Zd \cap \cu_{S}$,
\begin{multline}
\label{e.stupidpzpx}
\sup_{x\in z+\cu_r} \left( \left| \nabla p(x) - \nabla p(z) \right| + \left| \nabla q(x) - \nabla q(z) \right| \right)  \\
\leq Cr \left\| \left|\nabla^2 p\right| + \left|\nabla^2q\right| \right\|_{L^\infty(\cu_{S})} 
 \leq C\left( \frac rR \right) \left( \frac SR \right)^{(k-2)} 
 \leq CR^{\gamma -1 + \sigma(k-2)} 
 \leq C R^{-\frac12(1-\gamma)},
\end{multline}
where we have reduced $\sigma$, if necessary, so that $\sigma (k-2) \leq \frac12(1-\gamma)$. By the above estimates, we obtain
\begin{align*}
\lefteqn{
\int_{\cu_{S}} \Phi_R \left( -\frac12 \nabla v\cdot \a\nabla v - \left( \a \nabla p - \ahom \nabla q\right)\cdot \nabla v \right)
} \qquad & \\
& \leq \sum_{z\in r\Zd \cap \cu_{S}} \left( \Phi_R \right)_{z} \int_{z+\cu_r}  \left( -\frac12 \nabla v\cdot \a\nabla v - \left( \a \nabla p - \ahom \nabla q\right)\cdot \nabla v \right) \\
& \qquad + 
CR^{-\frac12(1-\gamma)} \sum_{z\in r\Zd \cap \cu_{S}} \left( \Phi_R \right)_{z}   \int_{z+\cu_r} \left(\left| \nabla v\right|^2+ \left| \nabla p\right|^2 + \left| \nabla q\right|^2\right)  \\
& \leq \sum_{z\in r\Zd \cap \cu_{S}} \left( \Phi_R \right)_{z} |\cu_r|  J\left( z+\cu_r, \nabla p(z),\nabla q(z) \right) \\
& \qquad + 
CR^{-\frac14(1-\gamma)} \sum_{z\in r\Zd \cap \cu_{S}} \left( \Phi_R \right)_{z}   \int_{z+\cu_r} \left(1 + \left| \nabla v\right|^2+ \left| \nabla p\right|^2+ \left| \nabla q\right|^2 \right)  .
\end{align*}
The function $v\in \A_k$ which attains the supremum in the first term on the left side of~\eqref{e.tailchoptcha} 
satisfies 
\begin{equation*}
\int_{\cu_{S}} \Phi_R \left| \nabla v \right|^2  \leq C\int_{\cu_{S}} \Phi_R\left( \left| \nabla p \right|^2  +  \left| \nabla q \right|^2 \right) \leq C.
\end{equation*}
Using this, the normalization of $p$ and $q$ and~\eqref{e.PhiRoscgamma}, we obtain
\begin{equation*}
CR^{-\frac14(1-\gamma)} \sum_{z\in r\Zd \cap \cu_{S}} \left( \Phi_R \right)_{z}   \int_{z+\cu_r} \left(1+\left| \nabla v\right|^2+ \left| \nabla p\right|^2+ \left| \nabla q\right|^2 \right)
\leq CR^{-\frac14(1-\gamma)}.
\end{equation*}
On the other hand,~\eqref{e.ngoodscale} and $\cu_S \subseteq \cu_{3^{2n}}$ imply that 
\begin{multline*} \label{}
\bigg| J\left( z+\cu_r, \nabla p(z),\nabla q(z) \right) 
 -  \frac12\left(\nabla q(z) -\nabla p(z)\right)\cdot \ahom \left(\nabla q(z)-\nabla p(z)\right)\bigg| 
 \\ \leq CR^{-\delta(d-s_2)/2}. 
\end{multline*}
Using $S\geq cR^{1+\sigma}$ to chop the tails of $\Phi_R$ again (like in~\eqref{e.tailchoptcha}),~\eqref{e.stupidpzpx} yields
\begin{multline*} 
\bigg| \sum_{z\in r\Zd \cap \cu_{S}} \left( \Phi_R \right)_{z} |\cu_r|  \frac12\left(\nabla q(z) -\nabla p(z)\right)\cdot \ahom \left(\nabla q(z)-\nabla p(z)\right)
\\ - \int_{\Phi_R}\frac12 \left(\nabla q-\nabla p\right)\cdot \ahom \left(\nabla q-\nabla p\right)  \bigg| \leq C R^{-\frac14(1-\gamma)} .
\end{multline*}
Connecting last two displays gives
\begin{multline*} \label{}
\bigg| \sum_{z\in r\Zd \cap \cu_{S}} \left( \Phi_R \right)_{z} |\cu_r| J\left( z+\cu_r, \nabla p(z),\nabla q(z) \right) 
- \int_{\Phi_R}\frac12 \left(\nabla q-\nabla p\right)\cdot \ahom \left(\nabla q-\nabla p\right) \bigg| \\
 \leq C\left( R^{-\delta(d-s_2)/2} + R^{-\frac14(1-\gamma)} \right).
\end{multline*}
Assembling the above estimates yields 
\begin{equation*}
J(0,R,p,q) \leq   \int_{\Phi_R}\frac12 \left(\nabla q-\nabla p\right)\cdot \ahom \left(\nabla q-\nabla p\right) + C\left( R^{-\delta(d-s_2)/2} + R^{-\frac14(1-\gamma)} \right).
\end{equation*}
This implies~\eqref{e.JbreakupUB} after a redefinition of~$\delta$.
\end{proof}

We next present the lower bound. 

\begin{lemma}
\label{l.JbreakupLB}
There exist $\delta(d,\Lambda)\in \left(0,\frac12\right]$ and $C(k,d,\Lambda)<\infty$ such that, for every $s\geq1$, $R\geq \mathcal{Y}_s$ and $p,q\in \Ahom_k(\Phi_R)$,
we have
\begin{equation} \label{e.JbreakupLB}
J(0,R,p,q) \\
\geq \int_{\Phi_R}\frac12 \left( \nabla p -\nabla q \right)  \cdot \ahom\left( \nabla p -\nabla q \right) - CR^{-\delta(d-s)}. 
\end{equation}
\end{lemma}
\begin{proof}
We take $\gamma\in (\frac12,1)$, $n\in\N$, $r=3^n$, $\sigma\in \left( 0 ,\frac12(1-\gamma) \right)$ and $S$ to be almost the same parameters as in the proof of Lemma~\ref{l.JbreakupUB}. The only difference is that we require $\sigma$ to be slightly smaller, if necessary, so that  $\sigma(k+ d) \leq \frac1{16}(1-\gamma)$. Notice in particular that we have $\cu_S \subseteq \cu_{3^{2n}}$ and~\eqref{e.ngoodscale}. 

\smallskip
\emph{Step 1.} 
We let $u\in H^1(\cu_S)$ be the unique solution of the Dirichlet problem
\begin{equation*} \label{}
\left\{ 
\begin{aligned}
& -\nabla \cdot \left( \a(x) \nabla u \right) = 0 & \mbox{in} & \ \cu_S, \\
& u = q-p & \mbox{on} & \ \partial \cu_S. 
\end{aligned}
\right.
\end{equation*}
%
For each $z\in r\Zd \cap \cu_{S}$, we let $v_{z}\in H^1(z+\cu_r)$ denote the solution of
\begin{equation*} \label{}
\left\{ 
\begin{aligned}
& -\nabla \cdot \left( \a(x) \nabla v_z \right) = 0 & \mbox{in} & \ z+\cu_r, \\
& v_z = q-p & \mbox{on} & \ \partial (z+\cu_r). 
\end{aligned}
\right.
\end{equation*}
Note that $v_z$ is the maximizer of $J(z+\cu_r,p-q,0)$. By gluing together the $v_z$'s, we see that there exists $v\in (q-p) + H^1_0(\cu_S)$ such that $v\vert_{z+\cu_r} = v_z$ for each $z\in r\Zd\cap \cu_S$. 

\smallskip

\emph{Step 2.} We show that $v$ defined above gives a good approximation for $J(0,R,p,q)$. Indeed, we have
\begin{multline} \label{e.stupidfluxmaps}
\int_{\Phi_R^{(\sigma)}} \left( - \frac12 \nabla v \cdot \a\nabla v - \left(\a\nabla p - \ahom \nabla q \right)\cdot \nabla v \right) 
\\ \geq \int_{\Phi_R}\frac12 \left( \nabla p -\nabla q \right)  \cdot \ahom\left( \nabla p -\nabla q \right)  - C  R^{-\delta_1(d-s)} 
\end{multline}
with $\delta_1 = \delta_1(\gamma,d,\Lambda)$, and where we recall that $\Phi^{(\sigma)}_R$ is defined in \eqref{e.Phideltadef}. According to~\eqref{e.stupidpzpx}, it is clear that 
\begin{equation*}
\left| J(z+\cu_r,\nabla (p-q)(z),0) - J(z+\cu_r,p-q,0) \right| \leq CR^{-\frac12(1-\gamma)}  |\nabla (p-q)(z)|^2 
\end{equation*}
and
\begin{equation*}
\left| J(z+\cu_r,\nabla p(z),\nabla q(z)) - J(z+\cu_r,p,q)  \right| \leq CR^{-\frac12(1-\gamma)} \left( |\nabla p(z)|^2 + |\nabla q(z)|^2 \right)  . 
\end{equation*}
By~\eqref{e.ngoodscale}, on the other hand, we have that 
\begin{multline*} 
\left| J(z+\cu_r,\nabla p(z),\nabla q(z)) - J(z+\cu_r,\nabla (p-q)(z) ,0) \right| 
\\  \leq CR^{-\delta(d-s_2)/2} \left( |\nabla p(z)|^2 + |\nabla q(z)|^2 \right)
\end{multline*}
and, combining the above three displays, 
\begin{multline} \label{e.v_z flux control000}
\left| \frac12 \nabla(q-p)(z)  \ahom \nabla (q-p)(z) - J(z+\cu_r,p-q,0) \right| 
\\ \leq C \left( R^{-\delta(d-s_2)/2} + R^{-\frac12(1-\gamma)} \right) \left( |\nabla p(z)|^2 + |\nabla q(z)|^2 \right)\,.
\end{multline}
Furthermore, we have that
\begin{multline}
 \label{e.blogchab}
\left| \fint_{z + \cu_r} \ahom \nabla v_z  -  \fint_{z + \cu_r} \a \nabla v_z \right| \\
\leq C\left( |\nabla p(z)| + |\nabla q(z)| \right) \left( R^{-\delta(d-s_2)/2} + R^{-\frac12(1-\gamma)} \right) \,.
\end{multline}
Indeed, if we denote, for $\xi,\zeta\in\Rd$, the maximizer of $J(z+\cu_r,\xi,\zeta)$ by $\tilde v(\cdot,z,\xi,\zeta)$, then by~\eqref{e.ngoodscale}, we have, for every $\xi \in B_1$,
\begin{multline*}
\fint_{z+\cu_r} \left| \nabla \tilde v(x,z,-\xi,0) - \nabla \tilde v(x,z,0,\xi) \right|^2\,dx \\
= \fint_{z+\cu_r} \left| \nabla \tilde v(x,z,\xi,\xi) \right|^2\,dx 
\leq CJ(z+\cu_r,\xi,\xi) 
\leq CR^{-\delta(d-s_2)/2}.
\end{multline*}
Then using the following formulas given by the first variation,
\begin{equation*}
\fint_{z+\cu_r} \xi\cdot \a(x) \nabla v_z(x)\,dx = \fint_{z+\cu_r} \nabla \tilde v(x,z,-\xi,0) \cdot \a(x) \nabla v_z(x)\,dx
\end{equation*}
and
\begin{equation*}
\fint_{z+\cu_r} \ahom \xi\cdot\nabla v_z(x)\,dx = \fint_{z+\cu_r} \nabla \tilde v(x,z,0,\xi) \cdot \a(x) \nabla v_z(x)\,dx,
\end{equation*}
we get, for any $\xi\in B_1$,
\begin{align*}
\fint_{z+\cu_r} \left( \xi\cdot \a(x)\nabla v_z (x) - \ahom \xi\cdot \nabla v_z(x) \right)\,dx \leq C \left( \fint_{z+\cu_r}\left|\nabla v_z (x)\right|^2\right)^{\frac12} R^{-\delta(d-s_2)/2}. 
\end{align*}
Taking the supremum over $\xi\in B_1$ and using~\eqref{e.stupidpzpx} yields~\eqref{e.blogchab}. 

\smallskip

Using the above displays, together with~\eqref{e.tailchoptcha},~\eqref{e.PhiRoscgamma} and~\eqref{e.stupidpzpx}, yields~\eqref{e.stupidfluxmaps}:
\begin{align*} 
\lefteqn{\int_{\Phi_R^{(\sigma)}} \left( - \frac12 \nabla v \cdot \a\nabla v - \left(\a\nabla p - \ahom \nabla q \right)\cdot \nabla v \right) } \quad & 
\\ \notag & = \sum_{z\in r\Zd\cap \cu_S} \int_{z + \cu_r} \left( - \frac12 \nabla v_z \cdot \a\nabla v_z - \left(\a\nabla p - \ahom \nabla q \right)\cdot \nabla v_z \right) \Phi_R^{(\sigma)}
\\ \notag & \geq \sum_{z\in r\Zd\cap \cu_S} \Phi_R^{(\sigma)}(z) |\cu_r| J(z+\cu_r,p-q,0) - C \left( R^{-\delta(d-s_2)/2}  + R^{-\frac12(1-\gamma)} \right)
\\ \notag & \geq  \int_{\Phi_R}\frac12 \left( \nabla p -\nabla q \right)  \cdot \ahom\left( \nabla p -\nabla q \right)  - C \left( R^{-\delta(d-s_2)/2}  + R^{-\frac12(1-\gamma)} \right)\,.
\end{align*}

\smallskip

\emph{Step 3.} We next show that $\nabla u -\nabla v$ is small in $L^2$, that is
\begin{equation}
\label{e.closetovzs}
\left\|  \nabla u -\nabla v \right\|_{\underline{L}^2(\cu_S)}^2 \leq C R^{-\delta_2(d-s)} 
\end{equation}
for $\delta_2(\gamma,d,\Lambda)>0$.
Observe that from this it follows trivially that 
\begin{equation}
\label{e.closetovzs000}
\left\|  \nabla u -\nabla v \right\|_{\underline{L}^2\left(\Phi_R^{(\sigma)}\right)} \leq C R^{\sigma d/2 -\delta_2 (d-s)} \leq C R^{-\delta_2 (d-s)/2}.
\end{equation}
Testing the energy of $v$ against $u$ yields, in view of~\eqref{e.v_z flux control000}, that 
\begin{multline}
\label{e.nabuanabu1}
\int_{\cu_S} \frac12 \nabla u\cdot \a \nabla u
 \leq \int_{\cu_S} \frac12 \nabla v\cdot \a \nabla v 
 = \sum_{z\in r\Zd\cap \cu_S} \int_{z+\cu_r} \frac12 \nabla v_z\cdot \a \nabla v_z  \\
 \leq |\cu_r|  \sum_{z\in r\Zd\cap \cu_S} \frac12 \nabla (q-p)(z) \cdot\ahom\nabla (q-p)(z)  +C \left| \cu_S \right| \left(R^{-\delta(d-s_2)/2}+ R^{-\frac12(1-\gamma)}\right). 
\end{multline}
On the other hand, by an integration by parts and~\eqref{e.stupidpzpx} again, we find that 
\begin{align*}
\int_{\cu_S} \frac12 \nabla u\cdot \a \nabla u 
& = \int_{\cu_S} \left(  \frac12 \nabla u\cdot \a \nabla u -\nabla (q-p)\cdot \ahom\nabla u \right) + \int_{\cu_S} \nabla (q-p) \cdot \ahom \nabla (q-p) \\
& \geq \sum_{z\in r\Zd\cap \cu_S} \int_{z+\cu_r} \left( \frac12 \nabla u\cdot \a \nabla u - \nabla (q-p)(z) \cdot \ahom \nabla u \right)  \\
& \qquad + \left| \cu_r \right| \sum_{z\in r\Zd\cap \cu_S} \nabla (q-p)(z) \cdot \ahom\nabla (q-p)(z)  \\
& \qquad - CR^{-\frac14(1-\gamma)}\int_{\cu_S} \left( 1 + \left| \nabla u \right|^2 +  \left| \nabla p \right|^2 +  \left| \nabla q \right|^2 \right).
\end{align*}
The summand of the first term on the right side above is bounded below by~$-J(z+\cu_r,0,\nabla (q-p)(z))$. Using the above to combine this with the second term and then the normalization of $p$ and $q$ to estimate the third, we obtain 
\begin{multline}
\label{e.nabuanabu2}
\int_{\cu_S} \frac12 \nabla u\cdot \a \nabla u \\
\geq  \frac12 \left| \cu_r \right| \sum_{z\in r\Zd\cap \cu_S} \nabla (q-p)(z) \cdot \ahom\nabla (q-p)(z) - C\left| \cu_S \right| R^{-\frac14(1-\gamma)} \left( \frac SR \right)^{2(k-1)}.
\end{multline}
Combining~\eqref{e.nabuanabu1} and~\eqref{e.nabuanabu2} and using the choice of $\sigma$ indicated in the first paragraph, we find that 
\begin{multline*} 
\left| \fint_{\cu_S} \frac12 \nabla u\cdot \a \nabla u  - \frac{\left|\cu_r\right|}{\left|\cu_S\right|}\sum_{z\in r\Zd\cap \cu_S} \frac12 \left( \nabla q(z) - \nabla p(z) \right) \cdot \ahom\left( \nabla q(z) - \nabla p(z) \right) \right| 
\\
\leq C \left(R^{-\delta(d-s_2)/2}+ R^{-\frac18(1-\gamma)}\right).
\end{multline*}
In particular, we deduce that 
\begin{equation*}
\left| \fint_{\cu_S} \frac12 \nabla u\cdot \a \nabla u
 - \fint_{\cu_S} \frac12 \nabla v\cdot \a \nabla v  \right| \leq C \left(R^{-\delta(d-s_2)/2}+ R^{-\frac18(1-\gamma)}\right)\,,
\end{equation*}
which, since $u-v\in H^1_0(U)$ and $u$ is a minimizer of the energy, implies~\eqref{e.closetovzs000}. 

\smallskip 

\emph{Step 4.} We next demonstrate that there is $\phi \in \A_k$ such that 
\begin{equation} \label{e.lbcorrector}
\left\| \nabla u - \nabla \phi  \right\|_{L^2\left(\Phi_R^{\sigma}\right)} \leq C \left( \frac{R}{S} \right)^{k+1} R^{-\delta_3(d-s)}.
\end{equation}
for $\delta_3(\gamma,d,\Lambda)$. 
Using the Poincar\'e inequality and~\eqref{e.closetovzs}, we see that
\begin{align} \label{e.u -(q-p) poincare}
\lefteqn{
\frac{1}{S^2} \left\| u - (q-p) \right\|_{\underline{L}^2(\cu_S)}^2
} \qquad & \\ \notag 
&   \leq \frac{C}{S^2} \left( \left\| u - v \right\|_{\underline{L}^2(\cu_S)}^2 +  \left\| v - (q-p)  \right\|_{\underline{L}^2(\cu_S)}^2 \right) \\ \notag 
&  \leq C \left\|  \nabla u - \nabla v) \right\|_{\underline{L}^2(\cu_S)}^2 +
C\left( \frac rS \right)^{2} \fint_{\cu_S} \left( \left| \nabla v \right|^2 +\left| \nabla q \right|^2+\left| \nabla p \right|^2\right)   \\ \notag 
& \leq C\left(R^{-\delta_2(d-s)}+ R^{-\frac14(1-\gamma)}\right) 
+ C\left( \frac rS \right)^2 \left( \left\| \nabla p \right\|^2_{\underline{L}^2(\cu_S)}  + \left\| \nabla q \right\|_{\underline{L}^2(\cu_S)}^2   \right) \\ \notag 
& \leq C\left(R^{-\delta_2(d-s)}+ R^{-\frac14(1-\gamma)}\right).
\end{align}
Since $S>R\geq \mathcal{Y}_s\geq \X$, by Proposition~\ref{p.regularity}, we find $\phi_{q-p} \in \Ahom_k$ such that 
$$
\left\| q-p - \phi_{q-p} \right\|_{\underline{L}^2(\cu_{S})} \leq C S^{-\delta} \left\| q-p \right\|_{\underline{L}^2(\cu_{S})} \leq C S^{-\delta} R^{\sigma k} R \leq C R^{-\tilde \delta(d-s)} 
$$
for suitably chosen $\tilde \delta (d,\Lambda)$. By the above display and~\eqref{e.u -(q-p) poincare}, and again by Proposition~\ref{p.regularity}, we may select $\phi\in \A_k$ such that, for every $R' \in \left[R,\frac12 S\right]$,
\begin{equation*} \label{}
\left\| u - \phi  \right\|_{\underline{L}^2(\cu_{R'})}  \leq C \left( \frac {R'}S \right)^{k+1}  \left\| u - \phi_{q-p} \right\| _{\underline{L}^2(\cu_S)} \leq C \left( \frac{R'}{R} \right)^{k+1} R' R^{-\tilde \delta_3(d-s)}
\end{equation*}
and thus by the Caccioppoli estimate
\begin{equation*} \label{}
\left\| \nabla u - \nabla \phi  \right\|_{L^2(\cu_{R'/2})} \leq C \left( \frac{R'}{S} \right)^{k+1} R^{-\tilde \delta_3(d-s)}.
\end{equation*}
This, together with the trivial estimate 
$$
\left\| \nabla u - \nabla \phi  \right\|_{L^2(\cu_{S})} \leq C R^{\sigma k}
$$
to handle the tail terms, yields after easy manipulations~\eqref{e.lbcorrector}. 

\smallskip

\emph{Step 5.} Conclusion. Combining~\eqref{e.closetovzs} and~\eqref{e.lbcorrector} yields, for $\delta_4(\gamma,d,\Lambda)$, 
\begin{equation*} 
\left\| \nabla v - \nabla \phi  \right\|_{L^2\left(\Phi_R^{\sigma}\right)} \leq C R^{-\delta_4(d-s)}.
\end{equation*}
Using this we deduce that 
\begin{align*}
\lefteqn{
 \int_{\Phi_R^{(\sigma)}}  \left( - \frac12 \nabla \phi \cdot \a\nabla \phi - \left(\a\nabla p - \ahom \nabla q \right)\cdot \nabla \phi \right) 
} \qquad &  \\
& \geq \sum_{z\in r\Zd\cap \cu_S} \int_{z+\cu_r} \left( - \frac12 \nabla v_z\cdot \a\nabla v_z - \left(\a\nabla p - \ahom \nabla q \right)\cdot \nabla v_z \right) \Phi_R^{(\sigma)} \\
& \qquad -  C 
\left\| |\nabla u| + |\nabla v| + |\nabla p| + |\nabla q|  \right\|_{L^2\left(\Phi_R^{(\sigma)}\right) } \left\| \nabla \phi - \nabla v \right\|_{L^2\left(\Phi_R^{(\sigma)}\right) }
\\
& \geq 
\frac12 \int_{\Phi_R}\frac12 \left( \nabla p -\nabla q \right)  \cdot \ahom\left( \nabla p -\nabla q \right)     - C R^{-\delta_4(d-s)}.
\end{align*}
This completes the proof since the Caccioppoli estimate and the polynomial growth implies, for all $T\geq S$, 
\begin{equation*} 
\left\| \nabla \phi  \right\|_{L^2(\cu_{T})} \leq C \left(\frac{T}{R} \right)^{k-1}\,, 
\end{equation*}
and this is enough to control the tail terms. 
\end{proof}

We now complete the proof of Proposition~\ref{p.basecase}.

\begin{proof}[{Proof of Proposition~\ref{p.basecase}}]
Lemma~\ref{l.Ys} gives us~\eqref{e.Ys}, and Lemmas~\ref{l.JbreakupUB} and~\ref{l.JbreakupLB} combine to yield~\eqref{e.Jminimalrad}. 
\end{proof}

\begin{proof}[Proof of Proposition~\ref{p.thebasecase}]
We use the normalization $\left\| \nabla p \right\|_{L^2(\Phi_R)}^2 + \left\| \nabla q \right\|_{L^2(\Phi_R)}^2 = 1$. 
Denote 
\begin{equation*} 
\mathsf{Err}(y,r,p,q) := \left| J(y,r,p,q) - \int_{\Phi_{y,r}} \frac12\left( \nabla p -\nabla q\right)\cdot  \ahom  \left( \nabla p -\nabla q\right)\right|\,.
\end{equation*}
According to~\eqref{e.Ys} and~\eqref{e.Jminimalrad}, we have by the homogeneity that 
\begin{equation*} 
\mathsf{Err}(y,r,p,q) \leq C \left( \left\| \nabla p \right\|_{L^2(\Phi_{z,r})}^2 + \left\| \nabla q \right\|_{L^2(\Phi_{z,r})}^2\right)   \left(r^{-\delta(d-1)} + \indc_{\{\Y_t(y) > r\}} \right) 
\end{equation*}
Clearly 
\begin{equation*}
\indc_{\{ \mathcal{Y}_1(y) > r \}} \le \frac{\mathcal{Y}_1(y)}{r} =  \O_1\left(Cr^{-1} \right)\,.
\end{equation*}
Taking expectation yields, via Lemma~\ref{l.snappingtheLs}, that 
\begin{multline} \label{e.LvsId000}
\left\|\nabla \Ll(q - L_{z,r} q\Rr)\right\|_{L^2(\Phi_{z,r})}  + \left\|\nabla \Ll(p - L_{z,r}^* p\Rr)\right\|_{L^2(\Phi_{z,r})} 
\\ \leq 
C \left( \left\| \nabla p \right\|_{L^2(\Phi_{z,r})}^2 + \left\| \nabla q \right\|_{L^2(\Phi_{z,r})}^2\right)   \left(r^{-\delta(d-1)} + r^{-1} \right) \,.
\end{multline}
Defining further the probability measure
\begin{equation*} 
\mu(dy) = \Phi_{\sqrt{R^2 - r^2}}(y) 
\left(\left\| \nabla p \right\|_{L^2(\Phi_{y,r})}^2 + \left\| \nabla q \right\|_{L^2(\Phi_{y,r})}^2   \right)
\, d y \,,
\end{equation*}
we have by Lemma~\ref{l.sum-O}, for all $t \geq 1$, 
\begin{equation}  \label{e.Err000}
\mathsf{Err}(0,R,p,q)  +  \int_{\Phi_{\sqrt{R^2-r^2}}} \mathsf{Err}(y,r,p,q) \leq \left(C r^{-\delta(d-1)} + \O_t \left(  C r^{-1/t} \right) \right)\,.
\end{equation}
Using~\eqref{e.LvsId000} and~\eqref{e.Err000} it is easy to see that both $\Add_t(\alpha_0/t)$ and $\Fluc_t(\alpha_0/t)$ hold for small enough $\alpha_0(d,\Lambda)$ and for all $t\geq 1$. Finally, to obtain $\Dual_k(\alpha_0)$, we clearly have, for $p \in \Ahom_1(\Phi_{z,r})$, 
\begin{equation*} 
J(z,r,p,p) = \mathsf{Err}(0,R,p,p) \leq \left(C r^{-\delta(d-1)} + \O_1 \left(  C r^{-1} \right) \right)\,.
\end{equation*}
Thus we get by~\eqref{e.first-var} and~\eqref{e.second-var}, for all $w \in \A_k(\Phi_{z,r})$,
\begin{align*} 
\left| \int_{\Phi_{z,r}}\nabla p \cdot \left( \a - \ahom \right)  \nabla w  \right| & = \left| \int_{\Phi_{z,r}} \nabla v(\cdot,z,r,p,p) \cdot \a \nabla w  \right|
\\ & \leq C J(z,r,p,p)^{\frac12}
\\& \leq \left(C r^{-\delta(d-1)/2} + \O_{2} \left(  C r^{-\frac12} \right) \right)\,.
\end{align*}
Taking supremum over $p \in \Ahom_1(\Phi_{z,r})$ and $w \in \A_k(\Phi_{z,r})$  then gives
\begin{equation*} 
\E\left[ \sup_{w \in \A_k(\Phi_{z,r})} \left| \int_{\Phi_{z,r}}  \left( \a - \ahom \right)  \nabla w  \right|  \right] \leq  C \left( r^{-\delta(d-1)/2} + r^{-\frac12}  \right)\,,
\end{equation*}
which gives $\Dual_k(\alpha_0)$ for small enough $\alpha_0$. The proof is complete. 
\end{proof}

\section{Improved scaling of the fluctuations}
\label{s.fluc-subopt}

In this section we prove Proposition~\ref{p.control-fluct}, which controls the size of the fluctuations of $J$ assuming sufficient additivity and localization.

\subsection{Preliminaries}
In order to streamline the presentation of the proof, we begin with some notation for centered random variables, and record a few elementary properties thereof. 

\smallskip

It will be convenient to phrase stochastic integrability in terms of the behavior of the Laplace transform. Before introducing new notation, we recall how our previous notation relates to this behavior: applying Chebyshev's inequality, one can show that for every $s > 1$, there exists $C(s) < \infty$ such that
\begin{equation}
\label{e.stoch-equiv1}
X \le \O_s(1)  \quad \implies \quad \mbox{for all} \ \lambda \ge 1, \ \log \E[\exp(\lambda X)] \le C \lambda^{\frac s {s-1}},
\end{equation}
\begin{equation}
\label{e.stoch-equiv2}
\mbox{for all} \ \lambda \ge 1, \ \log \E[\exp(\lambda X)] \le  \lambda^{\frac s {s-1}} \quad \implies \quad X \le \O_s(C).
\end{equation}

\medskip 

For every $s \in (1,2]$ and $\theta \ge 0$, we write
\begin{equation}
\label{e.def.Ost}
X = \bar \O_s(\theta)
\end{equation}
to mean that
$$
\mbox{for all} \ \lambda \in \R, \quad \log \E \left[ \exp\left( \lambda \theta^{-1}  X \right) \right] \leq  \lambda^2 \vee   | \lambda|^{\frac s {s-1}}  .
$$
For centered random variables, the notions of $\O_s$ and $\bar \O_s$-bounded random variables coincide, up to a multiplicative constant.
\begin{lemma}
\label{l.logL}
Let $s \in (1,2]$. There exists $C(s) < \infty$ such that for every random variable $X$,
$$
X = \bar \O_s(1) \quad \implies \quad X = \O_s(C),
$$
and conversely,
$$
X = \O_s(1) \quad \text{and} \quad \E[X] = 0 \quad  \implies \quad X = \bar \O_s(C). $$
\end{lemma}
\begin{proof}
The first part is a consequence of~\eqref{e.stoch-equiv2}; the second part is classical and can be derived from~\eqref{e.stoch-equiv1} and \cite[Lemma~5.3]{AKM}.
\end{proof}

The key ingredient of the proof of Proposition~\ref{p.control-fluct} is the simple observation that a sum of $k$ independent $\bar \O_s(\theta)$ random variables is $\bar \O_s(\sqrt{k} \, \theta)$, in agreement with the scaling of the central limit theorem. This is formalized in the second part of the next lemma. 
\begin{lemma}
\label{l.barO}
For every $s \in (1,2]$, there exists $C(s) < \infty$ such that the following statements hold.

\noindent \emph{(i)} Let $\mu$ be a measure over an arbitrary measurable space $E$, let $\theta : E \to \R_+$ be a measurable function and $(X(x))_{x \in E}$ be a jointly measurable family of random variables such that for every $x\in E$, $X(x) = \bar \O_s(\theta(x))$. We have
$$
\int X \, d\mu = \bar \O_s \Ll(C   \int \theta \, d\mu  \Rr) .
$$

\noindent \emph{(ii)} Let $\theta_1, \ldots, \theta_k \ge 0$ and $X_1, \ldots, X_k$ be random variables such that for every $i$, $X_i = \bar \O_s (\theta_i)$. If the random variables $(X_i)$ are independent, then the previous estimate can be improved to
\begin{equation}
\label{e.barO}
\sum_{i = 1}^k X_i = \bar \O_s \Ll( C \Big(\sum_{i = 1}^k \theta_i^2\Big)^{\frac 1 2} \Rr) .
\end{equation}
Moreover, if $\theta_i  = \theta_j$ for every $i, j \in \{1,\ldots,k\}$, then the constant $C$ in \eqref{e.barO} can be chosen equal to $1$. 
\end{lemma}
\begin{proof}[Proof of Lemma~\ref{l.barO}]
The first statement is a consequence of Lemmas~\ref{l.sum-O} and~\ref{l.logL}. 
We now turn to the second statement, and define 
$$
\bar \theta := \Ll(\sum_{i = 1}^k \theta_i^2\Rr)^{\frac 1 2}.
$$ 
By \eqref{e.stoch-equiv2} and Lemma~\ref{l.logL}, in order to prove \eqref{e.barO}, it suffices to show that there exists $C(s) < \infty$ such that for every $\lambda \in \R$,
\begin{equation}
\label{e.barO2}
\log \E \Ll[ \exp \Ll(  \bar \theta^{\, -1} \lambda \sum_{i = 1}^k X_i \Rr)  \Rr] \le C\Ll(1+|\lambda|^{\frac s {s-1}}\Rr).
\end{equation}
We use independence and then the assumption $X_i = \bar \O_s(\theta_i)$ to bound the term on the left side by
\begin{align*}  
\sum_{i = 1}^k \log \E \Ll[ \exp \Ll( \bar \theta^{\, -1}  \lambda   X_i \Rr)  \Rr] & \le \sum_{i = 1}^k \Ll( \bar \theta^{\, -1} \theta_i \lambda \Rr)^2 \vee \Ll|\bar \theta^{\, -1} \theta_i\lambda \Rr|^{\frac s {s-1}} \\
& \le \lambda^2  + |\lambda|^{\frac s {s-1}} \bar \theta^{\, -\frac s {s-1}}  \sum_{i = 1}^k \theta_i^{\frac s {s-1}}.
\end{align*}
Since $s \le 2$, we have $\frac s{s-1} \ge 2$, and thus \eqref{e.barO2} follows from the observation that
\begin{equation*}  
\bar \theta \ge \Ll(\sum_{i = 1}^k \theta_i^{\frac s {s-1}}\Rr)^{\frac {s-1} s}.
\end{equation*}
When $\theta_i  = \theta_j$ for every $i, j \in \{1,\ldots,k\}$, without loss of generality we may set $\theta_i = 1$, and observe that
\begin{align*}
\sum_{i = 1}^k \log \E \Ll[ \exp \Ll( k^{-\frac 1 2} \lambda   X_i \Rr) \Rr]
& \le \sum_{i = 1}^k \Ll(k^{-\frac 1 2 } \lambda\Rr)^2 \vee  \left| k^{-\frac 1 2 } \lambda\right|^{\frac s {s-1}}  \\
& \le \lambda^2 \vee |\lambda|^{\frac s {s-1}},
\end{align*}
where in the last step we used the fact that $\frac s {s-1} \ge 2$.
\end{proof}

In the proof of Proposition~\ref{p.control-fluct}, we will encounter sums of random variables with short-range dependence. We provide a version of Lemma~\ref{l.barO}(ii) adapted to this situation.
\begin{lemma}
\label{l.partition}
For every $s \in (1,2]$, there exists $C(s) < \infty$ such that the following holds. Let $\theta > 0$, $R \ge 1$, $\mcl Z$ be a subset of $(R\Z)^d$, and for each $x \in \mcl Z$, let $X(x)$ be an $\mcl F(\cu_{2R}(x))$-measurable random variable such that $X(x) = \bar \O_s(\theta(x))$. We have
$$
\sum_{x \in \mcl Z} X(x) = \bar \O_s\Ll(C\, \Big( \sum_{x \in \mcl Z} \theta(x)^2 \Big)^{\frac 1 2} \Rr).
$$
\end{lemma}
\begin{proof}
We partition $\mcl Z$ into $\mcl Z^{(1)}, \ldots, \mcl Z^{(3^d)}$ in such a way that for every $j \in \{1,\ldots, 3^d\}$, if $x \neq x' \in \mcl Z^{(j)}$, then $|x-x'| \ge 3R \ge 2R + 1$. (That is, we define $\mcl Z^{(1)} = (3R\Z)^d \cap \mcl Z$, and so on with translates of $(3R\Z)^d$.) For each $j$, the random variables $(X(x))_{x \in \mcl Z^{(j)}}$ are independent. By Lemma~\ref{l.barO},
$$
\sum_{x \in \mcl Z^{(j)}} X(x) = \bar \O_s \Ll(C\, \Big( \sum_{x \in \mcl Z^{(j)}} \theta(x)^2 \Big)^{\frac 1 2}  \Rr) ,
$$
and the conclusion follows by summing over $j$.
\end{proof}
In the next lemma, we show that we can transfer the localization and fluctuation properties from $J$ to $I$ or vice versa. 
\begin{lemma}
\label{l.fluc.transfer}
\label{l.loc.transfer}
Let $s,\de,\al \in (0,\infty)$. 

\noindent \emph{(i)} The property $\Fluc_k(s,\al)$ holds if and only if there exists $C(k,s,\alpha,d,\Lambda)<\infty$ such that, for every $z \in \Rd$, $r\geq r_0$ and $p,q\in \Ahom_k(\Phi_r)$,
\begin{equation*}
I(z,r,p,q) = \E \left[I(z,r,p,q)\right] + \O_s\left(Cr^{-\alpha}\right).
\end{equation*}
\noindent \emph{(ii)} The property $\Loc_k(s,\de,\al)$ holds if and only if there exist $C(k,s,\de,\al,d,\Lambda) < \infty$ and, for every $z \in \Rd$, $r \ge r_0$ and $p,q \in \Ahom_k(\Phi_{z,r})$, an $\mcl F\left(B_{r^{1+\de}}(z)\right)$-measurable random variable $I^{(\delta)}(z,r,p,q)$ such that
\begin{equation*}
I(z,r,p,q) = I^{(\delta)}(z,r,p,q) + \O_s(C r^{-\alpha}),
\end{equation*}
\begin{equation}
\label{e.bounded-Jd}
|I^{(\delta)}(z,r,p,q)| \le  \Lambda,
\end{equation}
and $(p,q) \mapsto I^{(\delta)}(z,r,p,q)$ is a quadratic form.
\end{lemma}
\begin{proof}
The equivalence between the statements concerning the localization or the fluctuations of $I$ and $J$ follow from \eqref{e.I.as.J}, \eqref{e.J.as.I} and \eqref{e.tail.I.series}. We can enforce the last two stated properties of $I^{(\delta)}$ since these hold for $I$ itself, see \eqref{e.bounded-I}. 
\end{proof}
For the purpose of proving Proposition~\ref{p.control-fluct}, we may increase $\delta > 0$ ever so slightly and assume that the random variables in the statement of $\Loc_k(s,\de,\al)$ or of Lemma~\ref{l.loc.transfer}(ii) are $\mcl F(\cu_{r^{1+\de}}(z))$-measurable instead of $\mcl F(B_{r^{1+\de}}(z))$-measurable. We will use this observation without further comment below.

\smallskip

Finally, we recall for future reference that there exists $C(k) < \infty$ such that for every $z \in \Rd$, $R \ge r \ge 1$ and $p,q \in \bar {\mcl A}_k$, 
\begin{equation}
\label{e.move-center}
\int_{\Phi_{z,r}} \Ll( |\nabla p|^2 + |\nabla q|^2 \Rr) \le C (1 + (R^{-1} |z|)^{2(k-1)}) \int_{\Phi_R} \Ll( |\nabla p|^2 + |\nabla q|^2 \Rr) .
\end{equation}

\subsection{Fluctuations at suboptimal scales}

In this subsection we prove the first statement of Proposition~\ref{p.control-fluct}, which improves the scaling of the fluctuations at suboptimal scales. 

\begin{proof}[Proof of Proposition~\ref{p.control-fluct}(i)]
Throughout the proof, the value of the constant $C(\al,\be,\de,s,k,d,\Lambda) < \infty$ may change from place to place. We decompose the proof into three steps.

\smallskip

\emph{Step 1.} By Lemma~\ref{l.fluc.transfer}(i), our goal is to show that for every $z \in \Rd$, $R \ge r \ge r_0$ and $p,q \in \bar \A_k(\Phi_{z,r})$,
\begin{equation}
\label{e.fluc-proper}
I(z,r,p,q) = \E \Ll[ I(z,r,p,q) \Rr] + \O_s(C r^{-\be}). 
\end{equation}
We use the assumption of $\Loc_k(s,\de,\al)$ to define the random variables $\Jd$ appearing in the conclusion of Lemma~\ref{l.loc.transfer}(ii)
and let
\begin{equation}
\label{e.def.tjd}
\tJd(z,r,p,q) := \Jd(z,r,p,q) - \E[\Jd(z,r,p,q)].
\end{equation}
By the construction of $\Jd$ and since $\al > \be$, in order to prove \eqref{e.fluc-proper}, it suffices to show that
\begin{equation}
\label{e.fluc-proper-de}
\tJd(z,r,p,q) = \O_s(C r^{-\be}). 
\end{equation}
In order to prove \eqref{e.fluc-proper-de}, we start by transfering the additivity  assumption on $I$ to $\tJd$, and chopping the tails of the heat kernel mask for convenience. 
For each $\td \de > 0$, we introduce the function $\Phi^{(\td \de)}_{z,R,r}$ defined by
\begin{equation}
\label{e.PhiRdef}
\Phi^{(\td \de)}_{z,R,r}:= \left\{ \begin{aligned}
& \Phi_{z,\sqrt{R^2 - r^2}} & \mbox{in} & \ \cu_{R^{1+\td \de}}(z), \\
& 0 & \mbox{in} & \ \Rd\setminus \cu_{R^{1+\td \de}}(z).
\end{aligned} \right.
\end{equation}
This cutoff differs slightly from that defined in \eqref{e.Phideltadef}. We use the notation in \eqref{e.not.int} with $\Phidt_{z,R,r}$ in place of $\Phi_{z,r}$ as well. 

\smallskip

By \eqref{e.bounded-I}, assumption $\Add_k(s,\alpha)$ implies that for every $z \in \Rd$, $R \ge r \ge r_0$ and $p,q \in \bar {\mcl A}_k(\Phi_{z,R})$, 
$$
I(z,R,p,q) = \int_{\Phidt_{z,R,r}} I(\cdot,r,p,q) + \O_s \Ll(C  r^{-\al}   \Rr) . 
$$
By the construction of $\Jd$ and Lemma~\ref{l.barO}, up to a redefinition of $C < \infty$, we infer that for every $z,R,r,p,q$ as above,
$$
\Jd(z,R,p,q) = \int_{\Phidt_{z,R,r}} \Jd(\cdot,r,p,q) + \O_s \Ll(C  r^{-\al}   \Rr) .
$$
This implies
$$
\Ll|\E[ \Jd(z,R,p,q)] - \int_{\Phidt_{z,R,r}} \E[ \Jd(\cdot,r,p,q)] \Rr| \le C  r^{-\al} ,
$$
and therefore
\begin{equation}
\label{e.additivity-centered}
\tJd(z,R,p,q) = \int_{\Phidt_{z,R,r}} \tJd(\cdot,r,p,q) +  \bar \O_s \Ll(C  r^{-\al}   \Rr) .
\end{equation}
Indeed,  we can write $\bar \O_s$ instead of $\O_s$ on the right side above by Lemma~\ref{l.logL}.

\smallskip

\emph{Step 2.}
For $R_1, \CC \ge 1$, let $\msf A(R_1,\CC)$ denote the statement that for every $R \in [r_0,R_1]$, $z \in \R^d$ and $p, q \in \bar \A_k(\Phi_{z,R})$,
$$
\tJd(z,R,p,q) = \bar \O_s \Ll( \CC \, R^{-\be}  \Rr) .
$$
By \eqref{e.bounded-Jd}, for any given $R_1$, there exists $\CC < \infty$ such that $\msf A(R_1, \CC)$ holds. In order to prove the result, it thus suffices to show that for every $\CC$ sufficiently large and $R_1$ sufficiently large,
\begin{equation}
\label{e.fluct-induction}
\msf A(R_1, \CC) \quad \implies \quad  \msf A(2R_1, \CC).
\end{equation}
Indeed, this yields the existence of a constant $\CC$ such that $\msf A(R_1,\CC)$ holds for every $R_1 \ge r_0$, and this is \eqref{e.fluc-proper-de}.

\smallskip

\emph{Step 3.} We prove \eqref{e.fluct-induction} for $\CC$ and $R_1$ sufficiently large. Let $R \in (R_1,2R_1]$, $z \in \Rd$ and $p,q \in \bar {\mcl A}_k(\Phi_{z,R})$. We will show that
$$
\tJd(z,R,p,q) = \bar \O_s \Ll( C \CC R^{-\beta - \eps} \Rr) + \bar \O_s \Ll( C R^{-\beta} \Rr) ,
$$
for some exponent $\eps > 0$ depending on $d$, $k$, $\al$, $\be$ and $\de$ (and a constant $C$ not depending on $\CC$). Indeed, 
In view of Lemma~\ref{l.barO}, this is sufficient to prove that \eqref{e.fluct-induction} holds for $R_1$ sufficiently large.

\smallskip

Without loss of generality, we assume that $z = 0$. We let $r := R^{\be/\al}$, and rewrite the additivity property \eqref{e.additivity-centered} as
\begin{equation}
\label{e.additivity-again}
\tJd(0,R,p,q) = \int_{\Phidt_{R,r}} \tJd(\cdot,r,p,q) + \bar \O_s \Ll(C  R^{-\be}   \Rr) ,
\end{equation}
where we use the notation $\Phidt_{R,r}$ instead of $\Phidt_{0,R,r}$ for concision.
The condition $\de < \frac {(\al - \be)}{\be}\frac{(d-2\be)}{d}$ implies in particular that $1+ \de < \frac \al \be$, so that informally, we have $r^{1+\de} \ll R$. Without loss of generality, we will assume that $\td \de > 0$ is adjusted so that $R^{1+\td\de}/(r^{1+\de})$ is an odd integer. 
We let 
$$
\mcl Z := \Ll(r^{1+\de} \, \Z^d\Rr) \cap \cu_{R^{1+\td\de}},
$$
and observe that up to a set of Lebesgue measure zero, $(\cu_{r^{1+\de}}(y))_{y \in \mcl Z}$ is a partition of $\cu_{R^{1+\td\de}}$. We use this partition to decompose the integral on the right side of \eqref{e.additivity-again} as
\begin{equation}
\label{e.splitting}
\sum_{y \in \mcl Z} \int_{\Phidt_{R,r}} \tJd(\cdot,r,p,q) \1_{\cu_{r^{1+\de}}(y)}.
\end{equation}
Recall that we assume $p,q \in \bar \A_k(\Phi_{R})$. By \eqref{e.move-center}, we have, for $R \ge R_1$ sufficiently large,
\begin{equation}
\label{e.move-delta}
\forall x \in \cu_{R^{1+\td \de}}, \  \int_{\Phi_{x,r}} (|\nabla p|^2 + |\nabla q|^2) \le R^{2k\td \de}.
\end{equation}
Hence, using also the induction hypothesis and the fact that $r \ge R_1$, we obtain
$$
\forall x \in \cu_{R^{1+\td \de}}, \ \tJd(x,r,p,q) = \bar \O_s \Ll( \CC r^{-\be}  R^{2k \td \de} \Rr) ,
$$
and by Lemma~\ref{l.barO}(i), we have for every $y \in \mcl Z$ that
\begin{equation*}
\int_{\Phidt_{R,r}} \tJd(\cdot,r,p,q) \1_{\cu_{r^{1+\de}}(y)} = \bar \O_s \Ll(C \CC r^{-\be}  R^{2k\td \de} \int_{\Phidt_{R,r}}  \1_{ \cu_{r^{1+\de}}(y)} \Rr).
\end{equation*}
By Lemma~\ref{l.partition}, we thus have
\begin{equation*}  
\sum_{y \in \mcl Z}\int_{\Phidt_{R,r}} \tJd(\cdot,r,p,q) \1_{\cu_{r^{1+\de}}(y)} \le \bar \O_s \Ll( C \CC r^{-\be} R^{2k\td\de} \, \Ll[\sum_{y \in \mcl Z} \Ll( \int_{\Phidt_{R,r}}  \1_{ \cu_{r^{1+\de}}(y)} \Rr)^2\Rr]^{\frac 1 2} \Rr) .
\end{equation*}
Since ${r^{1+\de}} \le \sqrt{R^2 - r^2}$ (for $R_1$ large enough), we have
$$
\sum_{y \in \mcl Z} \Ll( \int_{\Phidt_{R,r}}  \1_{ \cu_{r^{1+\de}}(y)} \Rr)^2 \le C r^{(1+\de)d} \, R^{-d}.
$$
Summarizing, we have shown that
\begin{equation*}
\int_{\Phidt_{R,r}} \tJd(\cdot,r,p,q) =  \bar \O_s \Ll( C \CC r^{-\be}  r^{(1+\de)\frac d 2} R^{2k\td \de - \frac d 2} \Rr) .
\end{equation*}
Recalling our choice of $r = R^{\frac \be \al}$, we observe that
$$
r^{-\be+\frac d 2}  R^{- \frac d 2} = R^{-\Ll[\frac d 2 \Ll( 1-\frac \be \al \Rr) + \be \Ll( \frac \be \al \Rr)  \Rr]},
$$
and that the exponent between square brackets is larger than $\be$.
It thus suffices to check that $\de$ and $\td \de$ are sufficiently small in terms of $d$, $\al$, $\be$ and $k$ to obtain the desired result. More precisely, we need that
$$
\Ll(\frac d 2 -  \be \Rr) \Ll( 1-\frac \be \al \Rr)  > \frac \be \al \frac d 2 \de + 2 k \td \de.
$$
Since $\td \de > 0$ can be chosen as small as desired, this condition reduces to our assumption $\de < \frac {(\al - \be)(d-2\be)}{\be d}$, so the proof is complete.
\end{proof}

\subsection{Fluctuations at the optimal scale}
In this subsection, we complete the proof of Proposition~\ref{p.control-fluct} by giving sufficient conditions for controlling the fluctuations at the optimal CLT scaling.  

\begin{proof}[{Proof of Proposition~\ref{p.control-fluct}(ii)}]
It is convenient to measure the size of elements of $\bar \A_k$ using a supremum norm: for every cube $\cu$ and $p,q \in \bar \A_k$, we write
$$
\|(p,q)\|_{\cu} := \sup_{x\in\cu} \Ll(\left|\nabla p(x)\right| \vee \left|\nabla q(x)\right|\Rr).
$$
It is clear that the norms on $\bar \A_k/\bar \A_0\times \bar \A_k/\bar \A_0$
$$
(p,q) \mapsto \Ll(\int_{\Phi_{z,r}} \Ll( |\nabla p|^2 + |\nabla q|^2 \Rr)  \Rr)^{\frac 1 2} 
\quad \mbox{and} \quad 
(p,q) \mapsto \|(p,q)\|_{\cu_r(z)}
$$
are equivalent, with multiplicative constants that do not depend on $r$ or $z$.

\smallskip


Since we assume $\al > (1+\de) \frac d 2$, there exist $0 < \eta_2 < \eta_1 < 1$ satisfying 
\begin{equation}
\label{e.cond1}
(1-\eta_1) \al > \frac d 2
\end{equation}
and
\begin{equation}
\label{e.cond-eta2}
(1-\eta_2)(1+\de) < 1.
\end{equation}

Two additional exponents $\eps(k,d,\eta_1,\eta_2,\de,\al) > 0$ and $\td \de(\eps, k,d,\eta_1,\eta_2,\de,\al) > 0$ appear in the argument below, and are assumed to be as small as needed. 
The value of the constant $C(k,d,\eta_1,\eta_2,\al,\eps,\td \de,\Lambda) < \infty$ may vary from place to place. 

\smallskip

We decompose the proof of $\Fluc_k(s,\frac d 2)$ into five steps.

\smallskip

\emph{Step 1.} We start by recalling some elements from Step 1 of the proof of Proposition~\ref{p.control-fluct}(i). We use the assumption of $\Loc_k(s,\de,\al)$ to construct the random variables $\Jd$ given in the conclusion of Lemma~\ref{l.loc.transfer}(ii)
and define $\tJd$ by \eqref{e.def.tjd}. By the construction of $\Jd$ and the fact that $\al \ge \frac d 2$, it suffices to show that for every $z \in \Rd$, $R \ge r \ge r_0$, and $p,q \in \bar {\mcl A}_k(\Phi_r)$,
$$
\tJd(z,r,p,q) = \bar \O_s\Ll(C r^{-\frac d 2} \Rr).
$$
Let $\td \de >0$. Arguing as in the proof of Proposition~\ref{p.control-fluct}(i), we see that our assumptions imply that for every $z \in \Rd$, $R \ge r \ge r_0$, and $p,q \in \bar {\mcl A}_k$ satisfying $\|(p,q)\|_{\cu_R(z)} \le 1$,
\begin{equation}
\label{e.additivity-centered-opt}
\tJd(z,R,p,q) = \int_{\Phidt_{z,R,r}} \tJd(\cdot,r,p,q) + \bar \O_s \Ll(C  r^{-\al}   \Rr).
\end{equation}


\smallskip

\emph{Step 2.}
For $R_1, \CC \ge 1$, we let $\msf A(R_1,\CC)$ denote the statement that for every $z \in \Rd$, $R \in [r_0,R_1]$, $r \in [R^{1-\eta_1},R^{1-\eta_2}]$, $p,q \in \bar \A_k$ satisfying $\|(p,q)\|_{\cu_R(z)} \le 1$ and (deterministic) $f : \R^d \to \R$ satisfying $\|f\|_{L^\infty} \le 1$,
\begin{equation}
\label{e.really719}
\fint_{\cu_R(z)} f \, \tJd(\cdot,r,p,q)  = \bar \O_s \Ll( \CC \, R^{-\frac d 2}  \Rr) .
\end{equation}
We aim to show that 
\begin{equation}
\label{e.induc.scale}
\mbox{there exists $\CC < \infty$ such that, for every $R \ge 1$, $\msf A(R,\CC)$ holds.}
\end{equation}
By \eqref{e.bounded-Jd} and Lemma~\ref{l.logL}, for any given $R_1$, there exists $\CC < \infty$ such that $\msf A(R_1, \CC)$ holds. In order to prove \eqref{e.induc.scale}, it thus suffices to show that there exists $\eps > 0$ such that for every $\CC$ and $R_1$ sufficiently large,
\begin{equation}
\label{e.induction-fluct}
\msf A(R_1, \CC) \quad \implies \quad \msf A(2R_1, (1+R_1^{-\eps})\CC).
\end{equation}
We now assume that $\msf A(R_1, \CC)$ holds (with $R_1$ as large as desired), and fix $z \in \Rd$, $R \in (R_1, 2R_1]$, $r \in [R^{1-\eta_1}, R^{1-\eta_2}]$, $p, q \in \bar \A_k$ satisfying $\|(p,q)\|_{\cu_R(z)} \le 1$, and $f : \R^d \to \R$ satisfying $\|f\|_{L^\infty} \le 1$. Proving \eqref{e.induction-fluct} amounts to showing that
\begin{equation}
\label{e.int-goal}
\fint_{\cu_R(z)} f \, \tJd(\cdot,r,p,q) = \bar \O_s \Ll( \CC (1+R_1^{-\eps}) R^{-\frac d 2} \Rr) .
\end{equation}
This is the purpose of the next two steps. 
Without loss of generality, we fix $z = 0$. 
It is in fact sufficient to show that
$$
\fint_{\cu_R} f \, \tJd(\cdot,r,p,q) = \bar \O_s \Ll( \CC (1+R^{-\eps})R^{-\frac d 2} \Rr) + \bar \O_s \Ll( C \CC R^{-\frac d 2 - \eps} \Rr)  .
$$

\smallskip

\emph{Step 3.}
We set $r_1 :=R_1^{1-\eta_1}$ and 
\begin{equation}
\label{e.def.g}
g(y) := \int_{\cu_R} f(x) \, \Phidt_{r,r_1}(y-x) \, dx,
\end{equation}
where we recall that write $\Phidt_{r,r_1} := \Phidt_{0,r,r_1}$ for convenience, see \eqref{e.PhiRdef}. Note that $\|g\|_{L^\infty} \le 1$. 
 In this step, we show that
\begin{equation}
\label{e.recompose}
\fint_{\cu_R} f \, \tJd(\cdot, r,p,q) = R^{-d} \int_{\cu_{R+r^{1+\td \de}}} g \, \tJd(\cdot,r_1,p,q)  + \bar \O_s \Ll( C R^{-\frac d 2 - \eps} \Rr) .
\end{equation}
By \eqref{e.additivity-centered-opt}, for every $x \in \cu_R$,
\begin{equation*}
\tJd(x, r,p,q) = \int_{\Phidt_{x,r,r_1}} \tJd(\cdot,r_1,p,q) + \bar \O_s\Ll( C r_1^{-\al} \Rr).
\end{equation*}
Multiplying by $f(x)$, integrating over $\cu_R$ and using Lemma~\ref{l.barO}, we get
\begin{equation}
\label{e.step3-fluc}
\fint_{\cu_R} f \, \tJd(\cdot, r,p,q)\\
 = \fint_{\cu_R} f(x) \Ll( \int_{\Phidt_{x,r,r_1}} \tJd(\cdot,r_1,p,q) \Rr) \, dx + \bar \O_s\Ll( C r_1^{-\al} \Rr) .
\end{equation}
The last term above is
$$
\bar \O_s \Ll( C R^{-\frac d 2 -\eps} \Rr) ,
$$
for $\eps=(1-\eta_1)\al - \frac d 2 >  0$.
By Fubini's theorem, the double integral on the right side of \eqref{e.step3-fluc} can be rewritten as
$$
R^{-d} \int_{\Rd} \tJd(y,r_1,p,q) g(y)\, dy.
$$
Since the function $g$ vanishes outside of $\cu_{R + r^{1+\td \de}}$, this proves \eqref{e.recompose}.

\smallskip

\emph{Step 4.} We prove \eqref{e.int-goal}. Without loss of generality, we may assume that the exponents $\td \de, \eps > 0$ are such that $(R+r^{1+\td \de})/R^{1-2\eps}$ is an odd integer. We set $\rho := (R+r^{1+\td \de} - 3 R^{1-2\eps})/2$ and 
$$
\mcl Z_1 := \frac 1 2 \Ll( \rho + R^{1-2\eps}  \Rr) \{-1,1\}^d \subset \R^d.
$$
This provides a decomposition of $\cu_{R + r^{1+\td \de}}$ into $2^d$ disjoint subcubes $(\cu_\rho(x))_{x \in \mcl Z_1}$ at distance at least $R^{1-2\eps}$ from one another, plus a remainder that we denote by
$$
\mcl B_1 := \cu_{R + r^{1+\td \de}} \setminus \bigcup_{x \in \mcl Z_1} \cu_\rho(x).
$$
Moreover, imposing $\td \de > 0, \eps > 0$ to be sufficiently small that $r^{1+\td \de} \le R^{1-2\eps}$ (for $R_1$ sufficiently large), we have
\begin{equation}
\label{e.cub.incl}
x \in \mcl Z_1 \quad \implies \quad \cu_\rho(x) \subset \cu_R.
\end{equation}
We first argue that the contribution of
$$
R^{-d} \int_{\mcl B_1} g \,  \tJd(\cdot,r_1,p,q) 
$$
is negligible. Indeed, note that for $\eps > 0$ sufficiently small, we have
\begin{equation*}  
r_1 = R_1^{1-\eta_1} \in [R^{(1-2\eps) (1-\eta_1)}, R^{(1-2\eps) (1-\eta_2)}],
\end{equation*}
and therefore, by the induction hypothesis, for every $x \in \cu_{R + r^{1+\td \de}}$, we have
\begin{equation*}  
\fint_{\cu_{R^{1-2\eps}}(x)} g \, \tJd(\cdot,r_1,p,q) = \O_s \Ll( C \CC R^{-(1-2\eps) \frac d 2} \Rr) .
\end{equation*}
This random variable is also $\mcl F(\cu_{R^{1-2\eps} + r_1^{1+\de}}(x))$-measurable, and by \eqref{e.cond1}, we can choose $\eps > 0$ sufficiently small that $r_1^{1+\de}< R^{1-2\eps}$ (for $R_1$ sufficiently large).
We can partition $\mcl B_1$ into at most $C R^{2\eps(d-1)}$ cubes of side length $R^{1-2\eps}$ (up to a set of null Lebesgue measure). By Lemma~\ref{l.partition}, we thus obtain that
\begin{align*}
R^{-d} \int_{\mcl B_1} g \,  \tJd(\cdot,r_1,p,q) &  = \bar \O_s \Ll( C \CC R^{\eps(d-1)} R^{-(1-2\eps) \frac d 2} R^{-2\eps d} \Rr)   \\
& = \bar \O_s \Ll( C \CC R^{-\frac d 2 - {\eps}} \Rr) ,
\end{align*}
and therefore this term is indeed negligible.

\smallskip

We now turn to the contribution of the integral over $\bigcup_{x \in \mcl Z_1} \cu_\rho(x)$. We first observe that
$$
\rho \le R_1 \quad \text{and} \quad r_1 \in [\rho^{1-\eta_1}, \rho^{1-\eta_2}].
$$
The first inequality follows from the fact that $R \le 2R_1$ and that we imposed $r^{1+\td \de} \le R^{1-2\eps}$. As a consequence, we also have $r_1 \ge \rho^{1-\eta_1}$. The last condition holds for $R_1$ sufficiently large, since $\eta_1 > \eta_2$:  
$$
\rho^{1-\eta_2} = \Ll( \frac{R+r^{1+\td \de} - 3 R^{1-2\eps}}{2} \Rr)^{1-\eta_2} \ge \Ll( \frac{R_1+r^{1+\td \de} - 3(2R_1)^{1-2\eps}}{2} \Rr)^{1-\eta_2}\ge R_1^{1-\eta_1} = r_1.
$$
We can thus apply the induction hypothesis to obtain that
$$
\fint_{\cu_\rho(x)} g \, \tJd(\cdot,r_1,p,q) = \bar \O_s \Ll( \CC \rho^{-\frac d 2}  \|(p,q)\|_{\cu_\rho(x)}^2\Rr).
$$
By \eqref{e.cub.incl}, for every $x \in \mcl Z_1$, we have $\|(p,q)\|_{\cu_\rho(x)} \le \|(p,q)\|_{\cu_R} \le 1$. Using also that $\rho \le R/2$, we can rewrite the estimate above as
$$
(R/2)^{-d}\int_{\cu_\rho(x)} g \, \tJd(\cdot,r_1,p,q) = \bar \O_s \Ll( \CC \rho^{-\frac d 2} \Rr)  .
$$
Moreover, the random variables on the left side above are independent as $x$ varies in $\mcl Z_1$. Therefore, by Lemma~\ref{l.barO},
$$
R^{-d} \sum_{x \in \mcl Z_1} \int_{\cu_\rho(x)} g \, \tJd(\cdot,r_1,p,q) = \bar \O_s \Ll( \CC 2^{-\frac d 2} \rho^{-\frac d 2} \Rr) . 
$$
By the definition of $\rho$, this completes the proof of \eqref{e.int-goal}.

\smallskip

\emph{Step 5.} We have now justified the induction \eqref{e.induction-fluct}, and therefore \eqref{e.induc.scale}. In this last step, we use additivity once more to obtain the desired pointwise control of $\tJd$.

Let $R \ge 1$ and $p, q \in \bar \A_k(\Phi_{R})$. We fix $r = R^{1-\eta_1}$, use \eqref{e.additivity-centered-opt} and \eqref{e.cond1} to get
\begin{equation}
\label{e.last-additivity}
\tJd(0,R,p,q) = \int_{\Phidt_{R,r}} \tJd(\cdot,r,p,q) + \bar \O_s \Ll( C R^{-\frac d 2 - \eps} \Rr) .
\end{equation}
We rewrite the integral on the right side as
\begin{equation*}  
\sum_{x \in R\Z^d} \int_{\cu_R(x)} \Phidt_{R,r} \ \tJd(\cdot,r,p,q) .
\end{equation*}
The summand indexed by $x$ in the sum above is $\mcl F(\cu_{R + r^{1+\de}}(x))$-measurable, and we recall that $(1-\eta_1)(1+\de) < 1$.
We choose $C < \infty$ sufficiently large that for every $x \in \Rd$,
\begin{equation*}  
\sup_{\cu_R(x)} \Phidt_{R,r} \le \frac{C}{R^{d}} \exp \Ll( - \frac {|x|^2}{CR^2} \Rr) .
\end{equation*}
By \eqref{e.move-center}, for every $x \in \Rd$, we have
\begin{equation*}  
\|(p,q)\|_{\cu_R(x)}^2 \le C \Ll(1+\Ll(\frac{|x|}{R}\Rr)^{2k}\Rr),
\end{equation*}
and therefore \eqref{e.induc.scale} yields
\begin{equation*}  
\int_{\cu_R(x)} \Phidt_{R,r} \ \tJd(\cdot,r,p,q) = \Ll(1+\Ll(\frac{|x|}{R}\Rr)^{2k}\Rr) \exp \Ll( -\frac{|x|^2}{CR^2} \Rr) \bar \O_s \Ll( C R^{-\frac d 2} \Rr) .
\end{equation*}
The conclusion follows from Lemma~\ref{l.barO}(i).
\end{proof}

\section{Improvement of additivity}
\label{s.additivity}

In this section, we prove Proposition~\ref{p.additivity}, which asserts roughly that good control of the fluctuations of~$J$ implies an improvement of additivity.

\smallskip

The argument, inspired by the ideas of~\cite[Section 3]{AS} and our previous paper~\cite{AKM}, relies on the connection between the gradient of the energy quantity~$J_k$ and the spatially averaged gradient and flux of its maximizers (given in~\eqref{e.gradient-J} above). Using this, we show that good control on the fluctuations of $J_k$ (in the form of assumption $\Fluc_k(s,\alpha)$) implies good control on the spatial averages of the gradients (and fluxes) of elements of $\A_k(\Rd)$. This allows us to match the maximizers of $I_k(0,R,p,q)$ and $I_k(0,r,p,q)$ on two different scales $R \ge r$, up to an error of order of $\O_s(Cr^{-\alpha})$, which improves the additivity to exponent~$2\alpha$ (at a loss of half of the stochastic integrability exponent). 

\smallskip

In other words, the fluctuations of the spatial averages of gradients and fluxes of maximizers are at most proportional to the fluctuations of $J_k$, but the additivity of $I_k$ is the \emph{square} of the fluctuations of the former. This is the basis of the bootstrap argument and it is what we focus on in this section.

\smallskip

Throughout this section, we fix parameters $k\in\N$, $s>0$ and $\alpha \in \left( 0,\frac ds \right)$ and suppose that 
\begin{equation}
\label{e.yourass}
\Fluc_k(s,\alpha)
\ \ \mbox{and} \ \ 
\Dual_k(\alpha) 
\quad \mbox{hold.}
\end{equation}
We denote by $\X$ the random variable $\X_{s\alpha}$ in Proposition~\ref{p.regularity}, and by $\X(x)$ its $\Zd$-stationary extension (that is, $\X(x):= \tau_x\X$). We also denote by $\mathcal{Y}$ the larger of the random variables $\mathcal{Y}_{s\alpha}$ in Proposition~\ref{p.basecase} and Corollary~\ref{c.convexity}, and by $\mathcal{Y}(x)$ its stationary extension. We also let $r_0$ be the deterministic scale introduced in Definition~\ref{def.r0}.

\subsection{Spatial averages and the coarse-grained flux}
We first use~$\Fluc_k(s,\alpha)$ to improve the stochastic integrability of~$\Dual_k(\alpha)$. 

\begin{lemma}
\label{l.fluxmaps}
There exists a constant $C(s,\alpha,k,d,\Lambda)<\infty$ such that, for every $z\in\Rd$ and $r \geq 1$,
\begin{equation} \label{e.fluxmaps}
\sup_{w\in \A_k(\Phi_{z,r}) }   \left| \int_{\Phi_{z,r}} \left( \a(y) - \ahom \right) \nabla w(y) \,dy  \right| \\
= \O_{s}\left(Cr^{-\alpha} \right).
\end{equation}
\end{lemma}
\begin{proof}
By the assumption of $\Fluc_s(\al)$ and \eqref{e.polarization}, for every $z \in \Rd$, $r \ge 1$ and $q \in \Ahom_k(\Phi_{z,r})$,
\begin{equation}
\label{e.fluc.to.sav}
\int_{\Phi_{z,r}}  \ahom \nabla v(\cdot,z,r,0,q) = \E \Ll[ \int_{\Phi_{z,r}}  \ahom \nabla v(\cdot,z,r,0,q)  \Rr] + \O_s(C r^{-\al})
\end{equation}
and
\begin{equation}
\label{e.fluc.to.fx}
\int_{\Phi_{z,r}}  \a \nabla v(\cdot,z,r,0,q) = \E \Ll[ \int_{\Phi_{z,r}} \a \nabla v(\cdot,z,r,0,q)  \Rr] + \O_s(C r^{-\al}).
\end{equation}
Indeed, to get~\eqref{e.fluc.to.fx} for instance, we see that, for every $p \in \Ahom_k(\Phi_{z,r})$,
\begin{align*}
\lefteqn{
\int_{\Phi_{z,r}} \nabla p\cdot \a \nabla v(\cdot,z,r,0,q)
} \qquad & \\
& = 
- J(z,r,p',q) + J(z,r,0,q) + J(z,r,p',0) \\
& = \E \left[ - J(z,r,p',q) \right] + \E \left[J(z,r,0,q)\right] + \E \left[J(z,r,p',0)\right] + \O_s\left( Cr^{-\alpha} \right) \\
& = \E \left[\int_{\Phi_{z,r}} \nabla p\cdot \a \nabla v(\cdot,z,r,0,q) \right] + \O_s\left( Cr^{-\alpha} \right).
\end{align*}
The finite dimensionality of $\Ahom_k$ allows us to extract an orthonormal basis $\{ p_i \}_{i=1}^{\dim(\Ahom_k)} \subseteq \Ahom_k(\Phi_{z,r})$ of $\Ahom_k$, i.e., 
\begin{equation*} 
\int_{\Phi_{z,r}} \nabla p_{i}(x) \cdot \ahom \nabla p_{j}(x) \, dx = \delta_{ij}\,,
\end{equation*}
and using this we obtain   
\begin{align*} \label{}
\lefteqn{
\left| \int_{\Phi_{z,r}}  \a \nabla v(\cdot,z,r,0,q) - \E \Ll[ \int_{\Phi_{z,r}} \a \nabla v(\cdot,z,r,0,q)  \Rr]  \right|
} \qquad & \\
& \leq \sup_{p \in \Ahom_k(\Phi_{z,r})}  \left| \int_{\Phi_{z,r}} \nabla p\cdot \a \nabla v(\cdot,z,r,0,q) - \E \left[\int_{\Phi_{z,r}} \nabla p\cdot \a \nabla v(\cdot,z,r,0,q) \right] \right| \\
& \leq \sum_{i\in \{ 1,\ldots,\dim(\Ahom_k)\}} \left| \int_{\Phi_{z,r}} \nabla p_i\cdot \a \nabla v(\cdot,z,r,0,q) - \E \left[\int_{\Phi_{z,r}} \nabla p_i\cdot \a \nabla v(\cdot,z,r,0,q) \right] \right| \\
& \leq \O_s\left( Cr^{-\alpha} \right). 
\end{align*}
This confirms~\eqref{e.fluc.to.fx} and the argument for~\eqref{e.fluc.to.sav} is similar.

\smallskip

By~\eqref{e.bounded-u} and the assumption of $\Dual_k(\al)$, 
$$
\Ll|\E \Ll[ \int_{\Phi_{z,r}}  (\a - \ahom) \nabla v(\cdot,z,r,0,q)  \Rr] \Rr| \le C r^{-\al}.
$$
Combining this with \eqref{e.fluc.to.sav} and \eqref{e.fluc.to.fx} and using that $\Ahom_k$ is finite dimensional, we obtain that
\begin{equation}
\label{e.flux-v}
\sup_{q \in \Ahom_k(\Phi_{z,r})} \Ll|\int_{\Phi_{z,r}} (\a - \ahom) \nabla v(\cdot,z,r,0,q) \Rr|= \O_s(Cr^{-\al}).
\end{equation}

\smallskip

We next study the surjectivity of the mapping
\begin{equation}
\label{e.map.q}
\Ll\{
\begin{array}{rcl}
\Ahom_k /\Ahom_0 & \longrightarrow & \A_k /\A_0 \\
q & \longmapsto & v(\cdot,z,r,0,q).
\end{array}
\Rr.
\end{equation}
For the purpose of proving the lemma, we may assume that $r\geq \Y(z)$ since for $w \in \A_k(\Phi_{z,r})$,
\begin{equation*}
\left| \int_{\Phi_{z,r}} \nabla p \cdot \left( \a - \ahom \right) \nabla w \,  \right| \, \1_{\{ r \le \Y(z)\} } \le C \, \1_{\{ r \le \Y(z)\} } \\
\le C \Ll( \frac{\Y(z)}{r} \Rr)^{\al} = \O_{s}\left( Cr^{-\alpha} \right).
\end{equation*}
By Corollary~\ref{c.convexity}, for every $r \ge\Y(z)$,
$$
J(z,r,0,q) \ge \frac 1 4 \int_{\Phi_{z,r}} \nabla q \cdot \ahom \nabla q.
$$
By \eqref{e.J-energy}, we get that for every $r \ge \Y(z)$,
\begin{equation}
\label{e.inj}
\|\nabla q\|_{L^2(\Phi_{z,r})} \le C \|\nabla v(\cdot,z,r,0,q)\|_{L^2(\Phi_{z,r})}.
\end{equation}
Hence, for $r \ge \Y(z)$, the mapping displayed in \eqref{e.map.q} is injective. By Proposition~\ref{p.regularity}, the spaces $\Ahom_k / \Ahom_0$ and $\A_k / \A_0$ have the same dimension and therefore the mapping in \eqref{e.map.q} is bijective. That is, for every $w \in \A_k(\Phi_{z,r})$, there exists $Q(w) \in \Ahom_k$ such that
$$
\nabla w = \nabla v(\cdot,z,r,0,Q(w)).
$$
By \eqref{e.inj}, we deduce
\begin{equation*}
\|\nabla(Q(w))\|_{L^2(\Phi_{z,r})}  \le  C \|\nabla w\|_{L^2(\Phi_{z,r})} \le C,
\end{equation*}
and the desired estimate follows by \eqref{e.flux-v}.
\end{proof}

We next give a technical lemma, used many times in what follows, which allows us to build a bridge between spatial averages, weighted by polynomials, on two different scales. Since it concerns only the convolution of polynomials with the heat kernel, it does not require the assumption~\eqref{e.yourass}. Recall that~$\mcl P_k$ is defined in~\eqref{e.defPk}.

\begin{lemma} \label{l.CKish}
For each $k\in\N$, there exists $C(k,d)<\infty$  and, for each $q \in\mathcal{P}_k$ and $1 \leq r \leq R/\sqrt 2$, a polynomial $\tilde{q} \in\mathcal{P}_k$ such that 
\begin{equation} \label{e.CKishbound}
\left\| \tilde{q} \right\|_{L^2\left( \Phi_{R} \right)} \leq C \left\| q\right\|_{L^2\left( \Phi_{R} \right)}\,
\end{equation}
and, for every $F  \in L^2(\Phi_R)$,
\begin{equation} \label{e.int by parts}
\int_{\Phi_R}  F(x) q(x)\,dx  =  \int_{\Phi_{\sqrt{R^2-r^2}}}  \tilde{q}(y)  \int_{\Phi_{y,r}} F(x)\,dx \, dy.
\end{equation}
\end{lemma}
\begin{proof}
We first use the Taylor expansion of $q$ at $y$ to write
$$
q(x) = \sum_{n=0}^{k} \frac{1}{n!} \nabla^{n} q(y) (x-y)^{\otimes n}\,. 
$$
We hence obtain by the semi-group property of the heat kernel that
\begin{align*} 
& \int_{\Phi_R}  F(x)   q(x)  \, dx 
\\ & \qquad = \int_{\Phi_{\sqrt{R^2-r^2}} }\int_{\Phi_{y,r}} F(x)   q(x)  \, dx \, dy
\\ & \qquad = \sum_{n=0}^{k} \frac1{n!}  \int_{\R^d}   F(x)  \int_{\R^d} \Phi_{\sqrt{R^2-r^2}}(y) \Phi_{r}(x-y) \nabla^{n} q(y)    (x-y)^{\otimes n}   \, dy \, dx. 
\end{align*}
Using the identities
\begin{equation*}
\left\{ 
\begin{aligned}
& (x-y) \Phi_{r}(x-y) = 2 r^2 \ahom \nabla_y \Phi_{r}(x-y)\,, \\
& \nabla \Phi_{\sqrt{R^2 - r^2}}(y) = - \frac{\ahom^{-1} y}{2(R^2 - r^2)}\Phi_{\sqrt{R^2 - r^2}}(y),
\end{aligned}
\right.
\end{equation*}
we get by integration by parts, for any smooth tensor $\mathbf{G}$ with polynomial growth and for $m\geq 1$, that
\begin{align}  \label{e.int-by-parts000}  
\lefteqn{\int_{\Phi_{\sqrt{R^2-r^2}}}  \Phi_{r}(x-y)  \mathbf{G}(y)    (x-y)^{\otimes m}   \, dy } \quad &
\\ \nonumber & = 2r^2 \int_{\R^d} \Phi_{\sqrt{R^2-r^2}}(y) \nabla_y \Phi_{r}(x-y) \cdot \ahom \mathbf{G}(y)    (x-y)^{\otimes (m-1)}   \, dy 
\\ \nonumber & = - 2r^2 \int_{\R^d} \ahom \nabla \Phi_{\sqrt{R^2-r^2}}(y)   \cdot  \mathbf{G}(y)    (x-y)^{\otimes (m-1)}   \Phi_{r}(x-y) \, dy 
\\ \nonumber& \qquad - 2r^2 \int_{\R^d} \Phi_{\sqrt{R^2-r^2}}(y)  \nabla_y \cdot \left( \ahom \mathbf{G}(y)    (x-y)^{\otimes (m-1)} \right)  \Phi_{r}(x-y)  \, dy 
\\ \nonumber& = \frac{r^2}{R^2-r^2} \int_{\Phi_{\sqrt{R^2-r^2}}}   \mathbf{G}(y) y^{\otimes 1}  (x-y)^{\otimes (m-1)} \Phi_{r}(x-y)  \, dy 
\\ \nonumber& \qquad - 2r^2 \int_{\Phi_{\sqrt{R^2-r^2}}}   \nabla_y \cdot \left( \ahom \mathbf{G}(y) \right)   (x-y)^{\otimes (m-1)} \Phi_{r}(x-y)  \, dy 
\\ \nonumber& \qquad + 2r^2(m-1) \int_{\Phi_{\sqrt{R^2-r^2}}}  \mathbf{G}(y) (\ahom e)^{\otimes 1} e^{\otimes 1}     (x-y)^{\otimes (m-2)}   \Phi_{r}(x-y)  \, dy \,,
\end{align}
where we denote $e = (1,\ldots,1)$. 
Set thus inductively, for $j\in\{0,\ldots,n-1\}$,
\begin{align*} 
\mathbf{G}_{n,n}(y) & := \nabla^n q(y)\,, \qquad \tilde{\mathbf{G}}_{n-1}(y) = 0
\\
\mathbf{G}_{j-1,n}(y) & := \frac{r^2}{R^2-r^2} \mathbf{G}_{j,n}(y) y^{\otimes 1}  - 2r^2 \nabla_y \cdot \left( \ahom \mathbf{G}_{j,n}(y)  \right)  +  \tilde{\mathbf{G}}_{j-1,n}(y)  
\\
\tilde{\mathbf{G}}_{(j-2)\vee 0,n}(y) & := 2r^2 (j-1) \mathbf{G}_{j,n}(y) (\ahom e)^{\otimes 1}  e^{\otimes 1}\,. 
\end{align*}
Since each step is contracting the tensor, we  have that $\mathbf{G}_{0,n}(y)$ is a scalar-valued polynomial in $y$ of degree $k$.
Applying~\eqref{e.int-by-parts000} repeatedly we deduce that
\begin{multline*} 
\int_{\R^d} \Phi_{\sqrt{R^2-r^2}}(y) \Phi_{r}(x-y)  \nabla^n q(y)   (x-y)^{\otimes n}   \, dy 
\\ =
\int_{\R^d} \Phi_{\sqrt{R^2-r^2}}(y)  \mathbf{G}_{0,n}(y) \Phi_{r}(x-y) \, dy\,.
\end{multline*}
 By changing the order of integration we get
\begin{equation}  \label{e.int by parts 2}
 \int_{\Phi_R}  F(x)   q(x)  \, dx  =  \int_{\Phi_{\sqrt{R^2-r^2}}} \sum_{n=0}^{k} \frac{1}{n!}  \mathbf{G}_{0,n}(y)  \int_{\Phi_{y,r}} F(x)   \, dx \, dy. 
\end{equation}
Therefore, we obtain the result provided that 
\begin{equation} \label{e.Gisbounded,goddamn}
\left\| \mathbf{G}_{0,n}(y) \right\|_{L^2\left(\Phi_{\sqrt{R^2-r^2}}\right) } \leq C \left\| q \right\|_{L^2\left(\Phi_{R}\right) }\,.
\end{equation}
To obtain this, we proceed inductively. Assuming that 
\begin{equation*} 
\sup_{j\leq m \leq n} r^{m-n} \left\| \mathbf{G}_{m,n}(y) \right\|_{L^2\left(\Phi_{\sqrt{R^2-r^2}}\right) } \leq C_j R^{-n}  \left\| q \right\|_{L^2\left(\Phi_{R}\right) }\,,
\end{equation*}
it is easy to see from the definition of $\mathbf{G}_{j-1,n}$ and the fact that $r \leq R/\sqrt{2}$ that there is a constant $C_{j-1}(C_j,d)$ such that 
\begin{equation*} 
\left\| \mathbf{G}_{j-1,n} \right\|_{L^2\left(\Phi_{\sqrt{R^2-r^2}}\right) } \leq C_{j-1} r^{n - (j-1)} R^{-n} \left\| q \right\|_{L^2\left(\Phi_{R}\right) }\,.
\end{equation*}
On the other hand, since $\mathbf{G}_{n,n} = \nabla^n q$, we have that the initial step is valid. This finishes the proof.  
\end{proof}

With the aid of the previous lemma, we can upgrade Lemma~\ref{l.fluxmaps} to include polynomial weights.
\begin{lemma}
\label{l.fluxmaps2}
Let $k,m \in \N$. There exists a constant $C(k,m,s,\alpha,d,\Lambda)<\infty$ such that, for every $x\in\Rd$ and $r \geq 1$,
\begin{equation} \label{e.fluxmaps2}
\sup_{w\in \A_k(\Phi_{z,r}) } \sup_{\mathbf{p} \in \mathcal{P}_m(\Phi_{z,r})} 
\left| \int_{\Phi_{z,r}} \mathbf{p}(y) \cdot \left( \a(y) - \ahom \right) \nabla w(y) \,dy  \right| \\
= \O_{s}\left(Cr^{-\alpha} \right)\,,
\end{equation}
where $\mathcal{P}_m(\Phi_{z,r}) $ stands for the set of $m^{th}$ degree vector-valued polynomials normalized so that $ \mathbf{p} \in \mathcal{P}_m(\Phi_{z,r})$ implies
$\left\| \mathbf{p} \right\|_{L^2(\Phi_{z,r})} \leq 1$.
\end{lemma}
\begin{proof}
Letting $\tilde{\mathbf{p}}$ be as in Lemma~\ref{l.CKish}, we obtain 
\begin{equation*} 
 \int_{\Phi_{z,r}} \mathbf{p}  \cdot  \left( \ahom  -\a \right)\nabla w 
  =   \int_{\Phi_{z,r/\sqrt{2}}}   \tilde{\mathbf{p}}(x) \cdot  \int_{\Phi_{x,r/\sqrt{2}}}  \left( \ahom  -\a(y) \right)\nabla w \, dy \, dx\,.
\end{equation*}
Applying Lemma~\ref{l.fluxmaps} (and Lemma~\ref{l.sum-O}) then yields the statement.
\end{proof}

The assumption that~$\Fluc_k(s,\alpha)$ holds gives us good control on the fluctuations of $J$ and therefore, by Lemma~\ref{l.fluxmaps} and the fact that $J$ is quadratic, of~$\nabla J$. This can be phrased in terms of an estimate on the spatial averages of maximizers of $J$.

\begin{lemma}
\label{l.whatflucgives}
There exists $C(s,\alpha,k,d,\Lambda)<\infty$ such that, for every $z\in\Rd$, $r\geq r_0$ and $p,q,q' \in \Ahom_k(\Phi_{z,r})$,
\begin{equation}
\label{e.whatflucgives}
\left| \int_{\Phi_{z,r}}\left(   \nabla (q-p)(x) - \nabla u(x,z,r,p,q)\right)  \cdot  \ahom \nabla q'(x)  \,dx \right| 
\leq  \O_s\left( Cr^{-\alpha} \right).
\end{equation}
\end{lemma}
\begin{proof}
By Remark~\ref{r.gradient} and~$\Fluc_k(s,\alpha)$, we have  that, for every  $p,q,q' \in \A_k(\Phi_{z,r})$,
\begin{equation} \label{e.whatflucgives000}
\left| \nabla_q J(z,r,p,q)(q') - \nabla_q \E \left[ J(z,r,p,q)(q') \right]  \right| = \O_s\left(Cr^{-\alpha} \right). 
\end{equation}
Let us compute $\nabla_q J(z,r,L_{z,r}^*p,L_{z,r}q)(q')$. We have by linearity that
\begin{align} \label{e.whatflucgives001}
\lefteqn{\nabla_q J(z,r,L_{z,r}^*p,L_{z,r}q)(q')  = \int_{\Phi_{z,r}} \ahom \nabla q' \cdot \nabla u(\cdot,z,r,p,q)    } \qquad & 
\\ \notag & =   \int_{\Phi_{z,r}} \ahom \nabla q' \cdot \nabla u(\cdot,z,r,0,q)   + \int_{\Phi_{z,r}} \a \nabla q' \cdot \nabla u(\cdot,z,r,p,0)   
\\ \notag & \qquad - \int_{\Phi_{z,r}} \nabla q' \cdot \left( \ahom  -\a \right)\nabla u(\cdot,z,r,p,0) \,.
\end{align}
To control the last term on the right, Lemma~\ref{l.fluxmaps2} yields
\begin{equation*} 
 \left| \int_{\Phi_{z,r}} \nabla q' \cdot \left( \ahom  -\a \right)\nabla u(\cdot,z,r,p,0)  \right|   \leq \O_s\left( C r^{-\alpha} \right)\,.
\end{equation*}
Taking expectation, we conclude by~\eqref{e.I-spat-av} and~\eqref{e.I-spat-flux} that
\begin{equation*} 
\left| \E\left[ \nabla_q J(z,r,L_{z,r}^*p,L_{z,r}q)(q') \right] - \int_{\Phi_{z,r}} \nabla q' \cdot \ahom \nabla(q-p) \right| \leq Cr^{-\alpha}\,.
\end{equation*}
This together with~\eqref{e.whatflucgives000} and the first line of~\eqref{e.whatflucgives001} finishes the proof. 
\end{proof}

We next show, using Lemma~\ref{l.fluxmaps}, that $r$-scale convolutions of elements of~$\A_k$ against the heat kernel are close to being~$\ahom$-harmonic functions.

\begin{lemma}
\label{l.coarsenedequation}
There exists $C(s,\alpha,k,d,\Lambda) < \infty$ and, for every $x\in\Rd$ and $r\geq 1$, a nonnegative random variable $\mathcal{H}_r(x)$ satisfying
\begin{equation}
\label{e.Hrest}
\mathcal{H}_r(x) \leq C \wedge \O_{s}(C r^{-\alpha})
\end{equation}
such that, for every $v \in \A_k(\Rd)$ and $\eta \in H^1_c(\Rd)$, we have 
\begin{multline} \label{e.almostahom}
\left| \int_{\Rd} \nabla \eta(x) \cdot \ahom \nabla \left( \int_{\Phi_{x,r} }  v(y)\,dy \right)\,dx  \right| \\
 \leq 
 \int_{\Rd} \left| \nabla \eta(x) \right| \left\| \nabla v \right\|_{L^2(\Phi_{x,r})}  \mathcal{H}_r(x)
\,dx.
\end{multline}
\end{lemma}
\begin{proof}
We define
\begin{equation*} \label{}
\mathcal{H}_r(x):= C \sup_{v\in \A_k} \left( \left\| \nabla v \right\|_{L^2\left(\Phi_{x,r}\right)}^{-1} \left| \int_{\Phi_{x,r}} \left( \a(y) - \ahom \right) \nabla v(y) \,dy  \right| \right).
\end{equation*}
It is clear that $\mathcal{H}_r(x) \leq C$.
Lemma~\ref{l.fluxmaps} then yields
\begin{align*} 
\lefteqn{
\left| \int_{\R^d} \nabla \eta(x) \cdot \ahom \nabla\left( \int_{\Phi_{x,r}}  v(y)\,dy \right) \, dx \right| 
} \qquad & \\
& = \left| \int_{\R^d} \nabla \eta(x) \cdot \int_{\Phi_{x,r}}   \ahom \nabla v(y)\,dy  \, dx \right|  \\
 & \leq  \left| \int_{\R^d} \nabla \eta(x) \cdot \int_{\Phi_{x,r}}   \a(y) \nabla v(y) \, dy  \, dx \right|  
+\int_{\Rd} \left| \nabla \eta(x) \right| 
 \left\| \nabla v \right\|_{L^2\left(\Phi_{x,r}\right)} \mathcal{H}_r(x) \, dx.
\end{align*}
The first integral is zero since
\begin{multline*}
 \int_{\R^d} \nabla \eta(x) \cdot \int_{\R^d}  \Phi_{r}(x-y) \a(y) \nabla v(y) \, dy  \, dx = 
 \\  \int_{\R^d} \left( \int_{\R^d}  \Phi_{r}(x-y) \nabla \eta(x) \,dx  \right) \cdot \a(y) \nabla v(y) \, dy  = 0\,.
  \end{multline*}
Indeed, the function $y\mapsto \int_{\R^d}  \Phi_{r}(x-y) \nabla \eta(x) \,dx$ is the gradient of an $H^1(\Rd)$ function decaying faster than any polynomial at infinity (due to the assumption that $\eta$ has compact support) and $\nabla v$ has almost surely at most polynomial growth at infinity.
\end{proof}

In the next lemma, we show that the spatial average of any element of $\A_k(\Rd)$ is close to an element of $\Ahom_k$ on every scale.

\begin{lemma}
\label{l.harmonicapprox}
For every $k\in\N$, there exists $C(k,s,\alpha,d,\Lambda) < \infty$ such that, for every $1\leq r \leq R/\sqrt{2}$, we have
\begin{equation}
\label{e.harmonicapprox}
 \sup_{v \in \A_k(\Phi_R)} \inf_{\mathsf{h} \in \Ahom_k}  \left\| \left( \nabla v \right)_{\Phi_{x,r}} - \nabla \mathsf{h} \right\|_{L^2\left(\Phi_{\sqrt{R^2 - r^2}}\right)}^2 
= \O_{s/2}\left(C r^{-2\alpha}\right) \,.
\end{equation}
\end{lemma}
\begin{proof}
Throughout, we fix $\tilde R:= \sqrt{R^2-r^2}$. 
Observe that the random variable 
\begin{equation*}
\tilde \X :=  \sup_{v \in \A_k(\Phi_R)} \inf_{\mathsf{h}\in \Ahom_k} \| \left( \nabla v \right)_{\Phi_{x,r}} - \nabla \mathsf{h} \|_{L^2(\Phi_{\tilde R})}^2
\end{equation*}
satisfies $\tilde \X \in [0,1]$ by the normalization~$\left\| \nabla v \right\|_{L^2(\Phi_R)} \leq 1$ (just take $\mathsf{h} = 0$ for any  given $v \in \A_k(\Phi_R)$). Thus, since $s \alpha < d$,
\begin{align*}
\tilde \X \indc_{\{\X \geq \tilde R\}} \leq \indc_{\{\X \geq \tilde R\}} \leq \Ll(\frac{\X}{\tilde R}\Rr)^{2\alpha} = \O_{s/2}\left(C \tilde R^{-2\alpha}\right) \leq \O_{s/2}\left(C r^{-2\alpha}\right)\,.
\end{align*}
Therefore, without loss of generality, we work on the event $\{\X \leq \tilde R\}$ throughout the rest of the argument. Furthermore, following the steps in the end of the proof of Lemma~\ref{l.fluxmaps}, we have that 
\begin{equation} \label{e.tildescriptXbound}
\tilde \X \leq C \sup_{p \in \Ahom_k(\Phi_{R}) }\inf_{\mathsf{h}\in \Ahom_k} \| \left( \nabla v(\cdot,z,r,0,p) \right)_{\Phi_{x,r}} - \nabla \mathsf{h} \|_{L^2(\Phi_{\tilde R})}^2 \,.
\end{equation}
Since $\Ahom_k(\Phi_{R})$ is a finite dimensional subspace of polynomials having an orthonormal basis $\left\{p_{j,R}\right\}_{j=1}^{\dim(\Ahom_k)} \subset \Ahom_k(\Phi_{R}) $, i.e.,
\begin{equation*} 
\int_{\Phi_{R}} \nabla p_{i,R}(x) \cdot \ahom \nabla p_{j,R}(x) \, dx = \delta_{ij}\,,
\end{equation*}
it is actually sufficient to prove that, for fixed $p_{j,R} $, we have 
\begin{equation*} 
\inf_{\mathsf{h}\in \Ahom_k} \| \left( \nabla v(\cdot,z,r,0,p_{j,R}) \right)_{\Phi_{x,r}} - \nabla \mathsf{h} \|_{L^2(\Phi_{\tilde R})}^2 \leq \O_{s/2}\left(C r^{-2\alpha}\right)\,.
\end{equation*}
The desired inequality~\eqref{e.harmonicapprox} follows from~\eqref{e.tildescriptXbound} by the following statement:
\begin{multline}
\label{e.harmonicapprox'}
  \inf_{\mathsf{h}\in \Ahom_k} \left\| \left( \nabla v(\cdot,z,r,0,p_{j,R}) \right)_{\Phi_{x,r}} - \nabla \mathsf{h} \right\|_{L(\Phi_{\tilde R})}^2  
\\ 
\leq \max\left\{\O_{s/2}\left(C r^{-2\alpha}\right)\,, \O_{\frac{s \alpha}{d+2\alpha} }\left(Cr^{-(d+2\alpha) } \right) \right\}\,.
\end{multline}
Indeed, by Remark~\ref{r.change-s},
$$
\O_{\frac{s \alpha}{d+2\alpha} }\left(Cr^{-(d+2\alpha) } \right) \wedge 1 \le \O_{s/2}\Ll( C r^{-2\al} \Rr) \,.
$$
Thus, since $\tilde \X \in [0,1]$, it suffices to prove~\eqref{e.harmonicapprox'}. For the rest of the argument, let us denote in short,  for fixed $j$, 
\begin{equation*} 
v = v(\cdot,z,r,0,p_{j,R}).
\end{equation*}

\smallskip

\emph{Step 1.} Harmonic approximation of $\left( v \right)_{\Phi_{x,r}}$ and iteration. For convenience, we will denote
\begin{equation*} \label{}
w(x) := \left( v \right)_{\Phi_{x,r}} =  \int_{\Phi_{x,r} }  v(y)\,dy\,.
\end{equation*}
For each $S\geq \tilde R$, we introduce an~$\ahom$-harmonic approximation of~$w$ in $B_S$, which we denote by~$h_S$. We take $h_S$ to be the unique element of $(w + H^1_0(B_S))\cap \Ahom(B_S)$. It follows from Lemma~\ref{l.coarsenedequation} (simply take $\eta = h_S-w$ and use Cauchy-Schwarz) that
\begin{align*}
\fint_{B_S} \left| \nabla w(x) - \nabla h_S(x) \right|^2\,dx 
& \leq 
 \fint_{B_S} \left\| \nabla v \right\|_{L^2(\Phi_{x,r})}^2 \left( \mathcal{H}_r(x) \right)^2  \,dx,
\end{align*}
where $\mathcal{H}_r(x) = C \wedge \O_s(C r^{-\alpha})$ is as in Lemma~\ref{l.coarsenedequation}.
Therefore, for every $\theta  \in (0,1]$,
\begin{equation}
\label{e.harmapprox}
\fint_{B_{\theta S}} \left| \nabla w(x) - \nabla h_S(x) \right|^2\,dx 
 \leq \theta^{-d}  \fint_{B_S}  \left\| \nabla v \right\|_{L^2(\Phi_{x,r})}^2  \left( \mathcal{H}_r(x) \right)^2  \,dx.
\end{equation}
By the regularity of $\ahom$-harmonic functions, we find, for every $\tilde{ \mathsf{h} } \in \Ahom_k$, 
\begin{align}
\label{e.tildep}
\inf_{\mathsf{h} \in \Ahom_k} \sup_{B_{\theta S}} \left|\nabla h_S- \nabla \mathsf{h} \right|  
& \leq C \theta^k  \fint_{B_S} \left| \nabla h_S(x)- \nabla \tilde{ \mathsf{h} }(x) \right| \, dx. 
\end{align}
From the triangle inequality and the previous two displays, if we denote 
\begin{equation*}
\omega(\varrho):= \varrho^{-k+\frac12} \inf_{\mathsf{h}\in \Ahom_k} \left\| \nabla w- \nabla \mathsf{h} \right\|_{\underline{L}^2(B_\varrho)},
\end{equation*}
then we obtain, for $\theta = (2C)^{-\frac12}$,
\begin{equation*}
\omega(\theta S) \leq \frac12 \omega(S) + C S^{-k+\frac12} 
\left( \fint_{B_S}  \left\| \nabla v \right\|_{L^2(\Phi_{x,r})}^2  \left( \mathcal{H}_r(x) \right)^2 \,dx \right)^{\frac12}.
\end{equation*}
Setting $S_j := \theta^{-j} \tilde R$ and summing over all the scales, using also the fact that $\omega(\varrho) \to 0$ as $\varrho \to \infty$ on the event $\{ \X(0) \leq \tilde R\}$, yields
\begin{equation}
\label{e.omegaRiter}
 \sum_{j=0}^\infty \omega(S_j) 
 \leq  C \tilde R^{-k + \frac12}\overline{\mathcal{H}}  \,,
\end{equation}
where we have defined
\begin{equation} \label{e.Y=sum}
\overline{\mathcal{H}} := \sum_{n=1}^\infty \theta^{n(k-\frac12)} \left( \fint_{B_{\theta^{-n} \tilde R}}  \left\| \nabla v \right\|_{L^2(\Phi_{x,r})}^2  \left( \mathcal{H}_r(x) \right)^2 \,dx \right)^{\frac12}.
\end{equation}
Letting $\mathsf{h}_j \in \Ahom_k$ be the minimizer appearing in the definition of $\omega(S_j)$, we obtain by the triangle inequality and the growth of polynomials in $\Ahom_k$ that, for $m>j$,
\begin{align*} 
\left\|\nabla \mathsf{h}_j - \nabla \mathsf{h}_{j+1} \right\|_{\underline{L}^2(B_{S_m})} & \leq C \left( \frac{S_m}{S_j} \right)^{k-1} \left\|\mathsf{h}_j - \mathsf{h}_{j+1} \right\|_{\underline{L}^2(B_{S_j})} 
\\ & \leq CS_m^{k-\frac12}  \left( \frac{S_j}{S_m} \right)^{\frac12} \left(\omega(S_j) + \omega(S_{j+1})\right) \,.
\end{align*}
Therefore we get, again by the triangle inequality, 
\begin{align*} 
 \left\|\nabla \mathsf{h}_0 - \nabla \mathsf{h}_{m} \right\|_{\underline{L}^2(B_{S_m})}  & \leq  \sum_{j=0}^{m-1} \left\|\nabla \mathsf{h}_j - \nabla \mathsf{h}_{j+1} \right\|_{\underline{L}^2(B_{S_m})} 
 \\ &  \leq C\tilde R^{k-\frac12}  \left( \frac{S_m}{\tilde R} \right)^{k-\frac12} \sum_{j=0}^m \omega(S_j) \leq C\left( \frac{S_m}{\tilde R} \right)^{k-\frac12} \overline{\mathcal{H}}  \,.
\end{align*}
Using now the decay properties of $\Phi_{\tilde R}$ we conclude that
\begin{align*} 
& \inf_{\mathsf{h} \in \Ahom_k} \left\| \nabla w - \nabla \mathsf{h} \right\|_{L^2(\Phi_{\tilde R})}   \leq  \left\| \nabla w - \nabla \mathsf{h}_0 \right\|_{L^2(\Phi_{\tilde R})}
\\ &  \qquad   \leq C \left(  \sum_{m=0}^\infty \exp\left( - c \theta^{-2m}\right) \theta^{-dm} \left( \left\| \nabla \mathsf{h}_m {-} \nabla \mathsf{h}_0 \right\|_{\underline{L}^2(B_{S_m})}^2 +  \left\| \nabla w {-} \nabla \mathsf{h}_m \right\|_{\underline{L}^2(B_{S_j})}^2 \right) \right)^{\frac12}
\\ &  \qquad   \leq C  \overline{\mathcal{H}}^{\frac12} \left(  \sum_{m=0}^\infty \exp\left( - c \theta^{-2m}\right) \theta^{-(d + 2k - 1)m}   \right)^{\frac12}
\\ &  \qquad   \leq C  \overline{\mathcal{H}}^{\frac12}\,.
\end{align*}
The rest of the proof is devoted to estimating the random variable $\overline{\mathcal{H}} $. 

\smallskip

\emph{Step 2.}  We use the Lipschitz estimate to pull the $\left\| \nabla v \right\|_{L^2(\Phi_{x,r})}^2 $ term outside the integrals appearing in $\overline{\mathcal{H}} $ defined in~\eqref{e.Y=sum}. The claim is that 
\begin{multline}
\label{e.lippullout}
 \fint_{B_{\theta^{-n} \tilde R}}  \left\| \nabla v \right\|_{L^2(\Phi_{x,r})}^2  \left( \mathcal{H}_r(x) \right)^2 \,dx  \\
  \leq C \theta^{-2n(k-1)} \fint_{B_{\theta^{-n}\tilde R} } \left( 1+\indc_{\{r\leq \X(x)\}} \left(  \frac{\X(x)}{r} \right)^{d} \right) \left(\mathcal{H}_r(x) \right)^2 \,dx \,.
\end{multline}
First, as we are assuming the event $\{ \tilde R \geq \X(0)\}$, we have that there is a polynomial $p_v \in \Ahom_k$ such that, for all $S \geq \tilde R$, 
$$
\left\|v \right\|_{\underline{L}^2(B_S)} \leq C \left\| p_v \right\|_{\underline{L}^2(B_S)} \leq C \left(\frac{S}{\tilde R}\right)^k \left\| p_v \right\|_{\underline{L}^2(B_{\tilde R})}  \leq 
C \left(\frac{S}{\tilde R}\right)^k \left\| v \right\|_{\underline{L}^2(B_{\tilde R})}.
$$
Without loss of generality we may assume that $(v)_{B_{\tilde R}} = 0$.  Poincar\'e's inequality and the normalization $ \left\| \nabla v  \right\|_{L^2(\Phi_R)} \leq 1$ then give that  
$$
\left\|v \right\|_{\underline{L}^2(B_S)} \leq C \left(\frac{S}{\tilde R}\right)^k \left\| v  \right\|_{\underline{L}^2(B_{\tilde R})} \leq C S \left(\frac{S}{\tilde R}\right)^{k-1} \,.
$$
Furthermore, suppose first that $r \leq S \leq  |x| + \tilde R$ and $\X(x) \leq |x|+\tilde R$. Then the Caccioppoli inequality and the Lipschitz bound (Proposition~\ref{p.mesoregularity} with $k=0$) yield
\begin{align*} 
\lefteqn{
\left\| \nabla v \right\|_{\underline{L}^2(B_S(x))} 
} \qquad & \\
& \leq \left(\frac{\X(x) \vee S}{S}\right)^{\frac d2} \left\| \nabla v \right\|_{\underline{L}^2(B_{\X(x) \vee S}(x))} 
\\
& \leq C \left(\frac{\X(x) \vee S}{S}\right)^{\frac d2} 
\left( \X(x) \vee S \right)^{-1}
\left\| v - (v)_{B_{2(\X(x) \vee S )}}\right\|_{\underline{L}^2(B_{2(\X(x) \vee S )}(x))}
\\& \leq C \left(\frac{\X(x) \vee S }{S}\right)^{\frac d2} 
\left(|x| + \tilde R \right)^{-1} \left\| v \right\|_{ \underline{L}^2(B_{2(|x| + \tilde R)})}
\\& \leq C \left(\frac{\X(x) \vee r}{r}\right)^{\frac d2} \left( \frac{|x| + \tilde R}{ \tilde R} \right)^{k-1}\,. 
\end{align*}
If, on the other hand, $ r\leq S \leq |x| + \tilde R$ and $\X(x) \geq |x|+\tilde R$, we get 
\begin{equation*} 
\left\| \nabla v \right\|_{\underline{L}^2(B_S(x))} \leq \left(\frac{|x|+\tilde R}{S}\right)^{\frac d2} \left\| \nabla v \right\|_{\underline{L}^2(B_{|x| + \tilde R }(x))} \leq 
C \left(\frac{\X(x) }{r}\right)^{\frac d2} \left( \frac{|x| + \tilde R}{\tilde R} \right)^{k-1}\,,
\end{equation*}
and finally if $ S \geq |x| + R$, then directly $\left\| \nabla v \right\|_{\underline{L}^2(B_S(x))}  \leq C (S/R)^{k-1}$. Using these gives
\begin{align*} 
\left\| \nabla v \right\|_{L^{2}(\Phi_{x,r})}^2 & \leq C \int_r^\infty \left(\frac{S}{r} \right)^{d} \exp \left(- c \frac{S^2}{r^2} \right) \fint_{B_{S}(x)} \left|\nabla v(y)\right|^2 \, dy\, \frac{dS}{S}
\\ & \leq C \left(\frac{\X(x) \vee r}{r}\right)^{d} \left( \frac{|x| + \tilde R}{\tilde R} \right)^{2(k-1)} \int_r^{|x|+\tilde R} \left(\frac{S^2}{r^2} \right)^{\frac d2} \exp \left(- c \frac{S^2}{r^2} \right) \, \frac{dS}{S}
\\ & \qquad + C \left( \frac {r}{\tilde R} \right)^{k-1} \int_{|x|+\tilde R}^\infty \left(\frac{S}{r} \right)^{d + k - 1} \exp \left(- c \frac{S^2}{r^2} \right) \, \frac{dS}{S}
\\ & \leq C \left(\frac{\X(x) \vee r}{r}\right)^{d} \left( \frac{|x| + \tilde R}{\tilde R} \right)^{2(k-1)}\,.
\end{align*}
We deduce that 
\begin{align*}
\lefteqn{
 \fint_{B_{\theta^{-n} \tilde R}}  \left\| \nabla v \right\|_{L^2(\Phi_{x,r})}^2  \left( \mathcal{H}_r(x) \right)^2  \,dx 
} \qquad & \\
 & \leq \fint_{B_{\theta^{-n} \tilde R}}  \left(\frac{\tilde R+|x|}{\tilde R}\right)^{2(k-1)} \left( 1+\indc_{\{r\leq \X(x)\}} \left(  \frac{\X(x)}{r} \right)^{d} \right)  \left( \mathcal{H}_r(x) \right)^2 \,dx  \\
 & \leq C \theta^{-2n(k-1)} \fint_{B_{\theta^{-n}\tilde R} } \left( 1+\indc_{\{r\leq \X(x)\}} \left(  \frac{\X(x)}{r} \right)^{d} \right) \left(\mathcal{H}_r(x)\right)^2 \,dx\,,
\end{align*}
which completes the proof of~\eqref{e.lippullout}. 

\smallskip

\emph{Step 3.} We next estimate the term in~\eqref{e.lippullout}.  The claim is that there exists $C(\alpha,s,k,d,\Lambda)<\infty$ such that
\begin{multline}
\label{e.Xbigweirdo}
\left( 1 + \indc_{\{r\leq \X(x)\}} \left(  \frac{\X(x)}{r} \right)^{d} \right) \left( \mathcal{H}_r(x) \right)^2 \\
= \max\left\{ \O_{\frac{s \alpha}{d + 2\alpha}}\left(C r^{-(d + 2\alpha)} \right)\,, \O_{s/2}(Cr^{-2\alpha}) \right\}.
\end{multline}
First, note that~\eqref{e.Hrest} gives that
\begin{equation*} \label{}
\left( \mathcal{H}_r(x) \right)^2 = \O_{s/2}(C r^{-2\alpha}).
\end{equation*}
Moreover, since $ s \alpha < d$,  we have for large enough~$C$ that
\begin{equation}
\label{e.Xweirdo}
\indc_{\{r\leq \X(x)\}} \left(  \frac{\X(x)}{r} \right)^{d}  \le \left(  \frac{\X(x)}{r} \right)^{d}  = \O_{s \alpha/d} \left(Cr^{-d }\right).
\end{equation}
By Remark~\ref{r.multiply}, applied with $s_1 = \frac s2$, $s_2 = \frac{s \alpha}{d}$, $\theta_1 = Cr^{-2\alpha}$ and $\theta_2 = Cr^{-d}$, we therefore obtain
$$
\indc_{\{r\leq \X(x)\}} \left(  \frac{\X(x)}{r} \right)^{d} \left( \mathcal{H}_r(x) \right)^2  = \O_{\frac{s \alpha}{d + 2\alpha}}\left(C r^{-(d+ 2\alpha)} \right).
$$
This proves~\eqref{e.Xbigweirdo}.  

\smallskip

\emph{Step 4.} We complete the proof. Combining~\eqref{e.omegaRiter},~\eqref{e.lippullout} and~\eqref{e.Xbigweirdo}, we get by Lemma~\ref{l.sum-O} that
\begin{align*}
 \overline{\mathcal{H}}
& \leq C\sum_{n=0}^\infty (\theta^{-n})^{-k+\frac12} \left( \fint_{B_{\theta^{-n} R}}  \left\| \nabla v \right\|_{L^2(\Phi_{x,r})}^2  \left( \mathcal{H}_r(x) \right)^2 \,dx \right)^{\frac12} \\
& \leq C\sum_{n=0}^\infty (\theta^{-n})^{-k+\frac12} \left( 
 \theta^{-2n(k-1)} \fint_{B_{\theta^{-n} R}}  \left( 1 + \indc_{\{r\leq \X(x)\}} \left(  \frac{\X(x)}{r} \right)^{d} \right) \left( \mathcal{H}_r(x) \right)^2   \,dx \right)^{\frac12}  \\
& \leq \max\left\{ \O_{\frac{s \alpha}{d + 2\alpha}}\left(C \sum_{n=0}^\infty \theta^{\frac n2}  r^{-(d + 2\alpha)} \right)\,, O_{s/2}\left(C \sum_{n=0}^\infty \theta^{\frac n2}  r^{-2\alpha} \right) \right\}^{\frac12}    
 \\
 & = \max\left\{ \O_{\frac{s \alpha}{d + 2\alpha}}\left(C r^{-(d + 2\alpha)} \right)\,, O_{s/2}(Cr^{-2\alpha}) \right\}^{\frac12}.
 \end{align*}
This combined with the result of Step~1 yields~\eqref{e.harmonicapprox'} and completes the proof by the discussion in the beginning of the proof.
 \end{proof}

The previous lemma says that every element of $v\in \A_k$ has spatial averages which are closely tracked by some element $\mathsf{h}\in \Ahom_k$. Our goal is to obtain more information about the~$\mathsf{h}$ which tracks the maximizer of $J_k(0,R,p,q)$ in terms of~$p$ and~$q$. This is accompished in the following lemma.  

\smallskip

\begin{lemma}
\label{l.parametermatching}
There exist $C(s,\alpha,d,\Lambda) < \infty$ such that, for every $r_0\leq r \leq R/\sqrt{2}$ and $p,q\in \Ahom_k(\Phi_R)$, we have
\begin{equation} \label{e.spatavgspoly}
\left\| \left( \nabla u(0,R,p,q) \right)_{\Phi_{\cdot,r}} - \nabla (q - p) \right\|_{L^2\left(\Phi_{\sqrt{R^2 - r^2}}\right)}^2  =\O_{s/2}\left(Cr^{-2\alpha}\right)\,.
\end{equation}
\end{lemma}
\begin{proof}
Fix $r_0\leq r \leq R/\sqrt{2}$, $p,q\in \Ahom_k(\Phi_R)$, and denote
\begin{equation*}
u:= u(\cdot,0,R,p,q). 
\end{equation*}
Applying Lemma~\ref{l.harmonicapprox}, we may select $\mathsf{h} \in \Ahom_k$ such that 
\begin{equation}
\label{e.harmapproxapp}
\left\| \left( \nabla u - \nabla \mathsf{h} \right)_{\Phi_{\cdot,r}}  \right\|_{L^2\left(\Phi_{\sqrt{R^2-r^2}}\right)} = \O_{s}\left( Cr^{-\alpha} \right). 
\end{equation}
Here we also applied the mean-value property in the form
\begin{equation*}
\nabla \mathsf{h} (x)  =  \left(\nabla \mathsf{h} \right)_{\Phi_{x,r}} \,.
\end{equation*} 
We apply Lemma~\ref{l.CKish}, using
$F_i \equiv \ahom_{ij} \partial_j (u - \mathsf{h})$, $q \equiv \partial_i q'$, and $r = \frac12R$, together with~\eqref{e.harmapproxapp} to obtain
\begin{equation*}
\left| \int_{\Phi_R} \left(  \nabla u - \nabla \mathsf{h}  \right)  \cdot \ahom\nabla q'(x)\,dx \right| \leq \O_s\left( CR^{-\alpha} \right). 
\end{equation*}
On the other hand, applying Lemma~\ref{l.whatflucgives} yields
\begin{equation*}
\left| \int_{\Phi_R}  \left(  \nabla (q-p)(x) - \nabla u(x) \right)  \cdot \ahom\nabla q'(x)\,dx \right| \leq  \O_s\left( CR^{-\alpha} \right).
\end{equation*}
The previous two displays give
\begin{equation*}
\left| \int_{\Phi_R}  \left( \nabla (q-p)(x) - \nabla \mathsf{h}(x)\right)  \cdot \ahom \nabla q'(x)\,dx \right| \leq  \O_s\left( CR^{-\alpha} \right).
\end{equation*}
Taking $q' := \left\| \nabla \mathsf{h}  - \nabla (q - p)  \right\|_{L^2(\Phi_R)}^{-1} \left( \mathsf{h} -(q-p)\right)$ yields
\begin{equation*} 
\left\| \nabla \mathsf{h}  - \nabla (q - p)  \right\|_{L^2(\Phi_R)} \leq  \O_{s}\left( CR^{-\alpha} \right)\,.
\end{equation*}
The triangle inequality and the mean-value property of $\ahom$-harmonic functions thus yield
\begin{equation*}
\left\| \left( \nabla u\right)_{\Phi_{\cdot,r}} - \nabla (q-p) \right\|_{L^2\left(\Phi_{\sqrt{R^2-r^2}}\right)} = \O_{s}\left( Cr^{-\alpha} \right),
\end{equation*}
which finishes the proof after squaring the previous display. 
\end{proof}

\subsection{Comparing maximizers on different scales}

The goal of this subsection is to compare maximizers of $J_k$ on different scales and thereby improve the additivity statement. This is accomplished by combining three ingredients: (i) Lemma~\ref{l.parametermatching}, above; (ii) the fact that spatial averages of the gradient of a function controls the oscillation of the function itself (the multiscale Poincar\'e inequality, Lemma~\ref{l.mspoincare}); and (iii) the fact that the oscillation of an element of $\A$ controls the $L^2$ norm of its gradient, by the regularity theory (precisely, Proposition~\ref{p.regularity} with $k=0$). 

\begin{lemma}
\label{l.additivitybitches}
There exists a constant $C(s,\alpha,k,d,\Lambda) < \infty$ 
such that, for every $r_0 \leq r \leq R /\sqrt{2}$ and $p,q\in \Ahom_k(\Phi_R)$,
\begin{multline}
\label{e.addtoS}
 \int_{\Phi_{\sqrt{R^2-r^2}} }  \left( \int_{\Phi_{x,r} } \left| \nabla u(y,0,R,p,q) - \nabla u(y,x,r,p,q) \right|^2 \,dy \right) \,dx   \\
=
\O_{s/2}\left(Cr^{-2\alpha}\right).
\end{multline}
\end{lemma}
\begin{proof}
For convenience, we fix $r_0 \leq r \leq R /\sqrt{2}$ and denote
\begin{equation*}
u:= u(\cdot,0,R,p,q) \quad \mbox{and} \quad u_{x,r}:= u(\cdot,x,r,p,q).
\end{equation*}
We then note that the left side of~\eqref{e.addtoS} is almost surely bounded by a constant $C(d,\Lambda)$. Therefore, as in the proof of Lemma~\ref{l.harmonicapprox} we may restrict ourselves to the event $\left\{\X(0) \leq \sqrt{R^2 - r^2}\right\}$. Indeed, otherwise we have that 
\begin{multline}
\label{e.addtoS000}
\indc_{\left\{\X(0) > \sqrt{R^2 - r^2}\right\}} \int_{\Phi_{\sqrt{R^2-r^2}} }  \left( \int_{\Phi_{x,r} } \left| \nabla u(y) - \nabla u_{x,r}(y) \right|^2 \,dy \right) \,dx   \\
\leq C \indc_{\left\{\X(0) > \sqrt{R^2 - r^2}\right\}} \leq \O_{s\alpha/d} \left(Cr^{-d}\right)\,.
\end{multline} 

\smallskip

We split the difference $u - u_{x,r}$ as
\begin{equation*}
u - u_{x,r} = \left( u - u_{x,r}\right) \indc_{\{ \X(x) \leq r \} }  + \left( u - u_{x,r}\right) \indc_{\{ \X(x) > r \} } = : v_{x,r} + \tilde v_{x,r}.
\end{equation*}
Most of the proof therefore concerns the estimate for~$v_{x,r}$, with the estimate for~$\tilde v_{x,r}$ coming in the final step, where we show that the error is the same as in~\eqref{e.addtoS000}. Denote
\begin{equation*}
w_{x,r} (y,t):= \int_{\Phi_{y,\sqrt{t}}} v_{x,r}(z)\,dz,
\end{equation*}
which is the solution of the parabolic equation
\begin{equation*}
\left\{
\begin{aligned}
& \partial_t w_{x,r}  - \nabla \cdot \left( \ahom \nabla w_{x,r} \right) = 0 & \mbox{in} & \ \Rd \times (0,\infty), \\
& w_{x,r}  = v_{x,r} & \mbox{on} & \ \Rd \times \{0 \}.
\end{aligned}
\right.
\end{equation*}

\smallskip

\emph{Step 1.} 
We show that  
\begin{equation} \label{e.additivityb3}
\left\| \nabla v_{x,r}  \right\|_{ \Phi_{x,r}}  \leq C \left\| \nabla w_{x,r}(\cdot,(cr)^2)    \right\|_{ \Phi_{x,\sqrt{r^2 - (cr)^2 }}}  
\end{equation}
According to Lemma~\ref{l.mspoincare2}, applied with $c = \sigma_0$,
\begin{equation} \label{e.additivityb2}
\int_{\Psi_{x,r}}  \left| v_{x,r} (y) \right|^2\,dy \leq C r^{-d}  \int_{B_{r/\theta}(x)} \left| w_{x,r}(y,(cr)^2) \right|^2\,dy\,.
\end{equation}
Subtracting a constant in the definition of $v_{x,r}$ so that $w_{x,r}(y,(cr)^2)$ has zero average in $B_{r/\theta}(x)$ we get by the Poincar\'e inequality that 
\begin{align*}
r^{-d}  \int_{B_{r/\theta}(x)} \left| w_{x,r}(y,(cr)^2) \right|^2\,dy 
 & \leq C r^{2} \int_{\Phi_{x,\sqrt{r^2-(cr)^2}}} \left| \nabla w_{x,r}(y,(cr)^2) \right|^2\,dy\,.
\end{align*}
On the other hand, by the Caccioppoli estimate we have
\begin{equation*} 
\left\| \nabla v_{x,r}  \right\|_{ \Phi_{x,r}}^2 \leq \frac{C}{r^2} \int_{\Psi_{x,r}}  \left| v_{x,r} (y) \right|^2\,dy\,.
\end{equation*}
Thus~\eqref{e.additivityb3} follows by~\eqref{e.additivityb2} and the previous two displays. 

\smallskip

\emph{Step 2.} We show that 
\begin{equation} \label{e.additivityb4}
\int_{\Phi_{\sqrt{R^2-r^2}}}\int_{\Phi_{x,{\sqrt{r^2-(cr)^2}}}} \left| \nabla w_{x,r}(y,(cr)^2) \right|^2\,dy\,dx 
\leq 
 \O_{s/2}\left(Cr^{-2\alpha}\right).
\end{equation}
The triangle inequality yields
\begin{multline*}
 \left| \nabla w_{x,r}(y,(cr)^2) \right|
 \leq \left| \nabla (q-p)(y) - \nabla \left( \int_{\Phi_{y,cr}} u(z)\,dz \right)  \right|  \\
 +  \left|  \nabla (q-p)(y) - \nabla \left(  \int_{\Phi_{y,cr}} u_{x,r}(z)\,dz\right) \right| \,.
\end{multline*}
Squaring this, integrating with respect to $\Phi_{x,\sqrt{r^2-(cr)^2}}$ and applying Lemma~\ref{l.parametermatching} gives  
\begin{multline*}
\int_{\Phi_{x,\sqrt{r^2-(cr)^2}}}  \left| \nabla w_{x,r}(y,(cr)^2) \right|^2\,dy  \\
 \leq \int_{\Phi_{x,\sqrt{r^2-(cr)^2}}}  \left| \nabla (q-p)(y) - \nabla \left( \int_{\Phi_{y,cr}} u(z)\,dz \right)  \right| \,dy 
 + \O_{s/2}\left( Cr^{-2\alpha} \right).
\end{multline*}
Integrating the result with respect to $\Phi_{\sqrt{R^2-r^2}}$, using the semigroup property for the heat kernels, and applying Lemma~\ref{l.parametermatching} once more, we obtain~\eqref{e.additivityb4}.

\smallskip

\emph{Step 3.} The conclusion. 
Combining the results of Steps~1 and~2 yields 
\begin{align*}
\int_{\Phi_{\sqrt{R^2-r^2}}} \int_{\Phi_{x,r}} \left| \nabla v_{x,r} (y) \right|^2\,dy \, dx
&\leq C \int_{\Phi_{\sqrt{R^2-r^2}}}  \int_{\Phi_{x,\sqrt{r^2-cr^2}}} \left| \nabla w_{x,r}(y,cr^2) \right|^2\,dy\, dx \\
& \leq  \O_{s/2}\left(Cr^{-2\alpha}\right).
\end{align*}
We now give the estimate for $\tilde v_{x,r}$. Following the Step 2 in the proof of Lemma~\ref{l.harmonicapprox}  
\begin{multline*}
 \indc_{ \left\{ \X(0) \leq \sqrt{R^2 - r^2} \right\}} \int_{\Phi_{\sqrt{R^2- r^2} } }\indc_{\{ \X(x) > r \} } \int_{\Phi_{x,r}}   \left|\nabla u(y) \right|^2 \, dy \, dx 
\\  \leq C \int_{\Phi_{\sqrt{R^2- r^2} } }\indc_{\{ \X(x) > r \} } \left(\frac{\X(x) \vee r}{r}\right)^d \, dx
\end{multline*}
and hence by~\eqref{e.Xweirdo}, using the assumption $s\alpha<d$, we have 
\begin{align}
\label{e.stupidtildevxr}
\lefteqn{  \indc_{ \left\{ \X(0) \leq \sqrt{R^2 - r^2} \right\}} 
\int_{\Phi_{\sqrt{R^2-r^2}}} \int_{\Phi_{x,r}} \left| \nabla \tilde v_{x,r} (y) \right|^2\,dy \, dx
} \qquad & \\
& \leq C\int_{\Phi_{\sqrt{R^2-r^2}}} \indc_{\{\X(x)>r\}} \left( 1+  \left( \frac{\X(x) \vee r}{r} \right)^d \right) \,dx \nonumber \\
& = \O_{s\alpha/d} \left( Cr^{-d} \right). \nonumber
\end{align}
Therefore the left side of~\eqref{e.stupidtildevxr} produces the same error as in~\eqref{e.addtoS000}, and it can be estimated similarly. 
\end{proof}

We next demonstrate an improvement of additivity and give the proof of Proposition~\ref{p.additivity}, up to the identification of $\theta$ in the assumption (which is accomplished in the next subsection).

\begin{lemma}
\label{l.improveadditivity}
For each $\ep>0$ and $\alpha\in \left(0,\frac ds-\ep \right]$, there exists $\eta(\ep,s,d,\Lambda)>0$ such that $\Add_k(s,\alpha+\eta)$ holds. Moreover, if we suppose in addition that $\theta\in (0,\alpha]$ is such that, for every $z\in\Rd$, $r\geq r_0$ and $p,q\in \Ahom_k(\Phi_{z,r})$,
\begin{equation}
\label{e.assofLtheta}
\left\|  \nabla L_{k,z,r}q - \nabla q \right\|_{L^2(\Phi_{z,r})} 
+ \left\|  \nabla L_{k,z,r}^*p - \nabla p \right\|_{L^2(\Phi_{z,r})}  \leq Cr^{-\theta},
\end{equation}
then $\Add_k\left(\frac{s\alpha}{\alpha+\theta},\alpha+\theta\right)$ holds.  
\end{lemma}
\begin{proof}
We begin with the argument for the second statement. 
Fix $r_0 \leq r \leq R /\sqrt{2}$. To shorten the notation, we denote, for each $x\in\Rd$,
\begin{equation*}
u:= u(\cdot,0,R,p,q) 
\quad \mbox{and} \quad 
u_{x,r}:= u(\cdot,x,r,p,q).
\end{equation*}
By Lemma~\ref{l.Iquadresponse} and the assumption~\eqref{e.assofLtheta}, 
we have, for each $x\in\Rd$, 
\begin{multline}
\label{e.IminusJ}
\left| I(x,r,p,q) -  \mathcal{J}(u,x,r,p,q) \right| \\
 \leq C \left\| \nabla u - \nabla u_{x,r} \right\|_{L^2(\Phi_{x,r})}^2 + Cr^{-\theta}  \left\| \nabla u - \nabla u_{x,r} \right\|_{L^2(\Phi_{x,r})}.
 \end{multline}
Therefore
\begin{align*}
\lefteqn{
\left| I(0,R,p,q) - \int_{\Phi_{\sqrt{R^2-r^2}}} I(\cdot,r,p,q) \right| 
} \qquad & \\
& =  \left| \int_{\Phi_{\sqrt{R^2-r^2}}} \left( \mathcal{J}\left( u,x,r,p,q \right) -   I(x,r,p,q)  \right)\,dx\right| \\
& \leq \int_{\Phi_{\sqrt{R^2-r^2}}}\left| \mathcal{J}\left( \nabla u,x,r,p,q \right) -   I(x,r,p,q)  \right|\,dx
 \\
& \leq \int_{\Phi_{\sqrt{R^2-r^2}}} \left( C \left\| \nabla u - \nabla u_{x,r} \right\|_{L^2(\Phi_{x,r})}^2 + Cr^{-\theta}  \left\| \nabla u - \nabla u_{x,r} \right\|_{L^2(\Phi_{x,r})}\right)\,dx.
\end{align*}
By Lemma~\ref{l.additivitybitches},
\begin{equation*}
\int_{\Phi_{\sqrt{R^2-r^2}}}  \left\| \nabla u - \nabla u_{x,r} \right\|_{L^2(\Phi_{x,r})}^2 \,dx \leq \O_{s/2}\left( Cr^{-2\alpha} \right). 
\end{equation*}
Since the left side is bounded almost surely, we also obtain
\begin{equation*}
\int_{\Phi_{\sqrt{R^2-r^2}}}  \left\| \nabla u - \nabla u_{x,r} \right\|_{L^2(\Phi_{x,r})}^2 \,dx \leq \O_{\frac{s\alpha}{\alpha+\theta}}\left( Cr^{-(\alpha+\theta)} \right). 
\end{equation*}
Lemma~\ref{l.additivitybitches} also gives that 
\begin{equation*}
\int_{\Phi_{\sqrt{R^2-r^2}}} r^{-\theta}  \left\| \nabla u - \nabla u_{x,r} \right\|_{L^2(\Phi_{x,r})}\,dx \leq \O_s \left( Cr^{-(\alpha+\theta)}\right).
\end{equation*}
Combining these, we obtain that 
\begin{equation}
\label{e.addup}
\left| I(0,R,p,q) - \int_{\Phi_{\sqrt{R^2-r^2}}} I(\cdot,r,p,q) \right|  \leq \O_{\frac{s\alpha}{\alpha+\theta}}\left( Cr^{-(\alpha+\theta)} \right).
\end{equation}
We can remove the restriction $r \le R/\sqrt{2}$ and obtain the same estimate for any $r_0 \le r < R$ by the triangle inequality and the semigroup property of the heat kernel.
This completes the proof of~$\Add_k({\frac{s\alpha}{\alpha+\theta}},\alpha+\theta)$. 

\smallskip

We turn to the proof of the first statement. According to Corollary~\ref{c.L-identity}, we have the assumption~\eqref{e.assofLtheta} for some $\theta_0(d,\Lambda)>0$. 
By Proposition~\ref{p.basecase}, we have, for every $t\in (0,d)$ and some exponent $\ep_0(d,\Lambda)>0$, the $\P$-almost sure bound
\begin{align*}
\left| \int_{\Phi_{\sqrt{R^2-r^2}}}\left(I(0,R,p,q) - I(x,r,p,q) \right) \indc_{ \{\Y_t(x) \leq r\}}\,dx \right| \indc_{\{ \Y_t \leq R\}} 
\leq Cr^{-\ep_0(d-t)}. 
\end{align*}
By boundedness and Remark~\ref{r.change-s}, we have 
\begin{multline*}
\left| I(0,R,p,q) - \int_{\Phi_{\sqrt{R^2-r^2}}} I(x,r,p,q) 
\right| \indc_{\{ \Y_t \geq R\}} \\
\leq C \indc_{\{ \Y_t \geq R\}} \leq C \Ll( \frac{\mcl Y_t}{R} \Rr)^{t/s} \le \O_s \left( CR^{-t/s } \right)
\end{multline*}
and, similarly,
\begin{equation*}
\int_{\Phi_{\sqrt{R^2-r^2}}} \left( \left| I(x,r,p,q) \right| + \left|I(0,R,p,q) \right| \right) \indc_{ \{\Y_t(x) \geq r\}}\,dx \, \indc_{\Y_t \leq R} \leq \O_s \left( Cr^{-t/s } \right).
\end{equation*} 
%
%
%
Combining these with~\eqref{e.addup} by the triangle inequality, and using that $s\alpha \leq d -\ep$ which allows us to take $t:=\frac12(d+s\alpha) = d- \frac12\ep$, which satisfies
\begin{equation*}
\frac12t - \alpha = \frac12(d-t) = \frac12\left(\frac d2-\alpha\right) 
\quad \mbox{and} \quad
\frac ts = \frac12\left( \alpha + \frac ds \right) \geq \alpha + \frac{\ep}{2s},
\end{equation*}
we obtain
\begin{equation*}
\left| I(0,R,p,q) - \int_{\Phi_{\sqrt{R^2-r^2}}} I(\cdot,r,p,q) \right|   
 = \O_s\left( C r^{-(\alpha+\eta)} \right) 
\end{equation*}
for the explicit exponent
\begin{equation*}
\eta:= \ep \left(  \frac{\ep_0 \theta_0}{2(\frac ds+\theta_0)} \wedge \frac{1}{2s} \right).
\end{equation*}
We have proved $\Add_k(s,\alpha+\eta)$ for $\eta(\ep,s,d,\Lambda)>0$, as desired. 
\end{proof}

\subsection{Two-scale expansion of maximizers of $J_k$}
\label{ss.twoscale}

In this subsection, we establish a quantitative  two-scale expansion for the maximizers of $J_k$ in terms of the first-order correctors. This is needed to identify an explicit exponent $\theta$ in the hypothesis of Lemma~\ref{l.improveadditivity} and thereby complete the proof of Proposition~\ref{p.additivity}. 

\smallskip
The first step is to match the correctors to the functions $u(\cdot,x,r,p,q)$ for $p,q\in\Ahom_1$. In fact, even though the existence of the correctors is classical, we actually just construct the correctors from the latter. Recall the definition of $\mathbb{L}^2_{\mathrm{pot}}$ around \eqref{e.def.L2pot}.

\begin{lemma}
\label{l.firstordercorrect}
Suppose that $s\leq 2$ and $\alpha \leq \frac d2$. 
There exists a linear map $\xi \mapsto \nabla \phi^{(1)} (\cdot,\xi)$ from $\Rd$ to $\mathbb{L}^2_{\mathrm{pot}}$ such that, up to an additive constant, every element of $\A_1$ has the form
\begin{equation*}
x\mapsto \underline{\phi}^{(1)}(x,\xi) := \xi \cdot x + \phi^{(1)} (x,\xi)
\end{equation*}
and there exist $\ep(d,\Lambda)>0$ and $C(d,\Lambda)<\infty$  such that
\begin{equation} 
\label{e.correctorbound}
\sup_{r\geq1} \sup_{\xi\in B_1} \left\| \nabla \phi^{(1)}(\cdot,\xi) \right\|_{\underline{L}^2(B_r)}  = \O_{2+\ep}(C).
\end{equation}
Finally, there exists $C(s,\alpha,k,d,\Lambda)<\infty$ such that, for every $t\in \left[ s,2+\ep\right)$, $z\in\Rd$, $r\geq r_0$ and $p,q\in \Ahom_1(\Phi_{x,r})$, we have
\begin{equation}
\label{e.snaptouyou}
\int_{\Phi_{z,r}} \left| \nabla u(x,z,r,p,q) -  \nabla \underline\phi^{(1)}(x,\nabla q-\nabla p) \right|^2\,dx \leq \O_{t/2}\left(Cr^{-\frac{2s\alpha}t} \right). 
\end{equation}
In particular, for every $t\in \left[ s,2+\ep\right)$, $\xi\in B_1$, $z\in\Rd$ and $r\geq 1$,
\begin{equation}
\label{e.spatavg1}
\left| \int_{\Phi_{z,r}}\nabla \phi^{(1)}(x,\xi)  \,dx \right| = \O_{t} \left(Cr^{-\frac{s\alpha}t }\right).
\end{equation}
\end{lemma}
\begin{proof}
In this proof, we let $\X = \X_\gamma = \O_\gamma(C)$, $\gamma \in [s \alpha,d)$, be as in Proposition~\ref{p.regularity}, where the parameter $\gamma$ will be chosen in Step 6 below, and we let $\X(z)$ be its $\Z^d$-stationary extension (that is, $\X(x) = \tau_x \X$). 

\emph{Step 1.} We define a candidate for the corrector field. 
According to Lemmas~\ref{l.parametermatching} and~\ref{l.mspoincare2}, we have that, for every $z\in\Rd$ and $R\geq r_0$,
\begin{equation}
\label{e.compoundclaim}
\int_{\Phi_{z,R}} \left| \nabla u(x,z,R,p,q) - \nabla u(x,z,2R,p,q)  \right|^2\,dx = \O_{s/2}\left( CR^{-2\alpha} \right). 
\end{equation}
Indeed, we have that, for $r_0\leq r \leq R/\sqrt{2}$,
\begin{align*}
\lefteqn{
\int_{\Phi_{z,\sqrt{R^2-r^2}}} 
\left|  \int_{\Phi_{x,r}} \left( \nabla u(y,z,R,p,q) - \nabla u(y,z,2R,p,q) \right) \,dy \right|^2\,dx 
} \qquad & \\
& \leq \int_{\Phi_{z,\sqrt{R^2-r^2}}} 
\left| \nabla (q-p)(x) -  \int_{\Phi_{x,r}} \nabla u(y,z,R,p,q) \,dy \right|^2\,dx  \\
& \qquad + \int_{\Phi_{z,\sqrt{R^2-r^2}}} 
\left| \nabla (q-p)(x) -  \int_{\Phi_{x,r}} \nabla u(y,z,2R,p,q) \,dy \right|^2\,dx
\end{align*}
and then using 
\begin{multline*}
\int_{\Phi_{z,\sqrt{R^2-r^2}}} 
\left| \nabla (q-p)(x) -  \int_{\Phi_{x,r}} \nabla u(y,z,2R,p,q) \,dy \right|^2\,dx \\
\leq C \int_{\Phi_{z,\sqrt{(2R)^2-r^2}}} 
\left| \nabla (q-p)(x) -  \int_{\Phi_{x,r}} \nabla u(y,z,2R,p,q) \,dy \right|^2\,dx,
\end{multline*}
we may apply Lemma~\ref{l.parametermatching} to get
\begin{equation*}
\int_{\Phi_{z,\sqrt{R^2-r^2}}} 
\left|  \int_{\Phi_{x,r}} \left( \nabla u(y,z,R,p,q) - \nabla u(y,z,2R,p,q) \right) \,dy \right|^2\,dx = \O_{s/2} \left( Cr^{-2\alpha} \right).
\end{equation*}
We then take $r= R/\sqrt{2}$ and apply Lemma~\ref{l.mspoincare2}, the Poincar\'e and Caccioppoli inequalities to get~\eqref{e.compoundclaim}. 
Using now the Lipschitz estimate (Proposition~\ref{p.regularity}), we obtain, for every $z\in\Rd$ and $r \leq R$,
\begin{equation}
\label{e.dyadicbluff}
\indc_{\{ \X(z) \leq r \} }   \int_{\Phi_{z,r}} \left| \nabla u(x,z,R,p,q) - \nabla u(x,z,2R,p,q)  \right|^2\,dx = \O_{s/2}\left( CR^{-2\alpha} \right). 
\end{equation}
By Remark~\ref{r.change-s} and~\eqref{e.bounded-u}, we also have, for every $t\geq s$,
\begin{multline}
\label{e.dyadicbluff-t}
\indc_{\{ \X(z) \leq r \} }   \int_{\Phi_{z,r}} \left| \nabla u(x,z,R,p,q) - \nabla u(x,z,2R,p,q)  \right|^2\,dx \\
 = \O_{t/2}\left( CR^{-\frac{2s\alpha}t} \right). 
\end{multline}
Summing~\eqref{e.dyadicbluff-t} over a dyadic sequence of scales, we obtain, almost surely with respect to $\P$, the existence of $w_{z,r}\in \A_1$ such that, for every $t\geq s$ and $ r \leq R$,
\begin{equation}
\label{e.needfortriangle}
\indc_{\{ \X(z) \leq r \}} \int_{\Phi_{z,r}} \left| \nabla u(x,z,R,p,q) - \nabla w_{z,r}(x) \right|^2\,dx \leq \O_{t/2}\left( CR^{-\frac{2s\alpha}t} \right). 
\end{equation}
It is clear from the construction that for $\X \leq r\leq R$, the gradients of $w_{z,r}$ and $w_{z,R}$ are the same. On the event $r < \X$, we may pick an arbitrary $R \ge \X$ and redefine $\nabla w_{z,r}$ to be $\nabla w_{z,R}$, without affecting \eqref{e.needfortriangle}. This modification does not depend on the choice of $R$, and ensures that $\nabla w_{z,r} = \nabla w_{z,R}$ for arbitrary $r, R \ge r_0$. We may then make $w_z=w_{z,r}$ itself independent of $r$ by choosing the additive constant so that 
\begin{equation}
\label{e.choosenormalization}
\left( w_{z} - (q-p) \right)_{\Phi_{1}} = 0.
\end{equation}
Moreover, by the construction we get that, for every $z\in\Zd$, 
\begin{equation}
\label{e.translaw}
\nabla w_z(\cdot + z) 
\quad \mbox{and} \quad
\nabla w_0(\cdot) 
\quad \mbox{have the same law.}
\end{equation}

\smallskip

\emph{Step 2.} We show that, for all $z\in\Rd$ and $r\geq |z|$, we have 
\begin{equation}
\label{e.snapzto0}
\indc_{\{ \X(0) \vee\X(z) \leq \sigma r\} }  \fint_{B_r} \left| \nabla w_z(x) - \nabla w_0(x) \right|^2\,dx \leq \O_{s/2}\left(C r^{-2\alpha} \right). 
\end{equation}
By the Caccioppoli inequality, 
\begin{equation*}
 \fint_{B_r} \left| \nabla w_z(x) - \nabla w_0(x) \right|^2\,dx
 \leq Cr^{-2} \inf_{a\in\R} \fint_{B_{2r}} \left| w_z(x) - w_0(x) -a \right|^2\,dx.
\end{equation*}
Using Lemma~\ref{l.mspoincare2},
\begin{equation}
\label{e.multiyourscale}
\left\| w_{0} - w_z -a \right\|_{L^2(\Psi_r)} 
 \leq C \fint_{B_{r/\theta}} \left|  \int_{\Phi_{y,\sigma r}} \left( w_0 - w_z -a \right)(x) \,dx \right|^2\,dy. 
\end{equation}
Choosing 
\begin{equation*}
a:= \fint_{B_{r/\theta}} \left(  \int_{\Phi_{y,\sigma r}} \left( w_0 - w_z \right)(x) \,dx \right)\,dy
\end{equation*}
and applying the Poincar\'e inequality, Lemma~\ref{l.parametermatching},~\eqref{e.needfortriangle} and the triangle inequality, we find that, for $r\geq |z|$ such that $\X(0) \vee \X(z) \leq \sigma r$,
\begin{align}
\label{e.thyuiop}
\lefteqn{
 \fint_{B_{r/\theta}} \left|  \int_{\Phi_{y,\sigma r}} \left( w_0 - w_z -a \right)(x) \,dx \right|^2\,dy 
 } \qquad & \\
 & \leq Cr^2  \fint_{B_{r/\theta}} \left|  \int_{\Phi_{y,\sigma r}} \left( \nabla w_0 - \nabla w_z \right)(x) \,dx \right|^2\,dy  \notag\\
 & \leq Cr^2  \int_{\Phi_{\sqrt{r^2-(\sigma r)^2}}} \left|  \int_{\Phi_{y,\sigma r}} \left( \nabla w_0 - \nabla w_z \right)(x) \,dx \right|^2\,dy \notag \\
 & \leq  \O_{s/2} \left( Cr^{2-2\alpha} \right).  \notag
\end{align}
Let us give more details on the last inequality claimed in the display above. We have, 
\begin{align*}
\lefteqn{ 
\int_{\Phi_{\sqrt{r^2-(\sigma r)^2}}} \left|  \int_{\Phi_{y,\sigma r}} \left( \nabla w_0 - \nabla w_z \right)(x) \,dx \right|^2\,dy
} \qquad & \\
& \leq \int_{\Phi_{\sqrt{r^2-(\sigma r)^2}}} \left| \nabla (q-p)(x)-  \int_{\Phi_{y,\sigma r}} \nabla w_0(x)\,dx \right|^2\,dy \\
& \qquad + \int_{\Phi_{\sqrt{r^2-(\sigma r)^2}}} \left| \nabla (q-p)(x)-  \int_{\Phi_{y,\sigma r}} \nabla w_z(x)\,dx \right|^2\,dy
\end{align*}
and, since $z\in B_r$, we get, in the case $\X(0) \vee \X(z) \leq \sigma r$,
\begin{multline*}
\int_{\Phi_{\sqrt{r^2-(\sigma r)^2}}} \left| \nabla (q-p)(x)-  \int_{\Phi_{y,\sigma r}} \nabla w_z(x)\,dx \right|^2\,dy  \\
\leq C\int_{\Phi_{z,\sqrt{(Cr)^2-(\sigma r)^2}}} \left| \nabla (q-p)(x)-  \int_{\Phi_{y,\sigma r}} \nabla w_z(x)\,dx \right|^2\,dy.
\end{multline*}
Now using the triangle inequality,~\eqref{e.needfortriangle} and Lemma~\ref{l.parametermatching}, the previous two displays yield, for $|z| \leq r$,
\begin{equation*}
\indc_{\{ \X(0) \vee\X(z) \leq r\} } \int_{\Phi_{\sqrt{r^2-(\sigma r)^2}}} \left|  \int_{\Phi_{y,\sigma r}} \left( \nabla w_0 - \nabla w_z \right)(x) \,dx \right|^2\,dy
\leq \O_{s/2} \left( Cr^{-2\alpha} \right),
\end{equation*}
as claimed. Combining~\eqref{e.thyuiop} and~\eqref{e.multiyourscale} yields~\eqref{e.snapzto0}.

\smallskip

\emph{Step 3.} We complete the construction of $\phi^{(1)}$. We first notice that
\begin{equation}
\label{e.statfld}
\P\left[ \forall z\in \Rd, \ \nabla w_z = \nabla w_0 \right] =1.
\end{equation}
Indeed, this follows from the previous step after sending $r\to \infty$ and applying the regularity estimate, which gives, for every $y\in\Zd$ and $r\geq  \X(y)+C$,
\begin{align*}
\fint_{B_{\sqrt{d}} (y)}\left| \nabla w_z(x) - \nabla w_0(x) \right|^2\,dx & \leq C\X(y)^{d/2} \fint_{B_r(y)} \left| \nabla w_z(x) - \nabla w_0(x) \right|^2\,dx \\
& \rightarrow 0 \quad \mbox{as} \ r\to\infty.
\end{align*}
Since we have $\X(y)<\infty$, for every $y\in\Zd$, $\P$-almost surely, we obtain~\eqref{e.statfld}. We may 
drop dependence on $z$ and write $w=w_z$, which is defined uniquely up to an additive constant. Moreover, it follows from~\eqref{e.translaw} that $\nabla w$ is a stationary field. That is, we have 
\begin{equation*}
\nabla w\in \mathbb{L}^2_{\mathrm{pot}}.
\end{equation*}

\smallskip

\emph{Step 4.} We next claim that, if we display the dependence $w$ on $p,q\in\Ahom_1$ by writing $w(\cdot,p,q)$, then, for every $p,q,p',q'\in\Ahom_1$ such that $\nabla (q-p) = \nabla (q'-p')$, we have
\begin{equation*}
\nabla w(\cdot,p,q) = \nabla w(\cdot,p',q').
\end{equation*}
Indeed, this is immediate from the argument in Steps~2 and~3. We simply compare $w(\cdot,p,q)$ and $w(\cdot,p',q')$ in the same way we previously compared $w_z$ and $w_0$. The main point is that the gradients of these functions have the same spatial averages, up to $\O_{s}(Cr^{-\alpha})$. 

\smallskip

This allows us to write $\nabla w(\cdot,p,q) = \nabla w(\cdot,q-p) = w(\cdot,\xi)$, after identifying~$\Rd$ with~$\Ahom_1$. We now define, for each $\xi\in\Rd$, 
\begin{equation*}
\phi^{(1)}(x,\xi) := w(x,\xi) - \xi\cdot x. 
\end{equation*}
We have that $\nabla \phi^{(1)}(\cdot,\xi) = \nabla w(\cdot,\xi) - \xi \in \mathbb{L}^2_{\mathrm{pot}}$.

\smallskip

\emph{Step 5.} We prove~\eqref{e.correctorbound}. Observe that, by the ergodic theorem, for each $\xi\in \Rd$,
\begin{equation*}
\lim_{R\to \infty} R^{-1} \left\| \underline{\phi}^{(1)}(\cdot,\xi) \right\|_{\underline{L}^2(B_R)} = \left| \xi \right|. 
\end{equation*}
Therefore the Lipschitz estimate (Proposition~\ref{p.regularity}(iii) with $k=0$) gives
\begin{equation} 
\label{e.firstlipcorrector}
\left\| \nabla \underline{\phi}^{(1)}(\cdot,\xi) \right\|_{\underline{L}^2(B_{\X})} \leq C |\xi|
\end{equation}
and from this we get
\begin{align}
\label{e.bluff}
\left\| \nabla \phi^{(1)}(\cdot,\xi) \right\|_{\underline{L}^2(B_1)} 
& \leq \left| \xi \right| + \left\| \nabla \underline{\phi}^{(1)}(\cdot,\xi) \right\|_{\underline{L}^2(B_1)} \\
& \leq \left| \xi \right| + \X^{\frac d2} \left\| \nabla \underline{\phi}^{(1)}(\cdot,\xi) \right\|_{\underline{L}^2(B_{\X})} \notag \\
& \leq C \left| \xi \right| \left( 1 + \X^{\frac d2} \right). \notag
\end{align}
Thus for every $s'<2$, there exists $C(s',d,\Lambda)<\infty$ such that 
\begin{equation}
\label{e.correctorboundpre}
\sup_{\xi\in B_1} \left\| \nabla \phi^{(1)}(\cdot,\xi) \right\|_{\underline{L}^2(B_1)}  = \O_{s'}(C).
\end{equation}
Actually, we can do slightly better in the second line of~\eqref{e.bluff} by using Meyers' estimate rather than give up the full volume factor. Indeed, by Meyers' and H\"older's inequalities, there exists $\ep(d,\Lambda)>0$ such that, for every $\xi\in B_1$, 
\begin{equation*} \label{}
\left\| \nabla \underline{\phi}^{(1)}(\cdot,\xi) \right\|_{\underline{L}^2(B_1)} 
\leq 
\left( \X^{\frac d2}\right)^{\frac{2}{2+\ep}} \left\| \nabla \underline{\phi}^{(1)}(\cdot,\xi) \right\|_{\underline{L}^{2+\ep} (B_{\X/2})}  
\leq \X^{\frac{d}{2+\ep}} \left\| \nabla \underline{\phi}^{(1)}(\cdot,\xi) \right\|_{\underline{L}^2(B_{\X})}.
\end{equation*}
This gives us the following slight improvement of~\eqref{e.correctorboundpre}: for some $\ep(d,\Lambda)>0$ and $C(d,\Lambda)<\infty$,
\begin{equation*} \label{}
\sup_{\xi\in B_1} \left\| \nabla \phi^{(1)}(\cdot,\xi) \right\|_{\underline{L}^2(B_1)}  = \O_{2+\ep}(C).
\end{equation*}
The same argument gives, for some $\ep(d,\Lambda)>0$ and $C(d,\Lambda)<\infty$,
\begin{equation}
\sup_{r\geq1} \sup_{\xi\in B_1} \left\| \nabla \phi^{(1)}(\cdot,\xi) \right\|_{\underline{L}^2(B_r)}  = \O_{2+\ep}(C).
\end{equation}
This completes the proof of~\eqref{e.correctorbound}.

\smallskip

\emph{Step 6.} 
 The conclusion. We have left to prove the estimates~\eqref{e.snaptouyou} and~\eqref{e.spatavg1}. To obtain~\eqref{e.snaptouyou}, we observe that~\eqref{e.needfortriangle} implies, for every $t\geq s$, 
\begin{equation} 
\label{e.snapyouonce}
\indc_{\{ \X(z) \leq r \} } \int_{\Phi_{z,r}} \left| \nabla u(x,z,r,p,q) - \nabla \underline{\phi}^{(1)} (x,q-p) \right|^2\,dx = \O_{t/2}\left( Cr^{-\frac{2s\alpha}t} \right). 
\end{equation}
By~\eqref{e.firstlipcorrector}, using Meyers' estimate as in the previous step, gives 
\begin{align*} \label{}
\indc_{\{ \X(z) \geq r \} } \int_{\Phi_{z,r}}\left| \nabla \underline{\phi}^{(1)} (x,q-p) \right|^2\,dx
& \leq C \indc_{\{ \X(z) \geq r \} }  \left( \frac{\X(z)}{r} \right)^{\frac{2d}{2+\ep}} \,. 
\end{align*}
Since we have that, for all $\sigma > 0$ and $\gamma \in [\alpha s,d)$, 
\begin{equation*} 
\indc_{\{ \X(z) \geq r \} } \leq \O_{\gamma/\sigma}\left( C r^{-\sigma \gamma}\right) \quad \mbox{and} \quad  \left( \frac{\X(z)}{r} \right)^{\frac{2d}{2+\ep}} \leq \O_{\frac{\gamma(2+\ep)}{2d}}\left( C r^{-\frac{2d}{2+\ep}}\right) \,,
\end{equation*}
Remark~\ref{r.multiply} implies that 
\begin{equation*} 
\indc_{\{ \X(z) \geq r \} } \left( \frac{\X(z)}{r} \right)^{\frac{2d}{2+\ep}}  \leq \O_{\frac{\gamma(2 +\ep)}{2d + \sigma (2+\ep)}} \left( C r^{-\frac{2d}{2+\ep} -\sigma \gamma} \right)\,.
\end{equation*}
For every $t \in [s,2+\ep)$ and $s \alpha < d$ we may choose $\sigma$ and $\gamma$ as above so that 
\begin{equation*} 
\sigma = \frac{2\gamma}{t} - \frac{2d}{2+\ep} \quad \mbox{and} \quad \frac{2d}{2+\ep} + \sigma \gamma \geq \frac{2\alpha s}{t} \,.
\end{equation*}
The previous displays yield, for every $t\in[s,2+\ep)$,
\begin{equation*} \label{}
\indc_{\{ \X(z) \geq r \} } \int_{\Phi_{z,r}}\left| \nabla \underline{\phi}^{(1)} (x,q-p) \right|^2\,dx = \O_{t/2} \left( Cr^{-\frac{2s\alpha}t} \right).
\end{equation*}
We also have, by~\eqref{e.bounded-u} and $s\alpha<d$ that, for every $t \in [s,2+\ep)$,
\begin{align*} \label{}
\indc_{\{ \X(z) \geq r \} } \int_{\Phi_{z,r}}\left| \nabla  u(x,z,r,p,q) \right|^2\,dx 
\leq C\indc_{\{ \X(z) \geq r \} }
& \leq  \O_{t/2}\left( Cr^{-\frac{2s\alpha}t} \right).
\end{align*}
From the previous two displays,~\eqref{e.snapyouonce} and the triangle inequality, we get~\eqref{e.snaptouyou}. We then obtain~\eqref{e.spatavg1} from~\eqref{e.snaptouyou}, Lemma~\ref{l.parametermatching} and the triangle inequality. This completes the proof of the lemma. 
\end{proof}

We next extend~\eqref{e.spatavg1} to general $p,q\in\Ahom_k$. Define, for each $p\in\Ahom_k$,
\begin{equation*}
\underline{\phi}^{(1)} (x,p):= p(x) + \phi^{(1)}\left(x,\nabla p(x) \right),
\end{equation*}
which is defined up to an arbitrary additive constant. Note that $\underline{\phi}^{(1)}(\cdot,p)$ is not, in general, an element of $\A$ unless $p\in\Ahom_1$. 

\begin{remark}
\label{r.correctorexpectation}
We next record the observation that, for every $p,p',q \in \A_k(\Phi_{z,r})$, 
\begin{multline*}
\left| \E \left[ \mathcal{J}(\underline{\phi}^{(1)} (\cdot,p'),z,r,p,q) \right] - \int_{\Phi_{z,r}} 
\left( \frac12 \nabla p' \cdot \ahom \nabla p' - 
\nabla p'\cdot \ahom\nabla p + \nabla p'\cdot \ahom \nabla q 
\right) \right|  \\
\leq
C\exp\left(-cr \right).
\end{multline*}
Indeed, if we were working with $\Rd$-stationarity rather than $\Zd$-stationarity, this could be obtained by putting the expectation inside the integral and using the stationarity of the first-order correctors. To make this argument work with $\Zd$-stationarity is an exercise that we leave to the reader.
\end{remark}

\begin{lemma}
\label{l.matchingutophi}
Suppose $s\leq 2$. 
Then for each $\beta \in (0,\alpha \wedge 1)$, there exists a constant $C(\beta,s,\alpha,k,d,\Lambda)<\infty$ such that, for every $z\in\Rd$, $r\geq r_0$ and $p,q\in \Ahom_k(\Phi_{z,r})$, 
\begin{equation}
\label{e.matchingutophi}
\int_{\Phi_{z,r}} \left| \nabla u(x,z,r,p,q) - \nabla \underline\phi^{(1)}(x,q-p) \right|^2\,dx \leq \O_{s/2} \left( Cr^{-2\beta} \right). 
\end{equation}
\end{lemma}
\begin{proof}
The argument starts from the assumption that, for some $l \in \{ 2, \ldots,k\}$ and $\gamma\in [0,\alpha\wedge 1)$, the following two statements hold:
\begin{enumerate}
\item[$\mathrm{(i)}_{l-1}$] For every $\beta \in (0,1)$, there exists $C(\beta,s,\alpha,k,d,\Lambda)<\infty$ such that, for every $y\in\Rd$, $R\geq 2$ and $p,q\in\Ahom_{l-1}(\Phi_{y,R})$, 
\begin{equation*}
\int_{\Phi_{y,R}} \left| \nabla u(x,y,R,p,q) - \nabla \underline\phi^{(1)}(x,q-p) \right|^2\,dx \leq \O_{s/2} \left( CR^{-2\beta} \right). 
\end{equation*}
\item[$\mathrm{(ii)}_{l,\gamma}$] There exists $C(s,\alpha,k,d,\Lambda) <\infty$ such that, for every $y\in\Rd$, $R\geq 2$ and $p,q\in\Ahom_{l}(\Phi_{y,R})$, 
\begin{equation*}
\int_{\Phi_{y,R}} \left| \nabla u(x,y,R,p,q) - \nabla \underline\phi^{(1)}(x,q-p) \right|^2\,dx \leq \O_{s/2} \left( CR^{-2\gamma} \right). 
\end{equation*}
\end{enumerate}
The goal is then to show that we can improve the exponent~$\gamma$ in~$\mathrm{(ii)}_{l,\gamma}$. Once this is accomplished, an easy induction argument will complete the proof. 

\smallskip

\emph{Step 1.} We improve the exponent $\gamma\in [0,\alpha\wedge 1)$ in~$\mathrm{(ii)}_{l,\gamma}$. The claim is that, for some $c(\gamma,\alpha)>0$, 
\begin{equation*}
\mathrm{(i)}_{l-1} \ \mbox{and} \ \mathrm{(ii)}_{l,\gamma} \quad \implies \quad \mathrm{(ii)}_{l,\gamma+c}.
\end{equation*}
(We will take $y=0$ for clarity.) Fix $R\geq 2$ and select $p,q\in\Ahom_{l}(\Phi_R)$. Owing to the assumption of~$\mathrm{(i)}_{l-1}$ and the linearity of the maps $(p,q) \mapsto \nabla u(\cdot,0,R,p,q)$ and $(p,q) \mapsto \nabla  \underline\phi^{(1)}(x,q-p)$, it suffices to consider the case that
\begin{equation}
\label{e.centerorigin2}
\nabla^{l-1} \left( q- p \right)(0) = \cdots = \nabla \left( q- p \right)(0) = (q-p)(0) = 0. 
\end{equation}
Observe that this implies, for every $x,y\in\Rd$,
\begin{equation*}
\nabla (q-p)(y) = \nabla (q-p)(y-x) + \nabla ( Q_x - P_x ) (y)
\end{equation*}
where $P_x$ and $Q_x$ are the polynomials of degree $l-1$ defined by
\begin{equation*}
P_x(y):= \sum_{n=0}^{l-1} \frac1{n!} \nabla^{n}p(x) (y-x)^{\otimes n}
\quad \mbox{and} \quad
Q_x(y):= \sum_{n=0}^{l-1} \frac1{n!} \nabla^{n}q(x) (y-x)^{\otimes n}.
\end{equation*}
We fix a mesoscale $r\in [1,R)$ to be selected below. 
We have, by linearity and the triangle inequality, 
\begin{align*}
\lefteqn{
\int_{\Phi_{x,r}} \left| \nabla u(y,x,r,p,q) - \nabla \underline\phi^{(1)}(y,q-p) \right|^2\,dy
} \qquad & \\
& \leq 2 \int_{\Phi_{x,r}} \left| \nabla u(y,x,r,\tau_xp,\tau_xq)  - \nabla \underline\phi^{(1)}(y,\tau_x(q-p)) \right|^2\,dy \\
& \qquad + 2\int_{\Phi_{x,r}} \left| \nabla u(y,x,r,P_x,Q_x)  - \nabla \underline\phi^{(1)}(y,Q_x-P_x) \right|^2\,dy.
\end{align*}
By the induction hypothesis~$\mathrm{(ii)}_l$ and the fact that our normalization implies
\begin{equation*}
\left( \left\| \nabla \tau_xp \right\|_{L^2(\Phi_{x,r})}^2 +\left\| \nabla \tau_xq\right\|_{L^2(\Phi_{x,r})}^2 \right) = \left( \left\| \nabla p \right\|_{L^2(\Phi_{r})}^2 +\left\| \nabla q\right\|_{L^2(\Phi_{r})}^2 \right) \leq C \left( \frac rR \right)^{2(l-1)},
\end{equation*}
we get 
\begin{equation*}
\int_{\Phi_{x,r}} \left| \nabla u(y,x,r,\tau_xp,\tau_xq)  - \nabla \underline\phi^{(1)}(y,\tau_x(q-p)) \right|^2\,dy \leq   \O_{s/2}\left( Cr^{-2\gamma}  \left( \frac rR \right)^{2(l-1)} \right).
\end{equation*}
By the induction hypothesis~$\mathrm{(ii)}_{l-1}$ and the fact that 
\begin{equation*}
\left( \left\| \nabla P_x \right\|_{L^2(\Phi_{x,r})}^2 +\left\| \nabla Q_x \right\|_{L^2(\Phi_{x,r})}^2 \right) \leq C,
\end{equation*}
we get, for a fixed $\beta \in (0,\alpha\wedge 1)$ to be selected below, 
\begin{equation*}
\int_{\Phi_{x,r}} \left| \nabla u(y,x,r,P_x,Q_x)  - \nabla \underline\phi^{(1)}(y,Q_x-P_x) \right|^2\,dy \leq \O_{s/2} \left( Cr^{-2\beta} \right). 
\end{equation*}
Therefore we obtain
\begin{multline*}
\int_{\Phi_{x,r}} \left| \nabla u(y,x,r,p,q) - \nabla \underline\phi^{(1)}(y,q-p) \right|^2\,dy \\
\leq   \O_{s/2}\left( Cr^{-2\gamma}  \left( \frac rR \right)^{2(l-1)} \right) + \O_{s/2} \left( Cr^{-2\beta} \right).
\end{multline*}
Using Lemma~\ref{l.additivitybitches} and the triangle inequality again, we obtain
\begin{align*}
\lefteqn{
\int_{\Phi_{R}} \left| \nabla u(x,0,R,p,q) - \nabla \underline\phi^{(1)}(x,q-p) \right|^2\,dx
} \qquad & \\
& = \int_{\Phi_{\sqrt{R^2-r^2}}} \int_{\Phi_{y,r}}  
\left| \nabla u(x,0,R,p,q) - \nabla \underline\phi^{(1)}(x,q-p) \right|^2\,dx\,dy \\
& \leq 2\int_{\Phi_{\sqrt{R^2-r^2}}} \int_{\Phi_{y,r}}  \left| \nabla u(x,0,R,p,q) - \nabla u(x,y,r,p,q) \right|^2  \,dx   \,dy \\
 & \qquad + 2\int_{\Phi_{\sqrt{R^2-r^2}}} \int_{\Phi_{y,r}} 
 \left| \nabla u(x,y,r,p,q) - \nabla \underline\phi^{(1)}(x,q-p) \right|^2\,dx   \,dy  \\
 & \leq \O_{s/2} \left( Cr^{-2\alpha} \right) + \O_{s/2}\left( Cr^{-2\gamma}  \left( \frac rR \right)^{2(l-1)} \right) + \O_{s/2} \left( Cr^{-2\beta} \right) \\
 & \leq \O_{s/2}\left( Cr^{-2\gamma}  \left( \frac rR \right)^{2(l-1)} \right) + \O_{s/2} \left( Cr^{-2\beta} \right).
\end{align*}
Note that $l \geq 2$. Optimizing the choice of $r$ yields
\begin{equation*}
\int_{\Phi_{R}} \left| \nabla u(x,0,R,p,q) - \nabla \underline\phi^{(1)}(x,q-p) \right|^2\,dx \leq \O_{s/2} \left( C R^{ -2 \theta }   \right) \,,
\end{equation*}
where
\begin{equation*}
\theta : =  \frac{(l-1) \beta}{\beta-\gamma+(l-1)}\,.
\end{equation*}
We have shown that $\mathrm{(ii)}_{l,\theta}$ holds. In view of the fact that $(l-1) \geq 1 > \beta$, it is easy to check that, for some $c(\gamma,\beta,\alpha)>0$,
\begin{equation*}
0\leq \gamma < \beta  \quad \implies \quad  \theta \geq \gamma + c.  
\end{equation*}
Since we are free to choose any $\beta \in (0,1)$, this implies that  
\begin{equation*}
0\leq \gamma <  1 \quad \implies \quad  \theta \geq \gamma + c. 
\end{equation*}
This completes the proof of the claim. 

\smallskip

\emph{Step 2.} The conclusion. Iterating Step~1 gives us that, for each $\ep >0$ and $l\in \{ 2,\ldots,k\}$, 
\begin{equation*}
\mathrm{(i)}_{l-1} \ \mbox{and} \ \mathrm{(ii)}_{l,0} \quad \implies \quad \mathrm{(i)}_{l}.
\end{equation*}
Lemma~\ref{l.firstordercorrect}, in particular~\eqref{e.snaptouyou}, gives that $\mathrm{(i)}_{1}$ holds. It is clear that $\mathrm{(ii)}_{l,0}$ holds for every $l\in\{2,\ldots,k\}$ by~\eqref{e.correctorbound} and the boundedness of $\| \nabla u(\cdot,z,r,p,q) \|_{L^2(\Phi_{z,r})}$ (see~\eqref{e.bounded-J} and~\eqref{e.J-energy}). We therefore obtain by induction that $\mathrm{(i)}_{k}$ holds. This completes the argument. 
\end{proof}

The previous two lemmas allow us to estimate the difference between $L_{z,r}$ (resp., $L^*_{z,r}$) and the identity. Since we obtain a better estimate in the case $k=1$, we separate the statements for $k=1$ and $k>1$ into the following two lemmas.

\begin{lemma}
\label{l.identifytheta1}
Assume that $s\leq 2$ and $\alpha\leq \frac d2$. 
There exists $C(s,\alpha,d,\Lambda)<\infty$ such that, for every $z\in\Rd$, $r\geq r_0$ and $p,q\in \Ahom_1(\Phi_{z,r})$, 
\begin{equation}
\label{e.scalingenergy1}
\left| 
\E\left[ J_1(z,r,p,q) \right] - \int_{\Phi_{z,r}} \frac12 \nabla (q-p) \cdot \ahom \nabla (q-p)
\right|  \leq Cr^{-2\alpha}
\end{equation}
and
\begin{equation}
\label{e.verifytheta1}
\left\|  \nabla L_{1,z,r}q - \nabla q \right\|_{L^2(\Phi_{z,r})} 
+ \left\|  \nabla L_{1,z,r}^*p - \nabla p \right\|_{L^2(\Phi_{z,r})}  \leq Cr^{-2\alpha}.
\end{equation}
\end{lemma}
\begin{proof}
According to~\eqref{e.second-var} and Lemma~\ref{l.firstordercorrect}, we have, for every 
$z\in\Rd$, $r\geq r_0$ and $p,q\in \Ahom_1(\Phi_{z,r})$,
\begin{multline}
\label{e.J1blague}
J_1(z,r,L^*_{1,z,r}p,L_{1,z,r}q) - \mathcal{J}(\underline{\phi}^{(1)}(\cdot,\nabla q- \nabla p),z,r,L^*_{1,z,r}p,L_{1,z,r}q) \\
= \O_{s/2}\left( Cr^{-2\alpha} \right).
\end{multline}
Identifying $p$ and $q$ with elements of $\Rd$ as usual and taking expectations, using Remark~\ref{r.correctorexpectation}, yields
\begin{equation*}
\left| \E \left[ J_1(z,r,0,L_{1,z,r}q) \right] - \frac12 q\cdot \ahom q + L_{1,z,r}q\cdot \ahom q \right| \leq Cr^{-2\alpha}. 
\end{equation*}
Comparing this with~\eqref{e.EJL}, we get
\begin{equation*}
\left|  q\cdot \ahom q - q\cdot \ahom L_{1,z,r}q \right| \leq Cr^{-2\alpha}. 
\end{equation*}
By the symmetry of $L_{1,z,r}$ (Lemma~\ref{l.Lsymm}), this yields the desired estimate~\eqref{e.verifytheta1} for~$L_{1,z,r}$. Obtaining the same estimate for $L^*_{z,r}$ is accomplished by a very similar argument. Returning to~\eqref{e.J1blague}, taking expectations and using~\eqref{e.verifytheta1} yields~\eqref{e.scalingenergy1}.
\end{proof}

%
%
%
%
%
%
%
%
%
%

%

The previous lemma is all that is needed to complete the proof of Proposition~\ref{p.additivity} in the case $k=1$. For more general $k\in\N$, obtaining a similar statement requires some more work, since $\underline\phi^{(1)}(\cdot,p)$ is not an element of $\A$, in general, for $p\in\Ahom_k$ with $k>1$. 

\smallskip

\begin{lemma}
\label{l.matchingenergyk} 
Assume that $s\leq 2$ and $\alpha\leq \frac d2$. 
Then for each $\beta \in \left(0,\frac{2\alpha}{\alpha+1} \wedge 1\right)$, there exists ~$C(\beta,\alpha,s,k,d,\Lambda)<\infty$ such that, for every $z\in\Rd$, $r\geq r_0$ and $p,q\in \Ahom_k(\Phi_R)$, 
\begin{equation}
\label{e.scalingenergyk}
\left| 
\E \left[ J_k(z,r,p,q) \right] - \int_{\Phi_{z,r}} \frac12 \nabla (q-p) \cdot \ahom \nabla (q-p)
\right|  \leq Cr^{-\beta}
\end{equation}
and
\begin{equation}
\label{e.verifythetak}
\left\|  \nabla L_{k,z,r}q - \nabla q \right\|_{L^2(\Phi_{z,r})} 
+ \left\|  \nabla L_{k,z,r}^*p - \nabla p \right\|_{L^2(\Phi_{z,r})}  \leq Cr^{-\beta}.
\end{equation}
\end{lemma}
\begin{proof}
We may assume that $\alpha\leq 1$. 
The proof is a multiscale argument similar to that of Lemma~\ref{l.matchingutophi}. 

\smallskip

\emph{Step 1.} We claim that, for each $\beta \in (0,\alpha\wedge 1)$, there exists $C(\beta,s,\alpha,k,d,\Lambda)<\infty$ such that, for every $z\in\Rd$, $r\geq r_0$ and $p,q\in \Ahom_k(\Phi_{z,r})$, 
\begin{equation}
\label{e.scalingenergykk}
\left| 
\E \left[ J_k(z,r,p,q) \right] - \int_{\Phi_{z,r}} \frac12 \nabla (q-p) \cdot \ahom \nabla (q-p)
\right|  \leq Cr^{-\beta}
\end{equation}
and
\begin{equation}
\label{e.verifythetakk}
\left\|  \nabla L_{k,z,r}q - \nabla q \right\|_{L^2(\Phi_{z,r})} 
+ \left\|  \nabla L_{k,z,r}^*p - \nabla p \right\|_{L^2(\Phi_{z,r})}  \leq Cr^{-\beta}.
\end{equation}
The argument is only very slightly different than the proof of the previous lemma. In what follows, we drop dependence on $k$. 
According to~\eqref{e.second-var} and Lemma~\ref{l.matchingutophi}, we have, for every 
$z\in\Rd$, $r\geq r_0$ and $p,q\in \Ahom_k(\Phi_{z,r})$,
\begin{equation*}
J(z,r,L^*_{z,r}p,L_{z,r}q) - \mathcal{J}(\underline{\phi}^{(1)}(\cdot,\nabla q- \nabla p),z,r,L^*_{z,r}p,L_{z,r}q) 
= \O_{s}\left( Cr^{-\beta} \right).
\end{equation*}
Taking expectations and using Remark~\ref{r.correctorexpectation} yields 
\begin{equation*}
\left| \E \left[ J(z,r,0,L_{z,r}q) \right] - \int_{\Phi_{z,r}} \left( \frac12 \nabla q\cdot \ahom \nabla q + \nabla L_{z,r}q\cdot  \ahom\nabla q\right) \right| \leq Cr^{-\beta}. 
\end{equation*}
Comparing this with~\eqref{e.EJL}, we get
\begin{equation*}
\left|  \int_{\Phi_{z,r}} \left(\nabla q\cdot \ahom \nabla q - \nabla q\cdot \ahom \nabla L_{z,r}q\right) \right| \leq Cr^{-\beta}. 
\end{equation*}
The symmetry of $L_{z,r}$ by Lemma~\ref{l.Lsymm} gives the estimate of the first term on the left of~\eqref{e.verifythetakk}.  The estimate for the second term is similar and from these we obtain~\eqref{e.scalingenergykk}.

\smallskip

\emph{Step 2.} We show that, for every $\beta\in (0,\alpha)$, there exists $C(\beta,\alpha,s,k,d,\Lambda)<\infty$ such that, for each $z\in\Rd$, $r\geq 1$ and $p,q\in\Ahom_{1}(\Phi_{z,r})$,
\begin{equation}
\label{e.affinehoney}
\left| \E \left[ I(z,r,p,q)\right] - \int_{\Phi_{z,r}} \frac12 \nabla (q-p) \cdot \ahom \nabla (q-p) \right| 
\leq Cr^{-2\beta}. 
\end{equation}
Using~\eqref{e.verifythetakk}, Lemmas~\ref{l.Iquadresponse},~\ref{l.matchingutophi}, and the fact that $ \underline\phi^{(1)}(\cdot,(q-p))\in \A_1\subseteq \A_k$ for every $p,q\in\Ahom_{1}(\Phi_{z,r})$, we find that, for every $\beta<\alpha$, 
\begin{equation*}
\left|  I(z,r,p,q) - \mathcal{J}\left( \underline{\phi}^{(1)}(\cdot,q-p), z,r, p,q\right)  \right| 
\leq \O_{s/2}\left(Cr^{-2\beta} \right).
\end{equation*}
Taking expectations yields the claim. 

\smallskip

\emph{Step 3.} We improve the estimate in Step~1 using a multiscale argument. The claim is that, for every $z\in\Rd$, $R\geq 1$ and $p,q\in\Ahom_{k}(\Phi_{z,R})$,
\begin{equation}
\label{e.haha}
\left| \E \left[ I(z,R,p,q)\right]  - \int_{\Phi_{z,R}} \frac12 \nabla (q-p) \cdot \ahom \nabla (q-p)
 \right| 
 \leq CR^{-2\beta/(1+\beta)}.
\end{equation}
For clarity, we consider only the case $z=0$. We fix $\beta\in (0,\alpha)$, $R\geq 2$ and $p,q\in \Ahom_k(\Phi_R)$. We choose a mesoscopic scale $r \in [1,\frac12 R]$ to be selected below.
For each $z\in\Rd$, set $p_z:= \pi_{z,r,1}p$ and $q_z:=\pi_{z,r,1}q$ (these are the projections defined in~\eqref{e.def.pim}). Note that $p_z,q_z\in\Ahom_1$. To keep the expressions short, we also put
\begin{equation*}
H_z:=  \left( \left\| \nabla p \right\|_{L^2(\Phi_{z,r})} +\left\| \nabla  q \right\|_{L^2(\Phi_{z,r})}  \right).
\end{equation*}
Observe that 
\begin{equation}
\label{e.normedtom} 
\left\| \nabla p-\nabla p_z\right\|_{L^2(\Phi_{z,r})}  +  \left\| \nabla q-\nabla q_z\right\|_{L^2(\Phi_{z,r})}  
 \leq CH_z\left( \frac rR \right).
\end{equation}
According to~\eqref{e.polarization}, we have,  
\begin{multline}
\label{e.polarforI}
I(z,r,p,q) 
 = I\left(z,r,p_z,q_z \right) + I\left(z,r,p-p_z,q-q_z \right) \\
+ \int_{\Phi_{z,r}} \nabla u(x,z,r,p_z,q_z ) \cdot \left( \a\nabla L_{z,r}^*(p-p_z)(x) - \ahom \nabla L_{z,r}(q-q_z)(x)  \right) \,dx.
\end{multline}
We want to take the expectation of~\eqref{e.polarforI}. The expectation of the first term on the right side is given by~\eqref{e.affinehoney}:
\begin{equation}
\label{e.Efirstterm}
\left| \E \left[ I\left(z,r,p_z,q_z \right) \right] - \int_{\Phi_{z,r}} \frac12 \nabla (q_z-p_z) \cdot \ahom \nabla (q_z-p_z) \right| \\
\leq CH_z^2 r^{-2\beta}.
\end{equation}
The expectation of the second term is given by~\eqref{e.scalingenergykk}:
\begin{multline}
\label{e.Esecondterm}
\Big| \, \E \left[ I\left(z,r,p-p_z,q-q_z \right)  \right]  \\
- \int_{\Phi_{z,r}} \frac12 \nabla ((q-q_z)-(p-p_z)) \cdot \ahom \nabla  ((q-q_z)-(p-p_z)) \Big| \\
 \leq C \left( \left\| p-p_z \right\|_{L^2(\Phi_{z,r})}^2 +\left\| q-q_z \right\|_{L^2(\Phi_{z,r})}^2  \right) r^{-\beta} 
 \leq C H_z^2\left( \frac rR \right)^{2} r^{-\beta}. 
\end{multline}
We turn to the expectation of the third term. By Lemma~\ref{l.fluxmaps} and~\eqref{e.verifythetakk},
\begin{align}
\label{e.Ethirdterm}
\bigg| \E \left[ \int_{\Phi_{z,r}} \nabla u(x,z,r,p_z,q_z ) \cdot \left( \a\nabla L_{z,r}^*(p-p_z)(x) - \ahom \nabla L_{z,r}(q-q_z)(x)  \right) \,dx  \right]\\
- \int_{\Phi_{z,r}} \nabla (q_z-p_z) \cdot \ahom \nabla \left((p-p_z) - (q-q_z) \right) \bigg| \notag\\
\leq CH_z^2 \left(\frac rR\right) r^{-\beta}. \notag
\end{align}
Finally, we observe that
\begin{align*}
\lefteqn{
\int_{\Phi_{z,r}} \frac12 \nabla (q-p) \cdot \ahom \nabla (q-p)
}  \ \ & \\
& = \int_{\Phi_{z,r}} \frac12 \nabla (q_z-p_z) \cdot \ahom \nabla (q_z-p_z) 
+ \int_{\Phi_{z,r}} \nabla  (q_z- p_z)  \cdot\ahom \nabla \left( (p-p_z)  -(q-q_z)  \right)\\
& \quad +\int_{\Phi_{z,r}} \frac12 \nabla ((q-q_z)-(p-p_z)) \cdot \ahom \nabla  ((q-q_z)-(p-p_z)).
\end{align*}
By the previous display,~\eqref{e.polarforI},~\eqref{e.Efirstterm},~\eqref{e.Esecondterm} and~\eqref{e.Ethirdterm}, we obtain
\begin{equation*}
\left| \E\left[ I(z,r,p,q)\right]  - \int_{\Phi_{z,r}} \frac12 \nabla (q-p) \cdot \ahom \nabla (q-p)
 \right| \\
 \leq CH_z^2 \left( r^{-2\beta} + \left( \frac rR\right) r^{-\beta}\right).
\end{equation*}
Integrating with respect to $\Phi_{\sqrt{R^2-r^2}}$ and applying Lemma~\ref{l.improveadditivity} (noting that by~\eqref{e.verifythetak} we can take any $\theta=\beta$), we obtain
\begin{equation*}
\left| \E \left[ I(0,R,p,q)\right]  - \int_{\Phi_R} \frac12 \nabla (q-p) \cdot \ahom \nabla (q-p)
 \right| 
 \leq C\left( r^{-2\beta}  + \left( \frac rR\right) r^{-\beta}   \right).
\end{equation*}
Now we choose $r:= R^{1/(1+\beta)}$ to obtain~\eqref{e.haha}.
\smallskip

\emph{Step 4.} The conclusion. Combining~\eqref{e.I.as.J},~\eqref{e.verifythetakk} and~\eqref{e.haha}, we obtain
\begin{equation}
\label{e.hahaJ}
\left| \E \left[ J_k(z,R,p,q)\right]  - \int_{\Phi_{z,R}} \frac12 \nabla (q-p) \cdot \ahom \nabla (q-p)
 \right| 
 \leq CR^{-2\beta/(1+\beta)}.
\end{equation}
This is~\eqref{e.scalingenergyk}. Applying Lemma~\ref{l.snappingtheLs} gives~\eqref{e.verifythetak}.
\end{proof}

\begin{proof}[{Proof of Proposition~\ref{p.additivity}}]
By Lemma~\ref{l.identifytheta1}, the assumption~\eqref{e.assofLtheta} of Lemma~\ref{l.improveadditivity} is valid in the case $\theta=\alpha$, $k=1$. Similarly, by Lemma~\ref{l.matchingenergyk},  assumption~\eqref{e.assofLtheta} is valid as well as in the case $\theta=\alpha \wedge 1$, $k\in\N$. The conclusion of Lemma~\ref{l.improveadditivity} therefore gives the proposition.
\end{proof}

\section{Improvement of localization and gradient-flux duality}
\label{s.localization}

In this section we prove Proposition~\ref{p.localization}, which contains two statements concerning the improvement of $\Loc_k(s,\alpha)$, as well as Proposition~\ref{p.min} concerning the improvement of $\Dual_k(\alpha)$.

\subsection{First improvement of localization}
\label{ss.localization1}

In this subsection, we give the proof of Proposition~\ref{p.localization}. The main step in the argument is to localize the vector space $\A_k$ itself. That is, given $\delta > 0$ and $R\gg1$, we identify a vector space $V^{(\delta)}_{k,R}$ that has the same dimension as $\A_k$, is~$\F(B_{R^{1+\delta}})$-measurable
and approximates $\A_k$ to within a suitable error. 
The argument crucially relies on the regularity theory stated in Proposition~\ref{p.regularity}. 

\smallskip

We assume throughout this subsection that, for fixed $k\in\N$, $s\in(0,\infty)$ and $\alpha \in \left( 0,\frac ds \right)$,
\begin{equation}
\label{e.yourassencoreencore}
\Fluc_k(s,\alpha)
\ \ \mbox{and} \ \ 
\Dual_k(\alpha) 
\quad \mbox{hold.}
\end{equation}
In particular, the lemmas proved in the previous section are applicable. We also take the same notational convention for $\X$ and $\Y$ as in the previous section.

\begin{lemma}
\label{l.localization}
For each $\delta > 0$ and $R\geq1$, there exists a vector space 
\begin{equation*} \label{}
V^{(\delta)}_{k,R}\subseteq \A\left(B_{R^{1+\delta}}\right)
\end{equation*}
such that
\begin{equation}
\label{e.Vdkmeas}
V^{(\delta)}_{k,R} \quad \mbox{is $\F(B_{R^{1+\delta}})$-measurable}
\end{equation}
and, for each $\beta \in (0,\alpha(1+ \delta) +\delta) \cap \left(0,\frac ds\right)$, a constant $C(\beta,\delta,s,\alpha,k,d,\Lambda)<\infty$ such that for every $R\geq (\X\vee \Y \vee C)$,
\begin{equation}
\label{e.Vdkdim}
\dim\left(V_{k,R}^{(\delta)}\right) = \dim\left( \Ahom_k \right)
\end{equation}
and
\begin{equation}
\label{e.Vdkapprox}
\sup_{u \in\A_k} \inf_{v \in V^{(\delta)}_{k,R}} \frac{\left\| \nabla u- \nabla v \right\|_{L^2\left(\Phi_R^{(\delta)}\right)} }{ \left\|\nabla u \right\|_{L^2\left(\Phi_R^{(\delta)}\right)} } = \O_s\left(CR^{-\beta} \right).
\end{equation}
\end{lemma}
\begin{proof}

\emph{Step 1.} We begin with the construction of $V^{(\delta)}_{k,R}$. We set $T:=R^{1+\delta}$ and $S:= R^{1+\delta - \tilde \ep}$, with $\tilde \ep>0$ chosen according to
\begin{equation*}
\beta = \alpha(1+\delta) + \delta - (1+ \alpha) \tilde \ep\,.
\end{equation*}
Let $n_k := \dim\left( \Ahom_k \right)$ and select a basis~$\{ p_1,\ldots,p_{n_k} \}$ of~$\Ahom_k$ so that $p_1\equiv 1$ and each $p_j$ with $j\geq 2$ satisfies the normalization $\left\| \nabla p_j \right\|_{L^2(\Phi_{R})} = 1$. Set $w_1:= p_1$ and, for each $j\in \{ 2,\ldots,n_k\}$, select $w_j\in \A\left(B_T \right)$ to minimize the quantity
\begin{equation*}
\left[ w - p_j \right]_S^2:= 
\fint_{B_{S/\theta}} \left| \int_{\Phi_{y,\sigma S}} \left( w(x) - p_j(x) \right) \1_{B_T}(x) \,dx \right|^2 \,dy 
\end{equation*}
among all functions in the class
\begin{equation*}
\mathcal{D}(B_T):= \left\{ w \in \A(B_T)\,:\,   \left\| \nabla w \right\|_{\underline{L}^2(B_{T})} \leq \lambda \left( \frac{T}{R} \right)^{k-1} \right\},
\end{equation*}
where $\lambda \in [1,\infty)$ will be chosen large enough.  The parameters $\theta(\delta,\beta,s,\alpha,k,d,\Lambda)$ and $\sigma(\delta,\beta,s,\alpha,k,d,\Lambda)$ are given by Lemma~\ref{l.mspoincare2}
corresponding $\A_m$, where the integer $m(\delta,\beta,s,\alpha,k,d,\Lambda)$ is chosen in Step 2 below. Precisely, we take $w_j \in \mathcal{D}(B_T)$ so that 
\begin{equation*}
\left[ w_j - p_j \right]_S =  \inf_{w \in  \mathcal{D}(B_T)} \left[ w - p_j \right]_S.
\end{equation*}
In order to simplify some of the expressions below, we assume that $w_j$ is canonically extended to be $p_j$ outside of $B_T$. 
The functional we minimize in this variational problem is clearly weakly continuous on the convex set $\mathcal{D}(B_T)$ with respect to the norm $\| \cdot \|_{H^1(B_T)}$ and therefore we deduce the existence of a minimizer. If the minimizer is not unique, we select the one that has the smallest $\|\cdot\|_{\underline{L}^2(B_T)}$ norm. 

\smallskip

As we now argue, the parameter $\lambda$ in the definition of $\mathcal{D}(B_T)$ can be chosen large enough that
\begin{equation}
\label{e.choose.lambda}
R \ge \X \ \ \text{and} \ \ \frac p 2 \in \Ahom_k(\Phi_R) \quad \implies \quad   \left\| \nabla u(\cdot,0,S,-p,0) \right\|_{\underline{L}^2(B_T)} \leq \lambda \left( \frac{T}{R} \right)^{k-1}.
\end{equation}
Indeed, for $R \ge \X$ and $\left\| \nabla p \right\|_{L^2(\Phi_R)} \leq 2$, Proposition~\ref{p.regularity} gives $q \in \Ahom_k$ such that, for $r \in [S,T]$, $u := u(\cdot,0,S,-p,0)$ and for any $a \in \R$ ,
\begin{equation*} 
\left\| u -a \right\|_{\underline{L}^2(B_{2T})}  \leq C\left\| q -a \right\|_{\underline{L}^2(B_{2T})}  \leq C \left( \frac{T}{S} \right)^k 
\left\| q -a \right\|_{\underline{L}^2(B_S)}  \leq 
C \left( \frac{T}{S} \right)^k  \left\| u-a \right\|_{\underline{L}^2(B_S)} .
\end{equation*}
Recall from \eqref{e.J-energy} that $\left\| \nabla u \right\|_{L^2(\Phi_S)} \leq C \left\| \nabla p \right\|_{L^2(\Phi_S)} \leq C \left(\frac{S}{R}\right)^{k-1}$. 
By the Caccioppoli and Poincar\'e inequalities, we thus get
\begin{equation*} 
\left\| \nabla u \right\|_{\underline{L}^2(B_{T})}  \leq \frac{C}{S}  \left( \frac{T}{S} \right)^{k-1} \inf_{a \in \R} \left\| u -a \right\|_{\underline{L}^2(B_{S})} \leq 
C \left( \frac{T}{S} \right)^{k-1}  \left\| \nabla u \right\|_{\underline{L}^2(B_S)}  \leq C  \left( \frac{T}{R} \right)^{k-1} \,,
\end{equation*}
and this shows \eqref{e.choose.lambda} for $\lambda(k,d,\Lambda)$ sufficiently large.

\smallskip

We then define
\begin{equation*}
V^{(\delta)}_{k,R} := \spn\left\{ w_1,\ldots,w_{n_k} \right\}
\end{equation*}
and denote by
\begin{equation*}
\mathcal{T}: \Ahom_k \to V^{(\delta)}_{k,R}
\end{equation*}
the linear map which satisfies $\mathcal{T}p_j = w_j$ for every $j\in \{1,\ldots,n_k\}$. 

\smallskip

It is immediate that $V^{(\delta)}_{k,R}\subseteq \A\left(B_{T} \right)$, that~$V^{(\delta)}_{k,R}$  satisfies the measurability condition~\eqref{e.Vdkmeas} and that $\dim ( V^{(\delta)}_{k,R}) \leq n_k$. The remainder of the proof, most of which is focused on~\eqref{e.Vdkapprox}, is broken into several steps.

\smallskip

Throughout the rest of the argument we may assume that $R \geq \X$, since we have by construction (by taking $v = w_1 =1$) that 
\begin{equation*}
\indc_{\{ R \leq \X \} } \sup_{u \in\A_k} \inf_{v \in V^{(\delta)}_{k,R}} \frac{\left\| \nabla u- \nabla v \right\|_{L^2\left(\Phi_R^{(\delta)}\right)} }{ \left\|\nabla u \right\|_{L^2\left(\Phi_R^{(\delta)}\right)} }  
\leq \indc_{\{ R \leq \X \} } = \O_s\left(CR^{-\beta} \right).
\end{equation*}

\smallskip

\emph{Step 2.} We claim that, for every $j\in\{2,\ldots,n_k\}$ and $u_j := u(\cdot,0,S,-p_j,0)$, we have
\begin{equation}
\label{e.closetoS}
 \left\|  \mathcal{T}p_j - u_j  \right\|_{\underline{L}^2\left(B_{S/\theta}\right)} \leq \O_{s}  \left( C \left\|\nabla p_j\right\|_{L^2\left( \Phi_S \right)} S^{1-\alpha} \right)\,.
\end{equation}
Observe that the selection of $\lambda$ ensures that, for each $j\in \{ 2,\ldots,n_k\}$, we have that $u_j \in \mathcal{D}(B_T)$. Thus, for every $j$,
\begin{equation*}
\left[ \mathcal{T}p_j - p_j \right]_S \leq \left[ u_j  - p_j \right]_S. 
\end{equation*}
By Lemma~\ref{l.parametermatching}, 
we have that 
\begin{equation*}
\left[ u_j - p_j \right]_S \leq \O_s\left( C \left\|\nabla p_j\right\|_{L^2\left( \Phi_S \right)} S^{1-\alpha} \right) .
\end{equation*}
Indeed, normalizing $ \int_{B_{S/\theta} } \int_{\Phi_{y,\sigma S}} u_j(z) \, dz \, dy = 0$  we have, by the Poincar\'e inequality and~\eqref{e.spatavgspoly}, that
\begin{align*}
 \frac{\left[ u_j - p_j \right]_S^2}{\left\|\nabla p_j\right\|_{L^2\left( \Phi_S \right)}^2 }
 & = \left\|\nabla p_j\right\|_{L^2\left( \Phi_S \right)}^{-2} \int_{B_{S/\theta}} \left| \int_{\Phi_{y,\sigma S}} \left( u_j(x) - p_j(x) \right) \,dx \right|^2 \,dy \\
&  \leq C \left\|\nabla p_j\right\|_{L^2\left( \Phi_S \right)}^{-2} S^2 \int_{B_{S/\theta}} \left| \int_{\Phi_{y,\sigma S}} \left( \nabla u_j(x) - \nabla p_j(x) \right) \,dx \right|^2 \,dy \\
& \leq C \left\|\nabla p_j\right\|_{L^2\left( \Phi_S \right)}^{-2} S^2\int_{\Phi_{\sqrt{S^2-(\sigma S)^2}}}  \left|  \int_{\Phi_{y,\sigma S}} \left( \nabla u_j(x) - \nabla p_j(x) \right) \,dx \right|^2 \,dy \\
& \leq \O_{s/2}\left( C  S^{2-2\alpha}\right)\,.
\end{align*}
We deduce that 
\begin{equation*}
\left[ \mathcal{T}p_j - p_j \right]_S \leq\O_{s}\left( C \left\|\nabla p_j\right\|_{L^2\left( \Phi_S \right)} S^{1-\alpha}\right)
\end{equation*}
and therefore, by the triangle inequality, 
\begin{equation*}
\left[ \mathcal{T}p_j - u_j \right]_S \leq \O_{s}\left( C \left\|\nabla p_j\right\|_{L^2\left( \Phi_S \right)}S^{1-\alpha} \right).
\end{equation*}
Furthermore, by the regularity theory, the definition of $\mathcal{D}(B_T)$, and the fact that $u_j, \mathcal{T}p_j \in \mathcal{D}(B_T)$  we have, for every $m\in\N$ with $m\geq k$, 
\begin{align*}
\inf_{v \in \A_m} \left\|  \mathcal{T}p_j - u_j - v  \right\|_{\underline{L}^2\left(B_{S/\theta}\right)} 
& \leq C \left( \frac{S}{\theta T} \right)^{m+1} 
\left\|  \mathcal{T}p_j - u_j  \right\|_{\underline{L}^2(B_T)} \\
& \leq C \lambda \left( \frac{S}{T} \right)^{m+1}  R^{1-k} T^k = C  R^{-\tilde \ep(m+1) + 1+ \delta k}.
\end{align*}
Taking $m(\tilde \ep,s,k,d)\in\N$ sufficiently large, we find $v\in \A_m$ such that 
\begin{equation*}
 \left\|  \mathcal{T}p_j - u_j - v  \right\|_{\underline{L}^2\left(B_{S/\theta}\right)} \leq CS^{1-\alpha}. 
\end{equation*}
Applying Lemma~\ref{l.mspoincare2} to $v$ and using the triangle inequality, we get that
\begin{multline*}
 \left\| v  \right\|_{L^2\left(\Psi_S\right)}
  \leq C \left( \fint_{S/\theta} \left| \int_{\Phi_{y,\sigma S}} v(z)\,dz \right|^2 \, dy \right)^{\frac12}  \\
    \leq C\left[ \mathcal{T}p_j - u_j \right]_S + CS^{1-\alpha} 
  \leq \O_{s}  \left( C \left\|\nabla p_j\right\|_{L^2\left( \Phi_S \right)} S^{1-\alpha} \right).
\end{multline*}
By the triangle inequality again, we obtain 
\begin{equation*}
 \left\|  \mathcal{T}p_j - u_j  \right\|_{\underline{L}^2\left(B_{S/\theta}\right)} \leq \O_{s}  \left( C \left\|\nabla p_j\right\|_{L^2\left( \Phi_S \right)} S^{1-\alpha} \right).
\end{equation*}
This completes the proof of~\eqref{e.closetoS}.

\smallskip

\emph{Step 3.} Observe that~\eqref{e.closetoS} implies that, for $R\geq (\X \vee\Y\vee C)$, the kernel of $\mathcal{T}$ is trivial and therefore $\dim\left(V^{(\delta)}_{k,R} \right) = n_k$. Indeed, 
the assumption $R \geq \Y\vee C$, Proposition~\ref{e.Ys} and~\eqref{e.J-energy} imply that, for some $\gamma(s,\alpha,k,d,\Lambda)>0$,
\begin{equation*} \label{}
\left| \int_{\Phi_{S}}  \nabla p_j \cdot \ahom \nabla p_j - \int_{\Phi_{S}} \nabla u_j \cdot\a\nabla u_j  \right| \leq R^{-\gamma}.
\end{equation*}
Then~\eqref{e.closetoS}, the Lipschitz estimate (using also~\eqref{e.upolygrowth1} and $R\geq \X$ to eliminate the tails of $\Phi_S$ in the integral) imply that  
\begin{equation*} \label{}
\left| \int_{\Phi_{S}} \nabla p_j \cdot \ahom \nabla p_j - \int_{B_{S/\theta}} \nabla \mathcal{T}p_j \cdot\a\nabla \mathcal{T}p_j  \right| \leq CR^{-\gamma}.
\end{equation*}
Therefore $\mathcal{T}$ is within $CR^{-\gamma}$ of an isometry between finite dimensional inner product spaces. It follows that $\mathcal{T}$ has a trivial kernel if $R^{-\gamma} \leq c$ for some $c>0$ depending on the dimension of~$\Ahom_k$, which depends only on $(k,d)$. This assumption is valid by $R\geq C$ if we enlarge the constant $C$. 

\smallskip

\emph{Step 4.} Application of the regularity theory. Applying Proposition~\ref{p.regularity} yields, for every $j\in \{ 1,\ldots,n_k\}$, the existence of $w \in \A_k$ such that, for every $r \in [R,S]$, we have
 \begin{equation*} \label{e.Tpw}
\left\| \mathcal{T}p_j -  w  \right\|_{\underline{L}^2\left(B_{2r} \right)} 
 \leq 
 C \left( \frac{r}{S} \right)^{k+1}  \left\| \mathcal{T}p_j - u_j \right\|_{\underline{L}^2\left(B_{S/\theta}\right)} \,.  
\end{equation*}
Then the Caccioppoli inequality,~\eqref{e.closetoS}, and the normalization $ \left\|\nabla p_j\right\|_{L^2\left( \Phi_R \right)} = 1$ yield
 \begin{equation*} \label{e.Tpwgrad}
\left\| \nabla \mathcal{T}p_j -  \nabla w  \right\|_{\underline{L}^2\left(B_{r} \right)} 
 \leq 
  \O_{s}  \left( C \left( \frac{r}{S} \right)^{k} \left\|\nabla p_j\right\|_{L^2\left( \Phi_S \right)} S^{-\alpha} \right) \leq  \O_{s}  \left( C \left( \frac{r}{S} \right)  S^{-\alpha} \right)   \,.  
\end{equation*}
On the other hand, if $r \in [S,T]$, we have, by the definition of $\mathcal{D}(B_T)$ and  growth of $w$,
$$
\left\| \nabla \mathcal{T}p_j -  \nabla w  \right\|_{\underline{L}^2\left(B_{r} \right)} \leq C  \left( \frac{T}{r} \right)^d \left( \frac{T}{R} \right)^{k-1} \,.
$$
By the previous two displays we then obtain 
\begin{align*} 
\lefteqn{ \left\| \nabla \mathcal{T}p_j -  \nabla w  \right\|_{\underline{L}^2\left(\Phi_R^{(\delta)} \right)}^2 } \qquad & 
 \\ &  \leq C  \int_R^\infty \exp\left(-c \left(\frac{r}{R} \right)^2 \right) \left(\frac{r}{R}\right)^d \fint_{B_r \cap B_{T}} \left| \nabla \mathcal{T}p_j(x) -  \nabla w   (x) \right|^2 \, dx \, \frac{dr}{r}
\\& =  \O_{s/2}  \left( C \left( \frac{R}{S} \right) ^2 S^{-2\alpha} \right)\,.
\end{align*}
Therefore, by the linearity of $\mathcal{T}$ and the finite dimensionality of $\Ahom_k$, we obtain, for every $p\in\Ahom_k$ with $\left\|\nabla p\right\|_{L^2(\Phi_R)} \leq 1$, 
\begin{equation*}
\inf_{w \in \A_k}  \left\|  \nabla \mathcal{T}p - \nabla w  \right\|_{L^2\left(\Phi_R^{(\delta)}\right)} = \O_s \left( CR^{-\alpha - \delta(\alpha + 1) + (1+ \alpha) \tilde \ep  } \right).
\end{equation*}
This completes the argument since $\beta = \alpha+\delta(\alpha + 1) - (1+ \alpha) \tilde \ep$. 
\end{proof}

With the aid of Lemma~\ref{l.localization}, we can define a localized version of $J$ and show the improvement of localization.  For each $\delta > 0$, $x\in \Rd$, $r\ge 1$ and $p,q\in \Ahom_k$ we set
\begin{equation*}
J^{(\delta)}(x,r,p,q) 
:= \sup_{u \in V^{(\delta)}_{k,r}(x)}  \int_{\Phi_{x,r}^{(\delta)}}\left(-\frac12 \nabla u \cdot\a  \nabla u - \nabla p\cdot \a\nabla u + \nabla q\cdot \nabla u \right),
\end{equation*}
where $\Phi^{(\delta)}_{x,r}$ is the truncated mask defined in \eqref{e.Phideltadef}.

\begin{lemma}
Fix $\delta > 0$ and $\beta \in (0,\alpha(1+\delta) + \delta) \cap (0,\frac d s)$. There exists a constant $C(\delta,\beta,s,\alpha,k,d,\Lambda)<\infty$ such that, for every $R\ge1$ and $p,q\in\Ahom_k(\Phi_R)$,
\begin{equation}
\label{e.Jlocal}
\left| J(0,R,p,q) - J^{(\delta)}(0,R,p,q) \right| \\
= \O_s\left(CR^{-\beta} \right). 
\end{equation}
In particular, we have that $\Loc_k\left(s,\delta,\beta \right)$ holds. 
\label{l.localization2}
\end{lemma}
\begin{proof}

\emph{Step 1.} 
We denote the unique (up to an additive constant) maximizer of $J^{(\delta)}(x,r,p,q)$ by $u^{(\delta)}(\cdot,x,r,p,q)$. It is immediate that 
\begin{equation}
\label{e.Jdeltalocal}
J^{(\delta)}(x,r,p,q) \quad \mbox{is $\F(B_{R^{1+\delta}})$-measurable}
\end{equation}
and
\begin{equation}
\label{e.udeltalocal}
u^{(\delta)}(\cdot,x,r,p,q) \quad \mbox{is $\F(B_{R^{1+\delta}})$-measurable.}
\end{equation}
We next compare these local versions~$J^{(\delta)}$ and $u^{(\delta)}$ to the original quantities. 

Fix $R\geq 1$ and $p,q\in \Ahom_k(\Phi_R)$. 
For concision, we write $u: = u(\cdot,0,R,p,q)$ and $u^{(\delta)}: = u^{(\delta)}(\cdot,0,R,p,q)$. We have that 
\begin{equation*}
J(0,R,p,q) + J^{(\delta)}(0,R,p,q) \leq C 
\end{equation*}
and
\begin{equation*}
\left\| \nabla u \right\|_{L^2\left(\Phi_R\right)} + \left\| \nabla u^{(\delta)} \right\|_{L^2\left(\Phid_R\right)} \leq C. 
\end{equation*}
We may work on the event $\{ R \geq \X\}$, since the boundedness of $J$, $J^{(\delta)}$ and their maximizers yields, for every $\theta < d$ and $t>0$, a constant $C(\theta,\delta,t,k,d,\Lambda)<\infty$ such that  
\begin{equation*}
\left( \left| J(0,R,p,q) \right| + \left| J^{(\delta)}(0,R,p,q) \right| \right)
\indc_{\{ R \leq \X\} }  
\leq \O_t \left( CR^{-\theta/t}\right)
\end{equation*}
and
\begin{equation*}
\left( \left\| \nabla u \right\|_{L^2\left(\Phid_R\right)} + \left\| \nabla u^{(\delta)}  \right\|_{L^2\left(\Phid_R\right)} \right)
 \indc_{\{ R \leq \X\} }  
\leq \O_t \left( CR^{-\theta/t}\right).
\end{equation*}
Therefore we assume that $R \geq \X$ for the rest of the argument. 

\smallskip

According to Lemma~\ref{l.localization}, we may select $v^{(\delta)} \in V^{(\delta)}_{k,R}$ and $v \in \A_k$ to satisfy
\begin{equation}
\label{e.maxdeltaapprox}
\left\| \nabla u - \nabla v^{(\delta)} \right\|_{L^2\left(\Phid_R\right)} + \left\| \nabla u^{(\delta)} - \nabla v \right\|_{L^2\left(\Phid_R\right)} \leq \O_s\left( CR^{-\beta}\right).
\end{equation}
Since $R \geq \X$, the $k^{th}$ degree polynomial growth of $u$ and $v$ given by Proposition~\ref{p.regularity} and the Caccioppoli estimate, together with $R^{k+\be} \exp\left( R^{-2\delta} \right) \leq C(k,\be, \delta)$ for all $R \geq 1$, yield
\begin{equation*}
\left\| \nabla u \right\|_{L^2\left(\Phi_R-\Phid_R\right)} + \left\|  \nabla v \right\|_{L^2\left(\Phi_R-\Phid_R\right)} \leq C R^{-\be}. 
\end{equation*}
We now compute 
\begin{align*}
\lefteqn{
J^{(\delta)}(x,r,p,q) 
} \qquad & \\
& \geq \int_{\Phid_R} \left(-\frac12 \nabla v^{(\delta)} \cdot\a  \nabla  v^{(\delta)} - \nabla p\cdot \a\nabla  v^{(\delta)} + \nabla q\cdot \nabla  v^{(\delta)} \right) \\
& \geq \int_{\Phid_R} \left(-\frac12 \nabla u \cdot\a  \nabla  u - \nabla p\cdot \a\nabla  u + \nabla q\cdot \nabla  u \right)
- C \left\| \nabla u - \nabla v^{(\delta)} \right\|_{L^2\left( \Phid_R \right)}  \\
& \geq J(0,R,p,q) 
- CR^{-\beta} 
- \O_s\left( CR^{-\beta}\right) \\ 
& \geq J(0,R,p,q) -  \O_s\left( CR^{-\beta}\right)
\end{align*}
and
\begin{align*}
J(0,R,p,q) 
& \geq \int_{\Phi_R}  \left(-\frac12 \nabla v \cdot\a  \nabla  v  - \nabla p\cdot \a\nabla  v  + \nabla q\cdot \nabla  v \right) \\
& \geq \int_{\Phid_R}  \left(-\frac12 \nabla v \cdot\a  \nabla  v  - \nabla p\cdot \a\nabla  v  + \nabla q\cdot \nabla  v \right) - CR^{-d} \\
& \geq \int_{\Phid_R}  \left(-\frac12 \nabla u^{(\delta)} \cdot\a  \nabla  u^{(\delta)}  - \nabla p\cdot \a\nabla  u^{(\delta)}  + \nabla q\cdot \nabla  u^{(\delta)} \right)  \\
& \qquad 
- C \left\| \nabla v - \nabla u^{(\delta)} \right\|_{L^2\left( \Phid_R \right)}
- CR^{-\beta} \\
& \geq J^{(\delta)}(0,R,p,q) 
- \O_s\left( CR^{-\beta}\right).
\end{align*}
The two previous displays imply~\eqref{e.Jlocal}.
\end{proof}

The proof of the first statement of Proposition~\ref{p.localization} is now complete, since the statement is contained in that of Lemma~\ref{l.localization2}.

\subsection{Localization using higher-order $J_k$}
\label{ss.localizationk}
In this subsection, we complete the proof of Proposition~\ref{p.localization} by giving the argument for the second assertion in its statement.

\smallskip

We begin with a lemma which compares~$J$ for different values of~$k$. We denote this dependence by writing $J_k$, $L_{z,r,k}$, $u_k$, and so on. What we show is that the quantities $J_{k'}(z,r,p,q)$ and $J_k(z,r,p,q)$ agree when $p,q\in \Ahom_k$ and $k \leq k'$ up to almost $\O_{s/2}\left( Cr^{-2\alpha} \right)$. 

\begin{lemma}
\label{l.skypepro}
Assume that $s\in (0,2]$, $\alpha\in\left(0,\frac ds\right)\cap \left(0, \frac d2\right]$ are such that $\Fluc_k(s,\alpha)$ and $\Dual_k(\alpha)$ hold for every $k\in\N$. Let $s\in (0,\infty)$, $\alpha \in \left( \ep ,\frac ds \right)$ and $k'\in\N$ with $k'\geq k$. Fix $\beta \in \left(0,\alpha\wedge 1\right)$. 
Then there exists a constant $C(\beta,s,\alpha,k,k',d,\Lambda)<\infty$ such that, for every $z\in\Rd$, $r\geq 1$ and $p,q\in \Ahom_k(\Phi_{z,r})$,
\begin{equation}
\label{e.JkequalsJm}
J_{k'}(z,r,p,q) - J_{k}(z,r,p,q) = \O_{s/2}\left( Cr^{-2\beta} \right).
\end{equation}
\end{lemma}
\begin{proof}
Fix $\beta \in(0,\alpha)$. 
By Lemmas~\ref{l.matchingutophi},~\ref{l.matchingenergyk} and the triangle inequality, for $z\in\Rd$, $r\geq 1$ and $p,q\in \Ahom_k(\Phi_{z,r})$,
\begin{equation*}
\left\| \nabla v_{k}(\cdot,z,r,p,q) -\nabla v_{k'}(\cdot,z,r,p,q)  \right\|_{L^2(\Phi_{z,r})} = \O_{s} \left( Cr^{-\beta} \right).
\end{equation*}
The previous estimate and \eqref{e.second-var} give
\begin{align*}
J_{k'}(z,r,p,q) - J_k(z,r,p,q) & = \int_{\Phi_{z,r}} \frac12 \left( \nabla v_k  - \nabla v_{k'} \right)\cdot \a \left( \nabla v_k  - \nabla v_{k'} \right) \\
&= \O_{s/2} \left( Cr^{-2\beta} \right).
\end{align*}
This yields~\eqref{e.JkequalsJm} and completes the proof.
\end{proof}

\begin{lemma}
\label{l.localizeopt}
Assume that $s\in (0,2]$, $\alpha\in\left(0,\frac ds\right)\cap \left(0, \frac d2\right]$ are such that $\Fluc_k(s,\alpha)$ and $\Dual_k(\alpha)$ hold for every $k\in\N$.  Then, for each $k\in\N$ and $\delta > 0$, we have that $\Loc_k(\frac s 2, \delta,2(\alpha\wedge1)-)$ holds.
\end{lemma}
\begin{proof}
The argument is based on the regularity theory of Section~\ref{s.regularity} and the previous lemma. We fix $\beta \in (0,\alpha\wedge 1)$, $k\in\N$ and take $k'\in\N$ with $k'\gg k$ to be selected below. We allow the constants $C$ to depend on $k'$ since it will eventually be chosen to depend on the appropriate quantities. 

\smallskip

We define the local quantity $J^{(\delta)}$, which is done a bit differently here than in the proof of Lemma~\ref{l.localization}. Fix $R\geq 1$, set $T:=R^{1+\delta}$ and define
\begin{equation*} \label{}
\mathcal{D}(B_T):= \left\{ w \in \A(B_T) \,:\, \left\| w \right\|_{\underline{L}^2(B_T)} \leq 2 R^{\delta k} \left\| w \right\|_{\underline{L}^2(B_R)} \right\}.
\end{equation*}
Note that $R\geq \X \vee C$ implies that $\A_k \subseteq \mathcal{D}(B_T)$. For $\delta > 0$, $x\in\Rd$ and $p,q\in \Ahom_k$, we set
 \begin{equation*}
J^{(\delta)}(0,R,p,q) 
:= \sup_{u \in \mathcal{D}(B_T)}  \int_{\Phi_{R}^{(\delta)}}\left(-\frac12 \nabla u \cdot\a  \nabla u - \nabla p\cdot \a\nabla u + \nabla q\cdot \nabla u \right),
\end{equation*}
where $\Phi^{(\delta)}_{R}$ is as defined in~\eqref{e.Phideltadef}. It is clear that
\begin{equation*} \label{}
J^{(\delta)}(0,R,p,q)  \quad \mbox{is $\F(B_{R^{1+\delta}})$-measurable} 
\end{equation*}
and that, for every $p,q\in\Ahom_k(\Phi_R)$, 
\begin{equation*}
0 \leq J^{(\delta)} (0,R,p,q) \leq C \left( \left\| \nabla p \right\|_{L^2(\Phi_R)}^2 +\left\| \nabla q \right\|_{L^2(\Phi_R)}^2  \right)
\end{equation*}
In the case $R\geq \X\vee C$, we deduce from $\A_k \subseteq \mathcal{D}(B_T)$ and easy tail estimates that, for every $p,q\in\Ahom_k(\Phi_R)$,
\begin{equation*} \label{}
J_k(0,R,p,q) \leq J^{(\delta)}(0,R,p,q) + CR^{-100\alpha}. 
\end{equation*}
On the other hand, the regularity theory implies that, provided $R \geq \X$, we have 
\begin{equation*} \label{}
\sup_{w\in \mathcal{D}(B_T)} \inf_{v \in \A_{k'}} \frac{\left\|  \nabla w - \nabla v\right\|_{\underline{L}^2(B_r) } }{\left\|\nabla w \right\|_{\underline{L}^2(B_r) }} \leq CR^{-\delta (k'-k)}. 
\end{equation*}
Hence, for every $p,q\in\Ahom_k(\Phi_R)$,
\begin{equation*} \label{}
J^{(\delta)}(0,R,p,q)  \leq J_{k'}(0,R,p,q) + CR^{-\delta(k'-k)}. 
\end{equation*}
Taking $k'$ large enough that $\delta(k'-k) > 2\alpha$ and applying~Lemma~\ref{l.skypepro}, we obtain, for every $R\geq C$ and $p,q\in\Ahom_k(\Phi_R)$,
\begin{equation*} \label{}
\left| J^{(\delta)}(0,R,p,q) - J_{k}(0,R,p,q) \right| \indc_{\{ \X \leq R \}} \\
\leq \O_{s/2}\left( CR^{-2\beta} \right).
\end{equation*}
The boundedness of $J^{(\delta)}$ and $J$ and the estimate for $\X$ in Proposition~\ref{p.regularity} give, for every $p,q\in\Ahom_k(\Phi_R)$,
\begin{equation*} \label{}
\left| J^{(\delta)}(0,R,p,q) - J_k(0,R,p,q) \right| \indc_{\{ \X \geq R \}}   \\
\leq C \indc_{\{ \X \geq R \}} 
\leq  \O_{s/2}\left( CR^{-2\beta} \right). 
\end{equation*}
This completes the argument. 
\end{proof}

The previous lemma implies the second assertion of Proposition~\ref{p.localization}. The proof of Proposition~\ref{p.localization} is now complete.

\subsection{Improvement of $\Dual_k(s,\alpha)$}
\label{ss.improvedual}

In this subsection we prove Proposition~\ref{p.min}.
In order to improve the duality between gradients and fluxes, we need to demonstrate (i) that the minimum of the map $q \mapsto \E \left[ J_k(z,r,p,q) \right]$ is small, and (ii) that the $q$ achieving the minimum is close to $p$. 

\smallskip

Throughout this subsection, we fix 
$s\in (0,2]$, $\alpha \in \left(0,\frac ds \right)\cap \left( 0,\frac d2\right]$ and $\beta\in (0,\alpha]$ and suppose that 
\begin{equation}
\label{e.dualsass}
\Fluc_k(s,\alpha) 
\ \mbox{and} \ 
\Dual_k(\beta)
\ \mbox{hold.}
\end{equation}
For each $z\in\Rd$ and $r\geq r_0$, we denote by $p\mapsto \bar M_{z,r}p$ the linear mapping satisfying, for every $p \in \Ahom_k$,
\begin{equation} \label{e.defM_zr}
\E\left[ J\left(z,r,p,\bar{M}_{z,r} p \right) \right] = \min_{q \in \Ahom_{k}} \, \E\left[ J\left(z,r,p,q\right) \right].
\end{equation}

\begin{lemma}
\label{l.superminalpha}
There exists $C(k,s,\beta,d,\Lambda)<\infty$ such that, for every $z \in \Rd$, $r \ge r_0$ and $p \in \Ahom_k(\Phi_{z,r})$,
\begin{equation}
\label{e.superminalpha}
J\left(z,r,p,\overline{M}_{z,r} p\right)  = \O_{s/2}\left(C r^{-2\beta}\right)\,.
\end{equation}
\end{lemma}
\begin{proof}
Let $ p \in \Ahom_k(\Phi_{z,r})$. By \eqref{e.J-energy} and homogeneity,
\begin{align*}
J(z,r,p,p)  
& =  \frac 1 2 \int_{\Phi_{z,r}} \nabla p \cdot (\ahom - \a) \nabla v(\cdot,z,r,p,p)  \\
& \leq \|\nabla v(\cdot,z,r,p,p)\|_{L^2(\Phi_{z,r})}  \sup_{w\in \A_k(\Phi_{z,r}) }   \left| \int_{\Phi_{z,r}} \nabla p \cdot \left( \a - \ahom \right) \nabla w \right|,
\end{align*}
and we deduce from \eqref{e.J-energy} and Lemma~\ref{l.fluxmaps2} that for every $p \in \Ahom_k(\Phi_{z,r})$,
\begin{equation}
\label{e.min.p.p}
J(z,r,p,p) = \O_{s/2}(C r^{-2 \beta}).
\end{equation}
We now prove the same estimate for $J(z,r,p,\bar M_{z,r} p)$. Let $\Y_{s\beta}(z)$ be as given by Corollary~\ref{c.convexity}. We may restrict our attention to the event $r \ge \Y_{s\beta}(z)+C$, since otherwise
\begin{equation*}
J(z,r,p, \bar M_{z,r} p) \1_{r \le \Y_{s \beta}(z)+C} \le C \, \1_{r \le \Y_{s\beta}(z)+C} 
\le C \Ll( \frac{\Y_{s\beta}(z)+C}{r} \Rr)^{2\beta} = \O_{s/2}\left( Cr^{-2\beta} \right).
\end{equation*}
Let $M_{z,r}$ be the linear mapping such that for every $p \in \Ahom_k$,
\begin{equation*}
J\left(z,r,p,M_{z,r} p \right)  = \min_{q \in \Ahom_{k}} J(z,r,p,q).
\end{equation*}
This is well-defined in view of the restriction on $r$.
In view of \eqref{e.min.p.p}, we clearly have, for every $p \in \Ahom_k(\Phi_{z,r})$,
$$
J(z,r,p,M_{z,r} p) = \O_{s/2}(C r^{-2 \beta}).
$$
By quadratic response, that is, Lemma~\ref{l.C11} with $p_1 = p_2 = p$, $q_1 = \bar M_{z,r} q$, $\frac{q_1 + q_2} 2 = M_{z,r} p$ and the definition of $M_{z,r}$, in order to conclude the proof, it suffices to show that
\begin{equation}
\label{e.comp.M.barM}
\|\nabla( M_{z,r} p) - \nabla (\bar M_{z,r} p)\|_{L^2(\Phi_{z,r})} = \O_s(C r^{-\beta}).
\end{equation}
For $r \ge \Y_{s\beta}(z)$, we have
$$
J(z,r,0,q) \ge \frac 1 4 \int_{\Phi_{z,r}} \nabla q \cdot \ahom \nabla q,
$$
and, by \eqref{e.polarization}, \eqref{e.bounded-J} and Young's inequality, 
\begin{align*}
J(z,r,p,q) & \geq J(z,r,0,q) - C \|\nabla p\|_{L^2(\Phi_{z,r})} \|\nabla q\|_{L^2(\Phi_{z,r})} \\
& \geq \frac 1 8 \|\nabla q\|_{L^2(\Phi_{z,r})}^2 - C \|\nabla p\|_{L^2(\Phi_{z,r})}^2.
\end{align*}
As a consequence, for each $r \ge \Y_{s\beta}(z)$, 
\begin{equation}
\label{e.Mzr.control}
\|\nabla(M_{z,r} p)\|_{L^2(\Phi_{z,r})} \le C \|\nabla p\|_{L^2(\Phi_{z,r})}.
\end{equation}
The polynomial $M_{z,r} p$ is characterized up to a constant by the property that
\begin{equation}
\label{e.gradient.Mzr}
\nabla_q J\left(z,r,p,M_{z,r} p\right) = 0.
\end{equation}
In order to plug random objects into the expectation of $J$, we denote 
\begin{equation*}
\E[J](z,r,p,q) := \E\left[J(z,r,p,q)\right] \,.
\end{equation*}
By the assumption of $\Fluc_k(s,\al)$, \eqref{e.Mzr.control}, \eqref{e.gradient.Mzr} and Remark~\ref{r.gradient}, we have that, for every $q' \in \Ahom_k(\Phi_{z,r})$, 
\begin{equation}
\label{e.grad.Mzr}
\nabla_q \E[J](z,r,q,M_{z,r} p)(q') = \O_s\left(C r^{-\alpha}\right) \,.
\end{equation}
For $r \ge C$, the polynomial $\overline{M}_{z,r} p$ is the unique (up to a constant) minimizer of the mapping $q \mapsto \E[J](z,r,p,q)$, that is, the unique zero of the linear mapping $q \mapsto \nabla_q \E[J](z,r,p,q)$. By Corollary~\ref{c.convexity}, this linear mapping is invertible with bounded inverse, and therefore \eqref{e.grad.Mzr} implies \eqref{e.comp.M.barM} with $\al$ in place of $\be$. Since $\al \ge \be$, the proof is complete.
\end{proof}

We next improve the previous lemma using a harmonic approximation argument similar to the one used for the proof of Lemma~\ref{l.harmonicapprox}.

\begin{lemma} 
\label{l.minimalE[J]}
There exists $C(\alpha,\beta,s,k,d,\Lambda)<\infty$ such that, for every $z \in \Rd$ and $r \ge r_0$,
\begin{equation} \label{e.minimalE[J]}
\sup_{p\in \Ahom_k(\Phi_{z,r})}  J\left(z,r,p,\overline{M}_{z,r} p \right)   \leq \O_{s/4}\left(C r^{-2(\alpha\wedge 2\beta)}\right)   \,.
\end{equation}
\end{lemma}
\begin{proof}
Denote $v := v(\cdot,z,r,p,\overline{M}_{z,r} p)$. In view of the estimate
\begin{equation*} \label{}
J\left(z,r,p,\overline{M}_{z,r} p \right) \leq C \left\| \nabla v \right\|_{L^2(\Phi_{z,r})}^2,
\end{equation*}
it suffices to prove that 
\begin{equation*} \label{}
\left\| \nabla v \right\|_{L^2(\Phi_{z,r})}^2 = \O_{s/4}\left(C r^{-2(\alpha\wedge 2\beta)}\right) .
\end{equation*}

\smallskip

\emph{Step 1.} We show that 
\begin{equation} 
\label{e.squish}
\int_{\Phi_{z,\sqrt{r^2 - (\sigma r)^2}}} \left|\int_{\Phi_{x,\sigma r}} \nabla v(y)  \right|^2 \, dx = \O_{s/2}\left(C r^{-2\alpha}\right) +  \O_{s/4}\left( C r^{-4\beta} \right).
\end{equation}
According to Lemma~\ref{l.superminalpha},
\begin{equation*} 
\left\| \nabla v \right\|_{L^2(\Phi_{z,r})}^2 \leq C J(z,r,p,\overline{M}_{z,r} p) \leq \O_{s/2}\left( C r^{-2\beta} \right)\,.
\end{equation*}
Therefore Lemma~\ref{l.harmonicapprox} yields the existence of $\mathsf{h} \in \Ahom_k$ such that 
\begin{equation} \label{e.minsetsmall000}
\int_{\Phi_{z,\sqrt{r^2-(\sigma r)^2}}} \left| \int_{\Phi_{x,\sigma r}} (\nabla v(y) - \nabla \mathsf{h} (y)) \, dy  \right|^2 \, dx  \leq \O_{s/4}\left( C r^{-4\beta} \right)\,,
\end{equation}
where we let $\sigma(k,d,\Lambda)$ be as in Lemma~\ref{l.mspoincare2}. Using that $\Fluc_k(s,\alpha)$ holds, we have, for every $p,q' \in \Ahom_k(\Phi_{z,r})$, 
\begin{multline*} 
\nabla_q J(z,r,p,\overline{M}_{z,r} p)(q')  = \int_{\Phi_{z,r}} \ahom \nabla v \cdot \nabla q'  \\ 
= \E\left[\nabla_q J(z,r,p,\overline{M}_{z,r} p)\right](q') + \O_s\left(C r^{-\alpha}\right) = \O_s\left(C r^{-\alpha}\right)
\end{multline*}
Using Lemma~\ref{l.CKish} we get, for all $q' \in \Ahom_k(\Phi_{z,r})$, the estimate 
\begin{equation*} 
\left| \int_{\Phi_{z,r}} \ahom (\nabla \mathsf{h} -\nabla v) \cdot \nabla q' \right|^2 \leq C \int_{\Phi_{z,\sqrt{r^2-(\sigma r)^2}}} \left| \int_{\Phi_{x,\sigma r}} (\nabla v(y) - \nabla \mathsf{h} (y)) \, dy  \right|^2 \, dx\,,
\end{equation*}
which implies that 
\begin{align*} 
  \int_{\Phi_{z,r}} \ahom \nabla \mathsf{h}  \cdot \nabla q'  & =  \int_{\Phi_{z,r}} \ahom \nabla v \cdot \nabla q'  + \int_{\Phi_{z,r}} \ahom (\nabla \mathsf{h} -\nabla v) \cdot \nabla q'
 \\ & \leq \O_s\left(C r^{-\alpha}\right) + \O_{s/2}\left( C r^{-2\beta} \right)\,.
\end{align*}
Taking $q' = \mathsf{h} /\|\nabla \mathsf{h} \|_{L^2(\Phi_{z,r})}$, this yields
\begin{equation*} 
\left\|\nabla \mathsf{h}  \right\|_{L^2(\Phi_{z,r})}^2 \leq \O_{s/2}\left(C r^{-2\alpha}\right) +  \O_{s/4}\left( C r^{-4\beta} \right)\,.
\end{equation*}
Returning to~\eqref{e.minsetsmall000} and using the triangle inequality, we obtain~\eqref{e.squish}.

\smallskip

\emph{Step 2.}  
Owing to Lemma~\ref{l.mspoincare2} we have
\begin{equation}  \label{e.msp2appl001}
\left\| v \right\|_{L^2(\Psi_{z,r})}^2 \indc_{\{ \X(z)\leq r \}}  \leq C \fint_{B_{r/\theta} } \left| \int_{\Phi_{y,\sigma r}} v(z) \, dz  \right|^2 \, dy \,.
\end{equation}
The Caccioppoli and Poincar\'e's inequalities (subtracting a constant from $v$, if necessary) yield 
\begin{align} \label{e.msp2appl002}
\left\| \nabla v \right\|_{L^2(\Phi_{z,r})}^2 \indc_{\{ \X(z)\leq r \}} & \leq C \fint_{B_{r/\theta} } \left| \int_{\Phi_{y,\sigma r}} \nabla v(z) \, dz  \right|^2 \, dy 
\\ \notag& \leq \O_{s/2}\left(C r^{-2\alpha}\right) +  \O_{s/4}\left( C r^{-4\beta} \right)
\end{align}
On the other hand,
\begin{equation*}
\left\| \nabla v \right\|_{L^2(\Phi_{z,r})}^2 \indc_{\{ \X(z)\geq r \}} 
\leq C \indc_{\{ \X(z)\geq r \}} 
\leq C \left( \frac{\X(z)}{r} \right)^{2\alpha} 
= \O_{\frac{d}{2\alpha}} \left( C r^{-2\alpha} \right).
\end{equation*}
Therefore, since $\frac{d}{\alpha}>s$,   
\begin{equation*} 
 \left\| \nabla v \right\|_{L^2(\Phi_{z,r})}^2  \leq \O_{s/2}\left(C r^{-2\alpha}\right) +  \O_{s/4}\left( C r^{-4\beta} \right)  \,.
\end{equation*}
This completes the proof. 
\end{proof}

We now complete the proof of Proposition~\ref{p.min}.

\begin{proof}[{Proof of Proposition~\ref{p.min}}]
We first consider the case $k=1$. 

\smallskip

\emph{Step 1.} The proof of the first statement of the proposition. Observe that Lemma~\ref{l.identifytheta1} implies, for each $p\in B_1$ and $r\geq r_0$,
\begin{equation*}
\left| p - \overline{M}_{z,r}p \right| \leq Cr^{-2\beta}. 
\end{equation*}
The uniform convexity of $q\mapsto   \E \left[ J_1(z,r,p,q) \right]$ and the fact that this map has its minimum at $q=\overline{M}_{z,r}p$, the previous line and~\eqref{e.minimalE[J]} imply that, for $p\in B_1$,
\begin{equation}
\label{e.affinecluptcha}
 \E \left[ J_1(z,r,p,p) \right] 
 \leq \E \left[ J_1(z,r,p,\overline{M}_{z,r}p) \right] + C \left| p - \overline{M}_{z,r}p \right|^2  \leq Cr^{-2(\alpha\wedge 2\beta)}. 
\end{equation}
By the first variation~\eqref{e.first-var}, we find that, for every $e\in \partial B_1$ and $w\in\A_1$,
\begin{align*}
\left| e\cdot \int_{\Phi_{z,r}} \left( \a(x) - \ahom \right)\cdot \nabla w(x)\,dx \right| 
& = \left| \int_{\Phi_{z,r}} \nabla v(x,z,r,e,e) \cdot \a \nabla w(x)\,dx \right|   \\
& \leq C \left\| \nabla w \right\|_{L^2(\Phi_{z,r})} \left\| \nabla v(\cdot,z,r,e,e) \right\|_{L^2(\Phi_{z,r})}  \\
& \leq C\left\| \nabla w \right\|_{L^2(\Phi_{z,r})} \left(J(z,r,e,e) \right)^{\frac12}.
\end{align*}
Thus, if we take $\{ e_1,\ldots,e_d\}$ to be the standard basis for $\Rd$, then 
\begin{equation*}
\sup_{w\in \A_1(\Phi_{z,r})} \left|  \int_{\Phi_{z,r}} \left( \a(x) - \ahom \right)\cdot \nabla w(x)\,dx \right| \leq C \sup_{i\in\{1,\ldots,d\}} \left(J(z,r,e_i,e_i) \right)^{\frac12}.
\end{equation*}
Taking expectations and applying~\eqref{e.affinecluptcha} gives
\begin{align*}
\E \left[ \sup_{w\in \A_1(\Phi_{z,r})} \left|  \int_{\Phi_{z,r}} \left( \a(x) - \ahom \right)\cdot \nabla w(x)\,dx \right| \right] 
& \leq C \E \left[  \sup_{i\in\{1,\ldots,d\}} \left(J(z,r,e_i,e_i) \right)^{\frac12} \right]   \\
& \leq C \sum_{i=1}^d  \E \left[ \left(J(z,r,e_i,e_i) \right)^{\frac12} \right]   \\
& \leq Cr^{-(\alpha\wedge 2\beta)}. 
\end{align*}
This completes the proof of the first statement of Proposition~\ref{p.min}. 

\smallskip

\emph{Step 2.} The proof of the second statement. Fix $k\in\N$. The reasoning is actually almost the same as in Step~1, the only difference being that Lemma~\ref{l.matchingenergyk} gives us 
$\left| p - \overline{M}_{z,r}p \right|  \leq Cr^{-\gamma}$ for every $\gamma \in (0,\frac{2\beta}{\beta+1} \wedge 1)$ rather than the stronger bound of $\left| p - \overline{M}_{z,r}p \right|  \leq Cr^{-2\beta}$. Percolating this change through the argument leads to the bound, for each $\gamma\in (0,\frac{2\beta}{\beta+1} \wedge 1)$,
\begin{equation*}
E \left[ \sup_{w\in \A_k(\Phi_{z,r})} \left|  \int_{\Phi_{z,r}} \left( \a(x) - \ahom \right)\cdot \nabla w(x)\,dx \right| \right]  \leq Cr^{-\alpha\wedge \gamma},
\end{equation*}
and the proof is complete.
\end{proof}

\section{Optimal quantitative estimates on the correctors}
\label{s.finalcorrectors}

In this section, we complete the proof of the main results stated in the introduction, Theorems~\ref{t.firstcorrectors} and~\ref{t.main1}, which are consequences of Theorem~\ref{t.additivitybelowd} and the theory developed in Sections~\ref{s.additivity} and~\ref{s.localization}.

\begin{proof}[{Proof of Theorem~\ref{t.main1}}]
Fix $s<1$. Theorem~\ref{t.additivitybelowd} implies that $\Add_1(s,d)$, $\Fluc_1(2s,\tfrac d2)$, $\Dual_1\left(\tfrac d2\right)$ and hold. Lemma~\ref{l.identifytheta1} then gives,  for $z\in\Rd$, $r\geq r_0$ and $p,q\in B_1$, 
\begin{equation}
\label{e.scalingenergy1111}
\left| 
\E\left[ J_1(z,r,p,q) \right] - \frac12 (q-p) \cdot \ahom (q-p) \right|  \leq Cr^{-d}
\end{equation}
and
\begin{equation}
\label{e.yesLswhipped}
\left| L_{1,z,r}^*p - p \right| + \left| L_{1,z,r}q - q \right| \leq Cr^{-d}.
\end{equation}
The latter yields via~\eqref{e.I.as.J} and~\eqref{e.bounded-J} that
\begin{equation}
\label{e.JvsId}
\left| J(z,r,p,q) - I(z,r,p,q) \right| \leq Cr^{-d}. 
\end{equation}
The first statement (i) of Theorem~\ref{t.main1} is a consequence of  $\Add_1(s,d)$ and~\eqref{e.JvsId}. 
 The second statement is~\eqref{e.scalingenergy1111}, the third statement is simply $\Fluc_1(2s,\tfrac d2)$, and the fourth statement is a consequence of the localization statements of Theorem~\ref{t.additivitybelowd}.
\end{proof}

\begin{proof}[{Proof of Theorem~\ref{t.firstcorrectors}}]
Fix $s<2$. The estimates~\eqref{e.gradient},~\eqref{e.flux} and~\eqref{e.energydensity} are immediate consequences of the triangle inequality, $\Fluc_1(s,\tfrac d2)$, $\Dual_1\left(\tfrac d2\right)$, the estimate~\eqref{e.yesLswhipped} and Lemmas~\ref{l.fluxmaps},~\ref{l.whatflucgives} and~\ref{l.firstordercorrect}.
Moreover, by Lemma~\ref{l.firstordercorrect}, there exist $\ep(d,\Lambda)\in \left(0,\frac12\right)$ and $C(d,\Lambda)<\infty$  such that
\begin{equation} 
\label{e.correctorboundpf}
\sup_{r\geq1} \sup_{\xi\in B_1} \left\| \nabla \phi^{(1)}(\cdot,\xi) \right\|_{\underline{L}^2(B_r)}  = \O_{2+\ep}(C)
\end{equation}
and, taking $s = 2 +\ep$, $\ep \in \left(0,\frac12\right)$ for $d>2$ and $s = 2$ when $d=2$, we have, for every $s'<s$, $\xi\in B_1$, $z\in\Rd$ and $r\geq 1$,
\begin{equation}
\label{e.spatavg1pf}
\left| \int_{\Phi_{z,r}}\nabla \phi^{(1)}(x,\xi)  \,dx \right| = \O_{s'} \left(Cr^{-\frac{d}s }\right).
\end{equation}
Note that $\frac ds>1$ in dimensions $d>2$. 

\smallskip

We have left to prove~\eqref{e.phiosc}. The argument is an application of~\eqref{e.correctorboundpf},~\eqref{e.spatavg1pf} and the multiscale Poincar\'e inequality (here in the form of Lemma~\ref{l.mspoincare}). Define, for each $y\in\Rd$ and $t>0$, 
\begin{equation*}
w(x,t):= \int_{\Rd} \Phi(x-y,t) \phi_e(y)\,dy\,.
\end{equation*}
Since 
\begin{equation*} 
\fint_{B_R} \left| \phi_e (x) - (\phi_e)_{B_R}\right|^2 \, dx = \inf_{a \in \R} \fint_{B_R} \left| \phi_e (x) - a \right|^2 \, dx 
\end{equation*}
we may normalize $\phi_e$ so that $w(0,R^2) = 0$. By~\eqref{e.spatavg1pf} and Lemma~\ref{l.sum-O} we get 
\begin{equation} \label{e.wincontrol}
\left|w(x,R^2) \right| = \O_{s'}\left( C|x|R^{-\frac {d}{s}} \right)
\end{equation}
Applying Lemma~\ref{l.mspoincare}, we obtain, for $\Psi_R$ defined as in~\eqref{e.PsiR},
\begin{align*}
\int_{\Psi_R} \left| \phi_e(x) \right|^2\,dx 
& \leq C \int_{\Psi_R} \left| w\left(x, R^2 \right) \right|^2\,dx + C \int_{0}^{R^2} \int_{\Psi_R} \left| \nabla w(y,t) \right|^2\,dy\,dt.
\end{align*}
To bound the first term on the right, we use~\eqref{e.wincontrol} and Lemma~\ref{l.sum-O} to obtain
\begin{align*}
 \int_{\Psi_R} \left| w\left(x, R^2 \right) \right|^2\,dx \leq  \O_{s'/2}\left( CR^{2-\frac {2d}{s}}   \right) \leq \O_{s'/2}(C)\,.
\end{align*}
We split the second term: 
\begin{multline*}
 \int_{0}^{R^2} \int_{\Psi_R} \left| \nabla w(y,t) \right|^2\,dy\,dt \\
 =  \int_{0}^{1} \int_{\Psi_R} \left| \nabla w(y,t) \right|^2\,dy\,dt
+  \int_{1}^{R^2} \int_{\Psi_R} \left| \nabla w(y,t) \right|^2\,dy\,dt.
\end{multline*}
For the first piece, we argue as in Step~1 of the proof of Lemma~\ref{l.mspoincare2}, using~\eqref{e.correctorboundpf}, to find that
\begin{equation*} 
\int_{0}^{1} \int_{\Psi_R} \left| \nabla w(y,t) \right|^2\,dy\,dt \leq C \int_{\Psi_R}  \left| \nabla \phi_e(x) \right|^2\,dx \leq \O_{1+\ep/2}(C)\,.
\end{equation*}
For the second piece, we use~\eqref{e.spatavg1pf} and integrate:
\begin{align*}
\int_{1}^{R^2} \int_{\Psi_R} \left| \nabla w(y,t) \right|^2\,dy\,dt
& \leq \int_{1}^{R^2} 
\O_{s'/2}\left(C t^{-\frac {d}{s}} \right) \, dt 
.
\end{align*}
The previous estimates give us, for every $s'<s$,
\begin{equation*}
\int_{\Psi_R} \left| \phi_e(x)  \right|^2\,dx  \leq 
\O_{s'/2} \left(  C \Rr)+ \int_1^{R^2} \O_{s'/2} \Ll( C   t^{-\frac {d}{s} }    \right)\,dt.
\end{equation*}
In the case $d>2$, we have $\frac ds>1$, 
and therefore by Lemma~\ref{l.sum-O}(i), we get that for every $s' < s =2+\ep$,
\begin{equation*}
\int_{\Psi_R} \left|  \phi_e(x) \right|^2\,dx  = \O_{s'/2}(C).
\end{equation*}
From this we deduce, after shrinking $\ep$ slightly,
\begin{equation*}
\left\|  \phi_e \right\|_{\underline{L}^2(B_R)} = \O_{2+\ep}(C)\,.
\end{equation*}
In dimension $d=2$, we use Lemma~\ref{l.sum-O}(ii) to obtain
\begin{align*}
\int_1^{R^2} \O_{s'/2} \Ll( C   t^{-1 }    \right)\,dt & \le \sum_{k = 0}^{\lfloor \log_2 R^2 \rfloor} \int_{2^k}^{2^{k+1}} \O_{s'/2} \Ll( C t^{-1}  \Rr) \, dt \le \O_{s'/2}(C \log R),
\end{align*}
 and thus, for every $s'<2$, 
\begin{equation*}
\int_{\Psi_R} \left| \phi_e(x) \right|^2\,dx  \leq \O_{s'/2} \left( C\log R \right)\,.
\end{equation*}
This gives for $d=2$ the bound, for every $s'<2$,
\begin{equation*}
\left\| \phi_e \right\|_{\underline{L}^2(B_R)} \leq \O_{s'}\left(C \log^{\frac12} R\right)\,.
\end{equation*}
This completes the proof of~\eqref{e.phiosc} and thus of the theorem. 
\end{proof}

\chapter{Scaling limits}
\label{part.two}

\section{Informal heuristics and statement of main result}
\label{s.heuristics}

The main purpose of this second part of the paper is to show that the first-order correctors converge to a non-Markovian variant of the Gaussian free field in the large-scale limit. 

\smallskip

We start by reviewing the heuristic derivation\footnote{This heuristic derivation was obtained by SA, Yu Gu and JCM. It was the object of a talk given in Banff in July 2015 and reproduced during the Oberwolfach seminar on stochastic homogenization shortly afterwards. The talk can be watched at \url{http://goo.gl/5bgfpR}.} of this result presented in \cite{GM}, putting more emphasis on the role played by the energy quantity $J$ introduced in \cite{AS,AM,AKM} and in this paper, and how it can be seen as a ``coarsening" of the coefficient field. To begin with, we recall the notions of white noise and Gaussian free field (GFF). Let $\msf Q$ be a symmetric non-negative definite $d$-by-$d$ matrix. We say that the random $d$-dimensional distribution\footnote{We only use the word ``distribution'' to refer to Schwartz distributions, and call the probability measure associated with a random variable its \emph{law}.} $W = (W_1,\ldots,W_d)$ is a \emph{vector white noise} with covariance matrix $\msf Q$ if for every $f = (f_1, \ldots, f_d) \in C^\infty_c(\Rd ; \Rd)$, the random variable
\begin{equation*}  
W(f) := W_1(f_1) + \cdots + W_d(f_d)
\end{equation*}
is a centered Gaussian with variance $\int_\Rd f \cdot \msf Q f$. The set of admissible test functions can be extended to $f\in L^2(\R^d; \R^d)$ by density. Given a vector white noise $W$ and a positive-definite symmetric matrix $\ahom$, we define the \emph{gradient Gaussian free field}, or gradient GFF for short, as the random $d$-dimensional distribution $\nabla \mathbf{\Psi}$ solving the equation
\begin{equation}  \label{e.GGFF}
- \nabla \cdot \ahom \nabla \mathbf{\Psi} = \nabla \cdot W.
\end{equation}
In other words, $\nabla \mathbf{\Psi}$ is the potential part in the Helmholtz-Hodge decomposition $W = -\ahom \nabla \mathbf{\Psi} + \mathbf{g}$, and $\mathbf{g}$ is the solenoidal (divergence-free) part. We interpret this definition by duality: for every $F \in C^\infty_c(\Rd; \Rd)$, we set
\begin{equation*}  
(\nabla \mathbf{\Psi})(F) =  W\Ll(\nabla (-\nabla \cdot \ahom \nabla)^{-1} (\nabla \cdot F)\Rr).
\end{equation*}
The function $\nabla (-\nabla \cdot \ahom \nabla)^{-1} (\nabla \cdot F)$ is the potential part in the Helmholtz-Hodge decomposition of $F$, and it belongs to $L^2(\Rd;\Rd)$. Our notion of Gaussian free field coincides with the standard one (see, e.g.,~\cite{Shef}) only when $\ahom$ and $\msf Q$ are proportional, and is otherwise a variant with the same scale invariance but which does not satisfy the spatial Markov property (see~\cite{GM2}). 

\smallskip

We say that the matrix-valued distribution $W = [W_{ij}]_{1 \le i,j \le d}$ is a \emph{matrix white noise} if the vector of its entries is a vector white noise. If such $W = [W_{ij}]_{1 \le i,j \le d}$ satisfies the additional symmetry constraint 
\begin{equation}  \label{e.sym.W}
\mbox{for every $i,j \in \{1,\ldots, d\}$}, \ \ W_{ij} = W_{ji},
\end{equation}
then we define the associated quadratic form $p \mapsto W(\cdot,p)$ via 
\begin{equation*}  
\mbox{for every $p \in \Rd$ and $f \in C^\infty_c(\R^d;\R)$}, \ \ W(f,p) := \tfrac 1 2 \, p\cdot [W_{ij}(f)] p,
\end{equation*}
and call $p \mapsto W(\cdot,p)$ a \emph{quadratic form white noise}. Note that the correspondence between the quadratic form $p \mapsto W(\cdot, p)$ and the matrix $[W_{ij}]_{ij}$ subject to the symmetry constraint \eqref{e.sym.W} is bijective, which justifies that we denote both objects by $W$. 

\smallskip

\smallskip

The coefficients of an elliptic operator define a correspondence between gradients of solutions and their fluxes. The matrix $\ahom$ describes this correspondence in the homogenized limit. In Part~\ref{part.one} we obtained optimal estimates on the error of this correspondence. The goal of this Part~\ref{part.two} is to obtain the next-order correction. Naturally, we rely heavily on both the results and ideas from~Part~\ref{part.one}. 

\smallskip

By Theorem~\ref{t.main1}, there exists a constant $C(d,\Lambda) < \infty$ such that for every $z \in \Rd$, $r \ge 1$ and $p,q \in B_1$,
\begin{equation*}  
\Ll| \E[J(z,r,p,q)] - \frac 1 2 (p-q) \cdot \ahom (p-q) \Rr| \le C r^{-d} .
\end{equation*}
Moreover, the quantity $J$ is additive, and its dependence on the coefficients local, up to errors we can neglect. This suggests that $J$ satisfies a form of central limit theorem. More precisely, we expect that
\begin{equation}
\label{e.J.clt}
J(z,r,p,q) \simeq \frac 1 2 (q-p) \cdot \ahom (q-p) + \int_{\Phi_{z,r}} \W(\cdot,p,q),
\end{equation}
where $(p,q) \mapsto \W(\cdot,p,q)$ is a quadratic form white noise. Note that by the scaling properties of white noise, the last term in \eqref{e.J.clt} has typical size of order $r^{-\frac d 2}$, and \eqref{e.J.clt} should be interpreted up to an error of lower order. 

\smallskip

By~\eqref{e.J.clt} and \eqref{e.J-energy}, the spatial average of $\nabla v(\cdot,z,r,p,q)$ is close to $(q-p)$, and its flux close to $\ahom(q-p)$, up to an error of order $r^{-\frac d 2}$. We now describe the next-order correction. By~\eqref{e.J.clt} and~\eqref{e.J-energy}, we have
\begin{align*}  
\int_{\Phi_{z,r}} p \cdot \Ll( \ahom - \a \Rr) \nabla v(\cdot, z,r,0,q) & = \frac 1 2 \Ll( J(z,r,p,p+q) - J(z,r,p,p-q) \Rr) \\
& \simeq \frac 1 2 \int_{\Phi_{z,r}} \Ll(\W(\cdot,p,p+q) - \W(\cdot, p, p-q)\Rr).
\end{align*}
Defining $\td \b_r(z)$ to be the symmetric matrix such that for every $p, q \in \Rd$,
\begin{equation*}  
p \cdot \td \b_r(z) q = \frac 1 2 \int_{\Phi_{z,r}} \Ll(\W(\cdot,p,p-q) - \W(\cdot, p, p+q)\Rr),
\end{equation*}
we obtain
\begin{equation*}  
\int_{\Phi_{z,r}} \a \nabla v(\cdot,z,r,0,q) \simeq \Ll(\ahom + \td \b_r(z)\Rr)  \int_{\Phi_{z,r}} \nabla v(\cdot,z,r,0,q).
\end{equation*}
Thus the matrix $\ahom + \td \b_r(z)$ is the ``coarsened" coefficient field, representing the sought-after next-order correspondence between spatial averages of gradients and fluxes of solutions at point $z$ and scale $r$. Variants of~$\td \b_r$ and estimates of its size have already played a central role in the quantitative theory developed in~\cite{AS,AM,AKM} and, of course, in the first part of this paper.

\smallskip

As explained in more details in \cite{GM}, this suggests that the spatial 
average of the corrector $\phi_{e,r}(z) := \int_{\Phi_{z,r}} \phi_e$ satisfies the approximate equation
\begin{equation}
\label{e.corrector.approx.pre}
-\nabla \cdot (\ahom + \td \b_r)(e+ \nabla \phi_{e,r} )\simeq 0.
\end{equation}
The term $\td \b_r \nabla \phi_{e,r}$ being of lower order, this simplifies to
\begin{equation}
\label{e.corrector.approx}
-\nabla \cdot \ahom \nabla \phi_{e,r} \simeq \nabla \cdot \Ll( \td \b_r e \Rr) ,
\end{equation}
which is the equation defining a gradient GFF, up to a convolution with the heat kernel. 
Similarly, if $f$ varies slowly on a scale much larger than $r$, $u$ solves $-\nabla \cdot \a \nabla u = f$, and $\bar u$ solves the homogeneous equation $-\nabla \cdot \ahom \nabla \bar u = f$, then we expect the spatial average of the difference $u_r(z) := \int_{\Phi_{z,r}} (u-\bar u)$ to satisfy
\begin{equation}
\label{e.Poisson.approx}
-\nabla \cdot \ahom \nabla u_{r} \simeq \nabla \cdot\Ll(\td \b_r \bar u\Rr). 
\end{equation}

\smallskip

Versions of~\eqref{e.J.clt} have been proved in \cite{N,rossignol,biskup,N2,GN}, with the quantity $J$ replaced by spatial averages of the energy density of the correctors and approximations to the corrector. A version of \eqref{e.corrector.approx} was proved in \cite{MoO,MN}, while a version of \eqref{e.Poisson.approx} with $u_r(z)$ replaced by $\int_{\Phi_{z,r}}(u - \E[u])$ was proved in \cite{GM}. The asymptotic identity in law \eqref{e.Poisson.approx} was shown to hold jointly over possible right-hand sides $f$ of the equation, see \cite[Remark~2.3]{GM}. Similarly, \eqref{e.corrector.approx} was shown to hold jointly over $e$, see \cite[Remark~1.4]{MN}. In these works, the space is the discrete lattice $\Z^d$, $d\ge 3$, and a key assumption is that the probability measure has an underlying product structure which makes available tools such as concentration inequalities, the Chatterjee-Stein \cite{chat1,chat2} method of normal approximation and the Helffer-Sj\"ostrand representation of correlations \cite{helsjo,sjo,NS2}. Each of these papers makes essential use of, and refines, the optimal quantitative estimates first proved in~\cite{GO1,GO2,GNO}.

\smallskip 

The heuristic argument recalled above is interesting for several reasons. First, it provides us with a very intuitive understanding of the results proved in \cite{MoO,MN,GM} by a less transparent method. Second, it opens the possibility to give a more direct proof of these results and without having to assume the product structure on the probability space. The goal of this second part of the paper is to give such a proof.

\smallskip

Here is the main result of Part~\ref{part.two}. See Figure~\ref{fig} below for an illustration of the result. 

\begin{theorem}
\label{t.main}
There exists a quadratic form white noise $(p,q) \mapsto \W(\cdot,p,q)$ such that for every $z, p,q \in \Rd$, we have $\W(\cdot,p,p) = 0$ and
\begin{equation}
\label{e.limit.J}
r^{\frac d 2} \Ll(J\Ll(r z,r,p,q\Rr) - \frac 1 2 (q-p) \cdot \ahom(q-p) \Rr)\xrightarrow[r \to \infty]{\mathrm{(law)}} \int_{\Phi_{z,1}} \W(\cdot,p,q).
\end{equation}
For every $e \in \Rd$, letting $\V(\cdot,e)$ be the vector white noise such that, for every $p \in \Rd$, 
\begin{equation}
\label{e.def.V}
p \cdot \V(\cdot,e) = \frac 1 2 \Ll[ \W(\cdot,p,p-e) - \W(\cdot,p,p+e) \Rr] 
\end{equation}
and $\nabla \mathbf{\Psi}_e$ be the gradient GFF such that
\begin{equation}
\label{e.def.Psi}
-\nabla \cdot \ahom \nabla \mathbf{\Psi}_e = \nabla \cdot (\V(\cdot, e)),
\end{equation}
we have
\begin{equation}
\label{e.limit.phi}
r^{\frac d 2} \,  (\nabla \phi_e) \Ll( r \  \cdot \,\Rr)  \xrightarrow[r \to \infty]{\mathrm{(law)}} \nabla \mathbf{\Psi}_e,
\end{equation}
with respect to the topology of $\mathcal C^{-\frac d 2 -}_\mathrm{loc}$. Moreover, the convergences in law in \eqref{e.limit.J} and \eqref{e.limit.phi} hold jointly over $z,p,q,e \in \Rd$.
\end{theorem}

We next explain some concepts and definitions invoked in the statement above. Here and throughout the rest of this paper, we say that the convergence
\begin{equation}  \label{e.convX}
X_r(p) \xrightarrow[r \to \infty]{\mathrm{(law)}} X(p),
\end{equation}
holds jointly over $p$ if for every $p_1, \ldots, p_n$, we have
\begin{equation*}  
\Ll(X_r(p_1), \ldots, X_r(p_n) \Rr) \xrightarrow[r \to \infty]{\mathrm{(law)}} \Ll(X(p_1), \ldots, X(p_n) \Rr).
\end{equation*}
When we also have, for every fixed $q$, that
\begin{equation}  \label{e.convY}
Y_r(q) \xrightarrow[r \to \infty]{\mathrm{(law)}} Y(q),
\end{equation}
we say that the convergences in law in \eqref{e.convX} and \eqref{e.convY} hold jointly over $p,q$ if for every $p_1, \ldots, p_n, q_1,\ldots,q_n$, we have
\begin{multline*}  
\Ll(X_r(p_1), \ldots, X_r(p_n), Y_r(q_1), \ldots,Y_r(q_n) \Rr) \\
\xrightarrow[r \to \infty]{\mathrm{(law)}} \Ll(X(p_1), \ldots, X(p_n) , Y(q_1), \ldots,Y(q_n)\Rr).
\end{multline*}
We let~$\mcl C^\al_\mathrm{loc}$ be the local Besov space with regularity exponent $\al$ and integrability exponents $\infty,\infty$. We say that the convergence \eqref{e.limit.phi} holds for the topology of $\mcl C^{-\frac d 2 -}_\mathrm{loc}$ if, for every~$\al < -\frac d 2$, it holds with respect to the topology of $\mcl C^{\al}_\mathrm{loc}$.

\smallskip

\begin{figure}
\begin{center}
\includegraphics[width=10cm, height=7cm]{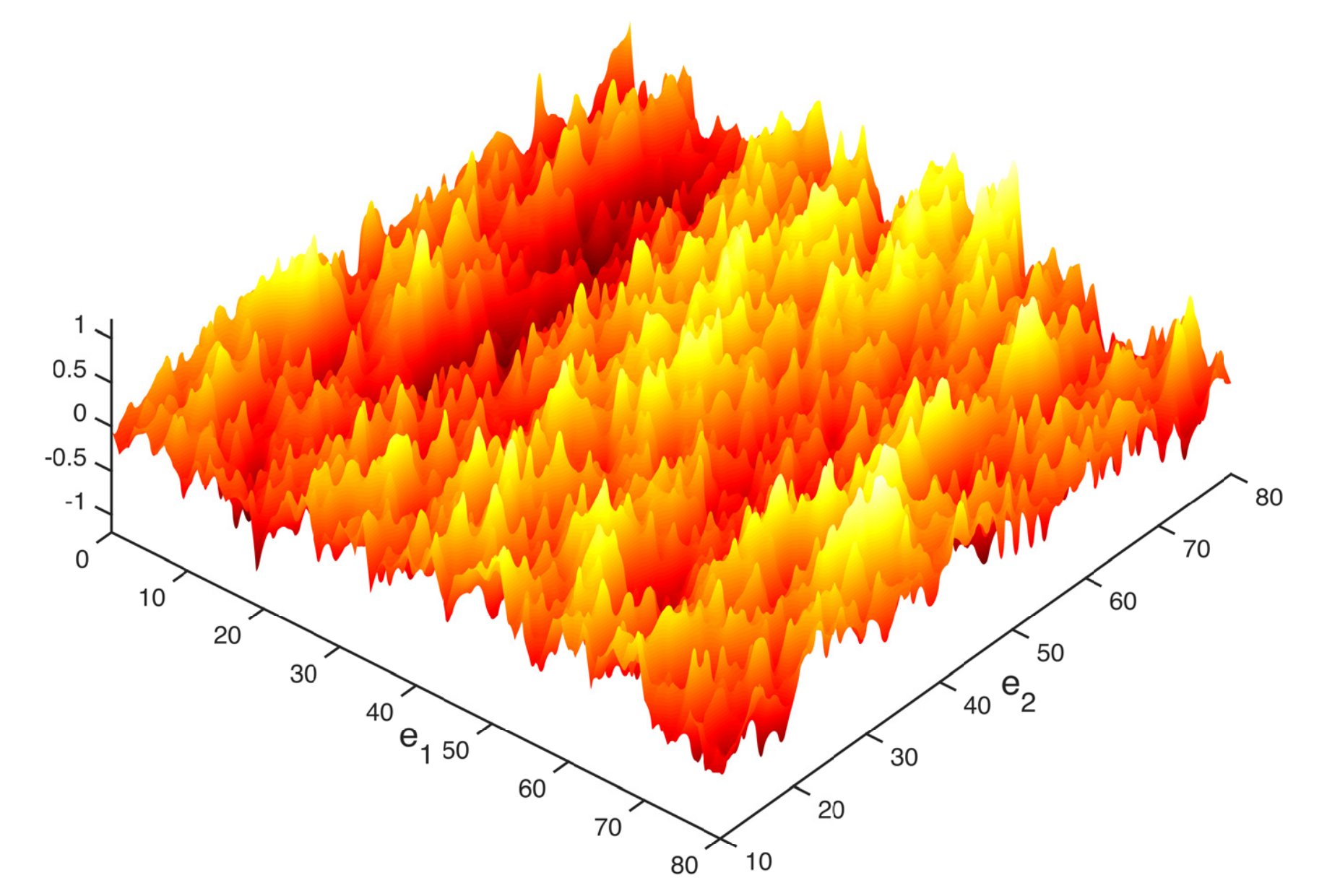}
\end{center}
\begin{center}
\includegraphics[width=10cm, height=7cm]{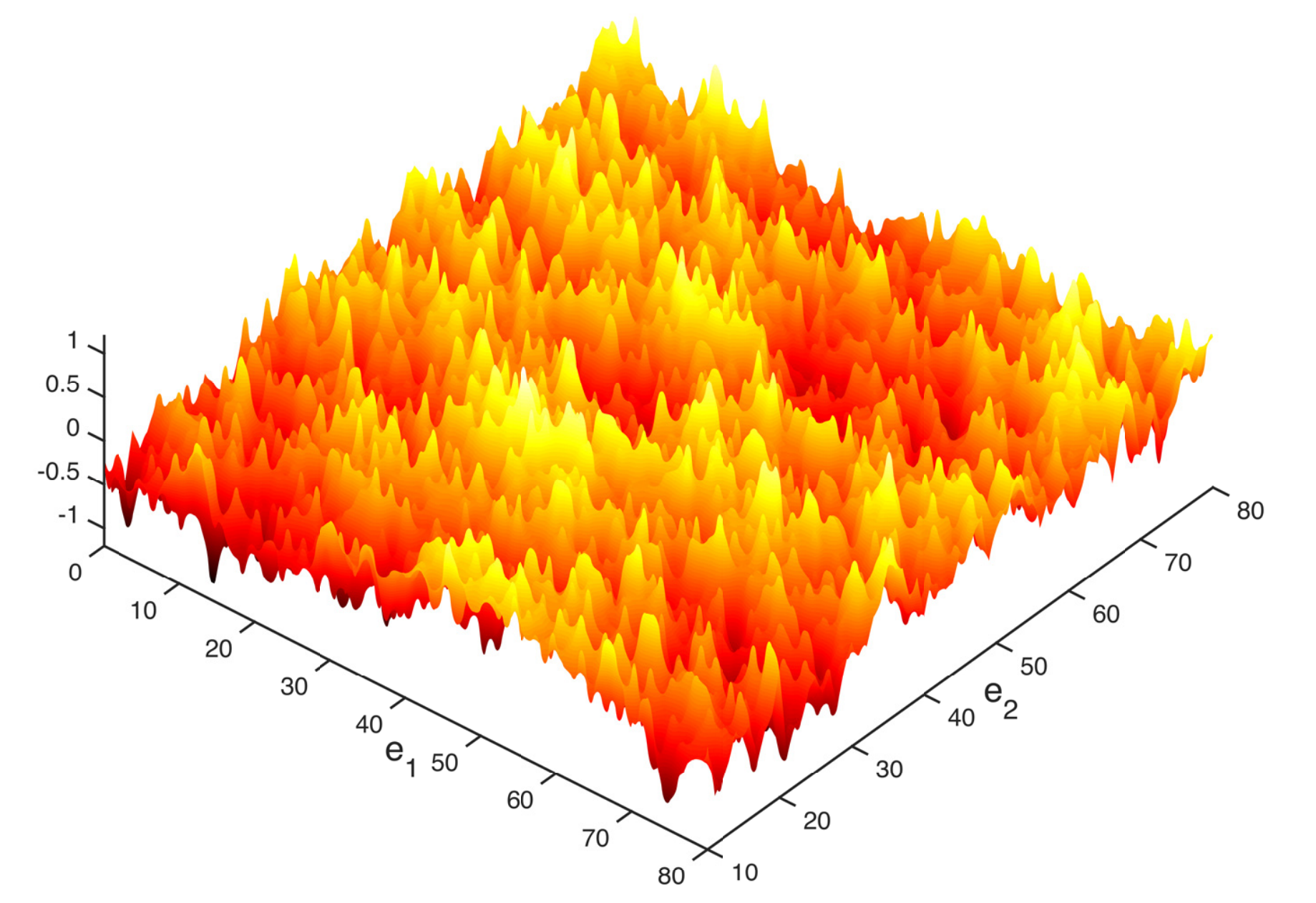}
\end{center}
\caption{{\small Graphs of the correctors $\phi_{e_1}$ and $\phi_{e_2}$ on the top and bottom, respectively, for a random checkerboard model in $d=2$ (with the same realization). The coefficient matrix is diagonal with independent entries equidistributed between $1$ and $10$. Notice that the mountain ranges for $\phi_e$ seem to line up in the orthogonal direction to~$e$. The image is courtesy of Antti Hannukainen (Aalto University).}}
\label{fig}
\end{figure}

In every dimension, the distribution $\nabla \mathbf{\Psi}_e$ can be realized as the gradient of a distribution $\mathbf{\Psi}_e$, as the notation suggests. However, only in dimension $d \ge 3$ can $\mathbf{\Psi}_e$ be realized as a stationary field. In dimension $2$, the field $\mathbf{\Psi}_e$ is only canonically defined up to an additive constant. In every dimension, the field $\mathbf{\Psi}_e$ can be recovered, up to an additive constant, from the knowledge of $\nabla \mathbf{\Psi}_e$, since every compactly supported mean-zero test function is the divergence of a compactly supported vector field (see Lemma~\ref{l.mean-zero}). 

\smallskip

Results partially overlapping with the argument presented here have recently appeared in~\cite{DGO}. There, the randomness of the environment is assumed to have a product structure, so that the Chatterjee-Stein method of normal approximation becomes available. A version of the statement of joint convergence in law of \eqref{e.J.clt}-\eqref{e.corrector.approx}-\eqref{e.Poisson.approx} simultaneously is then obtained. As far as we understand, the notion of ``path-wise theory'' introduced in~\cite{DGO} refers to this joint convergence. Note that an analogue of~$\td \b_r(z)$ is called ``the homogenization commutator" in~\cite{DGO}.

\smallskip

The proof of Theorem~\ref{t.main} is based on the results in the first part of this paper as well as refinements to the arguments there. Throughout, we keep the notation mostly consistent with that of Part~\ref{part.one}. The proof of Theorem~\ref{t.main} is split into two main steps. In the next section, we prove the scaling limit of the energy quantity~$J$, that is,~\eqref{e.limit.J}. The scaling limit of the correctors is then obtained in Section~\ref{s.scalingcorrectors}.

\section{Scaling limit of \texorpdfstring{$J$}{J}}
\label{s.scalingenergies}

According to~Theorem~\ref{t.main1}, for each $s < 2$, there exists $C(s,d,\Lambda) < \infty$ such that, for every $z \in \Rd$, $r \ge 1$ and $p,q \in B_1$,
\begin{equation}
\label{e.bounded.fluc}
J(z,r,p,q) = \frac 1 2 (q-p) \cdot \ahom (q-p) + \O_s\Ll(C r^{-\frac d 2} \Rr). 
\end{equation}
The goal of this section is to refine the estimate \eqref{e.bounded.fluc} and prove a central limit theorem for $J$. 

\smallskip

 We recall first that the expectation of $J(z,r,p,q)$ was identified up to a much smaller error. More precisely, by Theorem~\ref{t.main1}, there exists a constant $C(d,\Lambda) < \infty$ such that for every $z \in \Rd$, $r \ge 1$ and $p,q \in B_1$,
\begin{equation}
\label{e.expectation}
\Ll| \E[J(z,r,p,q)] - \frac 1 2 (q-p) \cdot \ahom (q-p)  \Rr| \le C r^{-d}.
\end{equation}
The estimate \eqref{e.bounded.fluc} then follows from the additivity and locality properties of $J$ stated in Theorem~\ref{t.main1}. We first fix some exponents that will be convenient to work with, and then recall these statements. For the rest of the paper, 
\begin{equation}
\label{e.fixs}
\text{we fix } s \in (1,2) \text{ sufficiently close to } 2,
\end{equation}
then $\de(s) > 0$ such that
\begin{equation}
\label{e.def.delta}
\frac{1+\de}{2} < \frac 1 s ,
\end{equation}
and $\al(s)$ such that 
\begin{equation}
\label{e.def.alpha}
\frac d 2(1+\de) < \al < \frac d s.
\end{equation}
For notational convenience, we think of $\de$, $\al$ as explicit functions of $s$, for instance we may fix $\de(s) = \frac 1 s - \frac 1 2$, etc. Taking $s < 2$ sufficiently close to $2$ brings $\de(s) > 0$ as close to $0$ as desired. 

\smallskip

The additivity property and~Remark~\ref{r.change-s} imply the existence of a constant $C(s,d,\Lambda) < \infty$ such that for every $z \in \Rd$, $R \ge r \ge 1$ and $p,q \in B_1$,
\begin{equation}
\label{e.additivity}
J(z,R,p,q) = \int_{\Phi_{z,\sqrt{R^2 - r^2}}} J(\cdot,r,p,q) + \O_{s}(C r^{-\al}).
\end{equation} 
The locality property implies the existence of a constant $C(s,d,\Lambda) < \infty$ and, for every $z \in \Rd$, $r \ge 1$ and $p, q \in B_1$, of an $\mcl F(\cu_{r^{1+\de}}(z))$-measurable random variable $\Jd(z,r,p,q)$ such that 
\begin{equation}
\label{e.localization}
J(z,r,p,q) = \Jd(z,r,p,q) + \O_s\Ll(C r^{-\al}\Rr).
\end{equation}

\smallskip

Define the centered random variable
\begin{equation}
\label{e.def.tJd}
\tJd(z,r,p,q) := \Jd(z,r,p,q) - \E[\Jd(z,r,p,q)].
\end{equation}
It follows from \eqref{e.additivity} and \eqref{e.localization} that there exists a constant $C(s,d,\Lambda) < \infty$ such that for every $z \in \Rd$, $R \ge r \ge 1$ and $p,q \in B_1$,
\begin{equation}
\label{e.additivity.centered}
\tJd(z,R,p,q) = \int_{\Phi_{z,\sqrt{R^2 - r^2}}} \tJd(\cdot,r,p,q) + \O_{s}(C r^{-\al}).
\end{equation}
We also note for future reference that by \eqref{e.bounded.fluc} and \eqref{e.localization}, for every $z \in \Rd$, $r \ge 1$ and $p, q \in B_1$,
\begin{equation}
\label{e.bound.tJd}
\tJd(z,r,p,q) = \O_s\Ll(C r^{-\frac d 2}\Rr). 
\end{equation}
We give ourselves three exponents $\eta_1(s) > \eta(s) >  \eta_2(s) > 0$ such that
\begin{equation}
\label{e.def.td.eta}
(1-\eta_1) \al > \frac d 2
\end{equation}
and
\begin{equation}
\label{e.def.eta}
(1-\eta_2)(1+\de) < 1.
\end{equation}

We recall from~\eqref{e.really719} that there exists $C(s,d,\Lambda) < \infty$ such that for every $z \in \Rd$, $R \ge 1$, $r \in [R^{1-\eta_1}, R^{1-\eta_2}]$,  $p,q \in B_1$ and (deterministic) $f :\R^d \to \R$ satisfying $\|f\|_{L^\infty} \le 1$,
\begin{equation}  
\label{e.bound.W0}
\fint_{\cu_R(z)} f \, \tJd(\cdot, r,p,q) = \O_s \Ll( C R^{-\frac d 2} \Rr) .
\end{equation}
We define cubes with a trimmed boundary:
\begin{equation}
\label{e.def.cut}
\cut_r(z) := \cu_{r-r^{\be}}(z),
\end{equation}
where the exponent $\be(s)$ is fixed according to
\begin{equation}
\label{e.def.beta}
\frac{1-\eta}{1-\eta_2} \vee (1-\eta_2)(1+\de) < \be < 1,
\end{equation}
and set, for every $z,p,q \in \Rd$ and $r \ge 1$, 
\begin{equation}
\label{e.def.W}
 W(z,r,p,q) :=  r^{-d} \int_{\cut_r(z)} \tJd\Ll(\cdot,r^{1-\eta},p,q\Rr) .
\end{equation}
Since $\be > (1-\eta)(1+\de)$, for $r$ sufficiently large, the random variable $W(z,r,p,q)$ is $\F(\cu_{r-1}(z))$-measurable, and by \eqref{e.bound.W0}, for every $z \in \Rd$, $r \ge 1$ and $p,q \in B_1$, we have
\begin{equation}
\label{e.bound.W}
W(z,r,p,q) = \O_s \Ll( C r^{-\frac d 2} \Rr) .
\end{equation}
This bound provides optimal stochastic tail estimates on the random variable~$W$. The goal of this section is to complement this information with an asymptotically exact description of the law of this random variable. This is achieved in the next proposition. There, a Gaussian quadratic form over $\Rd \times \Rd$ is a random quadratic form such that the entries of its associated $2d$-by-$2d$ matrix are a Gaussian vector.
\begin{proposition}
\label{p.clt}
There exists a Gaussian quadratic form $\msf Q(p,q)$ such that for every $z,p,q \in \Rd$, we have $\msf Q(p,p) = 0$ and
\begin{equation}
\label{e.clt}
r^{\frac d 2} W(z,r,p,q)  \xrightarrow[r \to \infty]{\mathrm{(law)}} \msf Q(p,q).
\end{equation}
Moreover, the convergence in law holds jointly over $p,q \in \Rd$.
\end{proposition}
We let $(p,q) \mapsto \W(\cdot,p,q)$ denote the quadratic form white noise such that, for every~$f \in L^2(\R^d,\R)$, the random variable $\W(f,p,q)$ has the same law as $\|f\|_{L^2(\Rd)} \msf Q(p,q)$. Abusing notation, we write $
\int_{\Rd} f(x) \W(x,p,q) \, dx := \W(f,p,q)$.
\begin{corollary}
\label{c.clt.convolution}
Let $\td \eta \in (\eta_1,\eta)$. For every $p,q \in \Rd$ and $f : \R^d \to \R$ satisfying
\begin{equation}
\label{e.hyp.f}
\int_{\Rd} \sup_{\cu_1(x)}  \Ll(|f| + |\nabla f|  \Rr)^2 \, dx < \infty,
\end{equation}
we have
\begin{equation}
\label{e.clt.convolution}
r^{-\frac d 2} \int_{\R^d} f\Ll( \tfrac x r \Rr) \, \tJd(x,r^{1-\td \eta}, p,q) \, dx \xrightarrow[r \to \infty]{\mathrm{(law)}} \int_{\Rd} f(x) \W(x,p,q) \, dx.
\end{equation}
Moreover, the convergence holds jointly over $f$ satisfying \eqref{e.hyp.f} and $p,q \in \Rd$. 
In particular, 
\begin{equation}
\label{e.limit.J2}
r^{\frac d 2} \Ll(J\Ll(r z,r,p,q\Rr) - \frac 1 2 (p-q) \cdot \ahom (p-q) \Rr)\xrightarrow[r \to \infty]{\mathrm{(law)}} \int_{\Phi_{z,1}} \W(\cdot,p,q),
\end{equation}
and the convergence holds jointly over $z,p,q \in \Rd$. 
\end{corollary}

\begin{remark}
\label{r.cond.f}
The condition imposed on $f$ in \eqref{e.hyp.f} is not optimal. For instance, we could weaken the differentiability condition to cover some H\"older continuous functions $f$, as the proof will make clear.
\end{remark}
We may use a trimmed or an untrimmed cube in the definition of $W$, or shift the center point $z$ over a distance smaller than $r$, without changing the random variable variable $W(z,r,p,q)$ significantly. We give a slightly more general version of this idea in the next lemma (recall that $\cut_R = \cu_{R - R^\be}$). 
\begin{lemma}
\label{l.trim}
There exists $C(s,d,\Lambda) < \infty$ such that for every $\be' \in[\be,1)$, $z \in \Rd$, $R \ge 1$, $r \in [R^{1-\eta_1}, R^{1-\eta}]$,  $p,q \in B_1$ and (deterministic) $f :\R^d \to \R$ satisfying $\|f\|_{L^\infty} \le 1$, we have
\begin{equation}  
\label{e.trim}
\fint_{\cu_R(z)} f \, \tJd(\cdot, r,p,q) = R^{-d} \int_{\cu_{R-R^{\be'}}(z)} f \, \tJd(\cdot, r,p,q) + \O_s \Ll( C R^{-\frac d 2-\frac{1-\be'}{2}} \Rr) ,
\end{equation}
and, for $|y-z| \le R^{\be'}$,
\begin{equation}
\label{e.move.z}
\fint_{\cu_R(z)} f \, \tJd(\cdot, r,p,q) = \fint_{\cu_{R}(y)} f \, \tJd(\cdot, r,p,q) + \O_s \Ll( C R^{-\frac d 2-\frac{1-\be'}{2}} \Rr) .
\end{equation}
\end{lemma}
\begin{remark}
\label{r.trim}
When $\be' \in (0,\be)$, the estimates \eqref{e.trim} and \eqref{e.move.z} clearly hold as well if we replace $\O_s \Ll( C R^{-\frac d 2-\frac{1-\be'}{2}} \Rr)$ by $\O_s \Ll( C R^{-\frac d 2-\frac{1-\be}{2}} \Rr)$ in the right-hand sides.
\end{remark}
\begin{proof}[Proof of Lemma~\ref{l.trim}]
Without loss of generality, we assume that $z = 0$. There exists a set $\mcl Z \subset R^{\be'} \Z^d$ of cardinality at most $R^{(d-1)(1-\be')}$ and such that up to a set of null Lebesgue measure, the cubes $(\cu_{R^{\be'}}(x))_{x \in \mcl Z}$ cover the region $\cu_R \setminus \cu_{R-R^{\be'}}$. By the definition of~$\be$ in \eqref{e.def.beta}, we have
\begin{equation*}  
r \in \Ll[R^{\be'(1-\eta_1)}, R^{\be'(1-\eta_2)}\Rr],
\end{equation*}
and therefore, by \eqref{e.bound.W0}, for every $x \in \Rd$, we have
\begin{equation*}  
\fint_{\cu_{R^{\be'}}(x)} f \, \tJd(\cdot,r,p,q) = \O_s \Ll( C R^{- \frac d 2\be'} \Rr) .
\end{equation*}
The random variable above is $\mcl F(\cu_{R^{\be'}+r^{1+\de}})$-measurable, and in view of the definition of $\be$, we have $r^{1+\de} \le R^{\be'}$. Using independence as in~Lemma~\ref{l.partition}, we obtain
\begin{equation*}  
R^{-d} \sum_{x \in \mcl Z} \int_{\cu_{R^{\be'}}(x)} f \, \1_{\cu_R \setminus \cu_{R-R^{\be'}}} \,  \tJd(\cdot,r,p,q) = \O_s \Ll( C R^{\frac{(d-1)(1-\be')}{2}-\frac d 2 \be' } \Rr) R^{-(1-\be') d},
\end{equation*}
which completes the proof of \eqref{e.trim}. The proof of \eqref{e.move.z} is identical.
\end{proof}
We next show that in the definition of $W(\, \cdot \,,R,p,q)$, we can replace the mesoscale $R^{1-\eta}$ appearing there by anything between $R^{1-\eta_1}$ and $R^{1-\eta_2}$, up to an error of lower order. 
\begin{lemma}
\label{l.var.meso}
There exist $\eps(s) > 0$ and $C(s,d,\Lambda) < \infty$ such that for every $z \in \Rd$, $R \ge 1$, $r \in [R^{1-\eta_1}, R^{1-\eta_2}]$ and $p,q \in B_1$, 
\begin{equation*}
\fint_{\cu_R(z)} \tJd(\cdot,r,p,q) = \fint_{\cu_R(z)} \tJd(\cdot,R^{1-\eta_1},p,q) + \O_s \Ll( C R^{-\frac d 2 - \eps} \Rr) .
\end{equation*}
\end{lemma}
\begin{proof}
Without loss of generality, we assume $z = 0$, and fix $p,q \in B_1$. Let $r_1 := R^{1-\eta_1}$. By the additivity property \eqref{e.additivity}, we have
\begin{equation*}  
\tJd(\cdot,r,p,q) = \int_{\Phi_{\sqrt{r^2 - r_1^2}}} \tJd (\cdot, r_1,p,q) + \O_s \Ll( C r_1^{-\al} \Rr) .
\end{equation*}
Note that $r_1^\al = R^{(1-\eta_1) \al}$ and that $(1-\eta_1) \al > \frac d 2$. 
Integrating over $\cu_R$, we get
\begin{equation*}  
\fint_{\cu_{R}} \tJd(\cdot,r,p,q) = R^{-d} \int_{\R^d} g \, \tJd(\cdot,r_1,p,q) + \O_s \Ll( C r_1^{-\al} \Rr) ,
\end{equation*}
where 
\begin{equation*}  
g(y) := \int_{\cu_R} \Phi_{\sqrt{r^2 - r_1^2}}(x-y) \, dx.
\end{equation*}
Let $\eps_1 > 0$. For $y \in \cu_{R - r^{1+\eps_1}}$, the quantity $|g(y) - 1|$ is smaller than any negative power of $r$. By \eqref{e.bound.W0}, we have in particular
\begin{equation*}  
R^{-d} \int_{\cu_{R-r^{1+\eps_1}}} (g-1) \, \tJd(\cdot,r_1,p,q) = \O_s \Ll( C R^{-10 d} \Rr) .
\end{equation*}
Similarly, we have
\begin{equation*}  
R^{-d} \int_{\R^d \setminus \cu_{R+r^{1+\eps_1}}} g \, \tJd(\cdot,r_1,p,q) = \O_s \Ll( C R^{-10 d} \Rr) .
\end{equation*}
By Lemma~\ref{l.trim}, if $\eps_1 >0$ is sufficiently small, then there exists $\eps > 0$ such that 
\begin{equation*}  
\int_{\cu_{R+r^{1+\eps_1}} \setminus \cu_{R-r^{1+\eps_1}} } (g - \1_{\cu_R}) \, \tJd(\cdot,r_1,p,q) = \O_s \Ll( C R^{-\frac d 2 - \eps} \Rr),
\end{equation*}
so the lemma is proved.
\end{proof}
For every real random variable $X$ and $\sigma, \cc, \lambda_1 \ge 0$, we write
$$
X = \mcl N(\sigma^2,\cc,\lambda_1)
$$
to mean that for every $\lambda \in (-\lambda_1,\lambda_1)$,
\begin{equation*}  
\Ll|\log \E\Ll[\exp\Ll(\lambda  X\Rr)\Rr] - \frac{\sigma^2 \lambda^2 }{2} \Rr| \le \cc \lambda^2.
\end{equation*}
The statement $X = \mcl N(\sigma^2,0,\infty)$ is equivalent to saying that $X$ is a centered Gaussian random variable of variance $\sigma^2$. By abuse of notation, we may write $\mcl N(\sigma^2,0,\infty)$ to denote the law of this Gaussian random variable. 

The proof of Proposition~\ref{p.clt} can be summarized as follows. On the one hand, if $X$ is a centered random variable with sufficient integrability and $\cc > 0$, then there exists $\lambda_1 > 0$ such that $X = \mcl N(\sigma^2,\cc,\lambda_1)$, where $\sigma^2 = \E[X^2]$. On the other hand, the rescaled sum of $2^d$ independent copies of $X$ satisfies the same estimate with $\lambda_1$ replaced by $2^{\frac d 2} \lambda_1$. Iterating this renormalization step, we see that we can come arbitrarily close to $\mcl N(\sigma^2, 0, \infty)$. 

\smallskip

In our setting, the second step of this argument is not exact. However, the errors become negligible as we move to larger and larger scales, so we simply need to start the induction argument at a large enough scale in order to conclude.  The first step of this argument is formalized in the next lemma. 

\begin{lemma}
\label{l.centered}
There exists $C < \infty$ such that if a random variable $X$ satisfies 
\begin{equation*}  
\E[\exp\left(2|X|\right)] \le 2 \ \text{ and } \ \ \E[X] = 0,
\end{equation*}
then for every $\lambda_1 \in (0,1]$,
\begin{equation*}  
X = \mcl N\Ll(\E\left[X^2 \right], C \lambda_1, \lambda_1\Rr). 
\end{equation*}
\end{lemma}
\begin{proof}
Let $\psi(\lambda) := \log \E\left[\exp(\lambda X)\right]$, and 
\begin{equation*}  
\E_{\lambda}[\cdot] := \frac{\E\left[\, \cdot \, \exp(\lambda X)\right]}{\E\left[\exp(\lambda X) \right]}
\end{equation*}
be the ``Gibbs measure'' associated with $\lambda X$. Since $\partial_\lambda \E_\lambda[F] = \E_\lambda[FX] - \E_\lambda[F] \E_\lambda[X]$, we have
\begin{align*}  
\psi'(\lambda) & = \E_\lambda[X], \\
\psi''(\lambda) & = \E_\lambda\left[X^2\right] - \left(\E_\lambda[X]\right)^2, \\
\psi'''(\lambda) & = \E_\lambda \left[X^3 \right] -3\E_\lambda\left[X^2\right] \E_\lambda[X] + 2 \left(\E_\lambda[X]\right)^3.
\end{align*}
In particular, there exists a constant $C < \infty$ such that for every $|\lambda| \le 1$, we have $|\psi'''(\lambda)| \le C$. The result then follows by a Taylor expansion of $\psi$ around $0$.
\end{proof}
We next give a slight restatement of~Lemma~\ref{l.logL}. 

\begin{lemma}
\label{l.Cheby}
There exists a constant $C(s) < \infty$ such that if $X = \O_s(1)$ and $\E[X] = 0$, then for every $\lambda \in \R$,
\begin{equation*}  
\log \E\left[\exp\left(\lambda C^{-1} X\right)\right] \le \lambda^2 \vee |\lambda|^{\frac s {s-1}} .
\end{equation*}
\end{lemma}
\begin{proof}
We use the Lemma~\ref{l.logL} to obtain the result for $\lambda$ in a neighborhood of the origin, and then Chebyshev's inequality for general $\lambda$. 
\end{proof}

We finally formalize a convenient approximation argument in the next lemma, before completing the proof of Proposition~\ref{p.clt}.
\begin{lemma}
\label{l.approx}
Let $\bar \lambda > 0$. There exists $C(s,\bar \lambda) < \infty$ such that for every $\sigma^2,\cc >0$, $\lambda_1 \in [0,\bar \lambda]$ and $\theta \in [0,1]$, if $X_1, X_2$ are two centered random variables satisfying
\begin{equation}
\label{e.X1}
X_1 = \mcl N \Ll( \sigma^2, \cc, \lambda_1 \Rr) \ \ \text{ and } \ \  X_2 = \O_s \Ll( \theta \Rr) ,
\end{equation}
then 
\begin{equation}  
\label{e.X12}
X_1 + X_2 = \mcl N \Ll( \sigma^2, \cc + C \sqrt{\theta}(1+\sigma^2+\cc), (1-\sqrt{\theta})\lambda_1 \Rr) .
\end{equation}
\end{lemma}
\begin{proof}
Write $X := X_1 + X_2$, and let $\zeta, \zeta'\in (1,\infty)$ be such that $\frac 1 \zeta + \frac 1 {\zeta'} = 1$. By H\"older's inequality,
\begin{equation}
\label{e.mydear1}
\log \E \Ll[ \exp \Ll( \lambda  X   \Rr)  \Rr] \le \frac 1 \zeta \log \E \Ll[ \exp \Ll( \zeta \lambda  X_1 \Rr)  \Rr] + \frac 1 {\zeta'} \log \E \Ll[ \exp \Ll(\zeta' \lambda  X_2 \Rr)  \Rr] ,
\end{equation}
and conversely,
\begin{equation}  
\label{e.mydear2}
\log \E \Ll[ \exp \Ll( \lambda  X_1 \Rr)  \Rr] \le \frac 1 \zeta \log \E \Ll[ \exp \Ll( \zeta \lambda  X \Rr)  \Rr] + \frac 1 {\zeta'} \log \E \Ll[ \exp \Ll(-\zeta' \lambda  X_2 \Rr)  \Rr] .
\end{equation}
It follows from the hypothesis that for every $\lambda$ satisfying $|\lambda| <  \zeta^{-1} \, \lambda_1$,
\begin{equation}  
\label{e.whatcanwedo}
\log \E \Ll[ \exp \Ll( \zeta \lambda  X_1 \Rr)  \Rr] \le \Ll( \frac {\sigma^2} 2 + \cc \Rr) (\zeta \lambda)^2.
\end{equation}
Recall from Lemma~\ref{l.Cheby} that there exists $C(s) < \infty$ such that for every $\lambda \in \R$,
\begin{equation*}  
\log \E \Ll[ \exp \Ll(C^{-1}  \zeta' \lambda  X_2 \Rr)  \Rr] \le  (\zeta' \theta \lambda)^2 \vee |\zeta' \theta\lambda|^{\frac s{s-1}}.
\end{equation*}
Fixing $\zeta  = (1-\sqrt{\theta})^{-1}$, and therefore $\zeta' = \theta^{-\frac 1 2}$, we deduce that there exists $C(s,\bar \lambda) < \infty$ such that for every $\lambda$ satisfying $|\lambda| \le \bar \lambda$,
\begin{equation*}  
\log \E \Ll[ \exp \Ll(  \zeta' \lambda  X_2 \Rr)  \Rr]  \le C \Ll(\sqrt{\theta} \lambda\Rr)^2.
\end{equation*}
Using \eqref{e.mydear1} with \eqref{e.whatcanwedo} and the previous display yields that for every $\lambda$ satisfying $|\lambda| < (1-\sqrt{\theta}) \lambda_1$,
\begin{equation*}  
\log \E \Ll[ \exp \Ll( \lambda  X   \Rr)  \Rr] - \frac {\sigma^2 \lambda^2}{2} \le \Ll( (\zeta - 1) \frac{\sigma^2}{2}  +  \cc + C \sqrt{\theta} \Rr) \lambda^2.
\end{equation*}
Since
\begin{equation*}  
\zeta - 1 = (1-\sqrt{\theta})^{-1} -1 \le C \sqrt{\theta},
\end{equation*}
we obtain one side of the two-sided inequality implicit in \eqref{e.X12}. The other inequality is proved in the same way, using \eqref{e.mydear2} instead of \eqref{e.mydear1}.
\end{proof}

\begin{proof}[Proof of Proposition~\ref{p.clt}] 

We begin by noting that $\msf Q(p,p) = 0$ is a consequence of the claimed convergence~\eqref{e.clt}. Indeed, it follows from~\eqref{e.clt} together with $\al(1-\eta) > \frac d 2$ (cf.~\eqref{e.def.td.eta}) and
\begin{equation*}  
r^{-d} \int_{\cut_r(z)} \tJd\Ll(\cdot,r^{1-\eta},p,q\Rr) = \O_{s/2} \Ll( C r^{-\al(1-\eta)} \Rr). 
\end{equation*}
The latter is obtained by~\eqref{e.min.p.p} and the fact that, by~\eqref{e.localization}, there exists $C(s,d,\Lambda) < \infty$ such that for every $z \in \Rd$ and $r \ge 1$,
\begin{equation*}  
\tJd(z,r,p,p) = \O_{s/2} \Ll( C r^{-\al} \Rr).
\end{equation*}

\smallskip

We now prove~\eqref{e.clt}. For every $z,p,q,p',q' \in \Rd$, we denote the bilinear form associated with $W$ by
\begin{equation}  \label{e.def.bilin}
W(z,r,p,q,p',q') := \frac 1 4 \Ll[ W(z,r,p+p',q+q') - W(z,r,p-p',q-q') \Rr] ,
\end{equation}
and similarly, we write $\tJd(z,r,p,q,p',q')$ for the bilinear form associated with~$\tJd(z,r,p,q)$. 

In order to prove Proposition~\ref{p.clt}, it suffices to show that for every $p$, $q$, $p'$, $q' \in B_1$, there exists a centered Gaussian random variable $\msf Q(p,q,p',q')$ such that for every $z \in \Rd$,  
\begin{equation}
\label{e.conv.W}
r^{\frac d 2} W(z,r,p,q,p',q') \xrightarrow[r \to \infty]{\mathrm{(law)}} \msf Q(p,q,p',q').
\end{equation}
Indeed, since $(p,q,p',q') \mapsto W(z,r,p,q,p',q')$ is a linear mapping, this suffices to guarantee that the convergence in law in \eqref{e.conv.W} holds jointly over $p,q,p',q' \in \Rd$ for a family of random variables $\msf Q$ such that $(p,q,p',q') \mapsto \msf Q(p,q,p',q')$ is linear, and thus to yield Proposition~\ref{p.clt}.

\smallskip

We therefore fix the parameters $p,q,p',q' \in B_1$, and ligthen the notation by simply writing $W(z,r)$ and $\tJd(z,r)$ instead of $W(z,r,p,q,p',q')$ and $\tJd(z,r,p,q,p',q')$ respectively. Our goal is to show that there exists $\sigma \in [0,\infty)$ such that for every $z \in \Rd$,
\begin{equation}
\label{e.conv.W2}
r^{\frac d 2} W(z,r) \xrightarrow[r \to \infty]{\mathrm{(law)}} \mcl N(\sigma^2,0,\infty).
\end{equation}
In order to show this, it suffices to prove that for any given $\bar \lambda  < \infty$ and $\cc > 0$, there exists $\sigma \in [0,\infty)$ such that for every $z \in \Rd$ and $r$ sufficiently large,
\begin{equation}
\label{e.clt.lambda}
r^{\frac d 2} W(z,r) = \mcl N\Ll(\sigma^2, \cc, \bar \lambda\Rr).
\end{equation}
Indeed, this ensures the convergence of the Laplace transform of $r^{\frac d 2} W(z,r)$ to that of a centered Gaussian random variable, which is sufficient to conclude. 

\smallskip

We fix $\bar \lambda  < \infty$ and proceed to prove~\eqref{e.clt.lambda}. In the argument, the value of the exponent~$\eps(s) > 0$ and of the constant $C( \bar \lambda,s,d,\Lambda) < \infty$ may vary in each occurrence. 
We break the proof into three steps.

\smallskip

\emph{Step 1.} 
For every $R \ge 1$ and $\sigma, \cc,\lambda_1 \ge 0$, we denote by $\msf A(R,\sigma^2,\cc,\lambda_1)$ the statement that for every $z \in \Rd$, and $r \in [R^{1-\eta_1},R^{1-\eta_2}]$,
$$
R^{\frac d 2} \fint_{\cu_R(z)} \tJd(\cdot,r) = \mcl N(\sigma^2, \cc,\lambda_1). 
$$

\smallskip

In this step, we show that for every $R \ge 1$, there exists $\sigma^2(R) \in [0,C]$ such that 
\begin{equation}
\label{e.init0}
\text{for every }\lambda_1 \in (0,1], \ \ \msf A(R,\sigma^2,C\lambda_1+ C R^{-\eps}, \lambda_1) \ \text{holds}.
\end{equation}
By \eqref{e.bound.W0} and Lemma~\ref{l.centered}, for each $R \ge 1$, there exists $\sigma^2(R) \in [0,C]$ such that for every $\lambda_1 \in (0,1]$,
\begin{equation*}  
R^{\frac d 2} \fint_{\cu_R} \tJd(\cdot,R^{1-\eta_1}) = \mcl N\Ll(\sigma^2, C \lambda_1, \lambda_1 \Rr).
\end{equation*}
By Lemmas~\ref{l.var.meso} and \ref{l.approx}, for every $r \in [R^{1-\eta_1}, R^{1-\eta_2}]$ and $\lambda_1 \in (0,1]$, we have
\begin{equation*}  
R^{\frac d 2} \fint_{\cu_R} \tJd(\cdot,r) = \mcl N\Ll(\sigma^2, C \lambda_1 + C R^{-\eps}, \lambda_1 \Rr).
\end{equation*}
In order to conclude, there remains to justify that this estimate still holds when the domain of integration $\cu_R$ is replaced by $\cu_R(z)$, for every $z \in \Rd$. By $\Z^d$-stationarity, it suffices to consider $z$ ranging in $[0,1)^d$. In this case, the result follows by an application of Lemmas~\ref{l.trim} and \ref{l.approx}.

\smallskip

\emph{Step 2.} We fix $0 < \ga_1 < \ga_2 < 1$ such that
\begin{equation}
\label{e.cond.gamma1}
\ga_1 \be > (1-\eta_1)(1+\de)
\end{equation}
and 
\begin{equation}
\label{e.cond.gamma2}
\ga_1(1-\eta)   > 1-\eta_1.
\end{equation}
In this step, we show that for every $R \ge 1$, $\sigma, \cc, \lambda_1 \ge 0$, and $\ga \in [\ga_1,\ga_2]$, we have
\begin{multline}
\label{e.induc.A0}
\msf A(R^\ga,\sigma^2,\cc,\lambda_1) \\ 
\implies \quad  \msf A\Ll(R,\sigma^2,\cc + CR^{-\eps}(1 + \cc + \sigma^2) ,\bar \lambda \wedge \Ll[\frac{R^{(1-\ga) \frac d 2}}{2} \lambda_1 \Rr] \Rr).
\end{multline}
In view of what needs to be proved and of Lemmas~\ref{l.trim} and \ref{l.approx}, we may replace the assumption of $\msf A(R^\ga,\sigma^2,\cc,\lambda_1)$ by the assumption that for every $z \in \Rd$ and $r \in [R^{\ga(1-\eta_1)}, R^{\ga(1-\eta_2)}]$,
\begin{equation}
\label{e.hyp.step1}
R^{-\ga \frac{d}{2} } \int_{\cut_{R^\ga}(z)} \tJd(\cdot,r) = \mcl N(\sigma^2,\cc,\lambda_1).
\end{equation}
By Lemmas~\ref{l.var.meso} and \ref{l.approx}, it suffices to show that for $r_1 := R^{1-\eta_1}$, we have
\begin{equation}
\label{e.concl.step1a}
R^{\frac d 2} \fint_{\cu_R} \tJd(\cdot,r_1) = \mcl N \Ll( \sigma^2, \cc + CR^{-\eps}(1 + \cc + \sigma^2) ,  \bar \lambda \wedge \Ll[\frac{R^{(1-\ga) \frac d 2}}{2} \lambda_1 \Rr]  \Rr) .
\end{equation}
Note that our choice of exponent $\ga_1$ in \eqref{e.cond.gamma2} ensures that
\begin{equation}
\label{e.gamma.inthere}
r_1 \in \Ll[ R^{\ga(1-\eta_1)}, R^{\ga(1-\eta)} \Rr].
\end{equation}
We pave $\cu_R$ by cubes of size $R^\ga$ that are at distance at least $2$ to one another, plus a remainder. More precisely, we let
\begin{equation*}  
\mcl Z_1 := \{x \in R^\ga \Z^d \text{ s.t. } \cu_{R^\ga}(x) \subset \cu_R\},
\end{equation*}
\begin{equation*}
\mcl B_1 := \cu_R \setminus \bigcup_{x \in \mcl Z} \cu_{R^\ga}(x),
\end{equation*}
and observe that
\begin{equation}  
\label{e.countZ1}
\Ll| \, |\mcl Z_1| - R^{(1-\ga)d} \, \Rr| \le C R^{(1-\ga)(d-1)},
\end{equation}
and
\begin{equation*}  
\int_{\cu_R} \tJd(\cdot,r_1) = \int_{\mcl B_1} \tJd(\cdot,r_1) + \sum_{x \in \mcl Z_1} \int_{\cu_{R^\ga}(x)} \tJd(\cdot,r_1) .
\end{equation*}
Since $\ga > (1-\eta_1)(1+\de)$, the random variable
\begin{equation*}  
\int_{\cu_{R^\ga}(x) \setminus \cut_{R^\ga}(x)} \tJd(\cdot,r_1)
\end{equation*}
is $\mcl F(\cu_{2R^\ga}(x))$-measurable. Moreover, by Lemma~\ref{l.trim} and \eqref{e.gamma.inthere}, we have
\begin{equation*}  
R^{-\gamma d} \int_{\cu_{R^\ga}(x) \setminus \cut_{R^\ga}(x)} \tJd(\cdot,r_1) = \O_s \Ll( R^{-\ga \Ll(\frac d 2 + \frac{1-\be}2\Rr)} \Rr) .
\end{equation*}
By independence (more precisely, by Lemma~\ref{l.partition}) and~\eqref{e.countZ1}, we get
\begin{equation}
\label{e.err.boundary}
R^{-d} \sum_{x \in \mcl Z_1} \int_{\cu_{R^\ga}(x) \setminus \cut_{R^\ga}(x)} \tJd(\cdot,r_1) = \O_s \Ll( C R^{-\frac d 2 - \gamma \frac{1-\be}{2}} \Rr) .
\end{equation}
Similarly, we obtain
\begin{equation}
\label{e.err.mcl.B}
R^{-d} \int_{\mcl B_1} \tJd(\cdot,r_1) = \O_s \Ll( C R^{-\frac d 2 - \frac {1-\ga}{2}} \Rr) .
\end{equation}
We now analyze
\begin{equation*}  
\sum_{x \in \mcl Z_1} \int_{\cut_{R^\ga}(x)} \tJd(\cdot,r_1).
\end{equation*}
By \eqref{e.cond.gamma1}, for $R$ sufficiently large, the summands are $\mcl F(\cu_{R^\ga-1}(x))$-measurable, and therefore independent of each other. By our assumption of \eqref{e.hyp.step1} and \eqref{e.gamma.inthere}, we obtain
\begin{equation}
\label{e.step1.concl}
|\mcl Z_1|^{-\frac 1 2} R^{-\ga \frac d 2} \sum_{x \in \mcl Z_1} \int_{\cut_{R^\ga}(x)} \tJd(\cdot,r_1) = \mcl N\Ll(\sigma^2,\cc, |\mcl Z_1|^\frac 1 2 \lambda_1\Rr).
\end{equation}
The result then follows by an application of Lemma~\ref{l.approx}.

\smallskip 

\emph{Step 3.} We conclude. Let $\cc_0 \in( 0,1]$. By \eqref{e.init0}, there exists $\lambda_1 > 0$ and, for every $R$ sufficiently large, a constant $\sigma^2(R) \in [0,C]$ such that 
\begin{equation*}  
\msf A(R,\sigma^2,\cc_0,\lambda_1) \ \ \text{holds}.
\end{equation*}
By \eqref{e.induc.A0}, for every $R$ sufficiently large, there exists $\sigma^2(R) \in [0,C]$ such that
\begin{equation}  \label{e.start}
\msf A(R,\sigma^2,\cc_0,\bar \lambda) \ \ \text{holds}.
\end{equation}
For $\sigma^2 \in [0,C]$ and $\cc \in [0, 2]$, we can rewrite \eqref{e.induc.A0} in the following form: 
\begin{equation*}  
\msf A(R,\sigma^2,\cc,\bar \lambda) \quad \implies \quad \forall R_1 \in \Ll[R^{1/\ga_2}, R^{1/\ga_1}\Rr], \ \  \msf A(R_1, \sigma^2, \cc + C R^{-\eps}, \bar \lambda).
\end{equation*}
Iterating this with the initialization \eqref{e.start} and $R$ sufficiently large, we obtain that there exists $\sigma^2(R) \in [0,C]$ such that
\begin{equation*}  
\forall R_1 \in \bigcup_{k = 1}^{\infty} \Ll[R^{k/\ga_1}, R^{k/\ga_2}\Rr], \ \ \msf A(R_1,\sigma^2,2\cc_0,\bar \lambda) \ \ \text{holds}.
\end{equation*}
One readily checks that the allowed range of values for $R_1$ contains an interval of the form $[R_2, +\infty)$ for some $R_2 < \infty$, so the proof is complete.
\end{proof}

We can now derive Corollary~\ref{c.clt.convolution} from Proposition~\ref{p.clt}.

\begin{proof}[Proof of Corollary~\ref{c.clt.convolution}]
Recall that $W(z,r,p,q,p',q')$ denotes the bilinear form associated with the quadratic form $W(z,r,p,q)$, see \eqref{e.def.bilin}. We use similar notation for $\tJd(z,r,p,q)$, $\msf Q(p,q)$ and $\W(\cdot,p,q)$. Our goal is to show that for every $p,q,p',q' \in \Rd$ and $f : \Rd \to \R$ satisfying \eqref{e.hyp.f}, we have
\begin{multline}
\label{e.clt.polar}
r^{-\frac d 2} \int_{\R^d} f\Ll( \tfrac x r \Rr) \, \tJd(x,r^{1-\td \eta}, p,q,p',q') \, dx \\
\xrightarrow[r \to \infty]{\mathrm{(law)}} \int_{\Rd} f(x) \W(x,p,q,p',q') \, dx.
\end{multline}
Since the terms in \eqref{e.clt.polar} depend linearly on $f$, $p$, $q$, $p'$ and $q'$, establishing the convergence \eqref{e.clt.polar} suffices to ensure joint convergence over these variables. From now on, we therefore fix the function $f$ satisfying \eqref{e.hyp.f} and $p$, $q$, $p'$, $q' \in B_1$, and focus on proving \eqref{e.clt.polar} for these parameters. 
We lighten the notation and simply write $\tJd(z,r)$ instead of $\tJd(z,r,p,q,p',q')$, and similarly for $W(z,r)$ and $\W(\cdot)$. We let $\sigma^2$ denote the variance of the white noise $\W$. The convergence \eqref{e.clt.polar} we need to prove can be restated as
\begin{equation}
\label{e.clt.conv}
r^{-\frac d 2} \int_{\R^d} f\Ll( \tfrac \cdot r \Rr) \, \tJd(\cdot,r^{1-\td \eta})
\xrightarrow[r \to \infty]{\mathrm{(law)}} \mcl N \Ll( \sigma^2 \int_\Rd f^2,0,\infty \Rr).
\end{equation}

Let $\kappa < 1$ be an exponent that will be chosen sufficiently close to $1$ in the course of the argument. We decompose the domain of integration along cubes of side length $r^\kappa$:
\begin{equation}  \label{e.decomp.cube}
\int_{\R^d} f\Ll( \tfrac \cdot r \Rr) \, \tJd(\cdot,r^{1-\td \eta})  = \sum_{y \in r^\kappa \Z^d} \int_{\cu_{r^\kappa}(y)} f\Ll( \tfrac \cdot r \Rr) \, \tJd(\cdot,r^{1-\td \eta}).
\end{equation}
We first show that we can replace each summand above by
\begin{equation*}  
f\Ll( \tfrac y r \Rr) \int_{\cu_{r^\kappa}(y)}  \tJd(\cdot,r^{1-\td \eta}),
\end{equation*}
up to an error of lower order. 
By \eqref{e.bound.W0}, we have
\begin{multline}  \label{e.repl.err}
  \int_{\cu_{r^\kappa}(y)} \Ll(f\Ll( \tfrac \cdot r \Rr) - f\Ll( \tfrac y r \Rr)\Rr) \, \tJd(\cdot,r^{1-\td \eta}) \\
=   \O_s \Ll( C r^{\kappa \frac d 2 - (1-\kappa)} \Rr) \, \sup_{\cu_{r^\kappa}(y)} \Ll|(\nabla f) \Ll( \tfrac \cdot r \Rr) \Rr|.
\end{multline}
Choosing $\kappa > (1-\eta)(1+\de)$ ensures that the random variable in \eqref{e.repl.err} is $\mcl F(\cu_{2r^\kappa}(y))$-measurable, and therefore, by~Lemma~\ref{l.partition},
\begin{multline*}  
\sum_{y \in r^\kappa \Z^d}  \Ll|\int_{\cu_{r^\kappa}(y)} \Ll(f\Ll( \tfrac \cdot r \Rr) - f\Ll( \tfrac y r \Rr)\Rr) \, \tJd(\cdot,r^{1-\td \eta})\Rr| \\
=   \O_s \Ll( C r^{ \frac d 2 - (1-\kappa) } \Rr) \, \Ll[r^{-d} \sum_{y \in r^\kappa \Z^d} r^{\kappa d} \Ll(\sup_{\cu_{r^\kappa}(y)} \Ll|(\nabla f) \Ll( \tfrac \cdot r \Rr) \Rr|\Rr)^2\Rr]^{\frac 1 2}.
\end{multline*}
By the assumption \eqref{e.hyp.f}, the expression between square brackets above remains bounded as $r$ tends to infinity. This justifies the replacement of $f\Ll( \tfrac \cdot r \Rr)$ by $f\Ll( \tfrac y r \Rr)$ in the right side of \eqref{e.decomp.cube}. By Lemmas~\ref{l.trim} and \ref{l.var.meso} and the same reasoning, we may replace each summand
\begin{equation*}  
f\Ll( \tfrac y r \Rr) \int_{\cu_{r^\kappa}(y)}  \tJd(\cdot,r^{1-\td \eta})
\end{equation*}
by 
\begin{equation*}  
r^{\kappa d} f\Ll( \tfrac y r \Rr) W(y, r^{\kappa}),
\end{equation*}
provided that $\kappa < 1$ is sufficiently close to $1$. 
There remains to analyze the asymptotic behavior of 
\begin{equation*}  
r^{-\frac d 2} \sum_{y \in r^\kappa \Z^d}r^{\kappa d}   f\Ll( \tfrac y r \Rr) \, W(y, r^{\kappa}).
\end{equation*}
Let $\bar \lambda < \infty$ and $\cc > 0$. By Proposition~\ref{p.clt}, the stochastic integrability of $W$ and the fact that $\|f\|_{L^\infty} < \infty$, for every $r$ sufficiently large and $y \in \Rd$, we have
\begin{equation*}  
r^{\kappa \frac d 2} f\Ll( \tfrac y r \Rr) W(y,r^\kappa) = \mcl N\Ll( f^2\Ll( \tfrac y r \Rr)\sigma^2, f^2\Ll( \tfrac y r \Rr)  \cc, \bar \lambda\Rr).
\end{equation*}
For $r$ sufficiently large, the random variable $W(y,r^\kappa)$ is $\mcl F(\cu_{r^\kappa-1}(y))$-measu\-rable. By independence, it thus follows that
\begin{multline*}  
r^{-\frac d 2} \sum_{y \in r^\kappa \Z^d}r^{\kappa d}   f\Ll( \tfrac y r \Rr) \, W(y, r^{\kappa}) \\
= \mcl N\Ll( r^{-d} \sum_{y \in r^\kappa \Z^d}r^{\kappa d}  f^2\Ll( \tfrac y r \Rr)\sigma^2, r^{-d} \sum_{y \in r^\kappa \Z^d}r^{\kappa d}  f^2\Ll( \tfrac y r \Rr)\cc, \bar \lambda\Rr).
\end{multline*}
Since
\begin{equation*}  
r^{-d} \sum_{y \in r^\kappa \Z^d}r^{\kappa d} f^2\Ll( \tfrac y r \Rr) \xrightarrow[r \to \infty]{} \int_\Rd f^2,
\end{equation*}
this completes the proof of \eqref{e.clt.conv}.

\smallskip

We now show \eqref{e.limit.J2}. By \eqref{e.expectation} and the definition of $\tJd$ in \eqref{e.localization}-\eqref{e.def.tJd}, it suffices to show that
\begin{equation*}  
r^{\frac d 2} \, \tJd(rz,r) \xrightarrow[r \to \infty]{\mathrm{(law)}}  \int_{\Phi_{z,1}} \W.
\end{equation*}
By the additivity property in \eqref{e.additivity.centered}, we have
\begin{equation*}  
\tJd(rz,r) = \int_{\Phi_{rz,\sqrt{r^2 - r^{2(1-\td \eta)}}}} \tJd(\cdot,r^{1-\td \eta}) + \O_s \Ll(C r^{-(1-\td \eta)\al} \Rr) .
\end{equation*}
The result follows from \eqref{e.def.td.eta} and \eqref{e.clt.convolution}.
\end{proof}

We record for future reference that the proof given above also enables to estimate the stochastic integrability of convolutions of $\tJd$. 
\begin{lemma}
\label{l.stoch.int}
For each $\td \eta \in (\eta_1,\eta)$, there exists $C(\td \eta,s,d,\Lambda) < \infty$ such that for every $p,q \in B_1$ and every $f : \R^d \to \R$ satisfying
\begin{equation*}
\int_{\Rd} \sup_{\cu_1(x)}  \Ll(|f| + |\nabla f|  \Rr)^2 \, dx  \le 1,
\end{equation*}
we have
\begin{equation*}  
r^{-\frac d 2} \int_{\R^d} f\Ll( \tfrac x r \Rr) \, \tJd(x,r^{1-\td \eta}, p,q) \, dx = \O_s \Ll( C \Rr) .
\end{equation*}
\end{lemma}
\begin{proof}
We let $\kappa < 1$ be as in the proof of Corollary~\ref{c.clt.convolution}. In view of this proof, it suffices to show that
\begin{equation}
\label{e.stoch.int}
r^{-\frac d 2} \sum_{y \in r^\kappa \Z^d}r^{\kappa d}   f\Ll( \tfrac y r \Rr) \, W(y, r^{\kappa},p,q) = \O_s \Ll( C \Rr) .
\end{equation}
By \eqref{e.bound.W}, we have
\begin{equation*}  
r^{\kappa \frac d 2}  \, W(y, r^{\kappa},p,q) = \O_s(C).
\end{equation*}
By Lemma~\ref{l.partition}, we deduce that 
\begin{equation*}  
r^{-\frac d 2} \sum_{y \in r^\kappa \Z^d}r^{\kappa d}   f\Ll( \tfrac y r \Rr) \, W(y, r^{\kappa},p,q) = \O_s \Ll( C \Rr) \Ll[ r^{-d} \sum_{y \in r^\kappa \Z^d}r^{\kappa d} f^2 \Ll( \tfrac y r \Rr)\Rr]^\frac 1 2 .
\end{equation*}
By the assumption on $f$, the quantity between square brackets remains bounded as $r$ tends to infinity, so the proof is complete.
\end{proof}

\section{Scaling limit of the correctors}
\label{s.scalingcorrectors}

In this section, we prove the second part of Theorem~\ref{t.main}, which gives the convergence in law of~$\nabla \phi_e$ to a gradient GFF. 

\smallskip

The coefficients of an elliptic operator define a correspondence between gradients of solutions and their fluxes. Our goal therefore is to describe the correspondence between \emph{spatial averages} of gradients and fluxes of solutions. As was demonstrated in Section~\ref{s.additivity}, this information can be conveniently read off from the quantity~$J$. Since we proved a scaling limit for~$J$ in the previous section, we can now revisit the arguments of~Section~\ref{s.additivity} and obtain much more precise information, which enables us to establish convergence to a gradient GFF. 

\smallskip

The exponents $s \in (1,2)$, $\de > 0$, $\eta_1 > \eta > \eta_2 > 0$ are those defined in the previous section. Recall that they are functions of $s$, and that they can be taken arbitrarily close to $0$ by taking $s < 2$ sufficiently close to $2$. We also introduce an exponent $\kappa$, which is allowed to depend on $(s,d,\Lambda)$ and satisfies
\begin{equation}
\label{e.def.kappa}
\kappa \in\left( \frac {1-\eta_1}{1-\eta} ,1 \right),
\end{equation}
and which we may redefine in each instance to be as close to $1$ as we wish.
Throughout, we allow $\ep>0$ to denote a positive exponent which may depend on $(s,d,\Lambda)$ and vary in each occurrence. 

\smallskip

We begin with the following definition, motivated by Lemma~\ref{l.ahomrz} below. 

\begin{definition}[Coarsened coefficients $\ahom_r(z)$]
For each $z\in\Rd$ and $r\geq 1$, we let $\b_r(z)$ denote the matrix representing the bilinear form 
\begin{equation*} \label{}
(p,q) \longmapsto \frac12 \left( \tJd(z,r^{1-\eta},p,p-q)  - \tJd(z,r^{1-\eta},p,p+q) \right).
\end{equation*}
That is, $\b_r(z)$ has the property that, for every $p,q\in\Rd$,  
\begin{equation*} \label{}
p\cdot \b_r(z) q = \frac12 \left( \tJd(z,r^{1-\eta},p,p-q)  - \tJd(z,r^{1-\eta},p,p+q) \right).
\end{equation*}
We then define
\begin{equation*} \label{}
\ahom_r(z):= \ahom + \b_r(z).
\end{equation*}
\end{definition}

We begin by noticing that since $\b_r$ is a linear function of $\tJd$, the results of the previous section give us complete quantitative information about the behavior of $\b_r$. For each $e \in \Rd$, we let $\V(\cdot,e)$ be the vector white noise defined by \eqref{e.def.V}. Notice that the mapping $e \mapsto \V(\cdot,e)$ is linear.

\begin{lemma}
\label{l.brz}
There exists $C(s,d,\Lambda)<\infty$ such that for every $z \in \Rd$ and $r \ge 1$,
\begin{equation} 
\label{e.brzbound}
\left| \b_r(z) \right| = \O_s\left( Cr^{-(1-\eta)\frac d2} \right).
\end{equation}
Moreover, for every $F : \Rd \to \Rd$ satisfying 
\begin{equation}
\label{e.hyp.F}
\int_{\Rd} \sup_{\cu_1(x)}  \Ll(|F| + |\nabla F|  \Rr)^2 \, dx \le 1
\end{equation}
and every $e \in B_1$, we have
\begin{equation}  
\label{e.brlawconv}
r^{-\frac d 2} \int_\Rd F\left(\tfrac xr \right) \cdot \b_{r^\kappa}(x) e \, dx \xrightarrow[r \to \infty]{\mathrm{(law)}} \int F(x) \cdot \V(x,e) \, dx,
\end{equation}
as well as, for every $r \ge 1$,
\begin{equation}
\label{e.brconvbound}
r^{-\frac d 2} \int_\Rd F\left(\tfrac xr \right) \cdot \b_{r^\kappa}(x) e \, dx = \O_s(C).
\end{equation}
Finally, the convergence in~\eqref{e.brlawconv} holds jointly with respect to $F$ and $e$ and jointly with~\eqref{e.limit.J2}.
\end{lemma}
\begin{proof}
The estimate \eqref{e.brzbound} follows from \eqref{e.bound.tJd}, the convergence in law from Corollary~\ref{c.clt.convolution}, and the bound \eqref{e.brconvbound} from Lemma~\ref{l.stoch.int}.
\end{proof}

It is also useful to notice that~$\b_r$ is continuous on scales smaller than~$r$. 

\begin{lemma}
\label{l.brcont}
There exist $\eps_1(s) > 0$ and $C(s,d,\Lambda) < \infty$ such that,
for every $x,y\in \Rd$ with $|x-y| \leq r^{1-\eta_1}$, we have
\begin{equation*}
\left| \b_r(x) - \b_r(y) \right| = \O_s \Ll(  C|x-y| r^{-(1-\eta_1) \Ll(\frac d 2 + 1\Rr)} +C r^{-\Ll(\frac d2+\eps_1\Rr)}\right). 
\end{equation*}
\end{lemma}
\begin{proof}
By the additivity property \eqref{e.additivity} and \eqref{e.localization}, we have 
\begin{equation*}  
\b_r(x) = \int_{\Phi_{x,\sqrt{r^{2(1-\eta)} - r^{2(1-\eta_1)}}}} \b_{r^{1-\eta_1}}(z) \, dz + \O_s \Ll( C r^{-(1-\eta_1) \al} \Rr) .
\end{equation*}
Recall from \eqref{e.def.td.eta} that $(1-\eta_1)\al > \frac d 2$. In order to prove the result, it therefore suffices to show that
\begin{multline}  \label{e.brcont}
\Ll| \int_\Rd \b_{r^{1-\eta_1}} (z) \Ll( \Phi(z-y,r^{2(1-\eta)} - r^{2(1-\eta_1)}) - \Phi(z-x,r^{2(1-\eta)} - r^{2(1-\eta_1)})\Rr)  \Rr| \\
=  \O_s \Ll(  C |x-y| r^{-(1-\eta_1) \Ll(\frac d 2 + 1\Rr)}\Rr).
\end{multline}
There exists a constant $C < \infty$ such that for every $x,y \in \Rd$ satisfying $|x-y| \le r^{1-\eta}$ and $z \in \Rd$, we have
\begin{multline*}  
\Ll( \Phi(z-y,r^{2(1-\eta)} - r^{2(1-\eta_1)}) - \Phi(z-x,r^{2(1-\eta)} - r^{2(1-\eta_1)})\Rr) \\
\le \frac{|x-y|}{r^{1-\eta}} \frac{C}{r^{(1-\eta)d}}\left( \exp \Ll( -\frac{|z-y|^2}{C r^{2(1-\eta)}} \Rr) +  \exp \Ll( -\frac{|z-x|^2}{C r^{2(1-\eta)}} \Rr) \right) .
\end{multline*}
The result then follows from \eqref{e.brzbound} and the fact that $\eta_1 > \eta$.
\end{proof}

We now show that $\ahom_r(z)$ is indeed acting like a coarsening of the coefficients. This is a next-order version of~Lemma~\ref{l.fluxmaps}.

\begin{lemma}
\label{l.ahomrz}
There exists $C(s,d,\Lambda)<\infty$ and $\ep(s,d,\Lambda)>0$ such that, for every $z\in\Rd$ and $r \ge 1$,
\begin{equation}
\label{e.ahomrz} 
\sup_{u\in\A_1(\Phi_{z,r})} 
\Ll| \int_{\Phi_{z,r}}\left( \a - \ahom_r \right) \nabla u  \Rr| 
= \O_s\left( Cr^{-\left( \frac d2+\eps \right) } \right). 
\end{equation}
\end{lemma}
\begin{proof}
We allow $\ep$ to be a positive exponent depending on $(s,d,\Lambda)$ which may vary in each occurrence. 

\smallskip

\emph{Step 1.} We show that, for every $z\in\Rd$, $r\geq 1$ and $q\in B_1$, 
\begin{equation}
\label{e.kyloren}
\left|  \int_{\Phi_{z,r}} \left( \a - \ahom \right) \nabla v(\cdot,z,r,0,q)  - \int_{\Phi_{z,r}} \b_rq \, \right| = \O_s\left( Cr^{-\left(\frac d2+\ep\right)} \right).
\end{equation}
By the additivity and localization properties, see~\eqref{e.expectation}, \eqref{e.additivity} and \eqref{e.localization}, there exists a constant $C(s,d,\Lambda)<\infty$ such that, for every $z \in \Rd$, $r \ge 1$ and $p,q \in B_1$, 
\begin{multline*}
J(z,r,p,q) \\
= \frac 1 2 (p-q) \cdot \ahom (p-q) + \int_{\Phi_{z,\sqrt{r^2 - r^{2(1-\eta)}}}} \tJd(\cdot,r^{1-\eta},p,q) + \O_s \Ll( C r^{- \al(1-\eta)}\Rr) .
\end{multline*}
By \eqref{e.def.td.eta}, we have $\al(1-\eta) > \frac d 2$. By \eqref{e.bound.W0} and Lemma~\ref{l.sum-O}, replacing $\Phi_{z,\sqrt{r^2 - r^{2(1-\eta)}}}$ by $\Phi_{z,r}$ in the integral above produces an error of $\O_s \Ll( C r^{-\Ll(\frac d 2 + \eps\Rr)} \Rr)$. 
By the definition of $\b_r(z)$, we deduce that
\begin{equation}
\label{e.flux.map}
\frac 1 2 \Ll( J(z,r,p,p-q) - J(z,r,p,p+q) \Rr) \\
=  \int_{\Phi_{z,r}} p \cdot \b_r q + \O_s \Ll( C r^{-\left( \frac d 2 + \ep\right)} \Rr) .
\end{equation}
By~\eqref{e.J-energy}, the left side above is equal to
\begin{equation*} \label{}
\int_{\Phi_{z,r}} p\cdot \left( \a - \ahom \right) \nabla v(\cdot,z,r,0,q),
\end{equation*}
which gives~\eqref{e.kyloren}.

\smallskip

\emph{Step 2.} We show that, for every $z\in\Rd$, $R\geq 1$ and $q\in B_1$, 
\begin{equation}
\label{e.captainphasma}
\left|  \int_{\Phi_{z,R}} \b_R \nabla v(\cdot,z,R,0,q) - \int_{\Phi_{z,R}} \b_R q \,\right| = \O_s \Ll( C R^{-\left( \frac d 2 + \ep\right)} \Rr) .
\end{equation}
This is a refinement of the fact that 
\begin{equation*}
\int_{\Phi_{z,R}} \nabla v(\cdot,z,R,0,q) = q + \O_s \Ll( C R^{-\frac d 2} \Rr),
\end{equation*}
which follows immediately from~\eqref{e.J-energy} and~\eqref{e.bounded.fluc}. The idea is to use additivity to pass to a slightly smaller mesoscale and use the regularity of~$\b_R$ in Lemma~\ref{l.brcont}. For clarity, we prove~\eqref{e.captainphasma} only for $z=0$. 

\smallskip

 Applying Lemmas~\ref{l.additivitybitches} and~\ref{l.identifytheta1} with $\alpha=\frac d2$, we have, for every $C\leq r \leq R/\sqrt{2}$,
\begin{multline*}
\int_{\Phi_{\sqrt{R^2-r^2}}} \left( \int_{\Phi_{z,r} }  \left| \nabla v(x,0,R,0,q) - \nabla v(x,z,r,0,q) \right|^2\,dx \right) \,dz 
= \O_{s/2}\left( Cr^{-d} \right). 
\end{multline*} 
Let $r \leq R/\sqrt{2}$ to be fixed shortly. Using the previous line and~\eqref{e.brzbound} we obtain
\begin{align*}
\lefteqn{
\left|  
\int_{\Phi_{R}} \b_R(x) \nabla v(x,0,R,0,q)\,dx 
-\int_{\Phi_{\sqrt{R^2-r^2}}} \int_{\Phi_{z,r}} \b_R(x) \nabla v(x,z,r,0,q)\,dx\,dz
 \right|  
} \qquad & \\ &
 \leq \left\| \b_R \right\|_{L^2(\Phi_{R})} \left( \int_{\Phi_{\sqrt{R^2-r^2}}}  \int_{\Phi_{z,r} }  \left| \nabla v(x,0,R,0,q) - \nabla v(x,z,r,0,q) \right|^2 \, dx \, dz  \right)^{\frac12}
\\ & \leq 
\O_{s} \left( C R^{-\left( 1-  \eta  \right)\frac{d}{2} } r^{- \frac{d}{2}  } \right)\,.
\end{align*}
We write the second term on the left as
\begin{align} \label{e.buildingcaptainphasma}
\lefteqn{
\int_{\Phi_{\sqrt{R^2-r^2}}} \int_{\Phi_{z,r}} \b_R(x) \nabla v(x,z,r,0,q)\,dx\,dz
} \qquad &  \\ 
\notag & = \int_{\Phi_{\sqrt{R^2-r^2}}} \int_{\Phi_{z,r}} \b_R(x) q \,dx\,dz \\ 
\notag & \qquad + \int_{\Phi_{\sqrt{R^2-r^2}}} \b_R(z) \int_{\Phi_{z,r}}  \left( \nabla v(x,z,r,0,q) - q\right)\,dx\,dz \\ 
\notag & \qquad + \int_{\Phi_{\sqrt{R^2-r^2}}} \int_{\Phi_{z,r}} \left(\b_R(x) - \b_R(z)  \right) \left( \nabla v(x,z,r,0,q) - q\right)\,dx\,dz \,.
\end{align}
The first term on the right appears in~\eqref{e.captainphasma}. We then estimate the two last terms on the right. For the third term we use the almost sure bound from~\eqref{e.bounded-u} which says
\begin{equation*} 
\left\| \nabla v(\cdot,z,r,0,q) \right\|_{L^2(\Phi_{z,r})}^2 \leq C |q|^2\,,
\end{equation*}
and obtain by H\"older's inequality that 
\begin{align*} 
\lefteqn{\left| \int_{\Phi_{z,r}} \left( \b_R(x) - \b_R(z) \right) \left( \nabla v(x,z,r,0,q) - q\right) \,dx  \right| } \qquad &
\\ & \leq C \left\| \b_R(\cdot) - \b_R(z) \right\|_{L^2(\Phi_{z,r})} \left\| \nabla v(\cdot,z,r,0,q) - q\right\|_{L^2(\Phi_{z,r})} 
\\ & \leq C |q| \left\| \b_R(\cdot) - \b_R(z) \right\|_{L^2(\Phi_{z,r})}\,.
\end{align*}
Lemma~\ref{l.brcont} yields
\begin{equation*} 
\left| \b_R(x) - \b_R(z) \right| \leq \O_s\left( C|x-y| R^{-(1-\eta_1) \Ll(\frac d 2 + 1\Rr)} +C R^{-\Ll(\frac d2+\ep_1\Rr)} \right)\,.
\end{equation*}
Setting $\ep_2 := \ep_1 + \eta_1 \left(\tfrac{d}{2} +1 \right)$, we have that the first term is dominant whenever $|x-y| \geq CR^{1-\ep_2}$.
Take thus $r:= R^{1-\ep_3}$ with $\ep_3 \in (\ep_2,1)$. The heat kernel bounds and~\eqref{e.brzbound} then imply
\begin{align*} 
 \left\| \b_R(\cdot) - \b_R(z) \right\|_{L^2(\Phi_{z,r})}  &
 \leq \left( \int_{B_{R^{1-\ep_2}}(z)} \Phi_{z,r}(x)\left| \b_R(x) - \b_R(z) \right|^2 \, dx    \right)^{\frac12}  
 \\ & \qquad + \left( \int_{B_{\R^d \setminus R^{1-\ep_2}}(z)} \Phi_{z,r}(x)\left| \b_R(x) - \b_R(z) \right|^2 \, dx    \right)^{\frac12}
 \\& \leq  \O_s\left( CR^{-\left(\frac d2+\ep_1\right)} + C R^{(\ep_3- \ep_2)d} \exp\left(- c R^{2(\ep_3 -\ep_2)} \right)\right) \,.
\end{align*}
In order to bound the second term on the right side of~\eqref{e.buildingcaptainphasma}, we use~\eqref{e.brzbound} and Lemmas~\ref{l.parametermatching} and~\ref{l.identifytheta1} to see that 
\begin{align*}
\lefteqn{
\left|
 \int_{\Phi_{\sqrt{R^2-r^2}}} \b_R(z) \int_{\Phi_{z,r}} \left( \nabla v(x,z,r,0,q)- q\right)\,dx \,dz
 \right|
 } \qquad & \\
 & \leq \int_{\Phi_{\sqrt{R^2-r^2}}} \left| \b_R(z) \right| \left|  \int_{\Phi_{z,r}} \left( \nabla v(x,z,r,0,q)- q\right)\,dx \right| \,dz  \\
 & = \O_s\left( CR^{-\left(1-\eta\right)\frac d2} r^{-\frac d2} \right). 
\end{align*}
Combining the previous displays we arrive at
\begin{multline*} 
\left|  \int_{\Phi_{z,R}} \b_R \nabla v(\cdot,z,R,0,q) - \int_{\Phi_{z,R}} \b_R q \,\right| 
\\ \leq \O_s\left( C \left( R^{- d + \frac d2 (\eta+\ep_3)}   +  R^{-\left(\frac d2+\ep_1\right)} + R^{(\ep_3- \ep_2)d} \exp\left(- c R^{2(\ep_3 -\ep_2)} \right)  \right) \right)
\end{multline*}
Choosing now $\ep_3$ close enough to $\ep_2$ yields~\eqref{e.captainphasma}. 

\smallskip

\emph{Step 3.} The conclusion. According to the definition of $\ahom_r$, the triangle inequality,~\eqref{e.kyloren} and~\eqref{e.captainphasma}, for every $z\in\Rd$, $r\geq 1$ and $q\in B_1$, 
\begin{multline}
\label{e.snoke}
\left|  \int_{\Phi_{z,r}}\left( \a - \ahom_r \right) \nabla v(\cdot,z,r,0,q)\right| 
\\
 = \left| \int_{\Phi_{z,r}} \left( \a - \ahom \right) \nabla v(\cdot,z,r,0,q) -  \int_{\Phi_{z,r}} \b_r \nabla v(\cdot,z,r,0,q)\, \right|  
 =  \O_s \Ll( C r^{- \left( \frac d2 + \ep \right)} \Rr).
\end{multline}
To conclude, we just need to argue that the set $\{ \nabla v(\cdot,z,r,0,q) \,:\, q\in\Rd\}$ is essentially $\A_1$.  
As argued in the proof of~Lemma~\ref{l.fluxmaps}, there exists a random variable $\mathcal{Y}(z)$ satisfying 
\begin{equation}
\label{e.Yz}
\mathcal{Y}(z) = \O_{s\left( \frac d2+\ep \right)}(C)
\end{equation}
such that, for every $r\geq \mathcal{Y}(z)$ and $u\in \A_1(\Phi_{z,r})$, there exists $q \in B_C$ such that 
\begin{equation*}
\nabla v(\cdot,z,r,0,q) = \nabla u. 
\end{equation*}
Thus~\eqref{e.snoke} implies
\begin{equation*}
\sup_{u\in\A_1(\Phi_{z,r})} 
\Ll| \int_{\Phi_{z,r}}\left( \a - \ahom_r \right) \nabla u  \Rr| \indc_{\{ r\geq \mathcal{Y}(z)\}}
= \O_s\left( Cr^{-\left( \frac d2+\eps \right) } \right).
\end{equation*}
By the almost sure boundedness of the left side of~\eqref{e.ahomrz} and~\eqref{e.Yz} yield
\begin{equation*}
\sup_{u\in\A_1(\Phi_{z,r})} 
\Ll| \int_{\Phi_{z,r}}\left( \a - \ahom_r \right) \nabla u  \Rr| \indc_{\{ r\leq \mathcal{Y}(z)\}}
\leq C \indc_{\{ r\leq \mathcal{Y}(z)\}} = \O_s\left( Cr^{-\left( \frac d2+\eps \right) } \right).
\end{equation*}
The previous two displays yield~\eqref{e.ahomrz}. 
\end{proof}

We next give a more precise version of~Lemma~\ref{l.fluxmaps2} by smuggling a polynomial weight into the statement of the previous lemma. We let $ \mathcal{P}_m(\Phi_{z,r})$ denote the set of $\mathbf{p}=(p_1,\ldots,p_d)\in (\mathcal{P}_m)^d$ such that each $p_i$ is a polynomial of degree at most $m$ and $\| p_i \|_{L^2(\Phi_{z,r})} \leq 1$. 

\begin{lemma}
\label{l.gff}
Fix $m\in\N$.
There exists $C(m,s,d,\Lambda)<\infty$ such that, for every $z\in\Rd$ and $r\geq 1$,
\begin{equation} \label{e.gff}
\sup_{u\in \A_1(\Phi_{z,r})} \,
\sup_{\mathbf{p} \in \mathcal{P}_m(\Phi_{z,r})} \,
\bigg|
\int_{\Phi_{z,r}}  \mathbf{p} \cdot\left( \a - \ahom_r \right) \nabla u
\bigg| 
= \O_s\left( Cr^{-\left( \frac d2+\eps \right)} \right). 
\end{equation}
\end{lemma}
\begin{proof}
Fix $z\in\Rd$, $r\geq r_0$ and $\mathbf{p} \in \mathcal{P}_m(\Phi_{z,r})$. According to Lemma~\ref{l.CKish}, there exists $C(m,d)<\infty$ and $\tilde{\mathbf{p}} \in (\mathcal{P}_m)^d$ such that $\| \mathbf{p} \|_{L^2(\Phi_{z,r})} \leq C$ and, for every $F\in (L^2(\Phi_{z,r}))^d$,
\begin{equation*}
\int_{\Phi_{z,r}} \mathbf{p} \cdot F = \int_{\Phi_{z,r/\sqrt{2}}} \tilde{\mathbf{p}}(x) \cdot \int_{\Phi_{x,r/\sqrt{2}}} F(y) \,dy\,dx.
\end{equation*}
Applying this with $F:= (\a-\ahom_r) \nabla u$ and then applying the result of Lemma~\ref{l.ahomrz} and the H\"older inequality, we obtain~\eqref{e.gff}. 
\end{proof}

Using the previous lemma, we give a next-order version of~Lemma~\ref{l.coarsenedequation}. This is now quite close to the statement of convergence in law to the GFF and formalizes the part of the heuristic which asserted that ``$\tilde{\mathbf{b}}_r \nabla \phi_{e,r}$ is of lower order" to pass from~\eqref{e.corrector.approx.pre} to~\eqref{e.corrector.approx}.

\begin{lemma}
\label{l.gffH1}
There exist an integer $k(s,d,\Lambda)\in\N$, exponents $\kappa(s,d,\Lambda)\in (0,1)$ and $\ep(s,d,\Lambda)>0$  and a constant~$C\left(s,d,\Lambda\right)<\infty$ such that, for every function $h \in C^k(\Rd)$ satisfying
\begin{equation*}  
\sup_{x\in\Rd} \sup_{j \in \{ 1,\ldots,k\}} \left( (1+|x|)^{d+j} \left| \nabla^j h(x) \right| \right) \le 1,
\end{equation*}
every $e\in  B_1$ and $r\geq 1$, we have 
\begin{equation*}
r^{-d} \left|
\int_{\Rd} \nabla h\left( \tfrac xr \right)
 \cdot \ahom \nabla \phi_e(x)\,dx 
- \int_{\Rd} \nabla h\left( \tfrac xr \right)
 \cdot \b_{r^\kappa}(x) e\,dx
\right| = \O_1\left( Cr^{-\left(\frac d2+\eps \right)} \right). 
\end{equation*}
\end{lemma}
\begin{proof}
For each $r\geq1$, we denote 
\begin{equation*}
h_r(x):= r^{1-d} h\left( \tfrac xr \right),
\end{equation*}
which is the rescaling of $h$ to length scale $r$, normalized to keep the $L^1$ norm of the gradient unchanged. The argument proceeds by approximating $\nabla h_r$ in mesoscopic regions by the product of a polynomial and $\Phi_r$, which allows us to then recover the conclusion of the lemma from Lemma~\ref{l.gff}. 

\smallskip

\emph{Step 1.} Mesoscopic polynomial approximation. We claim that, for every $m\in\{ 1,\ldots,k-2\}$, there exists $C\left(m,d,\Lambda\right)<\infty$ such that, for every $z\in\Rd$, 
\begin{multline}
\label{e.mesopolyapprox}
\inf_{\mathbf{p} \in (\mathcal{P}_m)^d} \,
\left( \int_{\Phi_{z,r^{\kappa}}} \left| \mathbf{p}(x) - \nabla h_r(x) \right|^2\,dx \right)^{\frac12} \\
\leq C r^{-(d+1)-(1-\kappa) (m+1)/2} \left( 1 + (|z| -r)_+\right)^{-(2+d+m)}.
\end{multline}
Taking $\mathbf{p}_z$ to be the $m$th order Taylor approximation to $\nabla h_r$ centered at the point $z$, we have
\begin{align*}
\sup_{x\in B_{r'}(z)} \left| \mathbf{p}_z(x) - \nabla h_r(x) \right|
& \leq C \left\| \nabla^{m+2} h_r \right\|_{L^\infty(B_{r'}(z))} (r')^{m+1} \\
& = C r^{-d-(m+2)}  \left\| \nabla^{m+2} h \right\|_{L^\infty\left(B_{r'/r}\left(\frac zr\right)\right)} (r')^{m+1}.
\end{align*}
Taking $r':= r^{1-(1-\kappa)/2}$ leads to the bound
\begin{multline*}
\int_{\cu_{r'}(z) } \Phi_{z,r^{\kappa}}(x) \left| \mathbf{p}_z(x) - \nabla h_r(x) \right|^2\,dx \\
\leq C r^{-(d+1) -(1-\kappa)(m+1)/2} \left( 1+(|z|-r')_+ \right)^{-(2+d+m)}.
\end{multline*}
On the other hand, using that $r' = r^{1-(1-\kappa)/2}\gg r^{1-\kappa}$ implies that $\Phi_{z,r^{\kappa}} \leq Cr^{-p}$ on $\Rd \setminus \cu_{r'}(z)$ for any $p$, we find
\begin{equation*}
\int_{\Rd \setminus \cu_{r'}(z) } \Phi_{z,r^{\kappa}}(x) \left| \mathbf{p}_z(x) - \nabla h_r(x) \right|^2\,dx 
\leq 
C r^{-100d} \left( 1 + ( |z| -r )_+ \right)^{-(2+d+m)}.
\end{equation*}
Combining the two previous displays yields~\eqref{e.mesopolyapprox}. In particular, we have 
\begin{equation}
\label{e.pzsize}
\left\|\mathbf{p}_z \right\|_{L^2(\Phi_{z,r^{\kappa}})} \leq Cr^{-d} \left(1+\frac{|z|}{r} \right)^{-(d+1)}.
\end{equation}

\smallskip

\emph{Step 2.} Application of Lemma~\ref{l.gff} and conclusion. 
We compute
\begin{align*}
\lefteqn{ 
\left| 
\int_{\Rd} \nabla h_r
 \cdot\left(  \a-\ahom_{r^\kappa} \right) \left(e+ \nabla \phi_e\right)\,dx 
\right|
} \qquad & \\
& = \left| \int_{\Rd} \int_{\Phi_{y,r^{\kappa}}} \nabla h_r(x)
 \cdot\left(  \a(x)-\ahom_{r^\kappa}(x) \right) \left(e+ \nabla \phi_e(x)\right)\,dx \,dy
\right| \\
& \leq \left| \int_{\Rd} \int_{\Phi_{y,r^{\kappa}}} \mathbf{p}_y(x)
 \cdot\left(  \a(x)-\ahom_{r^\kappa}(x) \right) \left(e+ \nabla \phi_e(x)\right)\,dx \,dy
\right| \\
& \qquad + \int_{\Rd} \int_{\Phi_{y,r^{\kappa}}} \left| \mathbf{p}_y(x) - \nabla h_r(x) \right| \left( 1 + \left| \nabla \phi_e(x)\right| \right) \,dx \,dy.
\end{align*}
By Lemma~\ref{l.gff},~\eqref{e.pzsize} and~Theorem~\ref{t.firstcorrectors}, Remark~\ref{r.multiply} and Lemma~\ref{l.sum-O},
\begin{align*}
\lefteqn{
\left| \int_{\Rd} \int_{\Phi_{y,r^{\kappa}}} \mathbf{p}_y(x)
 \cdot\left(  \a(x)-\ahom_{r^\kappa}(x) \right) \left(e+ \nabla \phi_e(x)\right)\,dx \,dy
\right|  
} \qquad & \\
& \leq \int_{\Rd} \left\| \mathbf{p}_y \right\|_{L^2(\Phi_{y,r^{\kappa}})} \left( 1 +  \left\| \nabla \phi_e \right\|_{L^2(\Phi_{y,r^{\kappa}})} \right)\O_s \left( Cr^{-\kappa\left(\frac d2+\ep \right)} \right) \,dy \\
& \leq Cr^{-d}  \int_{\Rd}\left(1+\frac{|z|}{r} \right)^{-(d+1)}\O_{(2+\ep)(1-\eta)} \left( C \right) \O_s \left( Cr^{-\kappa\left(\frac d2+\ep \right)} \right) \,dy \\
& = \O_1\left( Cr^{-\left(\frac d2+\ep\right)} \right),
\end{align*}
where the last display was obtained by taking $\kappa$ closer to $1$ and redefining $\ep$. Here and below, $\ep$ is assumed to be smaller than the exponent $\ep(d,\Lambda)$ in the statement of~Theorem~\ref{t.firstcorrectors}. We also assume that the exponent $\eta$ is close enough to~$1$ that $(2+\ep)(1-\eta) >2$. A similar computation which uses~\eqref{e.mesopolyapprox} in place of Lemma~\ref{l.gff} yields
\begin{align*}
\lefteqn{
 \int_{\Rd} \int_{\Phi_{y,r^{\kappa}}} \left| \mathbf{p}_y(x) - \nabla h_r(x) \right| \left( 1 + \left| \nabla \phi_e(x)\right| \right) \,dx \,dy
} \qquad & \\
& \leq  \int_{\Rd} \left\| \mathbf{p}_y - \nabla h_r \right\|_{L^2(\Phi_{y,r^{\kappa}})} \left( 1+ \left\| \nabla \phi_e  \right\|_{L^2(\Phi_{y,r^{\kappa}})} \right)\, dy \\
& \leq 
 \int_{\Rd} C r^{-d-(1-\kappa) (m+1)/2} \left( 1 + (|y| -r)_+\right)^{-(1+d+m)}  \O_{\kappa (2+\ep)} \left( C \right)\, dy \\
 & = \O_{\kappa(2+\ep)} \left( C  r^{-(1-\kappa) (m+1)/2}\right).
\end{align*}
Taking $m$ (and therefore $k$) sufficiently large that $(m+1)(1-\kappa) > d+2\ep$ and then combining the previous three displays, we obtain the lemma. 
\end{proof}

We now present the proof of the main result. 

\begin{proof}[{Proof of Theorem~\ref{t.main}}]
We decompose the proof into three steps. 

\smallskip

\emph{Step 1.} We show that the family of random distributions $\{r^{\frac d 2} \,  (\nabla \phi_e)(r \ \cdot \,), r \ge 1\}$ is tight with respect to the topology of $\mathcal C^{-\frac d 2 -}_\mathrm{loc}$. By \cite[Theorem~2.25]{M-tight}, it suffices to verify that, for some large $k\in\N$ and for every $F \in C^k(\Rd;\Rd)$
satisfying
\begin{equation}
\label{e.Ffastdecay}
\forall a\in (1,\infty), \quad \sup_{x\in\Rd} \sup_{1\leq m\leq k} \left( 1 + |x| \right)^a  \left| \nabla^m F(x) \right| < \infty
\end{equation}
and $\zeta \in [1,\infty)$, we have
\begin{equation*}
\sup_{r \ge 1} \, \sup_{\ell \ge 1} \E \Ll[\Ll| \Ll( \ell r \Rr)^{\frac d 2} \int_{\Rd}  F \Ll( {\ell x} \Rr) \cdot \nabla \phi_e(r x)\, dx   \Rr|^\zeta \Rr] < \infty.
\end{equation*}
By a change of variables, this reduces to
\begin{equation}
\label{e.moment-bound}
\sup_{r > 0} \E \Ll[\Ll|  r^{-\frac d 2} \int_{\Rd}  F \Ll( \tfrac{x} r \Rr) \cdot \nabla \phi_e(x)\, dx   \Rr|^\zeta \Rr] < \infty.
\end{equation}
We first show~\eqref{e.moment-bound} under the assumption that~$F$ has zero mean, that is,
\begin{equation} 
\label{e.Fmeanzero}
\int_{\Rd} F(x)\,dx = 0. 
\end{equation}
In this case, we split $F$ by writing 
\begin{equation*} \label{}
F = -\ahom \nabla h + \mathbf g,
\end{equation*}
where $h$ is the unique solution of 
\begin{equation*} \label{}
-\nabla \cdot \ahom \nabla h = \nabla \cdot F \quad \mbox{in} \ \Rd
\end{equation*}
satisfying
\begin{equation*} \label{}
\limsup_{|x|\to \infty} |h(x)| = 0. 
\end{equation*}
The existence of $h$ can be obtained by the Green's formula,  Lemma~\ref{l.mean-zero} below, and an easy density argument, which also gives us the estimate
\begin{equation}  \label{e.decay.h}
\sup_{0 \leq j \leq k } \, \sup_{x \in \Rd} \, (1+|x|)^{d+j+1} \left( \left|\nabla^{j+1} h(x)\right| + \left| \nabla^{j} \mathbf g(x) \right| \right)  < \infty.
\end{equation}
Using Lemma~\ref{l.gffH1} and~\eqref{e.brconvbound}, we obtain, for every $r\geq 1$, 
\begin{equation} 
\label{e.potparth}
r^{-\frac d2} \left| \int_{\Rd} \nabla h\left( \tfrac xr \right) \cdot \ahom \nabla \phi_e(x)\,dx \right| = \O_1\left( C \right). 
\end{equation}
For $r\in (0,1)$, we use the bound, for each $z\in\Rd$,
\begin{equation*}
\fint_{B_r(z)} \left| \nabla \phi_e(x) \right|^2\,dx \leq Cr^{d} \fint_{B_1(z)}   \left| \nabla \phi_e(x) \right|^2\,dx = \O_{2+\ep}\left(Cr^{d} \right),
\end{equation*}
which is a consequence of~\eqref{e.correctorboundpf} and stationarity, to get 
\begin{align*}
r^{-\frac d2} \left| \int_{\Rd} \nabla h\left( \tfrac xr \right) \cdot \ahom \nabla \phi_e(x)\,dx \right| = \O_{2+\ep}(C). 
\end{align*}
Thus we have~\eqref{e.potparth} for every $r>0$. 

\smallskip

Next we turn to the estimate for $\mathbf{g}$. Since $\mathbf{g}(\cdot/r)$ is solenoidal, we have by an integration by parts and~\eqref{e.decay.h} that, for any $R\geq1$,
\begin{equation*} \label{}
\int_{B_R} \mathbf{g} \left(\tfrac xr \right)\cdot \nabla \phi_e(x)\,dx
= \int_{\partial B_R} n(x)\cdot \mathbf{g}\left(\tfrac xr \right) \phi_e(x)\,dx
\leq \left(\frac Rr\right)^{-(d+1)} \int_{\partial B_R} \left| \phi_e(x) \right|\,dx. 
\end{equation*}
Given any $R' \geq1$, we may take $R \in \left[R' , 2R' \right]$ such that 
\begin{equation*} \label{}
 \int_{\partial B_R} \left| \phi_e(x) \right|\,dx  \leq \frac {C}{R'}  \int_{B_{2R'}} \left| \phi_e(x) \right|\,dx= R^{d-1} \O_{2}\left(C\log^{\frac12} R\right),
\end{equation*}
by~Theorem~\ref{t.firstcorrectors}. Combining these gives, for every $R\geq 1$,
\begin{equation*} \label{}
\int_{B_R} \mathbf{g} \left(\tfrac xr \right)\cdot \nabla \phi_e(x)\,dx = \O_2\left( Cr^{d+1} R^{-2} \log^{\frac12} R \right). 
\end{equation*}
Outside of $B_R$, we have 
\begin{equation*} \label{}
\left| \int_{\Rd\setminus B_R} \mathbf{g} \left(\tfrac xr \right)\cdot \nabla \phi_e(x)\,dx \right| \leq Cr^{d+1} \int_{\Rd\setminus B_R} \left| x \right|^{-(d+1)} \O_2\left( C \right) = \O_2(Cr^{d+1}R^{-1}). 
\end{equation*}
Combining these and sending $R\to \infty$ yields, for every $r>0$, 
\begin{equation}
\label{e.boomcha}
\int_{\Rd} \mathbf{g} \left(\tfrac xr \right)\cdot \nabla \phi_e(x)\,dx  = 0 \quad \mbox{$\P$-a.s.}
\end{equation}
Combining this with~\eqref{e.potparth} yields, for every $r>0$ and $F\in C^k(\Rd;\Rd)$ satisfying~\eqref{e.Ffastdecay} and~\eqref{e.Fmeanzero},
\begin{equation*} \label{}
\Ll|  r^{-\frac d 2} \int_{\Rd}  F \Ll( \tfrac{x} r \Rr) \cdot \nabla \phi_e(x)\, dx   \Rr| = \O_1\left( C \right),
\end{equation*}
which is a much stronger bound than we announced in~\eqref{e.moment-bound}. To remove the assumption~\eqref{e.Fmeanzero}, it suffices to exhibit a single function $f\in C^k(\Rd)$ satisfying the decay~\eqref{e.Ffastdecay} for which we can prove~\eqref{e.moment-bound} (because a general $F\in C^k(\Rd;\Rd)$ satisfying the above decay assumption can be written as the sum of a multiple of this function (times each basis vector) and a mean-zero element as above). According to~Theorem~\ref{t.firstcorrectors}, we have~\eqref{e.moment-bound} for the function~$\Phi_1$, which is clearly in the admissible class. This completes the proof of~\eqref{e.moment-bound}. 

\smallskip

\emph{Step 2.} In this step, we show that for every $F \in C^k(\Rd;\Rd)$ satisfying~\eqref{e.Ffastdecay} and $\int_\Rd F = 0$, we have
\begin{equation}  \label{e.conv.F}
r^{-\frac d 2} \int_{\Rd} F \Ll( \tfrac x r \Rr) \cdot \nabla \phi_e(x) \, dx \xrightarrow[r \to \infty]{\mathrm{(law)}} (\nabla \mathbf{\Psi}_e)(F),
\end{equation}
where $\nabla \Psi_e$ denotes the gradient GFF defined by \eqref{e.def.Psi}. 
As above, we decompose $F$ into $F = -\ahom \nabla h + \mathbf g$, where $h$ is the unique function tending to $0$ at infinity and such that $-\nabla \cdot \ahom \nabla h = \nabla \cdot F$. 
By the definition of $\mathbf \Psi_e$, in order to prove \eqref{e.conv.F}, it suffices to show that
\begin{equation}
\label{e.conv.h}
r^{-\frac d 2} \int_{\Rd} \ahom \nabla h \Ll( \tfrac x r \Rr) \cdot \nabla \phi_e(x) \, dx \xrightarrow[r \to \infty]{\mathrm{(law)}} \V(\nabla h,e),
\end{equation}
and
\begin{equation}
\label{e.conv.g}
r^{-\frac d 2} \int_{\Rd} \mathbf g \Ll( \tfrac x r \Rr) \cdot \nabla \phi_e(x) \, dx \xrightarrow[r \to \infty]{\mathrm{(prob.)}} 0.
\end{equation}
The limit \eqref{e.conv.h} follows from Lemmas~\ref{l.gffH1} and \ref{l.brz}, while the limit \eqref{e.conv.h} was already proved (more strongly) in~\eqref{e.boomcha}, above. Thus we have~\eqref{e.conv.F}.

\smallskip

\emph{Step 3.} In this step, we identify the limit of $r^{\frac d 2} \,  (\nabla \phi_e)(r \ \cdot \,)$ as $\nabla \mathbf{\Psi}_e$. More precisely, we show that the convergence \eqref{e.conv.F} holds for test functions $F \in C^k(\Rd;\Rd)$ satisfying~\eqref{e.Ffastdecay} but which are not necessarily of mean zero. For every $m \in [1,\infty)$, we set
\begin{equation*}  
\td F_m(x) := F(x) - \Phi_{m}(x) \int_\Rd F.
\end{equation*}
By definition, the function $\td F_m$ belongs to $C^k(\Rd;\Rd)$, satisfies~\eqref{e.Ffastdecay} and is of mean zero. By the result of the previous step, we have
\begin{equation}
\label{e.conv.td.F}
r^{-\frac d 2} \int_\Rd \td F_m \Ll( \tfrac x r \Rr) \cdot \nabla \phi_e(x) \, dx \xrightarrow[r \to \infty]{\mathrm{(law)}} (\nabla \mathbf{\Psi}_e)\Ll(\td F_m\Rr).
\end{equation}
In order to complete the proof, it suffices to show that for every $\td \eps > 0$, we have
\begin{equation}
\label{e.m1}
\lim_{m \to \infty} \limsup_{r \ge 1} \P \Ll[\Ll|r^{-\frac d 2}  \int_\Rd(F - \td F_m) \Ll( \tfrac x r \Rr) \cdot \nabla \phi_e(x) \, dx\Rr| \ge \td \eps\Rr] = 0,
\end{equation}
and 
\begin{equation}
\label{e.m2}
(\nabla \mathbf{\Psi}_e)\Ll(F - \td F_m\Rr) \xrightarrow[m \to \infty]{\mathrm{(prob.)}} 0,
\end{equation}
see e.g.~\cite[Theorem~3.2]{bil}. According to~Theorem~\ref{t.firstcorrectors}, we have
\begin{align*}  
r^{-\frac d 2}  \int_\Rd(F - \td F_m) \Ll( \tfrac x r \Rr) \cdot \nabla \phi_e(x) \, dx & =  r^{\frac d 2} \Ll(\int_\Rd F \Rr) \cdot \Ll( \int_{\Phi_{mr}} \nabla \phi_e\Rr) \\
& = \Ll(\int_\Rd |F| \Rr) \O_s\Ll(C m^{-\frac d 2}\Rr).
\end{align*}
This implies \eqref{e.m1}. The convergence in \eqref{e.m2} follows from the observation that
\begin{equation*}  
\lim_{m \to \infty} \Ll\| F - \td F_m \Rr\|_{L^2(\Rd)} = 0. 
\end{equation*}
The proof is now complete. 
\end{proof}

In the argument above, we invoked the following simple lemma, the proof of which we recall here for completeness. Denote by $\mcl S$ the Schwartz class of test functions:
\begin{multline}
\label{e.def.S}
\mcl S := \big\{f \in C^\infty(\Rd;\R) \ : \ \text{for every } k \in \N \text{ and } i \in \N^d \\ \sup_{x \in \Rd} (1+|x|)^k |\partial^i f(x)| < \infty \big\}.
\end{multline}

\begin{lemma}
\label{l.mean-zero}
For every $f \in \mcl S$ of mean zero, there exists $F \in \mcl S^d$ such that $\nabla \cdot F = f$. If $f \in C^\infty_c(\Rd;\R)$, then we can further require that $F \in C^\infty_c(\Rd;\Rd)$.
\end{lemma}
\begin{proof}
We first prove the statement assuming $f \in C^\infty_c(\Rd;\R)$. Without loss of generality, we assume that $\supp f \subset B_1$.  
Let $h \in C^\infty_c(\R;\R)$ be such that $\int_\R h = 1$ and $\supp h \subset B_1$. We define the following functions from $\Rd$ to $\R$:
\begin{align*}  
g_0(x) & := f(x), \\
g_1(x) & := h(x_1)  \int_{\R}f(y_1,x_2,\ldots,x_d) \, dy_1,  \\
g_2(x) & := h(x_1) h(x_2) \int_{\R^2} f(y_1,y_2,x_3,\ldots,x_d) \, dy_1 \, dy_2, \\
& \vdots \\
g_d(x) & := h(x_1) \, \cdots \, h(x_d) \int_{\R^d} f = 0,
\end{align*}
where we write $x = (x_1,\ldots,x_d)$. Note that $g_0, \ldots,g_d \in C^\infty_c(\Rd;\R)$. For every $i \in \{1, \ldots, d\}$, we set
\begin{align*}  
F_i(x) := \int_{-\infty}^{x_i} (g_{i-1} - g_i)(x_1,\ldots,x_{i-1},y,x_{i+1}, \ldots,x_d)\, dy.
\end{align*}
Note that
\begin{equation*}  
\partial_i F_i(x) = g_{i-1}(x) - g_i(x), 
\end{equation*}
so that $\nabla \cdot F = f$. In order to complete the proof, there remains to verify that $F$ is compactly supported. Since $g_0,\ldots,g_d$ are compactly supported, it suffices to check that 
\begin{equation*}  
x_i > 1 \quad \implies \quad F(x) = 0.
\end{equation*}
If $x_i > 1$, then
\begin{align*}  
& \int_{-\infty}^{x_i} g_{i-1}(x_1,\ldots,x_{i-1},y,x_{i+1}, \ldots,x_d)\, dy \\
& \qquad = h(x_1) \, \cdots \, h(x_{i-1}) \int_{\R^i} f(y_1,\ldots,y_i,x_{i+1},\ldots,x_d) \, dy_1 \, \cdots \, dy_i \\
& \qquad = \int_{-\infty}^{x_i} g_{i}(x_1,\ldots,x_{i-1},y,x_{i+1}, \ldots,x_d)\, dy ,
\end{align*}
since $\int_\R h = 1$. This completes the proof under the assumption that $f \in C^\infty_c(\Rd;\R)$. When $f \in \mcl S$, a similar argument shows that $F \in \mcl S^d$. 
\end{proof}

\noindent \textbf{Acknowledgments.} The second author was supported by the Academy of Finland project \#258000. We thank Antti Hannukainen (Aalto University) for performing the numerical computations of the corrector and producing Figure~\ref{fig}.

\small
\bibliographystyle{plain}
\bibliography{additivestructure}

\end{document}